\documentclass[12 pt]{book}
\usepackage[a4paper, total={6.5in, 9.5in}]{geometry}
\usepackage[T1]{fontenc}
\usepackage{imakeidx}
\makeindex
\usepackage[utf8]{inputenc}
\usepackage[english]{babel}
\usepackage{graphicx}

\usepackage{amsmath}
\usepackage{amsfonts}
\usepackage{amssymb}
\usepackage{amsthm}
\usepackage[alphabetic]{amsrefs}
\usepackage[titletoc]{appendix}

\allowdisplaybreaks

\usepackage{mathtools}

\usepackage[colorlinks=true,linkcolor=blue,urlcolor=blue,citecolor=blue]{hyperref}

\usepackage{wedn}
\usepackage[T1]{fontenc}
\newcommand{\XA}{{\mbox{\large{\wedn{a}}}}}
\newcommand{\XB}{{\mbox{\large{\wedn{b}}}}}
\newcommand{\XC}{{\mbox{\large{\wedn{c}}}}\,}
\newcommand{\XX}{\underline{\mbox{\large{\wedn{x}}}}\,}
\newcommand{\XY}{\underline{\mbox{\large{\wedn{y}}}}\,}
\newcommand{\XT}{\underline{\mbox{\large{\wedn{t}}}}\,}

\title{Local Density of Solutions to Fractional Equations\thanks{Supported by
the Australian Research Council Discovery Project 170104880 NEW ``Nonlocal
Equations at Work'',
the DECRA Project
DE180100957 ``PDEs, free boundaries
and applications'' and the Fulbright Foundation.
The authors are members of INdAM/GNAMPA.}}
\author{Alessandro Carbotti\thanks{Dipartimento di Matematica
e Fisica, Universit\`a del Salento,
Via Per Arnesano, 73100 Lecce, Italy. {\tt alessandro.carbotti@unisalento.it}}, Serena Dipierro\thanks{Department
of Mathematics and Statistics,
University of Western Australia,
35 Stirling Highway,
Crawley WA 6009, Australia. {\tt serena.dipierro@uwa.edu.au} },
and
Enrico Valdinoci\thanks{Department of Mathematics and Statistics,
University of Western Australia,
35 Stirling Highway,
Crawley WA 6009, Australia. {\tt enrico.valdinoci@uwa.edu.au} }}

\newtheorem{theorem}{Theorem}[chapter]
\newtheorem{remark}[theorem]{Remark}
\newtheorem{lemma}[theorem]{Lemma}
\newtheorem{proposition}[theorem]{Proposition}

\newtheorem{example}[theorem]{Example}
\newtheorem{corollary}[theorem]{Corollary}
\numberwithin{equation}{chapter}

\begin{document}
\maketitle\tableofcontents

\chapter*{Preface}

The study
of nonlocal operators of fractional type possesses a long tradition, 
motivated both by mathematical curiosity and by real world applications.
Though this line of research presents some similarities and analogies with
the study of operators of integer order, it also presents a number
of remarkable differences, one of the greatest being the recently discovered
phenomenon that {\em
all functions are (locally) fractionally
harmonic (up to a small error)}. This feature is quite
surprising, since it is in sharp contrast with the case of classical
harmonic functions, and it reveals a genuinely nonlocal peculiarity.
\medskip

More precisely, it has been proved in~\cite{MR3626547} that
given any $C^k$-function~$f$ in a bounded domain~$\Omega$
and given any~$\epsilon>0$, there exists a function~$f_\epsilon$ which
is fractionally harmonic in~$\Omega$ and such that the $C^k$-distance in~$\Omega$
between~$f$ and~$f_\epsilon$ is less than~$\epsilon$.

\begin{figure}[h]
\centering
\includegraphics[width=9.5 cm]{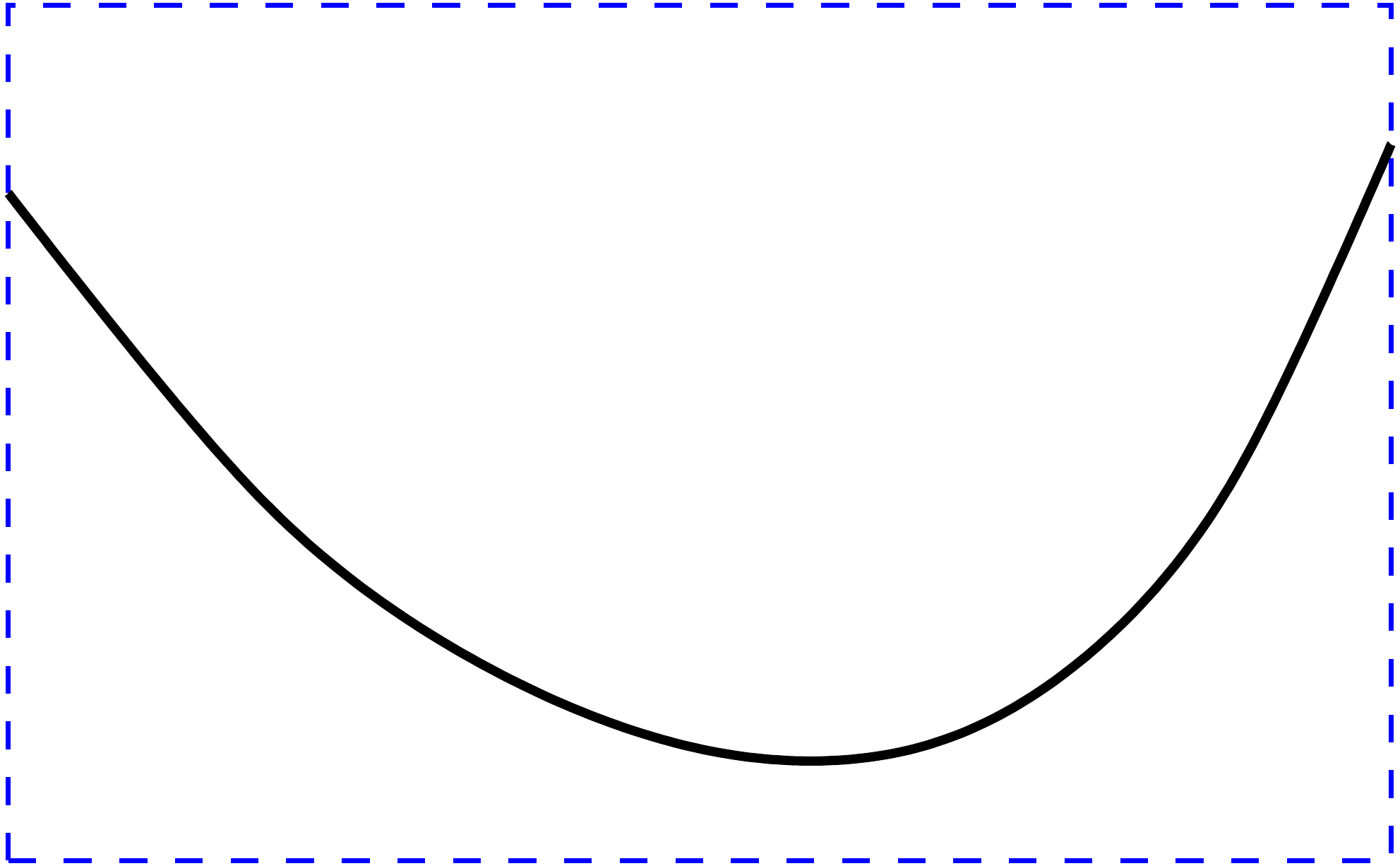}
\caption{\footnotesize\it All functions are fractional harmonic, at different scales (scale of the original function).}
\label{FIGSCAL}
\end{figure}

Interestingly, this kind of results can be also applied at any scale,
as shown in Figures~\ref{FIGSCAL}, \ref{FIGSCAL2} and~\ref{FIGSCAL3}.
Roughly speaking, given {\em any} function, without any special geometric
prescription, in a given bounded domain (as in Figure~\ref{FIGSCAL}), one can ``complete''
the function outside the domain in such a way that the resulting object
is fractionally harmonic. That is, one can endow the function given in the bounded domain
with a number of suitable oscillations outside the domain in order
to make an integro-differential
operator of fractional type vanish. This idea is depicted in Figure~\ref{FIGSCAL2}.
As a matter of fact, Figure~\ref{FIGSCAL2} must be considered just a ``qualitative''
picture of this method, and does not have any demand of being ``realistic''.
On the other hand, even if Figure~\ref{FIGSCAL2} did not provide a correct
fractional harmonic extension of the given function outside the given domain,
the result can be repeated at a larger scale, as in Figure~\ref{FIGSCAL3},
adding further remote oscillations in order to obtain a fractional harmonic function.
\begin{figure}[h]
\centering
\includegraphics[width=9.5 cm]{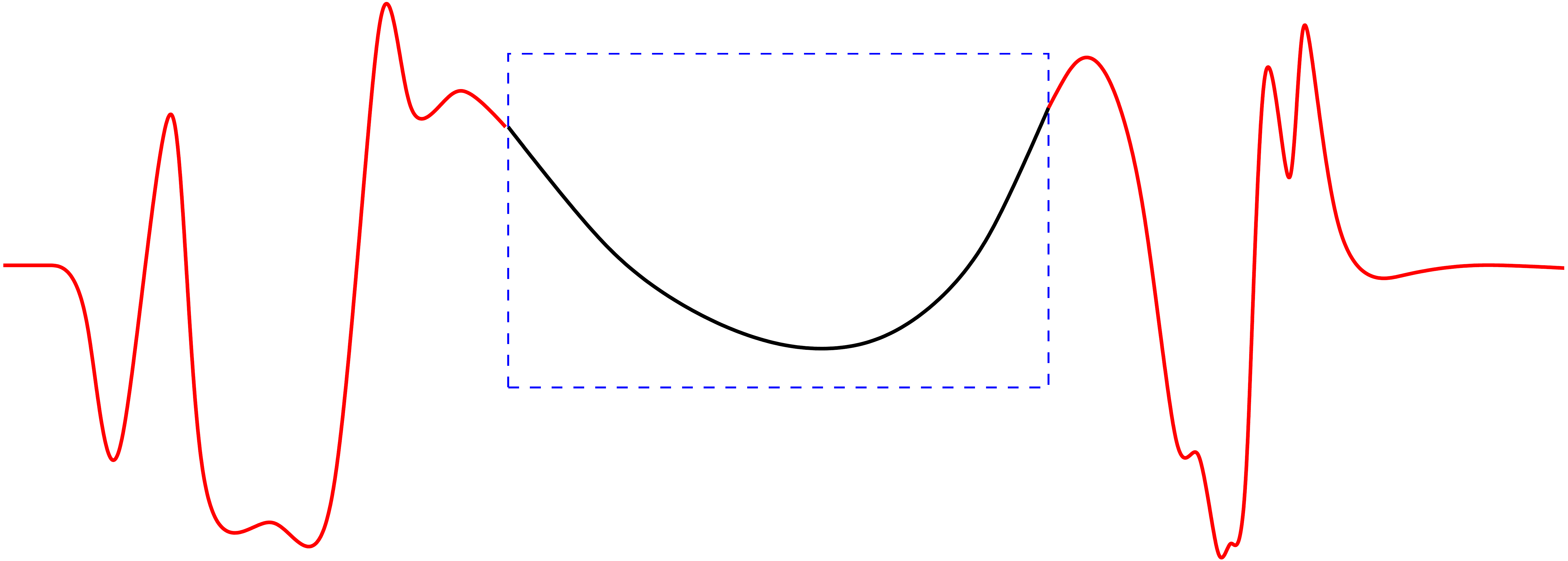}
\caption{\footnotesize\it All functions are fractional harmonic, at different scales
(``first''
scale of exterior oscillations).}
\label{FIGSCAL2}
\end{figure}

In this sense, this type of results really says that whatever graph we draw on a sheet of paper,
it is fractionally harmonic (more rigorously, it can be shadowed 
with an arbitrary precision by another graph, which can be appropriately continued
outside the sheet of paper in a way which makes it fractionally harmonic).\medskip

\begin{figure}[h]
\centering
\includegraphics[width=9.5 cm]{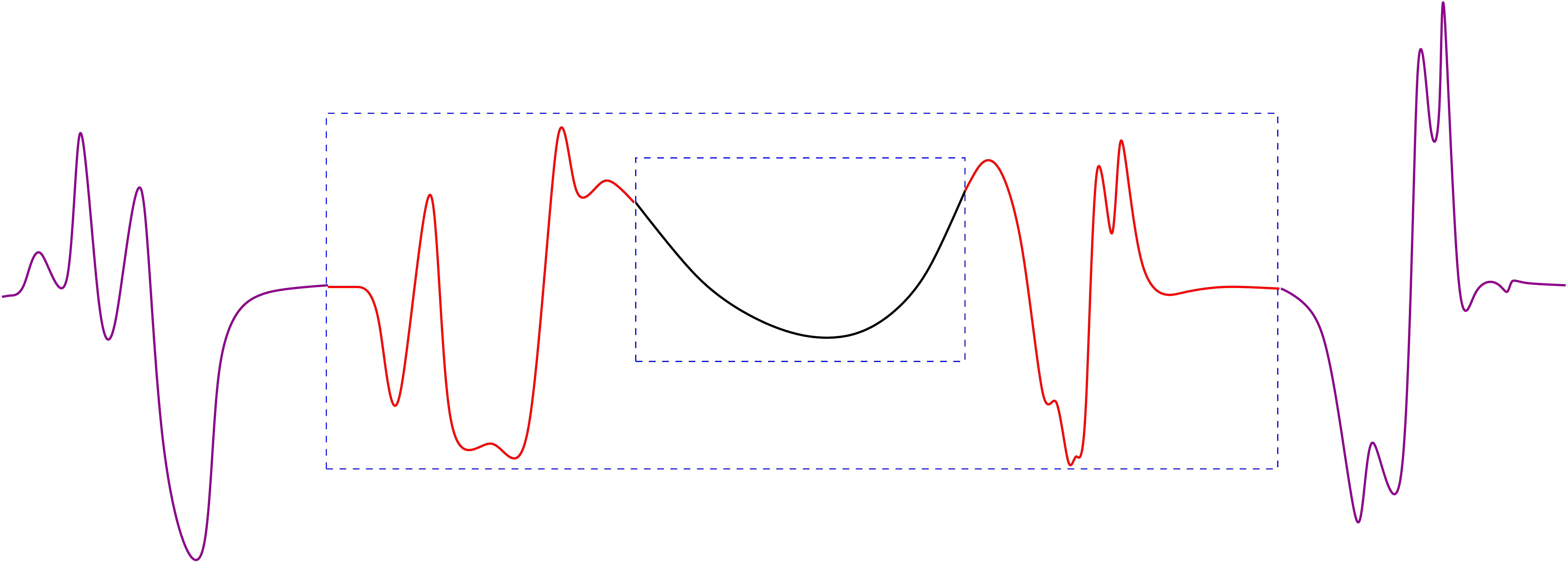}
\caption{\footnotesize\it All functions are fractional harmonic, at different scales
(``second''
scale of exterior oscillations).}
\label{FIGSCAL3}
\end{figure}

This book contains a {\em new result}
in this line of investigation, stating that {\em every
function lies in the kernel of every linear equation involving some fractional operator,
up to a small error}.
That is,
{\em any given function can be smoothly approximated by
functions lying in the kernel of a linear operator
involving at least one fractional component}.
The setting in which this result holds true is very general, since it takes into account
anomalous diffusion, with possible
fractional components
in both space and time. The operators taken into account comprise
the case of
the sum of classical and fractional Laplacians, possibly of different
orders, in the space variables, and classical or fractional derivatives
in the time variables.
Namely, the equation can be of any order, it does not need any
structure (it needs no ellipticity or parabolicity conditions), and the fractional
behaviour is either in time or space, or in both.\medskip

In a sense, this type of approximation results
reveals the true power of fractional equations, independently on the structural
``details'' of the single equation under consideration,
and
shows that {\em space-fractional and time-fractional
equations exhibit a variety of solutions which is much richer and more
abundant than in the case of classical diffusion}.
\medskip

Though space- and time-fractional diffusions can be seen
as related aspects of nonlocal phenomena, they arise in different
contexts and present important structural differences.
The paradigmatic example of space-fractional diffusion is embodied by the fractional Laplacian,
that is a fractional root of the classical Laplace operator. This
setting often surfaces from stochastic processes presenting jumps
and it exhibits the classical spatial symmetries such as invariance under
translations and rotations, plus a scale invariance of the integral
kernel defining the operator. Differently from this,
time-fractional diffusion is typically related to memory effects,
therefore it distinguishes very strongly between the ``past''
and the ``future'', and the arrow of time plays a major role (in particular,
since the past influences the future, but not viceversa, time-fractional
diffusion does not possess the same type of symmetries of the space-fractional one).
In these pages, we will be able to consider operators which arise
as superpositions of both space- and time-fractional diffusion, possibly taking
into account classical derivatives as well (the cases of diffusion which
is fractional just in either space or time are comprised as special situation
of our general framework). Interestingly, we will also consider fractional
operators of any order, showing, in a sense, that some properties
related to fractional diffusion persist also when higher order operators
come into play,
differently from what happens in the classical case, in which the theory
available for the Laplacian operator presents significant differences with respect
to the case of polyharmonic operators.\medskip

To achieve the original result presented here, we
develop a broad theory of some fundamental facts about space- and time-fractional
equations.
Some of these additional results were known in the literature,
at least in some particular cases, but some other are new and interesting
in themselves, and, in developing these
auxiliary theories,
this monograph presents a completely self-contained
approach to a number of basic questions, such as:
\begin{itemize}
\item Boundary behaviour for the time-fractional eigenfunctions,
\item Boundary behaviour for the time-fractional harmonic functions, 
\item Green representation formulas,
\item Existence and regularity for the first eigenfunction of the (possibly higher order)
fractional Laplacian,
\item Boundary asymptotics of the first eigenfunctions of the (possibly higher order)
fractional Laplacian,
\item Boundary behaviour of (possibly higher order) fractional harmonic functions.
\end{itemize}
We now dive into the technical details of this matter.

\chapter{Introduction: why fractional derivatives?}\label{WHY}

The goal of this introductory chapter is to provide a series
of simple examples in which fractional diffusion and fractional derivatives
naturally arise and give a glimpse on how
to use analytical methods to attack simple problems
arising in concrete situations. Some of the examples that we present are original,
some are modifications of known ones, all will be treated
in a {\em fully elementary} way that requires basically no
main prerequisites.
Other very interesting
motivations can be already found in the specialized literature, see e.g.~\cite{comb, MR1658022, MR2218073, MR2584076, MR2639369, MR2676137, MR3089369, claudia, MR3557159, 2017arXiv171203347G}
(also, in Chapter~\ref{DUEE} we will recall some other, somehow more advanced,
applications).

Some disclaimers here are mandatory. First of all, the motivations
that we provide do not aim at being fully exhaustive, since
the number of possible applications of fractional calculus are so abundant
that it is virtually impossible to discuss them all in one shot. Moreover
(differently from the rest of this monograph) while providing these motivations
we do not aim at the maximal mathematical rigor (e.g. all functions
will be implicitly assumed to be smooth and suitably decay at infinity, limits will freely taken and interchanged, etc.),
but rather at showing natural contests in which fractional objects appear
in an almost unavoidable way.

\begin{figure}[h]
\centering
\includegraphics[width=9.5 cm]{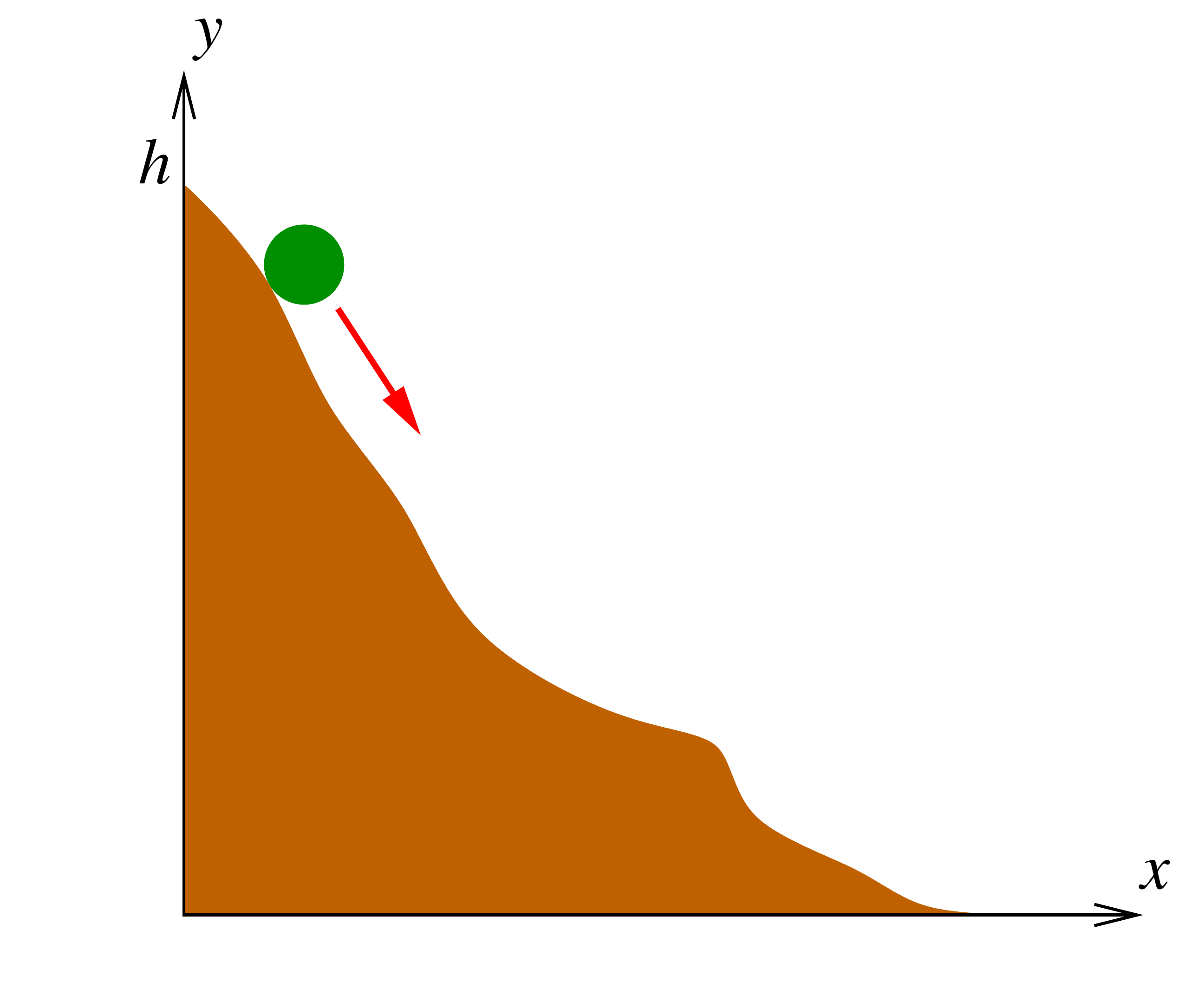}
\caption{\footnotesize\it A material point sliding down.}
\label{FIGSLi}
\end{figure}

\begin{example}[Sliding time and tautochrone problem\index{problem!tautochrone}]\label{TAUr}
{\rm Let us consider a point mass subject to gravity, sliding down on a curve
without friction. We suppose that the particle starts its motion
with zero velocity at height~$h$ and it reaches its minimal position
at height~$0$ at time~$T(h)$. Our objective is to describe~$T(h)$
in terms of the shape of the slide. To this end, see Figure~\ref{FIGSLi},
we follow an approach introduced by N. H. Abel (see pages~11-27 
in~\cite{MR1191901})
and use coordinates $(x,y)\in{\mathbb{R}}^2$ to describe the slide as a function
in the vertical variable, namely~$x=f(y)$. It is also convenient
to introduce the arclength parameter
\begin{equation}\label{VEL0}\phi(y):=\sqrt{|f'(y)|^2+1}\end{equation}
and to describe the position of the particle by the notation~$p(t):=(f(y(t)),y(t))$.
The velocity of the particle is therefore
\begin{equation}\label{VEL} v(t):=\dot p(t)=(f'(y(t),1)\,\dot y(t).\end{equation}
By construction, we know that~$y(0)=h$ and~$y(T(h))=0$, and
moreover~$v(0)=0$. Accordingly, by the Energy Conservation Principle,
for all~$t\in[0,T(h)]$,
$$ \frac{m |v(t)|^2}{2}+mg\,y(t)=\frac{m |v(0)|^2}{2}+mg\,y(0)=mgh,$$
where~$m$ is the mass of the particle and~$g$ is the gravity acceleration.
As a consequence, simplifying~$m$ and recalling~\eqref{VEL0} and~\eqref{VEL}
(and that the particle is sliding downwards),
$$ -\phi(y(t))\,\dot y(t)=\sqrt{|f'(y(t))|^2+1}\,|\dot y(t)|=|v(t)|=\sqrt{2g(h-y(t))}.
$$
Hence, separating the variables,
$$ T(h)=-\int_0^{T(h)}\frac{\phi(y(t))\,\dot y(t)}{\sqrt{2g(h-y(t))}}\,dt=
-\int_h^{0}\frac{\phi(y)}{\sqrt{2g(h-y)}}\,dy,
$$
that is
\begin{equation}\label{8sEQ}
T(h)=\int^h_{0}\frac{\phi(y)}{\sqrt{2g(h-y)}}\,dy.
\end{equation}
We observe that~$\phi(y)\ge1$, thanks to~\eqref{VEL0}, therefore~\eqref{8sEQ}
gives that
\begin{equation}\label{Th1} T(h)\ge\sqrt{\frac{2h}{g}},\end{equation}
which corresponds to the free fall, in which~$f$ is constant and the particle
drops down vertically.

Interestingly, the relation in~\eqref{8sEQ} can be seen as a fractional equation.
For instance, if we exploit the Caputo notation of fractional derivative
of order~$1/2$, as it will be discussed in detail in the forthcoming
formula~\eqref{defcap}, one can write~\eqref{8sEQ} in the fractional form
\begin{equation}\label{FRAcha} T(h) = \sqrt{\frac{\pi}{2g}}\,D_{h,0}^{1/2}\Phi(h),
\end{equation}
where~$\Phi$ is a primitive of~$\phi$, say
$$ \Phi(H):=\int_0^H\phi(y)\,dy.$$
It is instructive to solve the relation~\eqref{8sEQ} by obtaining explicitly~$\phi$
in terms of~$T(h)$. Of course, fractional calculus, operator algebra
and the theory of Volterra-type integral equations provide general
tools to deal with equations such as the one in~\eqref{FRAcha}, but for the scopes
of these pages, we try to perform our analysis using only elementary
computations. To this end, it is convenient to take advantage of
the natural scaling of the problem and convolve~\eqref{8sEQ}
against the kernel~$\frac{1}{\sqrt{h}}$. In this way, we obtain that
\begin{equation}\label{BRAbxdis}
\begin{split}
\int_0^H\frac{T(h)}{\sqrt{H-h}}\,dh\,&=\int_0^H\left[
\int^h_{0}\frac{\phi(y)}{\sqrt{2g(h-y)}}\,dy\right]\frac{dh}{\sqrt{H-h}}\\
&=\frac{1}{\sqrt{2g}}\int_0^H\phi(y)\,\left[
\int_y^H\frac{dh}{\sqrt{(h-y)(H-h)}}\right]\,dy.
\end{split}
\end{equation}
Using the change of variable~$\eta:=\frac{h-y}{H-y}$ we see that
$$ \int_y^H\frac{dh}{\sqrt{(h-y)(H-h)}}=\int_0^1\frac{d\eta}{\sqrt{\eta(1-\eta)}}=\pi,
$$
and hence~\eqref{BRAbxdis} becomes
\begin{equation}\label{BRAbxdis:2}
\int_0^H\frac{T(h)}{\sqrt{H-h}}\,dh=\frac{\pi}{\sqrt{2g}}\int_0^H\phi(y)\,dy=
\frac{\pi}{\sqrt{2g}}\Phi(H).\end{equation}
The main application of this formula consists in a quick solution of
the {\em tautochrone problem}, that is the determination
of the shape of the slide for which the sliding time~$T(h)$ is constant
(and therefore independent on the initial height~$h$). In this case,
we set~$T(h)=T$ and then~\eqref{BRAbxdis:2} gives that
$$ 2T\sqrt{H}=\frac{\pi}{\sqrt{2g}}\Phi(H),$$
and thus, differentiating in~$H$,
$$ \frac{T}{\sqrt{H}}=\frac{\pi}{\sqrt{2g}}\phi(H),$$
which, recalling \eqref{VEL0}, leads to
\begin{equation}\label{Cyah892}
|f'(y)|^2={\phi^2(y)-1}={\frac{2gT^2}{\pi^2\,y}-1}={\frac{2r-y}{y}},
\end{equation}
with~$r:=\frac{gT^2}{\pi^2}$, which is the equation of the {\em cycloid}\index{cycloid}.

Another interesting application of~\eqref{BRAbxdis:2}
surfaces when the sliding time~$T(h)$ behaves like a square root, say
\begin{equation}\label{Th2} T(h)=\kappa\sqrt{h},\end{equation}
for some~$\kappa>0$. In this case, formula~\eqref{BRAbxdis:2}
gives that
$$ \frac{\kappa\pi H}2=\kappa\int_0^H\sqrt{\frac{h}{H-h}}\,dh=\frac{\pi}{\sqrt{2g}}\Phi(H),$$
that says that
$$ \phi(H)=\Phi'(H)=\sqrt{\frac{g}{2}}\,\kappa,$$
and then, by~\eqref{VEL0},
$$ |f'(y)|^2=\phi^2(y)-1=\frac{g}{2}\,\kappa^2-1,$$
which is constant (and well-posed when~$\kappa\ge\sqrt{\frac2g}$,
coherently with~\eqref{Th1}). In this case~$f$ is linear
and therefore the slide corresponds to the classical 
{\em inclined plane}.
}\end{example}

\begin{example}[Sliding time and brachistochrone problem\index{problem!brachistochrone}]\label{TAUr2} {\rm Formula~\eqref{8sEQ}, as
well as its explicit fractional formulation in~\eqref{FRAcha}, is also useful to
solve the {\em brachistochrone problem}, that is detecting the curve
of fastest descent between two given points. In this setting,
the mathematical formulation of Example~\ref{TAUr} can be modified as follows.
The initial height is fixed, hence we can assume that~$h>0$ is a given parameter.
Instead, we can optimize the shape of the slide as given by the function~$f$.
To this end, it is convenient to write~$\phi=\phi_f$ in~\eqref{VEL0}, in order
to stress its dependence on~$f$. Similarly, the fall time in~\eqref{8sEQ}
and~\eqref{FRAcha} can be denoted by~$T(f)$ to emphasize its dependence on~$f$.
In this way, we write~\eqref{8sEQ} as
\begin{equation}\label{LIgh1}
T(f)=\int^h_{0}\frac{\phi_f(y)}{\sqrt{2g(h-y)}}\,dy.
\end{equation}
Now, given~$\epsilon\in(-1,1)$, we consider a perturbation~$\eta\in C^\infty_0([0,h])$
of an optimal function~$f$.
That is, given~$\epsilon\in(0,1)$, we define
$$ f_\epsilon(y):=f(y)+\epsilon \eta(y).$$
Since~$f_\epsilon(0)=f(0)$ and~$f_\epsilon(h)=f(h)$, we have that the endpoints
of the slide described by~$f_\epsilon$ are the same as the ones
described by~$f$ and therefore the minimality of~$f$ gives that
\begin{equation}\label{LIgh2}
T(f_\epsilon)=T(f)+o(\epsilon).
\end{equation}
In addition, by~\eqref{VEL0},
\begin{eqnarray*}&& \phi_{f_\epsilon}(y)=\sqrt{|f'_\epsilon(y)|^2+1}=
\sqrt{|f'(y)+\epsilon \eta'(y)|^2+1}\\&&\qquad=
\sqrt{|f'(y)|^2+2\epsilon f'(y)\eta'(y)+o(\epsilon)+1}\\&&\qquad=
\sqrt{|f'(y)|^2+1}+\frac{\epsilon f'(y)\eta'(y)}{\sqrt{|f'(y)|^2+1}}+o(\epsilon)\\
&&\qquad=\phi_f(y)+\frac{\epsilon f'(y)\eta'(y)}{\sqrt{|f'(y)|^2+1}}+o(\epsilon).
\end{eqnarray*}
Therefore, in light of~\eqref{LIgh1} and~\eqref{LIgh2},
\begin{eqnarray*}
o(\epsilon)&=&
T(f_\epsilon)-T(f)\\
&=& \int^h_{0}\frac{\phi_{f_\epsilon}(y)-\phi_f(y)}{\sqrt{2g(h-y)}}\,dy\\
&=& \epsilon\int^h_{0}\frac{ f'(y)\eta'(y)}{\sqrt{2g(h-y)(|f'(y)|^2+1)}}\,dy+o(\epsilon),
\end{eqnarray*}
and consequently
$$ \int^h_{0}\frac{ f'(y)\eta'(y)}{\sqrt{2g(h-y)(|f'(y)|^2+1)}}\,dy=0.$$
Accordingly, since~$\eta$ is an arbitrary compactly supported perturbation,
we obtain the optimality condition
\begin{equation*}
\frac{d}{dy}\left( \frac{ f'(y)}{\sqrt{2g(h-y)(|f'(y)|^2+1)}}\right)=0
\end{equation*}
and then
\begin{equation*}
\frac{ f'(y)}{\sqrt{2g(h-y)(|f'(y)|^2+1)}}=-c,
\end{equation*}
for some~$c>0$.

This gives that
\begin{equation}\label{fpri} |f'(y)|^2=\frac{2c^2g(h-y)}{1-2c^2g(h-y)}.\end{equation}
It is now expedient to consider a suitable translation of the slide, by considering
the rigid motion described by the function
$$ \zeta(y):=\frac{y-1+2c^2gh}{2c^2 g}$$
and we define
$$ \tilde f(y):=2c^2 g f(\zeta(y)).$$
We point out that~$\zeta'(y)=\frac1{2c^2g}$ and
$$ \frac{2c^2g(h-\zeta(y))}{1-2c^2g(h-\zeta(y))}=
\frac{2c^2g h-2c^2g\zeta(y)}{1-2c^2gh+2c^2g\zeta(y)}=
\frac{2c^2g h-(y-1+2c^2gh)}{1-2c^2gh+(y-1+2c^2gh)}=\frac{1-y}{y}.$$
Consequently, by~\eqref{fpri},
$$ |\tilde f'(y)|^2= (2c^2 g)^2 (\zeta'(y))^2 |f'(\zeta(y))|^2
=\frac{2c^2g(h-\zeta(y))}{1-2c^2g(h-\zeta(y))}=\frac{1-y}{y},$$
which is again an equation describing a {\em cycloid} (compare with~\eqref{Cyah892}).}\end{example}

Some additional comments about Examples~\ref{TAUr} and~\ref{TAUr2}.
The tautochrone problem was first solved by
Christiaan Huygens 
in his book {\em
Horologium Oscillatorium: sive de motu pendulorum ad horologia aptato demonstrationes
geometricae} in 1673.
Interestingly, the tautochrone problem
is also alluded in a famous passage from the 1851 novel {\em Moby Dick} by Herman Melville.

As for the brachistochrone problem,
Galileo was probably the first to
take it into account in~1638
in Theorem~XXII and Proposition~XXXVI of his book {\em Discorsi e Dimostrazioni Matematiche
intorno a due nuove scienze}, in which he seemed to
argue that the fastest motion from one end to the other does not take
place along the shortest line but along a circular arc
(now we know that ``circular arc'' should have been replaced
by ``cycloid'' to make the statement fully correct, and in fact the name
of cycloid is likely to go back to Galileo in its meaning of ``resembling a circle'').

The brachistochrone problem was then
posed explicitly in~1696
by Johann Bernoulli
in {\em Acta Eruditorum}, Although probably
knowing how to solve it himself, Johann Bernoulli
challenged all others to solve it, addressing his community with a rather bombastic
dare, such as {\em I, Johann Bernoulli, address the most brilliant
mathematicians in the world. Nothing is more attractive to intelligent
people than an honest, challenging problem, whose possible solution will
bestow fame and remain as a lasting monument.}

The coincidence of the solutions of the tautochrone and the brachistochrone
problems appears to be a surprising mathematical fact and to reveal
deep laws of physics: in the words of
Johann Bernoulli, {\em
Nature always tends to act in the simplest way, and so it here lets one curve
serve two different functions}.

For a detailed historical introduction to these problems, see e.g.~\cite{MR1281370}
and the references therein.

\begin{example}[Thermal insulation problems, Dirichlet-to-Neumann problems\index{Dirichlet-to-Neumann},
and the fractional Laplacian]\label{74747TGSCO}
{\rm Let us consider a room whose perimeter is given by walls of two different types.
One kind of walls is such that we can fix the temperature there,
i.e. by some heaters or coolers endowed with a thermostat.
The other kind of walls is made by insulating material which prevents the
heat flow to go through. The question that we want to address is:
\begin{equation}\label{QUE1}
{\mbox{{\em what is
the temperature at the insulating kind of walls?}}}
\end{equation}
We will see that the question in~\eqref{QUE1} can be conveniently set into a natural
fractional Laplacian framework. As a matter of fact,
To formalize this question, and address it at least in its simplest
possible formulation, let us consider the case in which the room
is modeled by the half-space~${\mathbb{R}}^{n+1}_+:={\mathbb{R}}^n\times(0,+\infty)$
(while the rooms in the real life are considered to be three-dimensional,
hence $n$ would be equal to~$2$ in this model,
we can also take into account the case of a general~$n$ in this discussion).
The walls of this room are given by~${\mathbb{R}}^n\times\{0\}$.
We suppose that the insulating material is placed in a nice bounded
domain~$\Omega\subset{\mathbb{R}}^n\times\{0\}$
and the temperature is prescribed at the remaining part of the walls~$({\mathbb{R}}^n\times\{0\}
)\setminus\Omega$, see Figure~\ref{2234CxC6783C}.

\begin{figure}[h]
\centering
\includegraphics[width=8.5 cm]{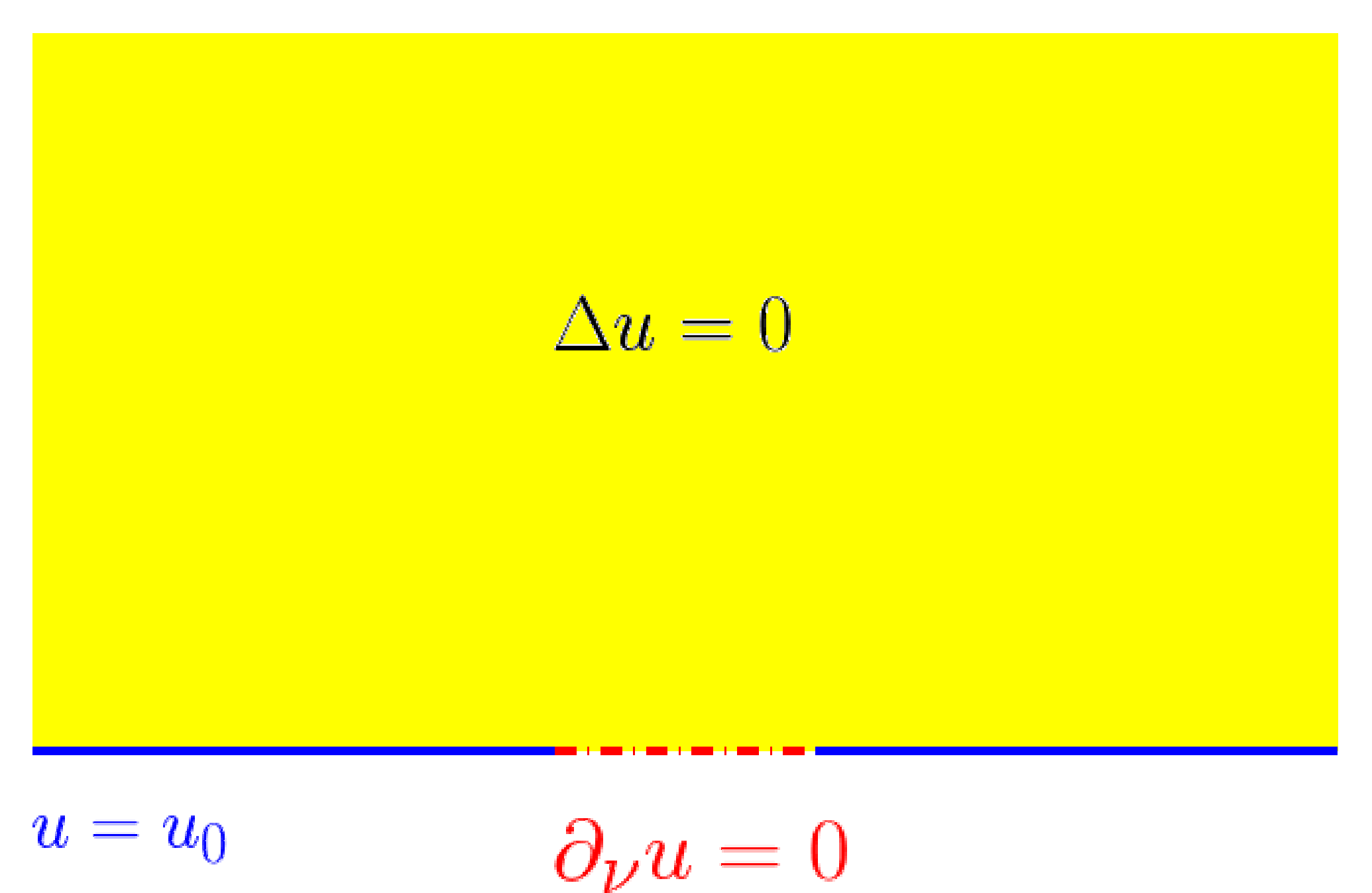}
\caption{\footnotesize\it 
The thermal insulation problem in Example \ref{74747TGSCO}.}
\label{2234CxC6783C}
\end{figure}

The temperature of the room at the point~$x=(x_1,\dots,x_{n+1})\in{\mathbb{R}}^n\times[0,+\infty)$
will be described by a function~$u=u(x)$. At the equilibrium, no heat flow occurs inside the room.
Taking the classical ansatz that the heat flow is produced by the gradient of the temperature,
since no heat sources are placed inside the room,
we obtain that for any ball~$B\Subset{\mathbb{R}}^{n+1}_+$ the heat
flux through the boundary of~$B$ is necessarily zero and therefore, by the Divergence Theorem,
$$ 0=\int_B {\rm div}(\nabla u)=\int_B \Delta u,$$
which gives that~$\Delta u=0$ in~${\mathbb{R}}^{n+1}_+$.

Complementing this equation with the prescriptions along the walls, we thereby obtain that
the room temperature~$u$ satisfies
\begin{equation}\label{UNja}
\begin{cases}
\Delta u(x)=0& {\mbox{ for all $x\in{\mathbb{R}}^{n+1}_+$,}}\\
\partial_{x_{n+1}}u(x)=0& {\mbox{ for all $x\in\Omega$,}}\\
u(x)=u_0(x_1,\dots,x_n)& {\mbox{ for all $x\in({\mathbb{R}}^{n}\times\{0\})\setminus\Omega$.}}\\
\end{cases}
\end{equation}
As a matter of fact, the setting in~\eqref{UNja} lacks uniqueness, since
if~$u$ is a solution of~\eqref{UNja}, then so are
$$ u(x)+x_{n+1},\qquad u(x)+x_1 x_{n+1},\qquad u(x)+e^{x_1}\sin x_{n+1} ,$$
and so on. Hence, to avoid uniqueness issues, we implicitly assume that the solution
of~\eqref{UNja} is constructed by energy minimization:
in this way, the strict convexity
of the energy functional
$$\int_{{\mathbb{R}}^{n+1}_+}|\nabla u(x)|^2\,dx$$
guarantees that the solution is unique.

The question in~\eqref{QUE1} is therefore reduced to find the value of~$u$
in~$\Omega$. To do so, one can observe that the model links the Neumann and the Dirichlet
boundary data of an elliptic problem: namely,
given on the boundary the homogeneous Neumann datum in $\Omega$ and
the (possibly inhomogeneous) Dirichlet datum in~$({\mathbb{R}}^{n}\times\{0\})\setminus\Omega$,
one can consider the (minimal energy) harmonic function satisfying these conditions
and then calculate its Dirichlet datum in~$\Omega$ to give an answer to~\eqref{QUE1}.

Computationally, it is convenient to observe that equation~\eqref{UNja} is linear
and therefore can be efficiently solved by Fourier transform\index{transform!Fourier}.
Indeed, we write~$\hat u=\hat u(\xi,x_{n+1})$ as the Fourier transform of~$u$
in the variables~$(x_1,\dots,x_n)$, that is, up to normalizing constants that
we neglect for the sake of simplicity, for any~$\xi=(\xi_1,\dots,\xi_n)\in{\mathbb{R}}^n$
and any~$x_{n+1}>0$ we define
$$ \hat u(\xi,x_{n+1}):=\int_{ {\mathbb{R}}^n } u(x_1,\dots,x_n,x_{n+1})
\exp\left(- i\sum_{j=1}^n x_j\xi_j\right)\,dx_1\dots dx_n.$$
Hence, integrating by parts, one sees that, for all~$k\in\{1,\dots,n\}$,
\begin{eqnarray*}
&&\int_{ {\mathbb{R}}^n } \partial_{x_k} u(x_1,\dots,x_n,x_{n+1})
\exp\left( -i\sum_{j=1}^n x_j\xi_j\right)\,dx_1\dots dx_n\\&=&-i\xi_k
\int_{ {\mathbb{R}}^n } u(x_1,\dots,x_n,x_{n+1})
\exp\left(- i\sum_{j=1}^n x_j\xi_j\right)\,dx_1\dots dx_n\\&=&-i\xi_k\, \hat u(\xi,x_{n+1})
.\end{eqnarray*}
Iterating this argument, one obtains that, for all~$k\in\{1,\dots,n\}$,
\begin{equation}\label{a78sddxvvvu}
\begin{split}
&\int_{ {\mathbb{R}}^n } \partial_{x_k}^2 u(x_1,\dots,x_n,x_{n+1})
\exp\left( -i\sum_{j=1}^n x_j\xi_j\right)\,dx_1\dots dx_n\\&\qquad=(-
i\xi_k)^2 \hat u(\xi,x_{n+1})=-\xi_k^2\,
\hat u(\xi,x_{n+1}),\end{split}\end{equation}
and hence, summing over~$k$ and taking into account also the derivatives in~$x_{n+1}$,
\begin{eqnarray*}
0&=&\widehat{\Delta u}(\xi,x_{n+1})\\&=& \sum_{k=1}^{n+1}
\int_{ {\mathbb{R}}^n } \partial_{x_k}^2 u(x_1,\dots,x_n,x_{n+1})
\exp\left( -i\sum_{j=1}^n x_j\xi_j\right)\,dx_1\dots dx_n
\\&=&- \sum_{k=1}^n\xi_k^2
\hat u(\xi,x_{n+1})+
\int_{ {\mathbb{R}}^n } \partial_{x_{n+1}}^2 u(x_1,\dots,x_n,x_{n+1})
\exp\left( -i\sum_{j=1}^n x_j\xi_j\right)\,dx_1\dots dx_n
\\ &=& -|\xi|^2 \hat u(\xi,x_{n+1})+\partial^2_{x_{n+1}}
\hat u(\xi,x_{n+1}).
\end{eqnarray*}
Consequently, for any~$x_{n+1}>0$,
\begin{equation*}
\begin{split}
-\widehat{ \partial_{x_{n+1}} u(\cdot,0)}
&=-\partial_{x_{n+1}} \hat u(\xi,0)
\\ &=-\partial_{x_{n+1}} \hat u(\xi,x_{n+1})+\int_0^{x_{n+1}}
\partial^2_{x_{n+1}}
\hat u(\xi,y)\,dy\\
&=-\partial_{x_{n+1}} \hat u(\xi,x_{n+1})+\int_0^{x_{n+1}}
|\xi|^2 \hat u(\xi,y)\,dy.
\end{split}
\end{equation*}
This equation has solution
\begin{equation*}
\hat u(\xi,x_{n+1})=\hat u_0(\xi) \,e^{-|\xi|\,x_{n+1}},\end{equation*}
and hence
\begin{equation}\label{FT:X} \partial_{x_{n+1}}\hat u(\xi,x_{n+1})=-|\xi|\,\hat u_0(\xi) \,e^{-|\xi|\,x_{n+1}}=
-|\xi|\,\hat u(\xi,x_{n+1}).\end{equation}
Combining this with the homogeneous Neumann condition in~\eqref{UNja} we obtain
that, if~${\mathcal{F}}^{-1}$
denotes the anti-Fourier transform, then
\begin{equation}\label{7ajsAHHBAB}
{\mathcal{F}}^{-1}\Big(|\xi|\,\hat u(\cdot,0)\Big)
=0\qquad{\mbox{ in }}\Omega.
\end{equation}
It is convenient to write this using the fractional Laplace formulation.
As a matter of fact, for every~$s\in[0,1]$ and a (sufficiently smooth and decaying)
function~$w:{\mathbb{R}}^n\to{\mathbb{R}}^n$, one can define
\begin{equation}\label{7A-DE}
(-\Delta)^s w:= {\mathcal{F}}^{-1} \Big(|\xi|^{2s} \hat w\Big).
\end{equation}
We observe that when~$s=1$ this definition, up to normalization constants,
gives the classical operator~$-\Delta$, thanks to~\eqref{a78sddxvvvu}.
By a direct computation (see e.g. Proposition~3.3 in~\cite{MR2944369})
one also sees that for every~$s\in(0,1)$
the operator in~\eqref{7A-DE}
can be written in integral form as
\begin{equation}\label{7A-DE2}
(-\Delta)^s w(x)=\int_{{\mathbb{R}}^n}\frac{2w(x)-w(x+y)-w(x-y)}{|y|^{n+2s}}\,dy.
\end{equation}
Comparing~\eqref{7ajsAHHBAB} with~\eqref{7A-DE}, we obtain that
equation~\eqref{UNja} can be written as a fractional equation for a function of~$n$
variables (rather than a classical reaction-diffusion equation for 
a function of~$n+1$
variables): indeed, if we set~$v(x_1,\dots,x_n):=u(x_1,\dots,x_n,0)$,
then
\begin{equation}\label{QUE1BB}
\begin{cases}
\sqrt{-\Delta} \,v =0 &{\mbox{ in $\Omega_0$}},\\
v =u_0 &{\mbox{ in ${\mathbb{R}}^n\setminus\Omega_0$}},
\end{cases} 
\end{equation}
where~$\Omega_0\subset{\mathbb{R}}^n$ is such that~$\Omega=\Omega_0\times\{0\}$.

The solution of the question posed in~\eqref{QUE1} can then
be obtained by considering the values of the solution~$v$ of~\eqref{QUE1BB}
in~$\Omega_0$

The equivalence between~\eqref{UNja} and~\eqref{QUE1BB} (which can also
be extended and generalized in many forms)
is often quite useful since it allows one to connect classical reaction-diffusion
equations and fractional equations and permits the methods typical of one research context
to be applicable to the other.
}\end{example}

\begin{example}[The thin obstacle problem\index{thin obstacle}]\label{THIN}
{\rm The classical obstacle problem considers an elastic membrane possibly subject to an external
force field
which is constrained above an obstacle.
The elasticity of the membrane makes its graph to be a
supersolution in the whole domain and a solution wherever it does not touch the obstacle.
For instance, if the vertical force field is denoted by~$h:{\mathbb{R}}^n\to{\mathbb{R}}$
and the obstacle is given by the subgraph of a function~$\varphi:{\mathbb{R}}^n\to{\mathbb{R}}$,
considering a global problem for the sake of simplicity,
the discussion above formalizes in the system of equations
\begin{equation*}
\begin{cases}
\Delta u\le h & {\mbox{ in }}{\mathbb{R}}^n,\\
u\ge\varphi & {\mbox{ in }}{\mathbb{R}}^n,\\
\Delta u=0 & {\mbox{ in }}{\mathbb{R}}^n\cap\{u>\varphi\}.
\end{cases}
\end{equation*}
As a variation of this problem, one can consider the case in which
the obstacle is ``thin'', i.e. it is supported along a manifold of smaller dimension -- concretely,
in our case, of codimension~$1$.
For concreteness, one can consider the case in which the obstacle
is supported on the hyperplane~$\{x_n=0\}$. In this case,
one considers the subgraph in~$\{x_n=0\}$ of a function~$\varphi:{\mathbb{R}}^{n-1}
\to{\mathbb{R}}$ and requires the solution to lie above it, see Figure~\ref{TH}.

\begin{figure}[h]
\centering
\includegraphics[width=7.8 cm]{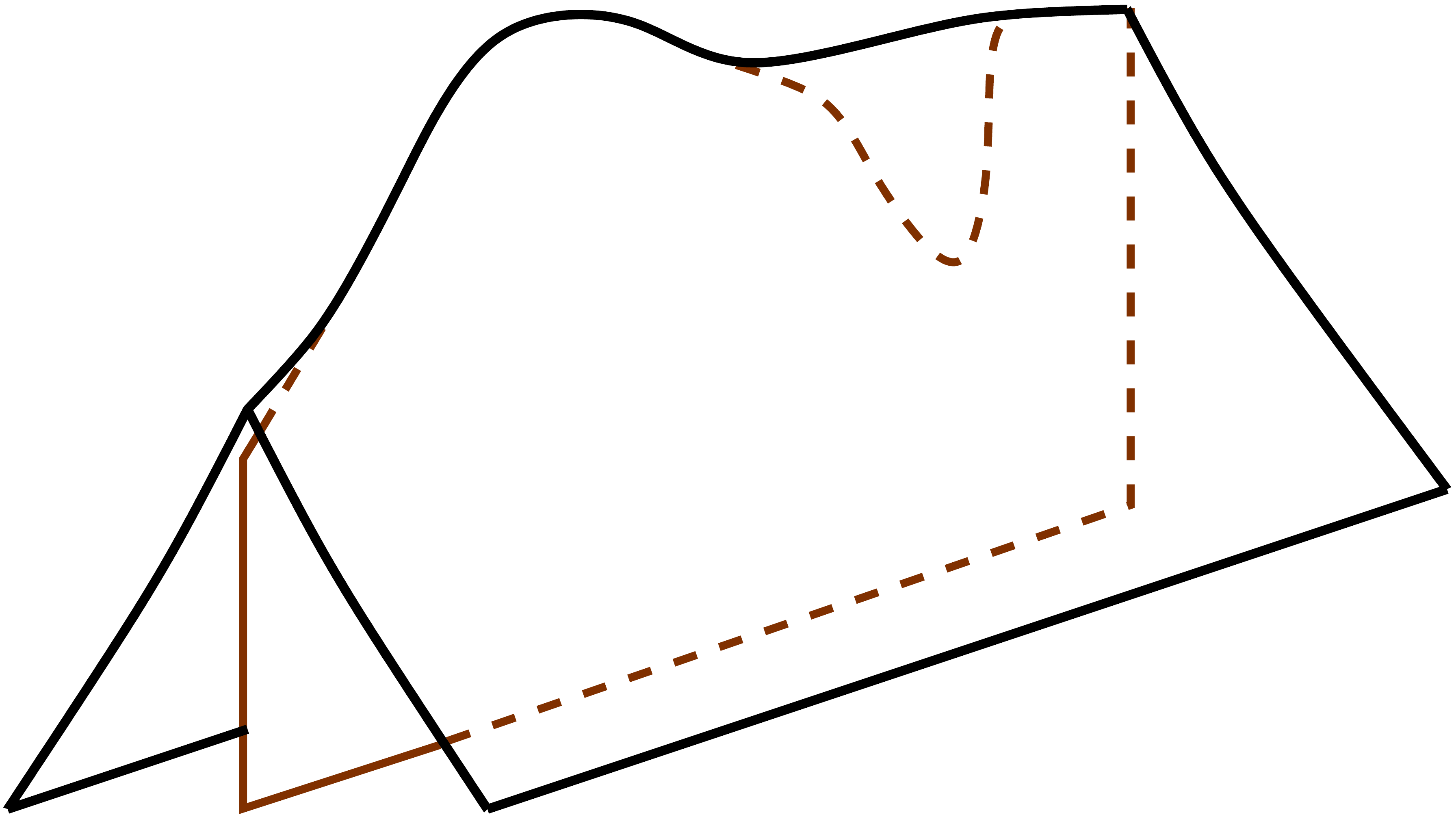}
\caption{\footnotesize\it 
The thin obstacle problem.}
\label{TH}
\end{figure}

Combining the thin
obstacle constrain with the elasticity of the membrane, we model this problem
by the system of equations
\begin{equation}\label{CO:AK}
\begin{cases}
\Delta u\le h & {\mbox{ in }}{\mathbb{R}}^n,\\
u\ge\varphi & {\mbox{ in }}\{x_n=0\},\\
\Delta u=0 & {\mbox{ in }}{\mathbb{R}}^n\setminus (\{x_n=0\}\cap\{u=\varphi\}).
\end{cases}
\end{equation}
For concreteness, we will take~$h:=0$ from now on
(the general case can be reduced to this by subtracting a particular solution). 
Also, given the structure of \eqref{CO:AK},
we
will focus on the case of even solutions with respect to the hyperplane~$\{x_n=0\}$,
namely
$$u(x_1,\dots,x_{n-1},-x_n)=u(x_1,\dots,x_{n-1},x_n).$$
We observe that, if~$\psi\in C^\infty_0({\mathbb{R}}^n,\,[0,+\infty))$, then
\begin{eqnarray*}
\int_{ {\mathbb{R}}^n } \nabla u(x)\cdot\nabla\psi(x)\,dx&=&
\int_{ \{x_n>0\} } \nabla u(x)\cdot\nabla\psi(x)\,dx+
\int_{ \{x_n<0\} } \nabla u(x)\cdot\nabla\psi(x)\,dx\\
&=&
\int_{ \{x_n>0\} } {\rm div}\,\big(\psi(x)\,\nabla u(x)\big)\,dx+
\int_{ \{x_n<0\} } {\rm div}\,\big(\psi(x)\,\nabla u(x)\big)\,dx
\\ &=&-2\int_{ \{x_n=0\} } \psi(x)\,\frac{\partial u}{\partial x_n}(x_1,\dots,x_{n-1},0^+)
\,d{\mathcal{H}}^{n-1}(x),
\end{eqnarray*}
where~${\mathcal{H}}^{n-1}$ denotes the standard
Hausdorff measure of codimension~$1$.
Therefore, the condition~$\Delta u\le0$ in~\eqref{CO:AK}
is distributionally equivalent to
$$ \frac{\partial u}{\partial x_n}(x_1,\dots,x_{n-1},0^+)\le0.$$
Similarly, given any point~$p\in \{x_n=0\}\cap\{u>\varphi\}$,
one can consider functions~$\psi\in C^\infty_0(B_\rho(p))$,
with~$\rho>0$ sufficiently small such that~$B_\rho(p)\subseteq\{u>\varphi\}$,
and thus find that
$$ \frac{\partial u}{\partial x_n}(x_1,\dots,x_{n-1},0^+)=0\qquad{\mbox{in}}\qquad
\{x_n=0\}\cap\{u>\varphi\}.$$
Hence, dropping the notation~$0^+$ for the sake of brevity,
one can write~\eqref{CO:AK} in the form
\begin{equation}\label{CO:AK2}
\begin{cases}
\Delta u=0 & {\mbox{ in }}{\mathbb{R}}^n\setminus (\{x_n=0\}\cap\{u=\varphi\}),\\
u\ge\varphi & {\mbox{ in }}\{x_n=0\},\\
\displaystyle\frac{\partial u}{\partial x_n}\le0& {\mbox{ in }}\{x_n=0\},\\
\displaystyle\frac{\partial u}{\partial x_n}=0& {\mbox{ in }}\{x_n=0\}\cap\{u>\varphi\}.
\end{cases}
\end{equation}
Interestingly, equation~\eqref{CO:AK2} can be written in a fractional Laplacian form.
Indeed, by using the Fourier transform in the variables~$(x_1,\dots,x_{n-1})$
(see e.g.~\eqref{FT:X}
and~\eqref{7A-DE}), we know that, up to normalization constants,
$$ \frac{\partial u}{\partial x_n}(x_1,\dots,x_{n-1},0^+)=-{\mathcal{F}}^{-1}
\left(
|(\xi_1,\dots,\xi_{n-1})| \hat u(\xi_1,\dots,\xi_{n-1},0)\right)=-\sqrt{-\Delta} u(x_1,\dots,x_{n-1},0).$$
Consequently, writing~$v:=u(x_1,\dots,x_{n-1},0)$, we can interpret~\eqref{CO:AK2}
as a fractional equation on~${\mathbb{R}}^{n-1}$, namely
\begin{equation*}
\begin{cases}
v\ge\varphi & {\mbox{ in }}{\mathbb{R}}^{n-1},\\
\sqrt{-\Delta}v\ge0& {\mbox{ in }}{\mathbb{R}}^{n-1},\\
\sqrt{-\Delta}v=0& {\mbox{ in }}\{v>\varphi\}.
\end{cases}
\end{equation*}
We do not go into the details of the classical and recent
developments of the mathematical theory of the thin obstacle problem
and of the many topics related to it:
for this, see e.g.~\cites{MR2962060, MR3709717}.}\end{example}

\begin{example}[The Signorini problem\index{problem!Signorini}]
{\rm In 1959, Antonio Signorini posed an engineering
problem consisting in finding the equilibrium configuration of
an elastic body, resting on a 
rigid frictionless
surface (see~\cite{MR0118021}).
Gaetano Fichera provided a rigorous mathematical framework
in which the problem is well-posed
in 1963, see~\cite{MR0176661}. Interestingly, this solution
was found just a few weeks before Signorini's death,
whose last words were spent to celebrate this discovery
as his {\em greatest contentment}.
The historical description of these moments is commemorated
in the survey
{\em La nascita della teoria delle disequazioni variazionali ricordata dopo trent'anni},
of the 1995 {\em Atti dei Convegni Lincei}.

Here, we recall a simplified version of the problem, its relation with
the thin obstacle problem, and its link with the fractional Laplace operator.

To this end, we introduce some notation from linear elasticity.
Namely, given a material body~$A\subset{\mathbb{R}}^n$ at rest
one describes its equilibrium configuration
under suitable forces by~$B=y(A)$, where~$y:{\mathbb{R}}^n\to{\mathbb{R}}^n$
is a (suitably regular and invertible) map (or, equivalently, one can consider~$B$ the
set at rest and~$A$ the equilibrium configuration in the ``real world'', up
to replacing~$y$ with its inverse).

In this setting, the displacement vector is defined by
\begin{equation}\label{Upda}
U(x):=y(x)-x.\end{equation}
We denote the components of~$U$ by~$U^{(1)},\dots,U^{(n)}$.
The ansatz of the linear elasticity theory is that, as a consequence
of Hooke's Law\index{law!Hooke's}, the infinitesimal elastic energy is proportional to the symmetrized
gradient\index{symmetrized gradient} of~$U$. That is, for every~$i,j\in\{1,\dots,n\}$, one defines the strain
tensor\index{strain tensor}
\begin{equation}\label{CALD} ({\mathcal{D}}U)_{ij}:=\frac{\partial{U^{(i)}}}{\partial x_j}+
\frac{\partial{U^{(j)}}}{\partial x_i}\end{equation}
and sets
$$ |{\mathcal{D}}U|:=\sum_{i,j=1}^n|({\mathcal{D}}U)_{ij}|^2=
\sum_{i,j=1}^n\left(
\frac{\partial{U^{(i)}}}{\partial x_j}+\frac{\partial{U^{(j)}}}{\partial x_i}\right)^2.$$
With this notation, the elastic component of the energy is
\begin{equation}\label{781:19haj} {\mathcal{E}}(U):=\frac12
\int_{A} |{\mathcal{D}}U(x)|^2\,dx.\end{equation}
The differential operator governing the elastostatic equations
(known in the literature as Navier-Cauchy equations, at least in the special
form of them that we take into account in our simplified approach)
are obtained from the first variation of the elastic energy functional in~\eqref{781:19haj}.
Since
\begin{eqnarray*}&&\frac{
|{\mathcal{D}}(U+\epsilon\Phi)|^2 -|{\mathcal{D}}U|^2}{2}\\&=&\frac12
\sum_{i,j=1}^n\left(
\frac{\partial{U^{(i)}}}{\partial x_j}+\frac{\partial{U^{(j)}}}{\partial x_i}+\epsilon
\left(\frac{\partial{\Phi^{(i)}}}{\partial x_j}+
\frac{\partial{\Phi^{(j)}}}{\partial x_i}\right)\right)^2 -\frac{|{\mathcal{D}}U|^2}2\\&=&
\epsilon\sum_{i,j=1}^n
\left(
\frac{\partial{U^{(i)}}}{\partial x_j}+\frac{\partial{U^{(j)}}}{\partial x_i}\right)
\left(\frac{\partial{\Phi^{(i)}}}{\partial x_j}+
\frac{\partial{\Phi^{(j)}}}{\partial x_i}\right)+o(\epsilon)\\&=&
\epsilon\sum_{i,j=1}^n
\left(
\frac{\partial{U^{(i)}}}{\partial x_j}
\frac{\partial{\Phi^{(i)}}}{\partial x_j}+
\frac{\partial{U^{(j)}}}{\partial x_i}
\frac{\partial{\Phi^{(i)}}}{\partial x_j}+
\frac{\partial{U^{(i)}}}{\partial x_j}
\frac{\partial{\Phi^{(j)}}}{\partial x_i}+
\frac{\partial{U^{(j)}}}{\partial x_i}
\frac{\partial{\Phi^{(j)}}}{\partial x_i}
\right)+o(\epsilon)\\
&=&
\epsilon\sum_{i,j=1}^n
\left(
\frac{\partial{U^{(i)}}}{\partial x_j}
\frac{\partial{\Phi^{(i)}}}{\partial x_j}+
\frac{\partial{U^{(j)}}}{\partial x_i}
\frac{\partial{\Phi^{(i)}}}{\partial x_j}+
\frac{\partial{U^{(j)}}}{\partial x_i}
\frac{\partial{\Phi^{(i)}}}{\partial x_j}+
\frac{\partial{U^{(i)}}}{\partial x_j}
\frac{\partial{\Phi^{(i)}}}{\partial x_j}
\right)+o(\epsilon)\\&=&
2\epsilon\sum_{i,j=1}^n
\left(
\frac{\partial{U^{(i)}}}{\partial x_j}+
\frac{\partial{U^{(j)}}}{\partial x_i}\right)\frac{\partial{\Phi^{(i)}}}{\partial x_j}
+o(\epsilon),
\end{eqnarray*}
it follows from~\eqref{781:19haj} that, for all~$\Phi\in C^\infty_0(A,{\mathbb{R}}^n)$,
\begin{equation}\label{DElap}
\begin{split}
\langle D{\mathcal{E}}(U),\Phi\rangle\,&=
2\sum_{i,j=1}^n\int_A
\left(
\frac{\partial{U^{(i)}}}{\partial x_j}(x)+
\frac{\partial{U^{(j)}}}{\partial x_i}(x)\right)\frac{\partial{\Phi^{(i)}}}{\partial x_j}(x)\,dx\\
&=
-2\sum_{i,j=1}^n\int_A\frac{\partial}{\partial x_j}
\left(
\frac{\partial{U^{(i)}}}{\partial x_j}(x)+
\frac{\partial{U^{(j)}}}{\partial x_i}(x)\right)\;\Phi^{(i)}(x)\,dx\\
&=
-2\sum_{i=1}^n\int_A
\left(
\Delta{U^{(i)}}(x)+
\frac{\partial}{\partial x_i}{{\rm div }}\,U(x)\right)\;\Phi^{(i)}(x)\,dx.
\end{split}
\end{equation}
One can also take into account the effect of a force field~$f=(f^{(1)},\dots,f^{(n)}):
{\mathbb{R}}^n\to
{\mathbb{R}}^n$ which is acting on the material body. If we think that the map~$y$
is the outcome of the deformation produced by the force field,
one can consider the infinitesimal work associated to this force
as given approximatively, for small displacements,
by the quantity
$$ f(x)\cdot (y(x)-x)\,dx=f(x)\cdot U(x)\,dx,$$
where the setting in~\eqref{Upda} has been exploited.
We obtain in this way a potential energy of the
form
$$ {\mathcal{P}}(U):=\int_A f(x)\cdot U(x)\,dx.$$
We see that
\begin{equation}\label{ed:AK} \langle\nabla{\mathcal{P}}(U),\Phi\rangle=\int_A
f(x)\cdot\Phi(x)\,dx,\end{equation}
for all $\Phi\in C^\infty_0(A,{\mathbb{R}}^n)$.

{F}rom this and~\eqref{DElap}, we obtain the elastic equation
\begin{equation}\label{ELAEQ}
\begin{split}&
\Delta{U^{(i)}}(x)+
\frac{\partial}{\partial x_i}{{\rm div }}\,U(x)
=\frac12\,f^{(i)}(x),\\&{\mbox{for all $i\in\{1,\dots,n\}$ and all~$x\in A$.}}\end{split}
\end{equation}
We also assume that~$y(\partial A)=\partial B$ and that~$B$
rests on a rigid surface, say a frictionless table.
In this case, we can write that~$A\subset\{x_n\ge0\}$
and~$B\subset\{ y_n\ge0\}$, that is~$y^{(n)}(x)>0$
for all~$x\in A$, and therefore~$U^{(n)}(x)+x_n>0$ for all~$x\in A$.

The contact set between~$B$ and the table is described by the points lying in~$T:=(\partial B)
\cap\{y_n=0\}$, and the contact set between~$A$ and the table is described by
the points lying in~$S:=(\partial A)
\cap\{x_n=0\}$. We describe~$S$ by distinguishing two classes of points,
namely the points~$S_1$
which do not leave the contact set,
and the points~$S_2$
which leave the contact set, namely
\begin{equation}\label{S1S2s}
\begin{split}
S_1\,&:=\{ x\in S {\mbox{ s.t. }} y^{(n)}(x)=0\}\\
&=\{ x\in S {\mbox{ s.t. }} U^{(n)}(x)+x_n=0\}\\
&=\{ x\in S {\mbox{ s.t. }} U^{(n)}(x)=0\}\\
{\mbox{and }}\qquad
S_2\,&:=\{ x\in S {\mbox{ s.t. }} y^{(n)}(x)>0\}\\
&=\{ x\in S {\mbox{ s.t. }} U^{(n)}(x)+x_n>0\}\\
&=\{ x\in S {\mbox{ s.t. }} U^{(n)}(x)>0\}.
\end{split}\end{equation}
We take into account the effect of a surface tension, acting by means
of a force field~$g:{\mathbb{R}}^n\to{\mathbb{R}}^n$.
In this case, the infinitesimal work, for small displacements, is approximately given by
$$ g(x)\cdot (y(x)-x)\,d{\mathcal{H}}^{n-1}(x)=g(x)\cdot U(x)\,d{\mathcal{H}}^{n-1}(x).$$
Therefore, we describe the surface tension energy
effect by an energy functional of the form
$$ {\mathcal{S}}(U):=
\int_{\partial A} g(x)\cdot U(x)\,d{\mathcal{H}}^{n-1}(x).
$$
We stress that if~$x_0\in S_2$ there exist~$\rho>0$ and~$\epsilon_0>0$ such that for all~$\epsilon
\in[-\epsilon_0,\epsilon_0]$
and~$\Phi\in C^\infty_0( B_\rho(x_0),{\mathbb{R}}^n)$
the perturbation~$U+\epsilon \Phi$ is admissible,
in the sense that it maps~$A$ into~$y(A)\subset\{y_n\ge0\}$.

On the other hand, if~$\Phi\in C^\infty_0( {\mathbb{R}}^n,{\mathbb{R}}^n)$
and~$\Phi^{(n)}\ge0$ we have that
the perturbation~$U+\epsilon \Phi$ is admissible for all~$\epsilon\ge0$,
since, for every~$x\in A$,
$$ 0\le y^{(n)}(x)=U^{(n)}(x)+x_n\le U^{(n)}(x)+\epsilon \Phi^{(n)}(x)+x_n.$$
Furthermore,
$$ \langle D{\mathcal{S}}(U),\Phi\rangle
=\int_{\partial A} g(x)\cdot\Phi(x)\,d{\mathcal{H}}^{n-1}(x),$$
for all~$\Phi\in C^\infty_0({\mathbb{R}}^n,{\mathbb{R}}^n)$.
Therefore, recalling~\eqref{DElap}, \eqref{ed:AK}
and~\eqref{ELAEQ}, it follows that
the variation of the full energy with respect to a boundary
perturbation~$\Phi\in C^\infty_0({\mathbb{R}}^n,{\mathbb{R}}^n)$
is equal to
\begin{eqnarray*}
&&
2\sum_{i,j=1}^n\int_A
\left(
\frac{\partial{U^{(i)}}}{\partial x_j}(x)+
\frac{\partial{U^{(j)}}}{\partial x_i}(x)\right)\frac{\partial{\Phi^{(i)}}}{\partial x_j}(x)\,dx
\\ &&\qquad+
\int_A
f(x)\cdot\Phi(x)\,dx+
\int_{\partial A} g(x)\cdot\Phi(x)\,d{\mathcal{H}}^{n-1}(x)\\
&=&
2\sum_{i,j=1}^n\int_A \frac{\partial}{\partial x_j}\left(
\left(
\frac{\partial{U^{(i)}}}{\partial x_j}(x)+
\frac{\partial{U^{(j)}}}{\partial x_i}(x)\right)\,\Phi^{(i)}(x)\right)\,dx
\\ &&\qquad+
\int_{\partial A} g(x)\cdot\Phi(x)\,d{\mathcal{H}}^{n-1}(x)
\\ &=&
2\sum_{i=1}^n\int_A {\rm div}\,\left( \Phi^{(i)}(x)\,
\left(
\nabla U^{(i)}(x)+
\frac{\partial{U}}{\partial x_i}(x)\right)\right)\,dx
\\ &&\qquad+
\int_{\partial A} g(x)\cdot\Phi(x)\,d{\mathcal{H}}^{n-1}(x)\\
\\ &=&
2\sum_{i=1}^n\int_{\partial A} \Phi^{(i)}(x)\,
\left(
\nabla U^{(i)}(x)+
\frac{\partial{U}}{\partial x_i}(x)\right)\cdot\nu(x)\,d{\mathcal{H}}^{n-1}(x)
\\ &&\qquad+\sum_{i=1}^n
\int_{\partial A} g^{(i)}(x)\;\Phi^{(i)}(x)\,d{\mathcal{H}}^{n-1}(x)
\\ &=&
2\sum_{i,k=1}^n\int_{\partial A} \Phi^{(i)}(x)\,({\mathcal{D}} U)_{ik}(x)\,
\nu_k(x)\,d{\mathcal{H}}^{n-1}(x)
\\ &&\qquad+\sum_{i=1}^n
\int_{\partial A} g^{(i)}(x)\;\Phi^{(i)}(x)\,d{\mathcal{H}}^{n-1}(x),
\end{eqnarray*}
where~$\nu$ is the exterior normal of~$A$.
This and the admissibility discussion of the perturbation
gives the boundary conditions
\begin{equation}\label{RA11}\begin{split}&
\sum_{k=1}^n({\mathcal{D}} U)_{nk}(x)\,
\nu_k(x)
=-\frac12\,
g^{(n)}(x)\qquad
{\mbox{for all $x\in S_2$}}\\
{\mbox{and }}\qquad &\sum_{k=1}^n({\mathcal{D}} U)_{nk}(x)\,
\nu_k(x)
\ge -\frac12\,
g^{(n)}(x)\qquad
{\mbox{for all $x\in S_1$.}}
\end{split}\end{equation}
If the surface forces are tangential to the boundary of the material body,
we have that~$g^{(n)}=0$ on the surface of the table, and
also the normal at these points is vertical,
therefore~\eqref{RA11}
reduces to
\begin{equation}\label{RA1167}\begin{split}&
({\mathcal{D}} U)_{nn}(x)
=0\qquad
{\mbox{for all $x\in S_2$}}\\
{\mbox{and }}\qquad &({\mathcal{D}} U)_{nn}(x)
\le0\qquad
{\mbox{for all $x\in S_1$.}}
\end{split}\end{equation}
Interestingly, the problem is naturally endowed with
``ambiguous'' boundary conditions\index{ambiguous boundary conditions}, in the sense that
there are two alternative boundary conditions in~\eqref{RA1167}
that are prescribed at the boundary, in terms of either equalities
and inequalities, and it is not a priori known
what condition is satisfied at each point. Also, by~\eqref{CALD},
we can write~\eqref{RA1167} in the form
\begin{equation}\label{RA1167-BIS}\begin{split}&
\frac{\partial U^{(n)}}{\partial x_n}(x)
=0\qquad
{\mbox{for all $x\in S_2$}}\\
{\mbox{and }}\qquad &\frac{\partial U^{(n)}}{\partial x_n}(x)
\le0\qquad
{\mbox{for all $x\in S_1$.}}
\end{split}\end{equation}
We now consider a magnified version of this picture at ``nice'' contact points.
For this, for simplicity we suppose that
\begin{equation}\label{Gh781-48}
\begin{split}&
0\in\partial A,\qquad B_{\rho}(\rho e_n)\subseteq A,\qquad
B_{\rho}(-\rho e_1)\cap \{x_n=0\}\subseteq S_1\\&\qquad{\mbox{and}}\qquad
B_{\rho}(\rho e_1)\cap \{x_n=0\}\subseteq S_2,
\end{split}\end{equation}
for some~$\rho>0$, see Figure~\ref{CON}.

\begin{figure}[h]
\centering
\includegraphics[width=6.5 cm]{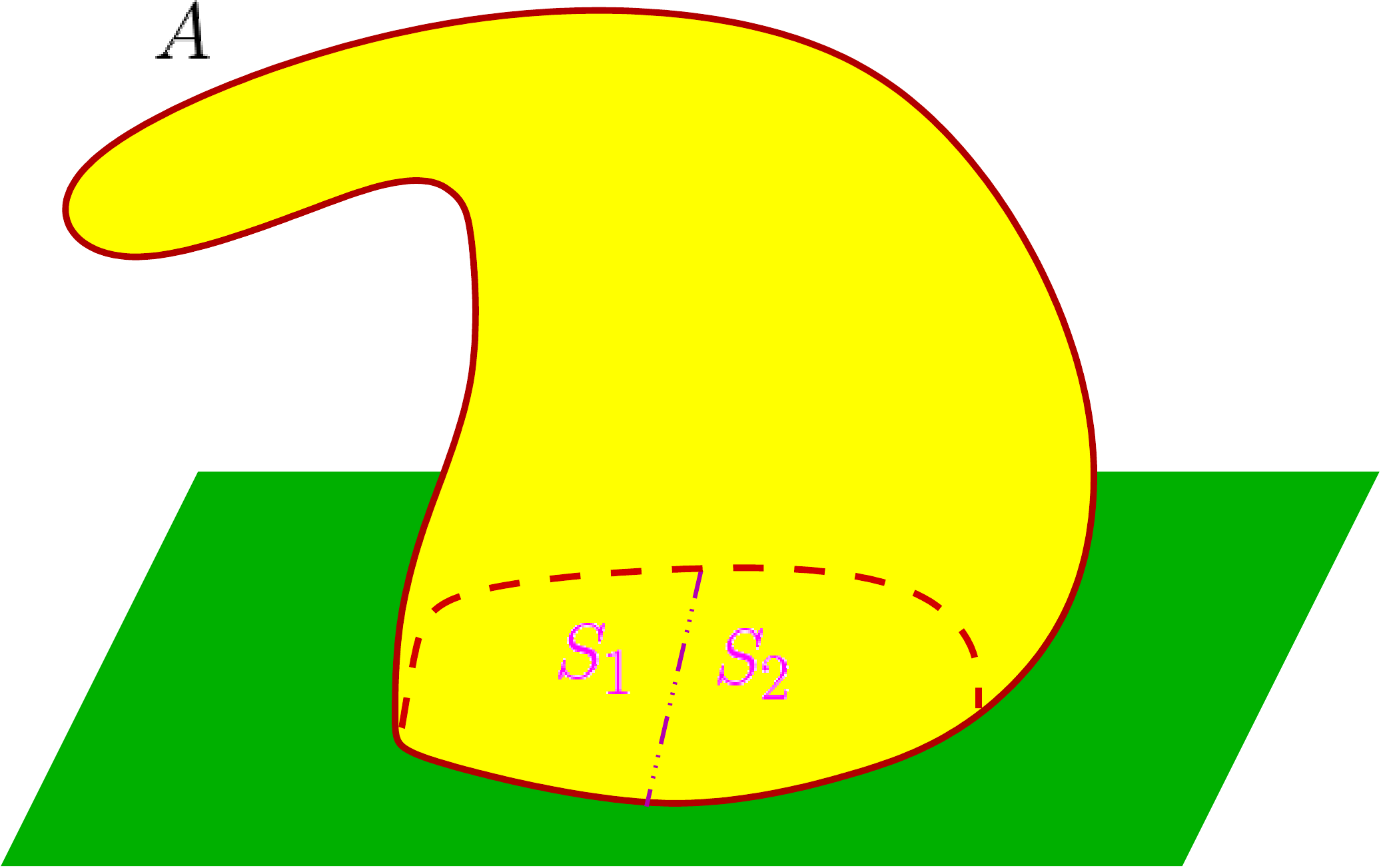}$\qquad$
\includegraphics[width=6.5 cm]{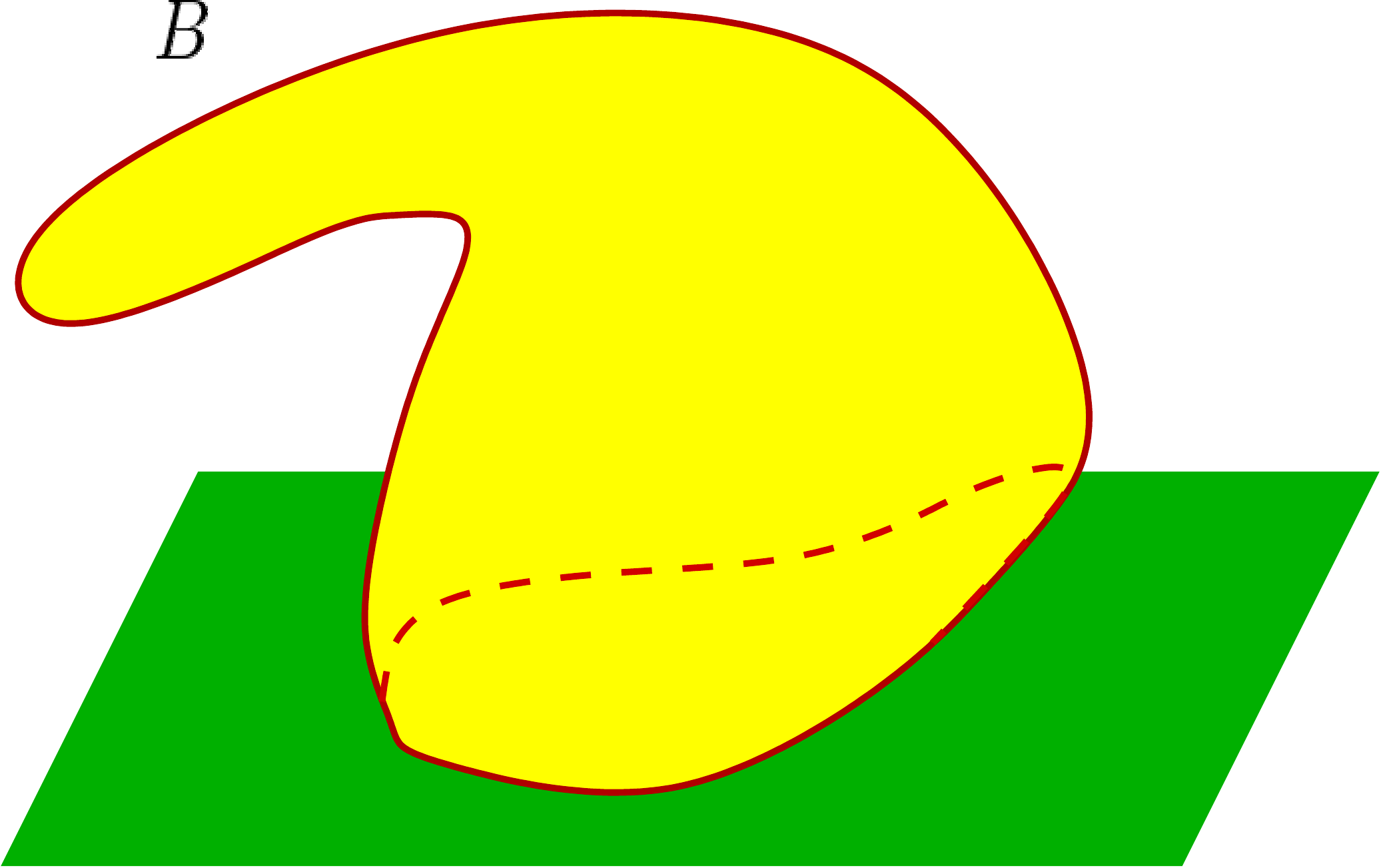}
\caption{\footnotesize\it 
Points leaving the contact set.}
\label{CON}
\end{figure}

Fixed~$\epsilon>0$, to be taken small in what follows,
we consider the transformation
$$ \Theta_\epsilon(x):=\left( \frac{x_1}{\epsilon},\dots,\frac{x_{n-1}}{\epsilon},
\frac{\sqrt{2}\;x_n}{\epsilon}\right).$$
By~\eqref{Gh781-48}, we have that, locally,
\begin{equation}\label{ConaVAU}
\begin{split}&
{\mbox{$\Theta_\epsilon(A)$
converges to~$\{x_n>0\}$,}}\\&
{\mbox{$\Theta_\epsilon(S_1)$
converges to~$\{x_n=0\}\cap\{x_1<0\}$}}\\
{\mbox{and }}\;&{\mbox{$\Theta_\epsilon(S_2)$
converges to~$\{x_n=0\}\cap\{x_1>0\}$,}} \end{split}\end{equation}
up to negligible sets.

Moreover, for each~$i\in\{1,\dots,n\}$, we define
\begin{equation} \label{89AK0} V_\epsilon^{(i)}(x):=
\epsilon^2\,\left[ U^{(i)}\big(\Theta_\epsilon(x)\big)-U^{(i)}(0)-
\sum_{k=1}^n \frac{\partial U^{(i)}}{\partial x_k}(0)\,\Theta^{(k)}_\epsilon(x)
\right].\end{equation}
Under a quadratic bound on~$U$, one can make the ansatz that~$V_\epsilon^{(i)}$
behaves well as~$\epsilon\searrow0$, and we will indeed assume that~$V_\epsilon=(
V_\epsilon^{(1)}, \dots,V_\epsilon^{(n)})$
converges smoothly to some~$V=(V^{(1)},\dots,V^{(n)})$. We also
set~$v:=V^{(n)}$. 
We point out that, by~\eqref{RA1167-BIS} and~\eqref{Gh781-48},
we can take sequences of points~$p_j\in S_1$ and~$q_j\in S_2$ which converge to the origin as~$j\to+\infty$,
and, assuming that the solution is regular enough, write that
\begin{equation}\label{89AK}
\begin{split}&
0=\lim_{j\to+\infty}{U^{(n)}}(p_j)=U^{(n)}(0)
\\{\mbox{and }}\qquad& 0=\lim_{j\to+\infty}\frac{\partial U^{(n)}}{\partial x_n}(q_j)
=\frac{\partial U^{(n)}}{\partial x_n}(0).
\end{split}\end{equation}
As a consequence, if~$\Theta_\epsilon(x)\in S_1\cup S_2$,
\begin{equation*}
\begin{split}
V^{(n)}_\epsilon(x)\,&=
\epsilon^2\,\left[ U^{(n)}\big(\Theta_\epsilon(x)\big)-
\sum_{k=1}^{n-1} \frac{\partial U^{(i)}}{\partial x_k}(0)\,\Theta^{(k)}_\epsilon(x)
\right]\\&\ge
-\epsilon^2\,
\sum_{k=1}^{n-1} \frac{\partial U^{(i)}}{\partial x_k}(0)\,\Theta^{(k)}_\epsilon(x)
\\&=-
\epsilon\,
\sum_{k=1}^{n-1} \frac{\partial U^{(i)}}{\partial x_k}(0)\,x_k,
\end{split}
\end{equation*}
and therefore, recalling~\eqref{ConaVAU}, we obtain that
\begin{equation}\label{1-32}
v(x_1,\dots,x_{n-1},0)\ge 0,\qquad{\mbox{for all }}(x_1,\dots,x_{n-1})\in{\mathbb{R}}^{n-1}.
\end{equation}
By~\eqref{89AK}, it also follows that
\begin{eqnarray*} \frac{\partial V_\epsilon^{(n)}}{\partial x_n}(x)&=&\sqrt{2}\;
\epsilon\,\left[ \frac{\partial U^{(n)}}{\partial x_n}
\big(\Theta_\epsilon(x)\big)-
\frac{\partial U^{(n)}}{\partial x_n}(0)
\right]\\&=&
\sqrt{2}\;
\epsilon\,\frac{\partial U^{(n)}}{\partial x_n}
\big(\Theta_\epsilon(x)\big).\end{eqnarray*}
Hence, by~\eqref{RA1167-BIS},
\begin{equation}\label{RA15100295683}\begin{split}&
\frac{\partial V_\epsilon^{(n)}}{\partial x_n}(x)
=0\qquad
{\mbox{if $\Theta_\epsilon(x)\in S_2$}}\\
{\mbox{and }}\qquad &\frac{\partial V_\epsilon^{(n)}}{\partial x_n}(x)
\le0\qquad
{\mbox{if $\Theta_\epsilon(x)\in S_1$.}}
\end{split}\end{equation}
This and \eqref{ConaVAU} formally lead to
\begin{equation}\label{RA1167-TRIS}\begin{split}&
\frac{\partial v}{\partial x_n}(x)
=0\qquad
{\mbox{in $\{x_n=0\}\cap\{x_1>0\}$}}\\
{\mbox{and }}\qquad &\frac{\partial v}{\partial x_n}(x)
\le0\qquad
{\mbox{in $\{x_n=0\}\cap\{x_1<0\}$.}}
\end{split}\end{equation}
One can make~\eqref{RA1167-TRIS} more precise in light of~\eqref{1-32}.
Namely, we claim that
\begin{equation}\label{RA1167-QUADRIS}
\frac{\partial v}{\partial x_n}(x)
=0\qquad
{\mbox{in $\{x_n=0\}\cap\{v>0\}$.}}
\end{equation}
To check this, let us take~$x\in\{x_n=0\}$ with~$a:=v(x)>0$
and let~$\epsilon>0$ be so small that~$V^{(n)}_\epsilon(x) \geq\frac{a}{2}$
and~$\left| \frac{\partial U^{(i)}}{\partial x_k}(0)\right|\,|x|\le\frac{a}{4\epsilon\,(n-1)}$,
for all~$k\in\{1,\dots,n-1\}$. Then, by~\eqref{89AK0}
and~\eqref{89AK}, we have that
$$ \epsilon^2 U^{(n)}
\big(\Theta_\epsilon(x)\big)=V^{(n)}_\epsilon(x)+\epsilon^2
\sum_{k=1}^{n-1} \frac{\partial U^{(i)}}{\partial x_k}(0)\,\Theta^{(k)}_\epsilon(x)
\ge\frac{a}{2}-\frac{a}{4}>0.$$
This and~\eqref{S1S2s} give that~$\Theta_\epsilon(x)\in S_2$
and consequently~\eqref{RA1167-QUADRIS} follows
from~\eqref{RA15100295683}.

We also introduce the notation
$$ \sigma_k:=\begin{cases}
1 & {\mbox{ if }} k\ne n,\\
{\sqrt2} & {\mbox{ if }} k= n.
\end{cases}$$
Then, we see that, for each~$m\in\{1,\dots,n\}$,
\begin{eqnarray*}
\frac{\partial^2 V_\epsilon^{(n)}}{\partial x_m^2}(x)&=&
\sigma_m^2\,
\frac{\partial^2 U^{(n)}}{\partial x_m^2}\big(\Theta_\epsilon(x)\big)
,\end{eqnarray*}
and accordingly, for every~$x\in{\mathbb{R}}^{n-1}\times(0,+\infty)$
such that~$\Theta_\epsilon(x)\in A$,
\begin{eqnarray*}
\Delta V_\epsilon^{(n)} (x)&=&\sum_{m=1}^n
\sigma_m^2\,
\frac{\partial^2 U^{(n)}}{\partial x_m^2}\big(\Theta_\epsilon(x)\big)
\\&=&
\sum_{m=1}^{n-1}
\frac{\partial^2 U^{(n)}}{\partial x_m^2}\big(\Theta_\epsilon(x)\big)
+2\,
\frac{\partial^2 U^{(n)}}{\partial x_n^2}\big(\Theta_\epsilon(x)\big).
\end{eqnarray*}
Hence, since, by~\eqref{ELAEQ},
\begin{eqnarray*}
\frac12\,f^{(n)}(x)&=&
\Delta{U^{(n)}}(x)+
\frac{\partial}{\partial x_n}{{\rm div }}\,U(x)\\
&=&\sum_{m=1}^n \frac{\partial^2 U^{(n)}}{\partial x_m^2}(x)+
\sum_{m=1}^n \frac{\partial^2 U^{(m)}}{\partial x_m\partial x_n}(x)\\
&=&\sum_{m=1}^{n-1} \frac{\partial^2 U^{(n)}}{\partial x_m^2}(x)+
\sum_{m=1}^{n-1} \frac{\partial^2 U^{(m)}}{\partial x_m\partial x_n}(x)
+2\,\frac{\partial^2 U^{(n)}}{\partial x_n^2}(x),
\end{eqnarray*}
we find that, if~$\Theta_\epsilon(x)\in A$,
\begin{eqnarray*}
\frac12\,f^{(n)}\big(\Theta_\epsilon(x)\big)&=&
\Delta V_\epsilon^{(n)} (x)+
\sum_{m=1}^{n-1} \frac{\partial^2 U^{(m)}}{\partial x_m\partial x_n}
\big(\Theta_\epsilon(x)\big).
\end{eqnarray*}
In particular, if the force field is due to vertical gravity, we have that~$f^{(n)}=-g$
for some constant~$g$, and accordingly
$$ \Delta V_\epsilon^{(n)} (x)=h(x),$$
as long as~$\Theta_\epsilon(x)\in A$, with
$$ h(x):=-\frac{g}2-\sum_{m=1}^{n-1} \frac{\partial^2 U^{(m)}}{
\partial x_m\partial x_n}\big(\Theta_\epsilon(x)\big).$$
Hence, by~\eqref{ConaVAU}, we can write
$$ \Delta v(x)=h(x),$$
for all~$x\in{\mathbb{R}}^{n-1}\times(0,+\infty)$. Combining this with~\eqref{1-32},
\eqref{RA1167-TRIS} and~\eqref{RA1167-QUADRIS},
we can write the system of equations
\begin{equation}\label{sau83eysap}\begin{cases}
& \Delta v(x)=h(x)\qquad
{\mbox{in $\{x_n>0\}$,}}\\
& v\ge 0\qquad
{\mbox{on $\{x_n=0\}$,}}\\
&
\displaystyle\frac{\partial v}{\partial x_n}(x)
\le0\qquad
{\mbox{on $\{x_n=0\}$}},\\
&\displaystyle\frac{\partial v}{\partial x_n}(x)
=0\qquad
{\mbox{on $\{x_n=0\}\cap\{v>0\}$.}}
\end{cases}\end{equation}
By taking even reflection, one can also define
$$ u(x)=u(x_1,\dots,x_n)=\begin{cases}
v(x_1,\dots,x_{n-1},x_n) & {\mbox{ if }} x_n\geq0,\\
v(x_1,\dots,x_{n-1},-x_n) & {\mbox{ if }} x_n<0.
\end{cases}$$
Similarly, we define~$h$ in~$\{x_n<0\}$ by even reflection,
and in this way
$$\Delta u(x)=\Delta v(x_1,\dots,x_{n-1},|x_n|)=h(x_1,\dots,x_{n-1},|x_n|)=h(x),$$
as long as~$x_n\ne0$.

We also observe that if~$\varphi\in C^\infty_0({\mathbb{R}}^n)$
is such that~$\varphi=0$ in~$\{x_n=0\}\cap\{v=0\}$,
\begin{eqnarray*}&&
\int_{ {\mathbb{R}}^n } \nabla u(x)\cdot\nabla\varphi(x)\,dx+
\int_{ {\mathbb{R}}^n } h(x)\,\varphi(x)\,dx
\\
&=& \int_{ {\mathbb{R}}^n \cap\{x_n>0\}} \nabla u(x)\cdot\nabla\varphi(x)\,dx
+ \int_{ {\mathbb{R}}^n \cap\{x_n<0\}} \nabla u(x)\cdot\nabla\varphi(x)\,dx\\
&&\qquad+
\int_{ {\mathbb{R}}^n \cap\{x_n>0\}} h(x)\,\varphi(x)\,dx+
\int_{ {\mathbb{R}}^n \cap\{x_n<0\}} h(x)\,\varphi(x)\,dx
\\
&=& \int_{ {\mathbb{R}}^n \cap\{x_n>0\}} {\rm div}\big(\varphi(x)\,\nabla u(x)\big)\,dx
+ \int_{ {\mathbb{R}}^n \cap\{x_n<0\}} {\rm div}\big(\varphi(x)\,\nabla u(x)\big)\,dx\\
&=& 
-\int_{ \{x_n=0\}} \varphi(x)\,\frac{\partial u}{\partial x_n}(x_1,\dots,x_{n-1},0^+)\,d
{\mathcal{H}}^{n-1}(x)\\&&\qquad
+
\int_{ \{x_n=0\}} \varphi(x)\,\frac{\partial u}{\partial x_n}(x_1,\dots,x_{n-1},0^-)\,d
{\mathcal{H}}^{n-1}(x)\\
&=& 
-2\int_{ \{x_n=0\}\cap\{ v>0\}} \varphi(x)\,\frac{\partial v}{\partial x_n}(x_1,\dots,x_{n-1},0^+)\,d
{\mathcal{H}}^{n-1}(x)\\&=&0,
\end{eqnarray*}
thanks to~\eqref{sau83eysap}.
This says that
\begin{equation*}
{\mbox{$\Delta u=h$ in~${\mathbb{R}}^n\setminus
(\{x_n=0\}\cap\{u=0\})$.}}\end{equation*}
This and~\eqref{sau83eysap} lead to the system
\begin{equation*}\begin{cases}
& \Delta u(x)=h(x)\qquad
{\mbox{in ${\mathbb{R}}^n\setminus
(\{x_n=0\}\cap\{u=0\})$,}}\\
& u\ge 0\qquad
{\mbox{on $\{x_n=0\}$,}}\\
&\displaystyle\frac{\partial u}{\partial x_n}(x)
\le 0\qquad
{\mbox{on $\{x_n=0\}$,}}\\
&\displaystyle\frac{\partial u}{\partial x_n}(x)
=0\qquad
{\mbox{on $\{x_n=0\}\cap\{u>0\}$,}}
\end{cases}\end{equation*}
which is in the form of the thin obstacle problem\footnote{As a matter of fact,
in the recent mathematical jargon,
there is some linguistic confusion about the ``Signorini problem'',
since this name is often used also for the
thin obstacle problem in
Example~\ref{THIN}.
In a sense,
the thin obstacle problem should be properly referred to
as the ``scalar'' Signorini problem, but the adjective ``scalar'' happens
to be often missing.
Of course, the ``original'' Signorini problem is technically even more demanding than
the thin obstacle problem, due to the vectorial nature of the question.
For the optimal regularity and the free boundary analysis of
the original Signorini problem and some important links with the
scalar Signorini problem, we refer to~\cite{MR3480553}
(see in particular Section~5 there).}
discussed in Example~\ref{THIN}
(compare with~\eqref{CO:AK2}).
}\end{example}

\begin{example}[Gamma function, Balakrishnan formula\index{formula!Balakrishnan}, the method of semigroups, and the Heaviside operational calculus]
{\rm Let us start with the classical definition of
the Euler's Gamma function\index{function!Gamma}:
$$ \Gamma (z)=\int _{0}^{+\infty }\tau^{z-1}\,e^{-\tau}\,d\tau. $$
We compute it at the point~$z:=1-s$, with~$s\in(0,1)$, and we integrate by parts,
thus obtaining
\begin{eqnarray*} \Gamma (1-s)&=&
\int_{0}^{+\infty }\tau^{-s}\,e^{-\tau}\,d\tau\\ &=&
-\int_{0}^{+\infty }\tau^{-s}\,\frac{d}{d\tau}(e^{-\tau}-1)\,d\tau\\&=&
-s\int_{0}^{+\infty }\tau^{-s-1}\,(e^{-\tau}-1)\,d\tau,
\end{eqnarray*}
which can be written as
$$ \Gamma(-s)=\int_{0}^{+\infty }\tau^{-s-1}\,(e^{-\tau}-1)\,d\tau.$$
Now, we take~$\delta>0$
and make the substitution~$t:=\delta^{-1} \tau$. In this way, we obtain that
\begin{equation}\label{7uAKKm1eee}
\delta^{s}=\frac1{\Gamma(-s)}\int_{0}^{+\infty }t^{-s-1}\,(e^{-t\delta }-1)
\,dt.\end{equation}
It turns out that one can make sense of this formula not only for a given
real parameter~$\delta>0$, but also when~$\delta$ is replaced by a suitably nice
operator such as the Laplacian (with the minus sign to make it positive).
Namely, formally taking~$\delta:=-\Delta$ in~\eqref{7uAKKm1eee},
one finds that
\begin{equation}\label{7uAKKm1eee2}
(-\Delta)^{s}=\frac1{\Gamma(-s)}\int_{0}^{+\infty }t^{-s-1}\,(e^{t\Delta}-1)
\,dt.\end{equation}
In spite of the sloppy way in which formula~\eqref{7uAKKm1eee2} was derived here,
it is possible to give a rigorous proof of it using operator theory.
Indeed, formula~\eqref{7uAKKm1eee2} was established by
Alampallam V. Balakrishnan in~\cite{MR0115096}.
Its meaning in the operator sense is that applying
the operator on the left hand
side of~\eqref{7uAKKm1eee2} to a nice (say, smooth and rapidly decreasing)
function
is equivalent to apply to it the operator on the right hand
side, namely
\begin{equation}\label{7uAKKm1eee3}
(-\Delta)^{s}u(x)=\frac1{\Gamma(-s)}\int_{0}^{+\infty }t^{-s-1}\,
\big(e^{t\Delta}u(x)-u(x)\big)
\,dt.\end{equation}
The meaning of~$e^{t\Delta}u(x)$ is also in the sense of operators.
To understand this notation one can set~$\Phi_u(x,t):=e^{t\Delta}u(x)$
and observe that, formally,
\begin{eqnarray*} &&\partial_t \Phi_u(x,t)=\partial_t\big(e^{t\Delta}u(x)\big)=
\Delta e^{t\Delta}u(x) =\Delta \Phi_u(x,t)\\
{\mbox{and }}&&\Phi_u(x,0)= e^{\Delta 0}u(x)=e^0 u(x)=u(x).
\end{eqnarray*}
These observations can be formalized by ``going backwards'' in the computation
and defining~$e^{t\Delta}u(x):=\Phi_u(x,t)$, where the latter is the solution
of the heat equation with initial datum~$u$, that is
$$ \begin{cases}
\partial_t \Phi_u(x,t)=\Delta \Phi_u(x,t) & {\mbox{ for all $x\in{\mathbb{R}}^n$ and~$t>0$,}}\\
\Phi_u(x,0)=u(x) & {\mbox{ for all $x\in{\mathbb{R}}^n$.
}}\end{cases}$$
With this notation, we can write~\eqref{7uAKKm1eee3} as
\begin{equation}\label{7uAKKm1eee4}
(-\Delta)^{s}u(x)=\frac1{\Gamma(-s)}\int_{0}^{+\infty }t^{-s-1}\,
\big(\Phi_u(x,t)-u(x)\big)
\,dt.\end{equation}
The power of formula~\eqref{7uAKKm1eee4} is apparent, since it reduces
a nonlocal, and in principle rather complicated, operator
such as the fractional Laplacian to the superposition
of classical heat semigroups, and can be exploited as a ``subordination
identity'' in which the well-established knowledge of the classical heat flow
leads to new results for the fractional setting, see e.g.~\cite{MR2858052}.
The importance and broad range of applicability of formula~\eqref{7uAKKm1eee4}
and of its various extensions is 
very clearly and extensively discussed in~\cites{2017arXiv171203347G, 2018arXiv180805159S}.

Now we make some historical comments about operator calculus and its successful
attempt to transform identities valid for real numbers into rigorous formulas
involving operators, under the appropriate assumptions. Without aiming
at reconstructing here the full history of the subject,
we recall that one of the first
attempts in the important directions sketched here was made by
Oliver Heaviside at the end of XIX century, who also tried to write
the solution of the heat equation in an operator form. Given the difficulty of the arguments
treated and the lack of mathematical technologies at that time,
some of the original arguments in the literature were probably not fully justified
and required the introduction of a brand new subject of mathematical analysis,
which indeed highly contribute to create and promote, see e.g.~\cite{MR555103}.
It is however plausible that the pioneering, albeit somewhat unrigorous, intuitions
of Heaviside were not always well-appreciated by the mathematical community
at that time. A footprint of this historical controversy has remained
in the work by Heaviside published in Volume~34 of the
periodical and scientific journal
{\em The Electrician}, in which Heaviside states that
{\em What one has a right to expect, however, is a fair field,
and that the want of sympathy should be kept in a neutral state, so as
not to lead to unnecessary obstruction. For even men who are not Cambridge
mathematicians deserve justice, which I very much fear they do not always get,
especially the meek and lowly}.

For an extensive treatment of the theory of operator calculus,
see e.g.~\cites{MR0105594, MR0361633, MR2244037, MR3468941} and the references therein.}\end{example}

\begin{example}[Fractional viscoelastic models, springs and dashpots]\label{LAD}
{\rm
A classical application of fractional derivatives occurs
in the phenomenological description of viscoelastic fluids. This is of course
a very advanced topic and we do not aim at fully cover it in these few pages:
see e.g. Section~10.2 of~\cite{MR1658022}, where
a number of fractional models for viscoelasticity\index{viscoelasticity} are discussed in detail.

Roughly speaking, a basic idea used in this context is that the viscoelastic effects
arise as a suitable ``ideal'' superposition of ``purely elastic'' and ``purely viscous''
phenomena, which are better understood when treated separately but whose
combined effect becomes quite difficult to comprise into classical equations.

On the one hand, the elastic effects are well-described by the displacement
of a classical spring subject to Hooke's Law, in which the displacement of the spring
is proportional to the force applied to it. That is, if~$\varepsilon$ denotes
the elongation of the spring and~$\sigma$ the force applied to it, one writes
\begin{equation}\label{VIS:000}
\sigma=\kappa\varepsilon ,
\end{equation}
for a suitable elastic coefficient~$\kappa>0$.

On the other hand, the viscous effects of fluids is classically described by
Newton's Law\index{law!Newton's}, according to which
forces are related to velocities, as in the formula
\begin{equation}\label{VIS:001}
\sigma=\nu\dot\varepsilon ,
\end{equation}
where the ``dot'' here above denotes derivative with respect to time, 
for a suitable viscous coefficient~$\nu>0$.

The rationale sustaining~\eqref{VIS:001} can be understood thinking about the free fall
of an object. In this case, if one takes into account the gravity and
the viscous friction of the air, the vertical position of a falling object is described
by the equation
\begin{equation}\label{8iwjs1324367000923}
mg-\nu\dot\varepsilon= m\ddot\varepsilon.
\end{equation}
For long times, the falling body reaches asymptotically a terminal velocity,
which can be guessed by~\eqref{8iwjs1324367000923} by formally imposing that the limit
acceleration is zero: in this limit regime, one thus obtain the velocity equation
\begin{equation}\label{008iwjs1324367000923}
mg-\nu\dot\varepsilon= 0,
\end{equation}
which formally coincides with equation~\eqref{VIS:001} when the force is
the gravitational one. Hence, in a sense, comparing~\eqref{VIS:001} with~\eqref{008iwjs1324367000923},
one can think that 
Newton's Law for viscid fluids describes the asymptotic velocity
in a regime in which the viscous effects are dominant and after a sufficient
amount of time (after which the acceleration effects become negligible).

Roughly speaking, the idea of viscoelasticity is that, in general,
fluids are neither perfectly elastic nor perfectly viscid, therefore
an accurate formulation of the problem requires the study of an operator which
interpolates between the ``derivative of order zero'' appearing
in~\eqref{VIS:000} and
the ``derivative of order one'' appearing
in~\eqref{VIS:001}, and of course fractional derivatives seem to perfectly fit
such a scope.

{F}rom the point of view of the notation, in the description of fluids,
the displacement function~$\varepsilon$ typically represents the ``strain'',
while the normalized force function~$\sigma$ typically represents the ``stress''
acting on the fluid particles.

\begin{figure}[h]
\centering
\includegraphics[width=4.5 cm]{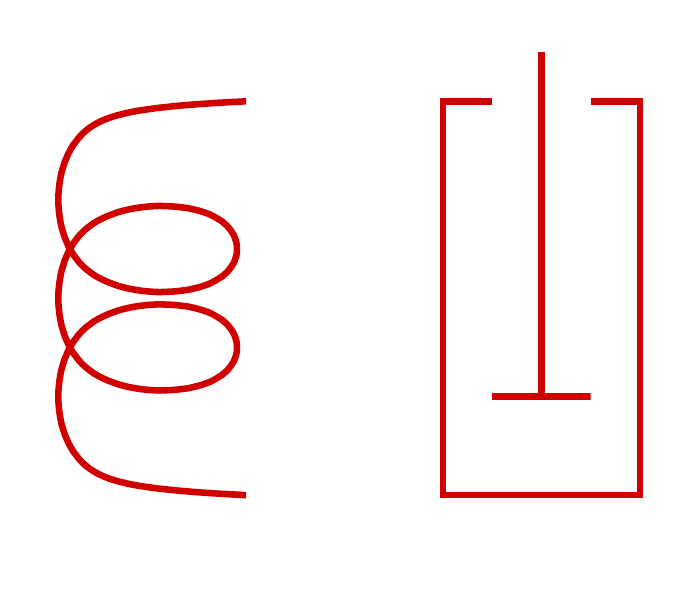}
\caption{\footnotesize\it Schematic representation of a spring (left) and a dashpot (right).}
\label{DASH}
\end{figure}

To give a concrete feeling of this superposition of elastic and viscid effects, we recall
here a purely mechanical model which was proposed by~\cite{Schiessel}
(we actually simplify the discussion presented in~\cite{Schiessel}, since we 
do not aim here at fully justified general statements).

The idea proposed by~\cite{Schiessel} is to take into account a system of springs,
which react to forces elastically according to
Hooke's Law, and
dashpots, or viscous dampers, in which strains and stresses are related by
Newton's Law.

\begin{figure}[h]
\centering
\includegraphics[width=6 cm]{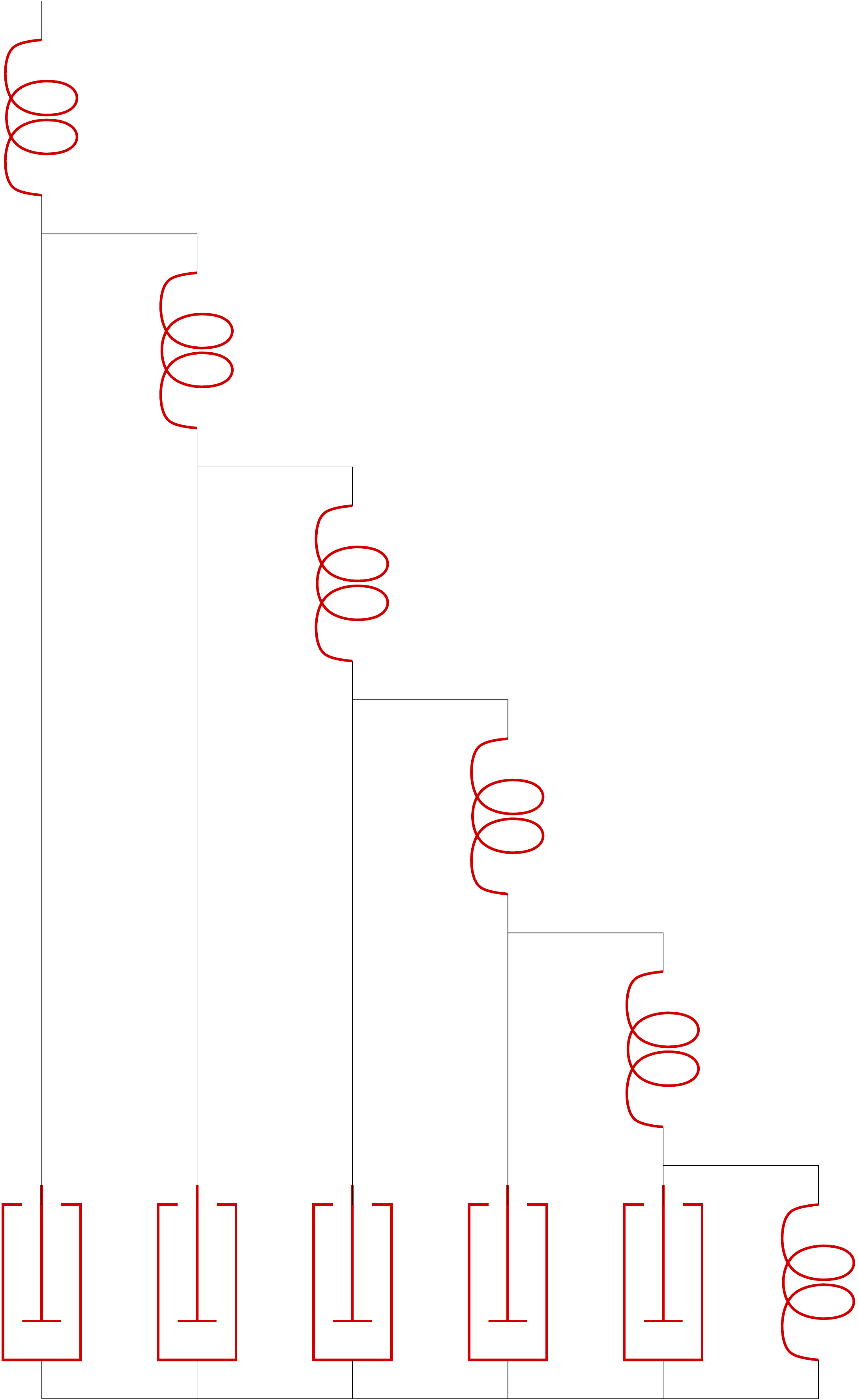}
\caption{\footnotesize\it The spring-dashpot ladder of Example~\ref{LAD}.}
\label{DASH2}
\end{figure}

To clarify the setting we will draw schematically springs and dashpots as described in Figure~\ref{DASH}.
Following~\cite{Schiessel}, the model that we discuss here
consists of a ladder-like structure with springs along
one of the struts and dashpots on the rungs of the ladder, see Figure~\ref{DASH2}.
The system contains~$n$ spring and~$n-1$ dashpots and we will formally
consider the limit as~$n\to+\infty$.
The superindexes~$d$ and~$s$ will refer to the dashpots and the springs, respectively,
while the subindexes will refer to the position in the ladder.
In particular, we can consider the elongation of the springs, denoted by~$\varepsilon^s_0,\dots,
\varepsilon^s_{n-1}$ and the ones of the dashpots, denoted by~$\varepsilon^d_0,\dots,
\varepsilon^d_{n-2}$. {F}rom Figure~\ref{DASH2}, we see that
\begin{equation}\label{EL:1}
\varepsilon^d_k=\varepsilon^s_{k+1}+\varepsilon^d_{k+1}.
\end{equation}
Moreover, if $\sigma^s_1,\dots,\sigma^s_n$ denote the stresses on the springs
and~$\sigma^d_1,\dots,\sigma^d_{n-1}$ the ones on the dashpots,
the parallel arrangement on the ladder gives that
\begin{equation}\label{EL:2}
\sigma^s_k=\sigma^s_{k+1}+\sigma^d_k.\end{equation}
Also, the total elongation of the system is given by
\begin{equation}\label{EL:TOT}
\varepsilon=\varepsilon^s_0+\varepsilon^d_0,\end{equation}
and the stress at the end of the ladder is given by the one of the first spring, namely
\begin{equation}\label{EL:TOT2}
\sigma=\sigma^s_0.\end{equation}
We will consider springs with the same elastic coefficient, that we normalize in such a way
that Hooke's Law writes in this case as
\begin{equation}\label{EL:3}
\varepsilon^s_k=2\sigma^s_k.
\end{equation}
Also, the dashpots will be taken to be all with the same viscous coefficients
and, in this way, Newton's Law will be written as
\begin{equation}\label{EL:4}
\dot\varepsilon^d_k=\sigma^d_k.
\end{equation}
The case of different springs and dashpots can also be taken into account
and it would produce quantitatively different analysis, see~\cite{Schiessel}
for details, but even this simpler case in which all the mechanical elements are
the same produces some interesting effects that we now analyze.

First of all, by~\eqref{EL:1} and~\eqref{EL:3},
\begin{equation}\label{EL:5}
\varepsilon^d_k=2\sigma^s_{k+1}+\varepsilon^d_{k+1}.
\end{equation}
Similarly, by~\eqref{EL:2} and~\eqref{EL:4},
\begin{equation}\label{EL:6}
\sigma^s_k=\sigma^s_{k+1}+\dot\varepsilon^d_k.
\end{equation}
It is now appropriate to consider the Laplace transform\index{transform!Laplace} of a function~$u$,
that we denote by
$$ \bar u(\omega):=
\int_{0}^{+\infty} u(t) \,e^{-t\omega}\,dt.$$
For further reference, we observe that
\begin{equation}\label{87ujJ8929ppPQ}
\begin{split}
\bar{\dot{u}}(\omega)\,&=\int_{0}^{+\infty} \dot u(t) \,e^{-t\omega}\,dt\\&=
\int_{0}^{+\infty}\left( \frac{d}{dt}\big( u(t) \,e^{-t\omega}\big) 
+\omega u(t) \,e^{-t\omega}\right)\,dt\\&=-u(0)+\omega\bar u(\omega).
\end{split}
\end{equation}
In a similar way,
considering the fractional derivative notation
$$ D^{1/2}_{t,0} u(t):=\int_0^t \frac{
\dot u(\tau)}{\sqrt{t-\tau}}\,d\tau,$$
to be compared with the general setting in the forthcoming formula~\eqref{defcap},
we have that
\begin{equation}\label{YTtatra}
\begin{split}
\overline{D^{1/2}_{t,0} u}(\omega)\,&=\int_0^{+\infty}\left[
\int_0^t \frac{
\dot u(\tau)}{\sqrt{t-\tau}}\,d\tau\right]\,e^{-t\omega}\,dt\\
&=\int_0^{+\infty}\left[\dot u(\tau)\int_\tau^{+\infty} \frac{e^{-t\omega}
}{\sqrt{t-\tau}}\,dt\right]\,d\tau\\
&=\int_0^{+\infty}\left[\dot u(\tau)e^{-\tau\omega}\int_0^{+\infty} \frac{e^{-\zeta\omega}
}{\sqrt{\zeta}}\,d\zeta\right]\,d\tau
\\&=\frac1{\sqrt{\omega}}
\int_0^{+\infty}\left[\dot u(\tau)e^{-\tau\omega}\int_0^{+\infty} \frac{e^{-\mu}
}{\sqrt{\mu}}\,d\mu\right]\,d\tau\\&=\frac{C}{\sqrt{\omega}}
\int_0^{+\infty}\dot u(\tau)e^{-\tau\omega}\,d\tau\\
&=\frac{C}{\sqrt{\omega}}
\int_0^{+\infty}\left(\frac{d}{d\tau}\Big(u(\tau)e^{-\tau\omega}\Big)+
\omega u(\tau)e^{-\tau\omega}
\right)
\,d\tau\\
&=-\frac{Cu(0)}{\sqrt{\omega}}+C\sqrt{\omega}
\int_0^{+\infty}u(\tau)e^{-\tau\omega}\,d\tau\\
&=-\frac{Cu(0)}{\sqrt{\omega}}+C\sqrt{\omega}\bar u(\omega),
\end{split}
\end{equation}
for a suitable~$C>0$.

Now, taking the Laplace transform of~\eqref{EL:5}, we find that
\begin{equation*}
\bar\varepsilon^d_k=2\bar\sigma^s_{k+1}+\bar\varepsilon^d_{k+1},
\end{equation*}
and therefore
\begin{equation}\label{EL:7}
\frac{\bar\varepsilon^d_k}{2\bar\sigma^s_{k+1}}=1+\frac{\bar\varepsilon^d_{k+1}}{2\bar\sigma^s_{k+1}}.
\end{equation}
Instead, taking the Laplace transform of~\eqref{EL:6}
and recalling~\eqref{87ujJ8929ppPQ}, assuming that the initial displacement vanishes,
we find that
\begin{equation}\label{7uAJJA93eirjj}
\bar\sigma^s_k=\bar\sigma^s_{k+1}+\omega\bar\varepsilon^d_k.
\end{equation}
We write this identity as
\begin{equation*}
\bar\sigma^s_{k+1}=\bar\sigma^s_{k+2}+\omega\bar\varepsilon^d_{k+1}
\end{equation*}
and we substitute it in the right
hand side of~\eqref{EL:7}, concluding that
\begin{equation}\label{EL:8}\begin{split}
\rho_k\,&:=
\frac{\bar\varepsilon^d_k}{2\bar\sigma^s_{k+1}}\\&=1+
\frac{\bar\varepsilon^d_{k+1}}{2(\bar\sigma^s_{k+2}+\omega\bar\varepsilon^d_{k+1})}\\&=
1+\frac{\rho_{k+1}}{1+\omega\frac{\bar\varepsilon^d_{k+1}}{\bar\sigma^s_{k+2}}}\\&=
1+\frac{\rho_{k+1}}{1+2\omega\rho_{k+1}}\\
&=1+\frac{1}{2\omega+\frac{1}{\rho_{k+1}}}.
\end{split}
\end{equation}
We can iterate~\eqref{EL:8} and then find that
\begin{equation*}
\rho_k=1+\frac{1}{2\omega+\frac{1}{1+\frac{1}{2\omega+\frac{1}{\rho_{k+2}}}}}=
1+\frac{1}{2\omega+\frac{1}{1+\frac{1}{2\omega+\frac{1}{
1+\frac{1}{2\omega+\frac{1}{\rho_{k+3}}}
}}}},
\end{equation*}
and so on. That is, in the formal limit of infinitely many springs and dashpots,
\begin{equation}\label{EL:10}
\frac{\bar\varepsilon^d_0}{2\bar\sigma^s_{1}}=
\rho_0=1+\frac{1}{2\omega+\frac{1}{1+\frac{1}{2\omega+\frac{1}{\rho_2}}}}=
1+\frac{1}{2\omega+\frac{1}{1+\frac{1}{2\omega+\frac{1}{
1+\frac{1}{2\omega+\frac{1}{\ddots}}
}}}},
\end{equation}
which is an infinite continuous fraction.

We observe that
\begin{equation}\label{EL:11}
{\mbox{the right hand side of~\eqref{EL:10} is equal to }}\frac{1+\sqrt{1+
\displaystyle\frac2\omega}}{2}.
\end{equation}
Indeed, the right hand side of~\eqref{EL:10} is a positive number,
say~$X$, and it satisfies that
$$ X=1+\frac{1}{2\omega+\frac{1}{X}}.$$
Solving~$X$ in this relation, we obtain~\eqref{EL:11}, as desired.

Then, from~\eqref{EL:10} and~\eqref{EL:11}, we conclude that
\begin{equation}\label{EL:13}
\frac{\bar\varepsilon^d_0}{2\bar\sigma^s_{1}}=
\frac{1+\sqrt{1+
\displaystyle\frac2\omega}}{2}.
\end{equation}
On the other hand, recalling~\eqref{EL:TOT}, \eqref{EL:TOT2}, \eqref{EL:3} and~\eqref{7uAJJA93eirjj},
\begin{equation*}
\begin{split}
\frac{\bar\varepsilon(\omega)}{2\bar\sigma(\omega)}\,&=
\frac{\bar\varepsilon^s_0+\bar\varepsilon^d_0}{2\bar\sigma^s_0}\\&=
\frac{\bar\varepsilon^s_0}{2\bar\sigma^s_0}+
\frac{\bar\varepsilon^d_0}{2\bar\sigma^s_0}\\&=
1+\frac{\bar\sigma^s_{1}}{\bar\sigma^s_1+\omega
\bar\epsilon^d_0}\times\frac{\bar\varepsilon^d_0}{2\bar\sigma^s_{1}},
\end{split}\end{equation*}
which combined to~\eqref{EL:13} gives that
\begin{equation}\label{EL:14}
\frac{\bar\varepsilon(\omega)}{2\bar\sigma(\omega)}=1+
\frac{\bar\sigma^s_{1}}{\bar\sigma^s_1+\omega
\bar\epsilon^d_0}\times\frac{1+\sqrt{1+
\displaystyle\frac2\omega}}{2}.
\end{equation}
For small~$\omega$, we see that~\eqref{EL:14}
becomes
\begin{equation}\label{EL:15}
\frac{\bar\varepsilon(\omega)}{\bar\sigma(\omega)}\simeq\sqrt{\frac2\omega}
.\end{equation}
We observe that, at least formally, the regime of small~$\omega$ corresponds to
that of large~$t$: a rigorous justification
of these asymptotics is likely to rely on a suitable use of an
Abelian-Tauberian Theorem (see e.g.~\cite{MR1015374})
and goes well beyond the scope of our heuristic argument (but see formulas~(38)
and~(39) in~\cite{Schiessel} for a quantitative analysis of these asymptotics).

Hence, from~\eqref{EL:15},
taking the initial displacement to be null, i.e. assuming that~$\varepsilon(0)=0$,
and normalizing the constants to be unitary
for the sake of simplicity,
in light of~\eqref{YTtatra}, we can write, for small~$\omega$,
that
\begin{equation*}
\bar\sigma(\omega)\simeq \sqrt{\omega}\,\bar\varepsilon(\omega)=
\overline{D^{1/2}_{t,0} \varepsilon}(\omega),\end{equation*}
and therefore, for large~$t$, that
$$ D^{1/2}_{t,0} \varepsilon(t)\simeq \sigma(t).$$
This provides a heuristic, but rather convincing, motivation
showing how fractional derivatives naturally surface in complex mechanical models
involving springs and dashpots and suggest that similar phenomena can arise
in models in which both elastic and viscid effects contribute effectively to the macroscopic
behaviour of the system.
}\end{example}

\begin{example}[Diffusion along a comb structure]\label{TGSCO}
{\rm In this example we show how the geometry of the diffusion media
can naturally produce fractional equations.
We take into account a diffusion model along a comb structure.
The diffusion along the backbone of the comb can be either classical or
of space-fractional type (the model that we present
was indeed introduced in~\cite{comb} in the case of classical diffusion
along the backbone, but we will present here an even more general setting
that comprises space-fractional diffusion as well).

The ramified medium that we take into account in this example
is a ``comb'', with a backbone on the horizontal axis and a fine grid
of vertical fingers which are located at mutual distance~$\epsilon>0$, see Figure~\ref{CCC}.

\begin{figure}[h]
\centering
\includegraphics[width=11.5 cm]{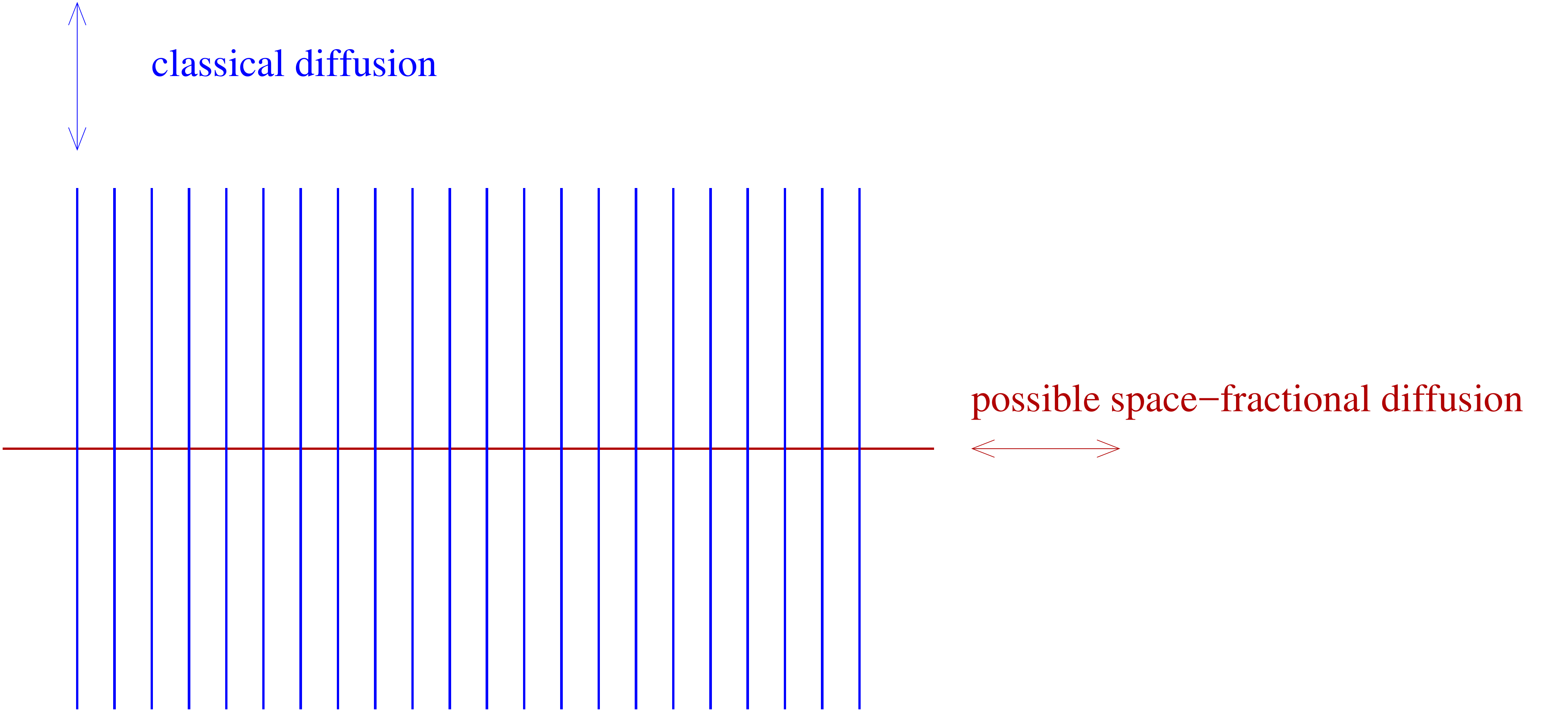}
\caption{\footnotesize\it 
The comb structure along which diffusion takes place in Example \ref{TGSCO}.}
\label{CCC}
\end{figure}

We consider the superposition of a horizontal diffusion along the backbone
driven by~$-(-\Delta)^s_x$, for
some~$s\in(0,1]$, the case~$s=1$ corresponding to classical diffusion
and the case~$s\in(0,1)$ to space-fractional diffusion
as in~\eqref{7A-DE} and~\eqref{7A-DE2}, and a vertical diffusion along the fingers of classical type.
To make the model attain a significant limit in the continuous approximation
in which~$\epsilon\searrow0$, since the fingers are infinitely
many and their distance tends to zero, it is convenient to assume
that the vertical diffusion is subject to a diffusive coefficient of order~$\epsilon$.
The initial position is taken for simplicity to be a concentrated mass 
at the origin. This model translates into the following mathematical formulation:
\begin{equation}\label{EQUco0}
\begin{cases}
\partial_t u(x,y,t)=-\delta_0(y)\,(-\Delta)^s_x u(x,y,t)+\epsilon\displaystyle\sum_{j\in{\mathbb{Z}}}
\delta_0(\epsilon j)\,\partial^2_y u(x,y,t)& \\ \qquad\qquad
{\mbox{ for all $(x,y)\in{\mathbb{R}}^2$ and $t>0$,}}\\
u(x,y,0)=\delta_0(x)\,\delta_0(y).
\end{cases}
\end{equation}
As customary, the notation~$\delta_0$ denotes here
the Dirac's Delta function centered at the origin.
We point out that
\begin{equation}\label{sedisy}
\lim_{\epsilon\searrow0}\epsilon\sum_{j\in{\mathbb{Z}}}
\delta_0(\epsilon j)=1,
\end{equation}
in the distributional sense. Indeed, if~$\varphi\in C^\infty_0({\mathbb{R}})$,
\begin{eqnarray*}
\epsilon\sum_{j\in{\mathbb{Z}}}\int_{\mathbb{R}}
\delta_0(\epsilon j)\,\varphi(y)\,dy=\epsilon\sum_{j\in{\mathbb{Z}}}\varphi(\epsilon j).
\end{eqnarray*}
Since the latter can be seen as a Riemann sum, we find that
$$\lim_{\epsilon\searrow0}
\epsilon\sum_{j\in{\mathbb{Z}}}\int_{\mathbb{R}}
\delta_0(\epsilon j)\,\varphi(y)\,dy=
\int_{\mathbb{R}}\varphi(y)\,dy,$$
which gives~\eqref{sedisy}.

Hence, by~\eqref{sedisy}, one can take into account the continuous limit
of~\eqref{EQUco0} as~${\epsilon\searrow0}$, that we can write in the form
\begin{equation}\label{EQUco}
\begin{cases}
\partial_t u(x,y,t)=-\delta_0(y)\,(-\Delta)^s_x u(x,y,t)+\partial^2_y u(x,y,t)& {\mbox{ for all $(x,y)\in{\mathbb{R}}^2$ and $t>0$,}}\\
u(x,y,0)=\delta_0(x)\,\delta_0(y).
\end{cases}
\end{equation}
We consider the effective transport along the backbone~$\{y=0\}$, given by the function
\begin{equation}\label{5INT-U} U(x,t):=\int_{\mathbb{R}} u(x,y,t)\,dy.\end{equation}
Our claim is that~$U$ satisfies a time-fractional equation given by
\begin{equation}\label{EQUcoU}
\begin{cases}
D^{1/2}_{t,0} U(x,t)= -(-\Delta)^s_x U(x,t)
& {\mbox{ for all $x\in{\mathbb{R}}$ and $t>0$,}}\\
U(x,0)=\delta_0(x).
\end{cases}
\end{equation}
Here, up to normalization constants, we are taking
\begin{eqnarray*} D^{1/2}_{t,0} U(x,t)&:=&
\int_0^t \frac{\partial_tU(x,\tau)}{\sqrt{t-\tau}}\,d\tau\\&=&
\int_0^t \frac{\partial_tU(x,t-\tau)}{\sqrt{\tau}}\,d\tau,\end{eqnarray*}
and one can compare this expression with the similar ones
in~\eqref{8sEQ} and~\eqref{FRAcha}, as well as with the more general
setting that we will introduce in~\eqref{defcap}.

To prove~\eqref{EQUcoU}, we first check the initial condition. For this,
using~\eqref{EQUco}, we see that
\begin{equation}\label{9AHK-2oakak} U(x,0)=
\int_{\mathbb{R}} u(x,y,0)\,dy=\int_{\mathbb{R}} \delta_0(x)\,\delta_0(y)\,dy=
\delta_0(x).\end{equation}
Having checked the initial condition in~\eqref{EQUcoU}, we now aim
at proving the validity of the evolution equation in~\eqref{EQUcoU}.
To this end, it is convenient
to consider
the Fourier-Laplace transform of a function~$v=v(x,t)$, namely
the Fourier transform\index{transform!Fourier-Laplace} in the variable~$x$ combined with the Laplace transform
in the variable~$t$.
Namely, up to normalization constants that we omit,
we define
$$ {\mathcal{E}}_v(\xi,\omega):=
\iint_{ {\mathbb{R}} \times(0,+\infty)} v(x,t) \,e^{-i x\xi-t\omega}\,dx\,dt.$$
We observe that, for every~$x\in{\mathbb{R}}$,
\begin{eqnarray*}&&
\int_0^{+\infty} D^{1/2}_{t,0} U(x,t)\,e^{-t\omega}\,dt \\&=&
\int_0^{+\infty}\left[
\int_0^t \frac{\partial_tU(x,t-\tau)
\;e^{-t\omega}
}{\sqrt{\tau}}\,d\tau\right]\,dt\\&=&
\int_0^{+\infty}\left[
\int_\tau^{+\infty} \frac{\partial_tU(x,t-\tau)
\;e^{-t\omega}
}{\sqrt{\tau}}\,dt\right]\,d\tau\\&=&
\int_0^{+\infty}\left[
\int_\tau^{+\infty} \Big( \partial_t\big( U(x,t-\tau)
\;e^{-t\omega}\big) +\omega U(x,t-\tau)\;e^{-t\omega}
\Big)\,dt\right]\,\frac{d\tau}{\sqrt{\tau}}\\
\\&=&
\int_0^{+\infty}\left[
- U(x,0)\;e^{-\tau\omega} +\int_\tau^{+\infty}\omega U(x,t-\tau)\;e^{-t\omega}
\,dt\right]\,\frac{d\tau}{\sqrt{\tau}}\\
&=&-\sqrt{\frac{\pi}{\omega}}\;U(x,0)+\omega\int_0^{+\infty}\left[
\int_0^{+\infty} U(x,\sigma)\;e^{-(\sigma+\tau)\omega}\,d\sigma
\right]\,\frac{d\tau}{\sqrt{\tau}}\\
&=&-\sqrt{\frac{\pi}{\omega}}\;U(x,0)+\sqrt{\pi\omega}
\int_0^{+\infty} U(x,\sigma)\;e^{-\sigma\omega}\,d\sigma.
\end{eqnarray*}
As a consequence,
\begin{equation}\label{841016810412}
\begin{split}
{\mathcal{E}}_{D^{1/2}_{t,0} U}(\xi,\omega)\,&=\int_{\mathbb{R}}
\left[-\sqrt{\frac{\pi}{\omega}}\;U(x,0)+\sqrt{\pi\omega}
\int_0^{+\infty} U(x,\sigma)\;e^{-\sigma\omega}\,d\sigma\right]
e^{-ix\xi}\,dx\\&=
-\sqrt{\frac{\pi}{\omega}}\;\hat U(\xi,0)+
\sqrt{\pi\omega}\,
{\mathcal{E}}_{U}(\xi,\omega),
\end{split}\end{equation}
where~$\hat U$ denotes the Fourier transform of~$U$ in the variable~$x$.

Moreover, by \eqref{7A-DE}, up to normalizing constants we can write that
\begin{eqnarray*}
{\mathcal{E}}_{(-\Delta)^s_x U}(\xi,\omega)&=&\int_0^{+\infty}
|\xi|^{2s} \,\hat U(\xi,t)\,
e^{-t\omega}\,dt\\&=&|\xi|^{2s}\,{\mathcal{E}}_{U}(\xi,\omega).
\end{eqnarray*}
By this and~\eqref{841016810412}, we see that, to check the validity
of the evolution equation in~\eqref{EQUcoU}, recalling~\eqref{9AHK-2oakak},
up to normalizing constants
we need to establish that
\begin{equation}\label{EQUcoU-F}
|\xi|^{2s}\,{\mathcal{E}}_{U}(\xi,\omega)-\sqrt{\frac{1}{\omega}}+
\sqrt{\omega}\,
{\mathcal{E}}_{U}(\xi,\omega)=0.
\end{equation}
To this end, we consider the Fourier-Laplace transform of the function~$u$
that solves~\eqref{EQUco}. Namely,
we define
\begin{equation}\label{5INT-W} W(\xi,y,\omega):={\mathcal{E}}_u(\xi,y,\omega)=
\iint_{ {\mathbb{R}} \times(0,+\infty)} u(x,y,t) \,e^{-i x\xi-t\omega}\,dx\,dt.\end{equation}
In view of~\eqref{EQUco}, we point out that
\begin{eqnarray*} &&
-\delta_0(y)\,|\xi|^{2s}W(\xi,y,\omega)+\partial^2_yW(\xi,y,\omega)\\
&=&-\delta_0(y)\iint_{ {\mathbb{R}} \times(0,+\infty)}
|\xi|^{2s}\,u(x,y,t) \,e^{-i x\xi-t\omega}\,dx\,dt
+\partial^2_yW(\xi,y,\omega)\\
&=&-\delta_0(y)\int_0^{+\infty} |\xi|^{2s}\,
\hat u(\xi,y,t) \,e^{-t\omega}\,dt
+\partial^2_yW(\xi,y,\omega)\\&=&
{\mathcal{E}}_{-\delta_0(y)\,(-\Delta)^s_x u+\partial^2_y u}(\xi,y,\omega)\\&=&
{\mathcal{E}}_{\partial_tu}(\xi,y,\omega)\\&=&
\iint_{ {\mathbb{R}} \times(0,+\infty)} \partial_t u(x,y,t) \,e^{-i x\xi-t\omega}\,dx\,dt\\
&=&
\iint_{ {\mathbb{R}} \times(0,+\infty)} \Big(
\partial_t \big(u(x,y,t) \,e^{-i x\xi-t\omega}\big)+
\omega u(x,y,t) \,e^{-i x\xi-t\omega}
\Big)
\,dx\,dt\\
&=&
-\int_{ {\mathbb{R}} } 
u(x,y,0)\,e^{-ix\xi}\,dx
+\omega\iint_{ {\mathbb{R}} \times(0,+\infty)}  u(x,y,t) \,e^{-i x\xi-t\omega}
\,dx\,dt\\
&=&-\delta_0(y)+\omega W(\xi,y,\omega).
\end{eqnarray*}
This gives that the map~${\mathbb{R}}\ni y\mapsto W(\xi,y,\omega)$
satisfies
\begin{equation}\label{6gW}
\partial^2_y W=a^2 W-b(2a+c)\delta_0(y)+c\delta_0(y)\,W,\end{equation}
where~$a=a(\omega):=\sqrt\omega$, $b=b(\xi,\omega):=\frac{1}{2\sqrt\omega+|\xi|^{2s}}$
and~$c=c(\xi):=|\xi|^{2s}$.

Hence, fixing~$\xi\in{\mathbb{R}}$ and~$t>0$, one considers~\eqref{6gW}
as an equation for a function of~$y\in{\mathbb{R}}$, in which~$a$, $b$
and~$c$ are coefficients independent of~$y$.

With this in mind, we observe that, if~$a>0$, $b\in{\mathbb{R}}$, 
and~$g(y):=b\,e^{-a|y|}$,
it holds that
\begin{equation}\label{7jALLAPP}
g''(y)=a^2 g(y)-b(2a+c)\delta_0(y)+c\delta_0(y)\,g(y),
\end{equation}
for any~$c\in{\mathbb{R}}$. To check this, let~$\varphi\in C^\infty_0({\mathbb{R}})$.
Then, we compute
\begin{eqnarray*}&&
\int_{\mathbb{R}} \big( g''(y)-a^2 g(y)+b(2a+c)\delta_0(y)-c\delta_0(y)\,g(y)\big)
\varphi(y)\,dy\\ &=&
\int_{\mathbb{R}} g(y)\varphi''(y)\,dy
-a^2\int_{\mathbb{R}}g(y)\varphi(y)\,dy
+b(2a+c)\varphi(0)-cg(0)\varphi(0)\\
&=&b\int_{\mathbb{R}}e^{-a|y|}\varphi''(y)\,dy
-a^2b\int_{\mathbb{R}} e^{-a|y|}\varphi(y)\,dy
+b(2a+c)\varphi(0)-bc\varphi(0)\\
&=&
b\int_{0}^{+\infty} e^{-ay}\varphi''(y)\,dy
+b\int_{-\infty}^0 e^{ay}\varphi''(y)\,dy\\&&\qquad
-a^2b\int_{0}^{+\infty} e^{-ay}\varphi(y)\,dy
-a^2b\int_{-\infty}^0 e^{ay}\varphi(y)\,dy\\&&\qquad
+b(2a+c)\varphi(0)-bc\varphi(0)\\&=&
-b\varphi'(0)+
ab\int_{0}^{+\infty} e^{-ay}\varphi'(y)\,dy
+
b\varphi'(0)
-ab\int_{-\infty}^0 e^{ay}\varphi'(y)\,dy\\&&\qquad
-a^2b\int_{0}^{+\infty} e^{-ay}\varphi(y)\,dy
-a^2b\int_{-\infty}^0 e^{ay}\varphi(y)\,dy\\&&\qquad
+b(2a+c)\varphi(0)-bc\varphi(0)\\&=&
-b\varphi'(0)-
ab \varphi(0)
+
a^2b\int_{0}^{+\infty} e^{-ay}\varphi(y)\,dy\\&&\qquad
+
b\varphi'(0)
-ab\varphi(0)
+a^2b\int_{-\infty}^0 e^{ay}\varphi(y)\,dy\\&&\qquad
-a^2b\int_{0}^{+\infty} e^{-ay}\varphi(y)\,dy
-a^2b\int_{-\infty}^0 e^{ay}\varphi(y)\,dy\\&&\qquad
+b(2a+c)\varphi(0)-bc\varphi(0)\\&=&0,
\end{eqnarray*}
which proves~\eqref{7jALLAPP}.

In the light of~\eqref{7jALLAPP}, we can write a solution of~\eqref{6gW} in the form
$$ W(\xi,y,\omega)=b(\xi,\omega) \,e^{-a(\omega)\,|y|}=
\frac{e^{-\sqrt\omega\,|y|}}{2\sqrt\omega+|\xi|^{2s}}.$$
As a consequence, recalling~\eqref{5INT-U} and~\eqref{5INT-W},
\begin{eqnarray*}
\frac{2}{ \sqrt\omega\big(2\sqrt\omega+|\xi|^{2s}\big) }&=&
\int_{\mathbb{R}}
\frac{e^{-\sqrt\omega\,|y|}}{2\sqrt\omega+|\xi|^{2s}}\,dy\\&=&
\int_{\mathbb{R}} W(\xi,y,\omega)\,dy\\&=&{\mathcal{E}}_U(\xi,\omega).
\end{eqnarray*}
Hence, we see that
$$\big(|\xi|^{2s}+2
\sqrt{\omega}\big)\,{\mathcal{E}}_{U}(\xi,\omega)=\frac{2}{ \sqrt\omega}.$$
With this, we have checked
validity of~\eqref{EQUcoU-F} and thus we have established the fractional equation
in~\eqref{EQUcoU}.
}\end{example}

\begin{example}[Fractional operators and option pricing]

{\rm One of the most common features in mathematical finance is option pricing.
Suppose that a holder wants
to buy some good, say, the ticket to the final match of Champions League
(an ``underlying''\index{underlying} in financial jargon), which
will occur at time~$T$.

The price of this ticket depends on time, and say that
at time~$t\in[0,T]$ this ticket costs~$S_t$
(we are using here a common notation in finance to use the
subindex~$t$ to denote the value of a function at a given time~$t$).
Rather than buying directly the ticket for the price~$S_t$,
he/she can buy an option\footnote{For simplicity,
we are discussing here the case of \emph{European options}\index{options!European},
in which
the holder can only
employ the option exactly at a given time $T$.
Instead, in the {\em American options}\index{options!American}
the holder has the right, but not the obligation, to buy the underlying
at an agreed-upon price at any time~$t\in[0,T]$
(and, of course,
the seller has an obligation to sell the underlying if the option is exercised).} that allow his/her
to buy the ticket at time~$T$ for a fixed price~$K$
(called in jargon ``strike price''\index{strike price}).

The question is: should the holder buy the ticket or the option?
Or, more precisely: what is the value that such an option has in the market?

We suppose that the value~$V$ of the option depends
on the
time $t$ and on the cost of the ticket~$S$, therefore we will write~$V=V(S,t)$.
What is obvious is the final value of the option, since
\begin{equation}\label{FIN}
V(S,T):=\max\{ S-K, 0\}.
\end{equation}
Indeed, at the final time, one can decide to either buy the ticket at a certain price~$S$
or the option at price~$V(S,T)$ and then the ticket at the strike price~$K$,
and these two operations should be equivalent, hence~$S=V(S,T)+K$ (this if~$S\ge K$,
otherwise the option has simply no value, since anyone can just buy
the ticket for a more convenient price, hence confirming~\eqref{FIN}).

Typically, in order to determine the option price, one can prove that it solves some partial differential equation. For instance, one can follow the Black-Scholes model,
and a variation of it, see e.g.~\cite{MR1904936}.
In our framework, a risk-neutral dynamic of the asset (the price of the ticket)
is
given by an exponential model
$$
S_t=S_0\exp(\mu t+\sigma W_t),
$$
where $\mu\in\mathbb{R}$
denotes a drift which
measures the expected yield of the underlying,
$\sigma\in(0,+\infty)$ is a diffusion coefficient
which measures the degree of variation of the trading 
price of the asset (in jargon, ``volatility''),
and $W_t$ is a ``reasonable'' stochastic process
modeling the unpredictable oscillations of the market. 
In the classical case, $W_t$ was taken to be simply the
Brownian motion, but recently fractional Brownian motions
and jump processes have been taken into account to
model possibly different evolutions of the market.

Without going into technical details, following e.g.~\cite{MR2584076},
we will suppose that the stochastic motion
arises from the superposition of a Brownian motion with a jump process.
Concretely, we assume that the ``infinitesimal generator'' of
the process is of the form
\begin{equation}\label{0-0usdjfcnv}
{\mathcal{A}}:= a\partial^2 -b (-\Delta)^s,\qquad{\mbox{ for some }}s\in(0,1),\end{equation}
with~$a$, $b\ge0$,
that is, if we denote by~${\mathbb{E}}$ the expected value, we assume that
\begin{equation}\label{CRAS} \lim_{\tau\to0} \frac{
{\mathbb{E}}\big( f(S_{t+\tau})\big)-f(S_t)
}{\tau}= {\mathcal{A}} f(S_t).\end{equation}
The heuristic rationale of this formula (say, with~$t=0$) is as follows:
in the appropriate scale, the probability
density of a particle traveling under
the superposition of a Brownian motion with a jump process satisfies a heat
equation in which the classical Laplacian is replaced by the operator~${\mathcal{A}}$,
see e.g.~\cite{MR2584076}, hence, if~$S_{t,x}$ denotes the
stochastic evolution starting at~$x$ (i.e. $S_{0,x}=x$), and~$
u(x,t):=
{\mathbb{E}}\big( f(S_{t,x})\big)$ we expect that~$u$ is a solution of
$$ \begin{cases}
\partial_t u(x,t)={\mathcal{A}}u(x,t), & {\mbox{ if }} t>0,\\
u(x,0)=f(x).
\end{cases}$$
Hence we can expect that
\begin{eqnarray*}&&
\lim_{\tau\to0} \frac{{\mathbb{E}}\big( f(S_{\tau,x})\big)-f(S_{0,x})}{\tau}=
\lim_{\tau\to0} \frac{u(x,\tau)-f(x)}{\tau}\\&&\qquad=
\lim_{\tau\to0} \frac{u(x,\tau)-u(x,0)}{\tau}
=\partial_t u(x,0)\\&&\qquad=
{\mathcal{A}}u(x,0)
={\mathcal{A}}f(x)\\&&\qquad={\mathcal{A}}f(S_{0,x}),
\end{eqnarray*}
which can provide a justification for~\eqref{CRAS}.

We also introduce a suitable It\^{o}'s formula, of the form
\begin{equation}\label{ITO}
\frac{d}{dt} f(S_t) = f'(S_t)\,\frac{dS_t}{dt}+{\mathcal{A}}f(S_t).
\end{equation}
To try to justify~\eqref{ITO}, one can notice that the stochastic process
can move ``indifferently'' up or down, say~$S_{t+\tau}-S_t$
have no ``definite sign'', hence
$$ \lim_{\tau\to0} \frac{
{\mathbb{E}}\big( f'(S_t)(S_{t+\tau}-S_t)\big)
}{\tau}=0.$$
That is, formally Taylor expanding, and recalling~\eqref{CRAS},
\begin{eqnarray*}
{\mathcal{A}} f(S_t)&=&
\lim_{\tau\to0} \frac{{\mathbb{E}}\big( f(S_{t+\tau})\big)-f(S_{t+\tau})}{\tau}
+\frac{f(S_{t+\tau})-f(S_t)}{\tau}\\
&=&\lim_{\tau\to0} \frac{{\mathbb{E}}\big( f(S_t)+
f'(S_t)(S_{t+\tau}-S_t)
\big)-\big(
f(S_t)+
f'(S_t)(S_{t+\tau}-S_t)
\big)}{\tau}
+\frac{d}{dt} f(S_t)\\
&=&\lim_{\tau\to0} -\frac{f'(S_t)(S_{t+\tau}-S_t)}{\tau}
+\frac{d}{dt} f(S_t)\\
&=& -f'(S_t)\,\frac{dS_t}{dt}
+\frac{d}{dt} f(S_t),
\end{eqnarray*}
which gives a heuristic justification of~\eqref{ITO}.

Clearly, formula~\eqref{ITO} can be extended to
functions which also depend explicitly on time, thus yielding
\begin{equation}\label{ITO2}
\frac{d}{dt} f(S_t,t) = 
\frac{\partial f}{\partial t}(S_t,t)+
\frac{\partial f}{\partial S}(S_t,t)\,\frac{dS_t}{dt}+{\mathcal{A}}f(S_t,t).
\end{equation}

Now, the idea is to try to determine an equation for the value~$V$
of the option, under reasonable assumptions on the market.
It would be desirable to release~$V$ from any randomness
and uncertainty offered by the market. In a sense, the oscillations of
the market could affect the price~$S_t$ of the ticket at time $t$,
but we would like to know~$V=V(S,t)$ in a way which is independent
of this randomness (only in dependence of the time~$t$
and any possible price of the ticket~$S$), and then evaluate~$V(S_t,t)$
to have the value of the option at time~$t$, for the price~$S_t$ of the ticket.

To do so, we try to build a ``risk-free'' portfolio.
Given the oscillations of the market, the strategy of possessing only
a certain number of options is highly subject to the market uncertainties
and it would be safer for a holder, or for a company, not only
to buy or sell one or more options but also to buy or sell one or more tickets
for the Champions League final. That is, a suitable portfolio should be of the form
\begin{equation}\label{PORT} P(S,t)=V(S,t)+\delta S,\end{equation}
with~$\delta\in{\mathbb{R}}$ giving the number of tickets
to possess with respect to
the option to make the total portfolio as ``stable''
as possible (negative values of~$\delta$ would
correspond to selling tickets, rather than buying them). To choose~$\delta$ in such a way that~$P$ becomes
as close to risk-free as possible, it is desirable to reduce the oscillations
of~$P$ with respect to the variations of~$S$.
To this end, we make the simplifying assumption that~$V(0,t)=0$,
i.e. the value of the option is null if so is the price of the ticket,
and we observe that
$$ V(S,t)=V(S,t)-V(0,t)=\frac{\partial V}{\partial S}(0,t)\,S+O(S^2),$$
that is
$$ V(S,t)-\frac{\partial V}{\partial S}(0,t)\,S=O(S^2).$$
Comparing this with~\eqref{PORT}, we see that
a reasonable possibility to make the portfolio as independent as possible
to the fluctuating value~$S$ is to choose~$\delta:=-\frac{\partial V}{\partial S}(0,t)$
in~\eqref{PORT}, which leads to
\begin{equation}\label{PORT2} P(S,t)=V(S,t)-\frac{\partial V}{\partial S}(0,t)\, S.\end{equation}
Let us make a brief comment on the minus sign in~\eqref{PORT2}.
One can suppose, for instance, that~$V$ is monotone increasing with
respect to~$S$ (the more the ticket costs, the more the option is valuable)
and thus~$\frac{\partial V}{\partial S}(0,t)>0$.
In the setting of~\eqref{PORT2}, this means that, to maintain
a balanced portfolio, if one buys options it is appropriate to sell tickets.

One then makes the ansatz that~\eqref{PORT2}
gives indeed a risk-free portfolio. Under this assumption,
the time evolution of~$P$ is just governed by the interest rate and one can write that
the time evolution of~$P(S_t,t)$ is simply~$P_0\,e^{rt}$, for
some~$r$, $P_0\in{\mathbb{R}}$.
As a consequence,
$$ \frac{d}{dt} P(S_t,t)=\frac{d}{dt}(P_0\,e^{rt})=
r\,P_0\,e^{rt}=r\,P(S_t,t).$$
Hence, recalling~\eqref{PORT2},
\begin{equation}\label{7:AKs}
\frac{d}{dt}\left( V(S_t,t)-\frac{\partial V}{\partial S}(0,t)\, S_t\right)=
r\,V(S_t,t)-r\,\frac{\partial V}{\partial S}(0,t)\, S_t.
\end{equation}
On the other hand, by~\eqref{ITO2},
\begin{eqnarray*}\frac{d}{dt}V(S_t,t) = 
\frac{\partial V}{\partial t}(S_t,t)+
\frac{\partial V}{\partial S}(S_t,t)\,\frac{dS_t}{dt}+{\mathcal{A}}V(S_t,t),\end{eqnarray*}
and so~\eqref{7:AKs} becomes
\begin{equation}\label{67123498dgwfsu}\begin{split}&
\frac{\partial V}{\partial t}(S_t,t)+
\frac{\partial V}{\partial S}(S_t,t)\,\frac{dS_t}{dt}+{\mathcal{A}}V(S_t,t)
-\frac{ d}{dt}\left(\frac{\partial V}{\partial S}(0,t)\, S_t\right)\\ =\;&
r\,V(S_t,t)-r\,\frac{\partial V}{\partial S}(0,t)\, S_t.\end{split}\end{equation}
It is also customary to neglect the
dependence of~$\frac{\partial V}{\partial S}(0,t)$ on~$t$
(say, the values of the options of very cheap tickets
are more or less the same independent on time), and thus
replace the term~$\frac{ d}{dt}\left(\frac{\partial V}{\partial S}(0,t)\, S_t\right)$
with~$ \frac{\partial V}{\partial S}(0,t)\, \frac{ dS_t}{dt}$. With this
approximation, one obtains from~\eqref{67123498dgwfsu}, after
simplifying two terms, that
\begin{equation}\label{BSS}
\frac{\partial V}{\partial t}(S_t,t)+{\mathcal{A}}V(S_t,t)=
r\,V(S_t,t)-r\,\frac{\partial V}{\partial S}(0,t)\, S_t.\end{equation}
Recalling~\eqref{0-0usdjfcnv}, when~$b=0$ one obtains
from~\eqref{BSS} the classical Black-Scholes equation
$$
\frac{\partial V}{\partial t}+rS_t\frac{\partial V}{\partial S_t}+\frac{\sigma^2}{2}S_t^2\frac{\partial^2 V}{\partial S_t^2}-rV=0\quad\text{in}\quad\mathbb{R}\times(0,T].
$$
In general, when~$b\ne0$ (and possibly~$a=0$)
one obtains in~\eqref{BSS} a nonlocal evolution equation
of fractional type. Such an equation
is complemented with the terminal condition in~\eqref{FIN}.}
\end{example}

\begin{example}[Complex analysis and Hilbert transform]{\rm
Given a (nice) function~$u:{\mathbb{R}}\to{\mathbb{R}}$,
the Hilbert transform\index{transform!Hilbert} of~$u$ is defined by
$$ H_u (t):=-{\frac {1}{\pi }}\int _{-\infty }^{\infty }{\frac {u(t)-u(\tau ) }{t-\tau }}\,d\tau .$$
We observe that
$$ \int _{|t-\tau|\in (r,R) }{\frac {u(t) }{t-\tau }}\,d\tau =u(t)\,
\int _{|\vartheta|\in (r,R) }{\frac {d\vartheta }{\vartheta }}=
u(t)\,\left(
\int _{r }^R{\frac {d\vartheta }{\vartheta }}+
\int _{-R}^{-r}{\frac {d\vartheta }{\vartheta }}\right)=
0,$$
for all~$0<r<R$. Hence,
in the principal value sense, after cancellations,
one can also write
$$ H_u (t):={\frac {1}{\pi }}\int _{-\infty }^{\infty }{\frac {u(\tau ) }{t-\tau }}\,d\tau .$$
Among the others, a natural application of
the Hilbert transform occurs in complex analysis.
Indeed, identifying points~$(x,y)\in{\mathbb{R}}\times[0,+\infty)$
of the real upper half-plane with points~$x+iy\in\{ z\in{\mathbb{C}} {\mbox{ with }} \Im z\ge0\}$
of the complex upper half-plane, one has that, {\em
on the boundary of the half-plane, the two harmonic conjugate functions of a holomorphic function are
related via the Hilbert transform}. More precisely,
if~$f$ is holomorphic in the complex upper half-plane,
we write~$z=x+iy$ with~$x\in{\mathbb{R}}$ and~$y\ge0$, and
$$ f(z)=u(x,y)+i v(x,y).$$
We also set~$u_0(x):=u(x,0)$ and~$v_0(x):=v(x,0)$.
Then, under natural regularity assumptions, we have that
\begin{equation}\label{HIL}
v_0(x)=H_{u_0}(x),\qquad{\mbox{ for all }}x\in{\mathbb{R}}.
\end{equation}
For rigorous complex analysis
results, we refer e.g.
to Theorem~93 on page~125 of~\cite{MR942661} and
to Theorems~3 and~4
on pages~77-78 of~\cite{MR1363489}. See also Sections 2.6--2.9
in~\cite{MR1363489} for a number of concrete applications of
the Hilbert transform.
We sketch two arguments to establish~\eqref{HIL},
one based on classical complex methods, and one exploiting
fractional calculus.

The first argument goes as follows. For any~$z=x+iy$ with~$y>0$, we consider the Hilbert transform of~$u_0$, up to normalization constants,
defined with a complex variable, namely we set
$$ F(z):=\frac1{\pi i} \int_{\mathbb{R}}
\frac{u_0(t)}{t-z}\,dt.$$
Then, $F$ is holomorphic in the upper half-plane. Moreover,
\begin{equation}\label{DeSTA}
\begin{split}
F(z)\,&=\frac1{\pi i} \int_{\mathbb{R}}
\frac{u_0(t)}{t-x-iy}\,dt
\\ &=\frac1{\pi } \int_{\mathbb{R}}
\frac{u_0(t)\, (i(x-t)+y)}{(t-x)^2+y^2}\,dt.
\end{split}\end{equation}
We let~$\tilde u(x,y):= \Re F(x+iy)$
and~$\tilde v(x,y):= \Im F(x+iy)$.
Then, using the substitution~$w:=(t-x)/y$,
\begin{eqnarray*}
\lim_{y\searrow0}\tilde u(x,y)&=&
\lim_{y\searrow0}
\frac1{\pi } \int_{\mathbb{R}}
\frac{u_0(t)\, y}{(t-x)^2+y^2}\,dt
\\
&=&
\lim_{y\searrow0}
\frac1{\pi } \int_{\mathbb{R}}
\frac{u_0(x+wy)}{w^2+1}\,dw
\\&=&
\frac{u_0(x)}{\pi } \int_{\mathbb{R}}
\frac{dw}{w^2+1}
\\&=& u_0(x).
\end{eqnarray*}
That is, the real part of~$F$ coincides with
the harmonic extension of~$u_0$ to the upper half-plane, 
up to harmonic functions vanishing on the trace.
Therefore (reducing to finite energy solutions), we suppose that~$\tilde u=u$. Since~$\tilde v$ and~$v$ are the 
conjugate harmonic functions of~$\tilde u$ and~$u$, respectively, from the Cauchy–Riemann equations we thereby find that
\begin{eqnarray*}&&
0=\partial_y (\tilde u-u)=-\partial_x (\tilde v-v)\\
{\mbox{and }}&&0=\partial_x (\tilde u-u)=\partial_y (\tilde v-v).
\end{eqnarray*}
Hence (restricting to functions with finite energy), we have that~$\tilde v=v$.

{F}rom these observations and~\eqref{DeSTA}, we find that
\begin{eqnarray*}
v_0(x)&=&\lim_{y\searrow0} \tilde v(x,y)
\\&=& \lim_{y\searrow0}
\frac1{\pi } \int_{\mathbb{R}}
\frac{u_0(t)\, (x-t)}{(t-x)^2+y^2}\,dt
\\&=& \lim_{y\searrow0}
\frac1{\pi } \int_{\mathbb{R}}
\frac{(u_0(t)-u_0(x))\, (x-t)}{(t-x)^2+y^2}\,dt
\\&=&
\frac1{\pi } \int_{\mathbb{R}}
\frac{(u_0(t)-u_0(x))\, (x-t)}{(t-x)^2}\,dt
\\&=&
-\frac1{\pi } \int_{\mathbb{R}}
\frac{u_0(t)-u_0(x)}{t-x}\,dt,\end{eqnarray*}
that gives~\eqref{HIL}.

We now present another argument to establish~\eqref{HIL}
based on fractional calculus.
Since~$u$ is harmonic in the upper half-plane,
using the Fourier transform in the~$x$ variable,
and dropping normalization constants for the
sake of simplicity,
we see that, for all~$y\in(0,+\infty)$,
$$ 0 ={\mathcal{F}} (\Delta u)(\xi,y)=-|\xi|^2 \hat u(\xi,y)+\partial^2_y \hat u(\xi,y),$$
and thus~$\hat u(\xi,y)= C(\xi)\,e^{-|\xi| y}$, with~$C(\xi)\in{\mathbb{R}}$.
This gives that
$$ \partial_y \hat u(\xi,0)= -|\xi|\,C(\xi) = -|\xi| \,\hat u(\xi,0)=-|\xi|\,\hat u_0(\xi).
$$
This and~\eqref{7A-DE} give that
$$ \partial_{y} u(x,0)=
{\mathcal{F}}^{-1}\Big(
\partial_{y}\hat u(\xi,0) \Big)
=
-{\mathcal{F}}^{-1}\Big(|\xi|\,\hat u_0(\xi)\Big)=
-\sqrt{-\Delta}u_0(x).$$
Hence, recalling~\eqref{7A-DE2},
dropping normalization constants for the
sake of simplicity, and using a parity cancellation, we see that
\begin{eqnarray*}
\partial_{y} u(x,0)&=&\frac12\,
\int_{{\mathbb{R}}}\frac{u_0(x+\theta)+
u_0(x-\theta)-2u_0(x)}{\theta^{2}}\,d\theta\\
&=&\lim_{\delta\searrow0}
\int_{{\mathbb{R}}\setminus(-\delta,\delta)}\frac{u_0(x+\theta)-u_0(x)}{\theta^{2}}\,d\theta\\
&=& \lim_{\delta\searrow0}
\int_{{\mathbb{R}}\setminus(-\delta,\delta)}\frac{u_0(x+\theta)
-u_0(x)-u_0'(x)\theta\Psi(\theta)
}{\theta^{2}}\,d\theta\\
&=& \int_{{\mathbb{R}}}\frac{u_0(x+\theta)
-u_0(x)-u_0'(x)\theta\Psi(\theta)
}{\theta^{2}}\,d\theta\\
&=& \int_{{\mathbb{R}}}\frac{u_0(\tau)
-u_0(x)
-u_0'(x)(\tau-x)\Psi(\tau-x)
}{(\tau-x)^{2}}\,d\tau,
\end{eqnarray*}
being~$\Psi\in C^\infty_0([-2,2], [0,1])$ an even function such that~$
\Psi=1$ in $[-1,1]$. Hence, from the Cauchy–Riemann equations,
$$ -\partial_{x}v_0(x)=-\partial_{x} v(x,0)=
\int_{{\mathbb{R}}}\frac{u_0(\tau)
-u_0(x)
-u_0'(x)(\tau-x)\Psi(\tau-x)
}{(\tau-x)^{2}}\,d\tau.$$
As a consequence,
\begin{equation}\label{701-39485x}
\begin{split}-
v_0(x)\,&= -\int_{-\infty}^x \partial_{x}v_0(t)\,dt\\&=
\int_{-\infty}^x\left[
\int_{{\mathbb{R}}}\frac{u_0(\tau)
-u_0(t)
-u_0'(t)(\tau-t)\Psi(\tau-t)
}{(\tau-t)^{2}}\,d\tau
\right]\,dt\\&=\lim_{R\to+\infty}
\int_{-\infty}^x\left[
\int_{-R+t}^{R+t}\frac{u_0(\tau)
-u_0(t)
-u_0'(t)(\tau-t)\Psi(\tau-t)
}{(\tau-t)^{2}}\,d\tau
\right]\,dt.
\end{split}\end{equation}
By an integration by parts, we notice that
\begin{eqnarray*}&&
\int_{-\infty}^x \frac{u_0(\tau)
-u_0(t) -u_0'(t)(\tau-t)\Psi(\tau-t)
}{(\tau-t)^{2}}\,dt\\
&=&
\int_{-\infty}^x \Big( u_0(\tau)
-u_0(t)
-u_0'(t)(\tau-t)\Psi(\tau-t)
\Big)\,
\frac{d}{dt}\left(\frac{1}{\tau-t}\right)\,dt\\&
=&\frac{ u_0(\tau)
-u_0(x)-u_0'(x)(\tau-x)\Psi(\tau-x)
}{\tau-x}\\ &&\quad+
\int_{-\infty}^x \frac{
u_0'(t)+
u_0''(t)(\tau-t)\Psi(\tau-t)-
u_0'(t)\Psi(\tau-t)-
u_0'(t)(\tau-t)\Psi'(\tau-t)
}{\tau-t}\,dt\\
&=&\frac{ u_0(\tau)
-u_0(x)}{\tau-x}
-u_0'(x)\Psi(\tau-x)
\\ &&\quad+
\int_{-\infty}^x 
\left(
\frac{u_0'(t)}{\tau-t}+
u_0''(t)\Psi(\tau-t)-\frac{
u_0'(t)\Psi(\tau-t)}{\tau-t}-
u_0'(t)\Psi'(\tau-t)\right)\,dt\\
&=&\frac{ u_0(\tau)
-u_0(x)}{\tau-x}
-u_0'(x)\Psi(\tau-x)
\\ &&\quad+
\int_{-\infty}^x 
\left(
\frac{u_0'(t)-u_0'(t)\Psi(\tau-t)}{\tau-t}+\frac{d}{dt}\Big(
u_0'(t)\Psi(\tau-t)\Big)\right)\,dt\\
&=&\frac{ u_0(\tau)
-u_0(x)}{\tau-x}
+
\int_{-\infty}^x u_0'(t)\Phi(\tau-t)\,dt,
\end{eqnarray*}
where we defined the odd function
$$ \Phi(r):=\frac{1-\Psi(r)}{r}.$$
By inserting this information into~\eqref{701-39485x}
and exchanging integrals, one obtains that
\begin{equation*}
\begin{split}-
v_0(x)\,&=\lim_{R\to+\infty}
\int_{-R+t}^{R+t}\left[
\frac{ u_0(\tau)
-u_0(x)}{\tau-x}
+
\int_{-\infty}^x u_0'(t)\Phi(\tau-t)\,dt
\right]\,d\tau
\\ &=\int_{\mathbb{R}}\frac{ u_0(\tau)
-u_0(x)}{\tau-x}\,d\tau
+\lim_{R\to+\infty}
\int_{-\infty}^x u_0'(t)\;\left[ \int_{-R}^{R} \Phi(\theta)\,d\theta\right]\,dt
\\ &=\int_{\mathbb{R}}\frac{ u_0(\tau)
-u_0(x)}{\tau-x}\,d\tau
.\end{split}\end{equation*}
This gives~\eqref{HIL}
(up to the normalizing constant that we have neglected).
}\end{example}

\begin{example}[Caputo derivatives and the fractional Laplacian]
{\rm The time-fractional diffusion that we model by the Caputo derivative
has fundamental differences with respect to the space-fractional diffusion
driven by the fractional Laplacian, since the latter possesses many invariances,
such as the ones under rotations and translations, that are not valid in the time-fractional
setting given by the Caputo derivative, in view of a memory effect which clearly distinguishes
between ``the past'' and ``the future'' and determines 
the ``time's arrow''.

On the other hand, the sum of two Caputo derivatives with
opposite time directions reduces to the fractional Laplacian,
as we now discuss in detail.

Let $\alpha\in(0,1)$ and $u$ sufficiently smooth. An integration by parts\footnote{As a historical
remark, we observe that
formulas such as in~\eqref{eq:fromcaptomar} naturally relate different
definitions of time-fractional derivatives, such as the Caputo derivative
and the Marchaud\index{fractional!derivative!Marchaud} fractional derivative.}
allows us to write the (left) Caputo derivative as
\begin{equation}
\label{eq:fromcaptomar}
D^{\alpha}_{t,a,+}u(t):=\frac{u(t)}{\Gamma(1-\alpha)(t-a)^{\alpha}}+\frac{\alpha}{\Gamma(1-\alpha)}\int_a^t\frac{u(t)-u(\tau)}{(t-\tau)^{\alpha+1}} d\tau.
\end{equation}
Choosing $a:=-\infty$, formula \eqref{eq:fromcaptomar} reduces to
\begin{equation}
\label{eq:marleft}
D^{\alpha}_{t,-\infty,+}u(t)=\frac{\alpha}{\Gamma(1-\alpha)}\int_{-\infty}^t\frac{u(t)-u(\tau)}{(t-\tau)^{\alpha+1}} d\tau=\frac{\alpha}{\Gamma(1-\alpha)}\int_0^{+\infty}\frac{u(t)-u(t-\tau)}{\tau^{\alpha+1}}d\tau.
\end{equation}
Since we will be interested in the core of this monograph in given fractional
derivatives with a prescribed arrow of time, we will focus
on this type of definition (and, in fact, also on higher order ones,
as in~\eqref{defcap}). Nevertheless, by an inversion of the time's arrow,
one can also define a notion of right Caputo derivative, which, when $a:=+\infty$, can be written as
\begin{equation}
\label{eq:marright}
D^{\alpha}_{t,+\infty,-}u(t):=\frac{\alpha}{\Gamma(1-\alpha)}\int_t^{+\infty}\frac{u(t)-u(\tau)}{(\tau-t)^{\alpha+1}} d\tau=\frac{\alpha}{\Gamma(1-\alpha)}\int_0^{+\infty}\frac{u(t)-u(t+\tau)}{\tau^{\alpha+1}}d\tau.
\end{equation}
See the forthcoming footnote on page~\pageref{NOTION}
for further comments on the notion of left and right fractional derivatives.
Summing up \eqref{eq:marleft} and \eqref{eq:marright}, and dropping normalization constants
for simplicity, we have that
\begin{eqnarray*}
D^{\alpha}_{t,-\infty,+}u(t)+D^{\alpha}_{t,+\infty,-}u(t)&=&
\int_0^{+\infty}\frac{2u(t)-u(t+\tau)-u(t-\tau)}{\tau^{\alpha+1}}d\tau \\
&=&\frac12\,\int_{-\infty}^{+\infty}\frac{2u(t)-u(t+\tau)-u(t-\tau)}{|\tau|^{\alpha+1}}d\tau \\
&=&\left(-\Delta_t\right)^{\alpha/2}u(t),
\end{eqnarray*}
where we used the integral representation of the fractional Laplacian
(recall~\eqref{7A-DE2}) and the
obvious one-dimensional notation
$$ \Delta_t:=\frac{\partial^2}{\partial t^2}.$$
Therefore, the sum of left and right Caputo derivatives with initial points $-\infty$ and $+\infty$ respectively, gives, up to a multiplicative constant, the one-dimensional fractional Laplacian of fractional order $\frac{\alpha}{2}$.
%%% It is interesting to remark that the equality \eqref{capandlap} provides an alternative representation of the anisotropic fractional Laplacian
%%% $$\int_{\mathbb{R}^n} (2u(x)-u(x+y)-u(x-y))d\mu(y),$$
%%% where we choose $\displaystyle d\mu(y):=\sum_{i=1}^n\frac{d\delta_{e_i}(y)+d\delta_{-e_i}(y)}{|y|^{n+s}}$, and, for any $i=1,\ldots,n$, $e_i$ denotes the $i$-th element of the canonical basis of $\mathbb{R}^n$.
%%% Indeed
%%% $$\int_{\mathbb{R}^n} (2u(x)-u(x+y)-u(x-y))d\mu(y)=\sum_{i=1}^n \left(-\frac{\partial^2}{\partial {x_i}^2}\right)^{s/2}u(x)=\sum_{i=1}^n \left({}_{x_i}D^{s}_{-\infty+}+{}_{x_i}D^{s}_{+\infty-}\right)u(x)$$
}\end{example}
\medskip

Other applications of fractional equations will be discussed in Appendix~\ref{APPEA}.

\chapter{Main results}\label{DUEE}
\label{s:first}
After having discussed in detail several motivations for fractional equations
in Chapter~\ref{WHY},
we begin here the mathematically rigorous part of this monograph,
and we start presenting the main original mathematical
results contained in this book and their relation with the existing literature.
\medskip

In this work we prove the local density of functions
which annihilate a linear operator built by classical and 
fractional derivatives, both in space and time.\medskip

Nonlocal operators of fractional type
present a variety of challenging problems in pure mathematics,
also in connections with long-range phase transitions and nonlocal
minimal surfaces,
and are nowadays commonly exploited in a large number of models
describing complex phenomena related
to anomalous diffusion and boundary reactions
in physics, biology and material sciences (see e.g.~\cite{claudia}
for several examples, for instance in atom dislocations in crystals and
water waves models).
Furthermore, anomalous diffusion in the space variables can be seen
as the natural counterpart of discontinuous
Markov processes, thus providing important connections
with problems in probability and statistics, and several applications to
economy and finance (see e.g.~\cite{MR0242239,MR3235230} for pioneer works relating
anomalous diffusion and financial models).\medskip

On the other hand, the development of time-fractional derivatives
began at the end of the seventeenth century, also in view of
contributions by mathematicians
such as Leibniz, Euler, Laplace, Liouville, Abel, Heaviside, and many others,
see e.g.~\cite{MR2624107, MR0444394, MR860085, MR125162492, ferrari} and the references therein for several
interesting scientific and historical discussions.
{F}rom the point of view of the applications, time-fractional derivatives 
naturally provide a model to comprise memory effects in the description of the
phenomena under consideration. \medskip

In this work, the time-fractional derivative will be mostly described
in terms of the so-called
Caputo fractional derivative (see~\cite{MR2379269}),  which
induces a natural ``direction'' in
the time variable, distinguishing between ``past'' and ``future''. In particular, the time direction encoded in this setting
allows the analysis of ``non anticipative systems'', namely phenomena in which
the state at a given time depends on past events, but not on future ones.
The Caputo derivative is also related to other types of time-fractional derivatives,
such as the Marchaud fractional derivative, which has
applications in modeling anomalous time
diffusion, see e.g.~\cite{MR3488533, AV, ferrari}.
See also~\cite{MR1219954, MR1347689} for more details on fractional
operators and several applications.\medskip

Here, we will take advantadge of the nonlocal structure
of a very general linear operator containing fractional derivatives
in some variables (say, either time, or space, or both),
in order to approximate, in the smooth sense and with arbitrary precision,
any prescribed function. Remarkably, {\em no structural assumption}
needs to be taken on the prescribed function: therefore
this approximation property reveals a {\em truly nonlocal behaviour},
since it is
in contrast with the rigidity of the functions that lie in the kernel
of classical linear
operators (for instance, harmonic functions cannot approximate 
a function with interior maxima or minima, functions with null first derivatives 
are necessarily constant, and so on).

The approximation
results with solutions of nonlocal operators have been first introduced
in~\cite{MR3626547}
for the case of the fractional Laplacian,
and since then widely studied under different perspectives,
including harmonic analysis,
see~\cite{MR3774704, 2016arXiv160909248G, 2017arXiv170804285R, 2017arXiv170806294R, 2017arXiv170806300R}.
The approximation result for the
one-dimensional case of a fractional derivative of Caputo type
has been considered in~\cite{MR3716924, CDV18}, and
operators involving classical time derivatives and additional classical derivatives
in space have been studied in~\cite{DSV1}.
\medskip

The great flexibility of solutions of fractional problems established
by this type of approximation results
has also consequences that go beyond
the purely mathematical curiosity. 
For example, these results 
can be applied to study the evolution of
biological populations, showing how a nonlocal hunting or dispersive
strategy can be 
more convenient than one based on classical diffusion,
in order to avoid waste of resources and optimize the search for food
in sparse environment, see~\cite{MR3590678, MR3579567}.
Interestingly, the theoretical descriptions
provided in this setting can be compared with a series of concrete
biological data and real world experiments, confirming
anomalous diffusion behaviours in many biological species, see~\cite{ALBA}.

\medskip

Another interesting application of time-fractional derivatives
arises in neuroscience, for instance in view of the anomalous diffusion
which has been experimentally measured in neurons, see e.g.~\cite{SANTA}
and the references therein.
In this case, the anomalous diffusion could be seen as the effect
of the highly ramified structure of the biological cells
taken into account, see~\cite{comb, DV1}. 
\medskip

In many applications, it is also natural to consider the case in
which different types of diffusion take 
place in different variables: for instance, classical diffusion
in space variables could be
naturally combined to anomalous diffusion with respect to variables
which take into account genetical information, see~\cite{GEN1, GEN2}.
\medskip

Now, to state the main original results of this work, we introduce some notation.
In what follows, we will denote the ``local variables''
with the symbol $x$, the ``nonlocal variables''
with $y$, the ``time-fractional variables''
with $t$.
Namely, we consider the variables
\begin{equation}\label{1.0}\begin{split}
&x=\left(x_1,\ldots,x_n\right)\in\mathbb{R}^{p_1}\times\ldots\times\mathbb{R}^{p_n},
\\&
y=\left(y_1,\ldots,y_M\right)\in\mathbb{R}^{m_1}
\times\ldots\times\mathbb{R}^{m_M}\\
{\mbox{and }}\;&
t=\left(t_1,\ldots,t_l\right)\in\mathbb{R}^l,\end{split}\end{equation}
for some $p_1,\dots,p_n$, $M$, $m_1,\dots,m_M$, $l \in\mathbb{N}$, and we let
$$\left(x,y,t\right)\in\mathbb{R}^N,\qquad{\mbox{
where }}\;N:=p_1+\ldots+p_n+m_1+\ldots+m_M+l.$$ 
When necessary, we will use the notation $B_R^k$ to denote the $k$-dimensional
ball of radius $R$, centered at the origin in $\mathbb{R}^k$; otherwise, when there are no ambiguities, we will use the usual notation $B_R$.

Fixed $r=\left(r_1,\ldots,r_n\right)\in\mathbb{N}^{p_1}\times\ldots\times\mathbb{N}^{p_n}$,
with~$|r_i|\ge1$ for each~$i\in\{1,\dots, n \}$,
and~$\XA=\left(\XA_1,\ldots,\XA_n\right)\in\mathbb{R}^n$, we consider the local operator
acting on the variables~$x=(x_1,\dots,x_n)$ given by
\begin{equation}\label{ILPOAU-1}
\mathfrak{l}:=\sum_{i=1}^n {\XA_i\partial^{r_i}_{x_i}}.
\end{equation}
where the multi-index notation has been used.

Furthermore, given $\XB=\left(\XB_1,\ldots,\XB_M\right)\in\mathbb{R}^M$ and
$s=\left(s_1,\ldots,s_M\right)\in\left(0,+\infty\right)^M$, we consider the operator
\begin{equation}\label{ILPOAU-2}
\mathcal{L}:=\sum_{j=1}^M {\XB_j(-\Delta)^{s_j}_{y_j}},
\end{equation}
where each operator~$(-\Delta)^{s_j}_{y_j}$
denotes the fractional Laplacian\index{fractional!Laplacian} of order $2s_j$ acting
on the set of space variables~$y_j\in\mathbb{R}^{m_j}$. More precisely,
for any~$j\in\{1,\dots,M\}$,
given $s_j>0$ and $h_j\in\mathbb{N}$ with $h_j:=\min_{q_j\in\mathbb{N}}$ such that $s_j\in(0,q_j)$,
in the spirit of~\cite{AJS2}, we consider
the operator
\begin{equation}
\label{nonlocop}
(-\Delta)^{s_j}_{y_j}u\left(x,y,t\right):=\int_{\mathbb{R}^{m_j}} 
{\frac{\left(\delta_{h_j} u\right)
\left(x,y,t,Y_j\right)}{|Y_j|^{m_j+2s_j}} \,dY_j},
\end{equation}
where
\begin{equation}\label{898989ksdc}
\left(\delta_{h_j} u\right)\left(x,y,t,Y_j\right):=
\sum_{k=-h_j}^{h_j} {\left(-1\right)^k \binom{2h_j}{h_j-k}u
\left(x,y_1,\ldots,y_{j-1},y_j+kY_j,y_{j+1},\ldots,y_M,t\right)}.\end{equation}

In particular, when $h_j:=1$, this setting comprises the case
of the fractional Laplacian~$\left(-\Delta\right)^{s_j}_{y_j}$ of order~$2s_j\in(0,2)$, given by
\begin{equation*}\begin{split}
\left(-\Delta\right)^{s_j}_{y_j}u\left(x,y,t\right) 
&\, := c_{m_j,s_j}\; 
\int_{\mathbb{R}^{m_j}} 
\Big(2u(x,y,t)-
u(x,y_1,\ldots,y_{j-1},y_j+Y_j,y_{j+1},\ldots,y_M,t)\\&\qquad\qquad-
u(x,y_1,\ldots,y_{j-1},y_j-Y_j,y_{j+1},\ldots,y_M,t)
\Big)\;
\frac{dY_j}{|Y_j|^{m_j+2s_j}},\end{split}
\end{equation*}
where $s_j\in(0,1)$ and $c_{m_j,s_j}$ denotes a multiplicative normalizing constant
(see e.g. formula~(3.1.10) in~\cite{claudia}).

It is interesting to recall that
if $h_j=2$ and $s_j=1$ the setting in~\eqref{nonlocop}
provides a nonlocal representation for the classical Laplacian,
see \cite{AV}.
\medskip

In our general framework,
we take into account also nonlocal operators of time-fractional type.
To this end,
for any~$\alpha>0$, letting $k:=[\alpha]+1$ and $a\in\mathbb{R}\cup\left\{-\infty\right\}$,
one can introduce the left\footnote{In the literature, one often finds also \label{NOTION}
the notion of right Caputo fractional derivative, defined for $t<a$ by 
$$ \frac{(-1)^k}{\Gamma\left(k-\alpha\right)}\int^{a}_{t}
\frac{\partial_{t}^{k} u(\tau)}{\left(\tau-t\right)^{\alpha-k+1}}
\, d\tau.$$
Since the right time-fractional derivative boils down to the left one
(by replacing~$t$ with~$2a-t$), in this work we focus only
on the case of left derivatives.

Also, though there are several time-fractional derivatives that are studied
in the literature under different perspectives,
we focus here on the Caputo derivative, since it possesses well-posedness properties
with respect to classical initial
value problems, differently than other time-fractional derivatives, such as the
Riemann-Liouville derivative\index{fractional!derivative!Riemann-Liouville},
in which the initial value setting involves data containing
derivatives of
fractional order.}
Caputo fractional derivative\index{fractional!derivative!Caputo} of order $\alpha$
and initial point\index{initial point} $a$, defined, for~$t>a$,
as
\begin{equation}
\label{defcap}
D_{t,a}^{\alpha}u(t):=\frac{1}{\Gamma(k-\alpha)}\int_a^t \frac{\partial_{t}^{k} u\left(\tau\right)}{(t-\tau)^{\alpha-k+1}} d\tau,
\end{equation}
where\footnote{For notational simplicity, we will often denote~$\partial_t^k u
=u^{(k)}$.} $\Gamma$ denotes the Euler's Gamma function. 

%%%%			For sufficiently smooth functions, a simple integration by parts allows to write the Caputo fractional derivative, when $\alpha\in(0,1)$, in the following equivalent form
%%%%			\begin{equation}
%%%%			\Gamma(1-\alpha)D_{t,a}^\alpha u(t):=\frac{u(t)-u(a)}{(t-a)^\alpha}+\alpha\int_a^t \frac{u(t)-u(\tau)}{(t-\tau)^{\alpha+1}} d\tau;
%%%%			\end{equation}
%%%%			notice that this definition takes into account only the values $t>a$, but if we assume that $u(t)=u(a)$ for any $t<a$, then we have
%%%%			\begin{equation}
%%%%			\begin{split}
%%%%			\label{march}
%%%%			&\frac{u(t)-u(a)}{(t-a)^\alpha}+\alpha\int_a^t \frac{u(t)-u(\tau)}{(t-\tau)^{\alpha+1}} d\tau=\alpha\int_{-\infty}^a \frac{u(t)-u(a)}{(t-\tau)^{\alpha+1}} d\tau+\alpha\int_a^t \frac{u(t)-u(\tau)}{(t-\tau)^{\alpha+1}} d\tau= \\
%%%%			&\alpha\int_{-\infty}^a \frac{u(t)-u(\tau)}{(t-\tau)^{\alpha+1}} d\tau+\alpha\int_a^t \frac{u(t)-u(\tau)}{(t-\tau)^{\alpha+1}} d\tau=\alpha\int_{-\infty}^t \frac{u(t)-u(\tau)}{(t-\tau)^{\alpha+1}} d\tau.
%%%%			\end{split}
%%%%			\end{equation}
%%%%			This formulation of \eqref{defcap} is better known as Marchaud fractional derivative with initial point $-\infty$; notice that \eqref{march} recalls the structure of the one-variable fractional Laplacian, but in this case, for the well-posedness of the boundary value problems, we need only a left-sided prescription of the solution. \\
In this framework,
fixed $\XC=\left(\XC_1,\ldots,\XC_l\right)\in\mathbb{R}^l$,
$\alpha=(\alpha_1,\dots,\alpha_l)\in(0,+\infty)^l$
and~$a=(a_1,\dots,a_l)\in(\mathbb{R}\cup\left\{-\infty\right\})^l$, 
we set
\begin{equation}\label{ILPOAU-3}
\mathcal{D}_a:=\sum_{h=1}^l \XC_h \,D_{t_h, a_h}^{\alpha_h}\,.
\end{equation} 
Then, in the notation introduced in~\eqref{ILPOAU-1}, \eqref{ILPOAU-2}
and~\eqref{ILPOAU-3}, we consider here the superposition of the
local, the space-fractional, and the time-fractional operators, that is, we set
\begin{equation}\label{1.6BIS}
\Lambda_a:=\mathfrak{l}+\mathcal{L}+\mathcal{D}_a.
\end{equation}
With this, the statement of our main result goes as follows:

\begin{theorem}\label{theone}
Suppose that 
\begin{equation}\label{NOTVAN}\begin{split}&
{\mbox{either there exists~$i\in\{1,\dots,M\}$ such that~$\XB_i\ne0$
and~$s_i\not\in{\mathbb{N}}$,}}\\
&{\mbox{or there exists~$i\in\{1,\dots,l\}$ such that~$\XC_i\ne0$ and $\alpha_i\not\in{\mathbb{N}}$.}}\end{split}
\end{equation}
Let $\ell\in\mathbb{N}$, $f:\mathbb{R}^N\rightarrow\mathbb{R}$,
with $f\in C^{\ell}\big(\overline{B_1^N}\big)$. Fixed $\epsilon>0$,
there exist
\begin{equation}\label{EX:eps}\begin{split}&
u=u_\epsilon\in C^\infty\left(B_1^N\right)\cap C\left(\mathbb{R}^N\right),\\
&a=(a_1,\dots,a_l)=(a_{1,\epsilon},\dots,a_{l,\epsilon})
\in(-\infty,0)^l,\\ {\mbox{and }}\quad&
R=R_\epsilon>1\end{split}\end{equation} such that
\begin{equation}\label{MAIN EQ}\left\{\begin{matrix}
\Lambda_a u=0 &\mbox{ in }\;B_1^N, \\
u=0&\mbox{ in }\;\mathbb{R}^N\setminus B_R^N,
\end{matrix}\right.\end{equation}
and
\begin{equation}\label{IAzofm}
\left\|u-f\right\|_{C^{\ell}(B_1^N)}<\epsilon.
\end{equation}
\end{theorem}

We observe that the initial points of the Caputo type operators
in Theorem~\ref{theone} also depend on~$\epsilon$, as detailed in~\eqref{EX:eps}
(but the other parameters, such as the orders of the operators
involved, are fixed arbitrarily).\medskip

We also stress that condition~\eqref{NOTVAN} requires that
the operator~$\Lambda_a$ contains at least one nonlocal operator
among its building blocks in~\eqref{ILPOAU-1}, \eqref{ILPOAU-2}
and~\eqref{ILPOAU-3}. This condition cannot be avoided, since
approximation results in the same spirit of Theorem~\ref{theone}
cannot hold for classical differential operators.\medskip

Theorem~\ref{theone} comprises, as particular cases,
the nonlocal approximation results established in the recent literature of this
topic. Indeed, when
\begin{eqnarray*}
&&\XA_1=\dots=\XA_n=\XB_1=\dots=\XB_{M-1}=\XC_1=\dots=\XC_l=0,\\
&& \XB_M=1,\\
{\mbox{and }}&& s\in(0,1)
\end{eqnarray*}
we see that Theorem~\ref{theone} recovers the main result
in~\cite{MR3626547}, giving the local density of $s$-harmonic functions
vanishing outside a compact set.\medskip

Similarly, when
\begin{eqnarray*}&&
\XA_1=\dots=\XA_n=\XB_1=\dots=\XB_{M}=\XC_1=\dots=\XC_{l-1}=0,\\
&&\XC_l=1,\\{\mbox{and }}&&\mathcal{D}_a=D_{t,a}^{\alpha},
\quad{\mbox{ for some~$\alpha>0$, $a<0$}}\end{eqnarray*}
we have that
Theorem~\ref{theone} reduces to
the main results in~\cite{MR3716924} for~$\alpha\in(0,1)$
and~\cite{CDV18} for~$\alpha>1$, in which such approximation result
was established for Caputo-stationary functions, i.e, functions that annihilate
the Caputo fractional derivative.\medskip

Also, when
\begin{eqnarray*}
&&p_1=\dots=p_{n}=1,\\
&&\XC_1=\dots=\XC_{l}=0,\\ {\mbox{and }}
&&s_j\in(0,1),\quad{\mbox{for every~$j\in\{1,\dots,M\},$}}\end{eqnarray*}
we have that
Theorem~\ref{theone} recovers the cases taken into account in~\cite{DSV1},
in which approximation results have been established
for the superposition of a local operator
with a superposition of fractional Laplacians of order~$2s_j<2$.\medskip

In this sense, not only Theorem~\ref{theone} comprises the existing literature,
but it goes beyond it, since it combines 
classical derivatives, fractional Laplacians and Caputo fractional derivatives
altogether.
In addition, it comprises the cases in which the space-fractional Laplacians 
taken into account are of order
greater than~$2$. 

As a matter of fact, this point is also a novelty
introduced by Theorem~\ref{theone} here
with respect to the previous literature.

Theorem~\ref{theone} was announced in~\cite{CDV18}, and
we have just received the very interesting preprint~\cite{2018arXiv181007648K}
which also considered the case of
different, not necessarily fractional, powers of the Laplacian,
using a different and innovative methodology.
\medskip

The rest of this book is organized as follows. 
Chapter~\ref{CH3} focuses on time-fractional operators.
More precisely, in Sections~\ref{s:second}
and~\ref{s:grf0}
we study the boundary behaviour of the eigenfunctions of the
Caputo derivative and of functions with vanishing Caputo derivative, respectively,
detecting their singular
boundary behaviour in terms of
explicit representation formulas. These type of results are also
interesting in themselves and can find further applications. 

Chapter~\ref{CH4} is devoted to some properties
of the higher order fractional Laplacian. 
More precisely, Section~\ref{s:grf} provides
some representation formula of the solution
of~$(-\Delta)^s u=f$ in a ball, with~$u=0$ outside this ball, for all~$s>0$,
and extends
the Green formula methods introduced
in~\cite{MR3673669} and \cite{AJS3}.

Then, in Section~\ref{SEC:eigef} we 
study the boundary behaviour of the first Dirichlet
eigenfunction of higher order fractional equations, and in Section~\ref{sec5}
we give some precise asymptotics
at the boundary for the first Dirichlet eigenfunction\index{Dirichlet!eigenfunction}
of~$(-\Delta)^s$ for any~$s>0$.

Section~\ref{s:hwb} is devoted to
the analysis of the asymptotic behaviour of $s$-harmonic
functions, with a ``spherical bump function'' as exterior Dirichlet datum.

Chapter~\ref{CH5} is devoted to the proof of our main result. To this end,
Section~\ref{s:fourthE} contains an auxiliary statement, namely
Theorem~\ref{theone2}, which will imply
Theorem~\ref{theone}. This is technically convenient,
since the operator~$\Lambda_a$ depends in principle on the initial
point~$a$: this has the disadvantage that if~$\Lambda_a u_a=0$
and~$\Lambda_b u_b=0$ in some domain, the function~$u_a+u_b$
is not in principle a solution of any operator, unless~$a=b$.
To overcome such a difficulty, in Theorem~\ref{theone2}
we will reduce to the case in which~$a=-\infty$, exploiting
a polynomial extension that we have introduced and used in~\cite{CDV18}.

In Section~\ref{s:fourth0}
we make the main step towards the proof of Theorem~\ref{theone2}.
Here, we prove that functions in the kernel
of nonlocal operators such as the one in~\eqref{1.6BIS}
span with their derivatives a maximal Euclidean space.
This fact is special for the nonlocal case
and its proof is based on the boundary analysis of the fractional operators in both
time and space.
Due to the general form of the operator in~\eqref{1.6BIS},
we have to distinguish here several cases,
taking advantage of either the time-fractional
or the space-fractional components of the operators.

Finally, in Section~\ref{s:fourth} we
complete the proof of
Theorem~\ref{theone2}, using the previous approximation
results and suitable rescaling arguments.

The final appendix provides concrete examples in which our main result can be applied.

\chapter{Boundary behaviour of solutions of time-fractional equations}\label{CH3}

In this chapter, we give precise asymptotics for the boundary behaviour of
solutions of time-fractional equations. The cases of the eigenfunctions
and of the Dirichlet problem with vanishing forcing term
will be studied in detail (the latter will be often referred to
as the time-fractional harmonic case, borrowing a terminology from
elliptic equations, with a slight abuse of notation in our case).

\section{Sharp boundary behaviour\index{boundary behaviour} for the time-fractional eigenfunctions}\label{s:second}

In this section we show that the eigenfunctions of the Caputo fractional derivative 
in~\eqref{defcap}
have
an explicit representation via the Mittag-Leffler function\index{function!Mittag-Leffler}.
For this,
fixed $\alpha$, $\beta\in\mathbb{C}$ with $\Re\left(\alpha\right)>0$,
for any $z$ with $\Re\left(z\right)>0$, we recall that
the Mittag-Leffler function is defined as
\begin{equation}\label{Mittag}
E_{\alpha,\beta}\left(z\right):=\sum_{j=0}^{+\infty} {\frac{z^j}{\Gamma
\left(\alpha j+\beta\right)}}.
\end{equation} 
The Mittag-Leffler function plays an important role in equations
driven by the Caputo derivatives, replacing the exponential function
for classical differential equations, as given by the following well-established result
(see \cite{MR3244285} and the references therein):

\begin{lemma}\label{lemma1}
Let~$\alpha\in(0,1]$,
$\lambda\in{\mathbb{R}}$, and
$a\in\mathbb{R}\cup\left\{-\infty\right\}$.
Then, the unique solution of the boundary value problem
\begin{equation*}\left\{
\begin{matrix}
D_{t,a}^{\alpha}u(t)=\lambda\, u(t) &
\mbox{ for any }t\in (a,+\infty),\\
u(a)=1 &
\end{matrix}\right.
\end{equation*}
is given by $E_{\alpha,1}\left(\lambda \left(t-a\right)^\alpha\right)$.
\end{lemma}

Lemma~\ref{lemma1} can be actually generalized\footnote{It is easily seen that
for~$k:=1$ Lemma~\ref{MittagLEMMA} boils down
to Lemma~\ref{lemma1}.} to any
fractional order of differentiation~$\alpha$:

\begin{lemma}\label{MittagLEMMA}
Let $\alpha\in(0,+\infty)$, with~$\alpha\in(k-1,k]$ and~$k\in\mathbb{N}$,
$a\in\mathbb{R}\cup\left\{-\infty\right\}$, and~$\lambda\in{\mathbb{R}}$. Then,
the unique 
continuous solution of the boundary value problem
\begin{equation}\label{CHE:0}\left\{
\begin{matrix}
D_{t,a}^{\alpha}u(t)=\lambda \,u(t) &
\mbox{ for any }t\in (a,+\infty),\\
u(a)=1 ,\\
\partial^m_t u(a)=0&\mbox{ for any }
m\in\{1,\dots,k-1\}
\end{matrix}\right.
\end{equation}
is given by $u\left(t\right)=E_{\alpha,1}\left(\lambda \left(t-a\right)^\alpha\right)$.

\begin{proof}
For the sake of simplicity we take $a=0$.
Also, the case in which~$\alpha\in{\mathbb{N}}$ can be checked with a direct computation,
so we focus on the case~$\alpha\in(k-1,k)$, with~$k\in{\mathbb{N}}$.

We
let $u\left(t\right):=E_{\alpha,1}\left(\lambda t^\alpha\right)$.
It is straightforward to see that~$ u(t)=1+{\mathcal{O}}(t^k)$ and therefore
\begin{equation}\label{CHE:1}
u(0)=1 \qquad{\mbox{and}}\qquad
\partial^m_t u(0)=0\;\mbox{ for any }\;
m\in\{1,\dots,k-1\}.
\end{equation}
We also claim that
\begin{equation}\label{CHE:2}
D_{t,a}^{\alpha}u(t)=\lambda \,u(t) \;
\mbox{ for any }\;t\in (0,+\infty).
\end{equation}
To check this,
we recall~\eqref{defcap} and~\eqref{Mittag} (with~$\beta:=1$), 
and we have that
\begin{eqnarray*}&&
D_{t,a}^{\alpha} u\left(t\right) \\&= &
\frac{1}{\Gamma\left(k-\alpha\right)}\int_0^t {\frac{u^{(k)}\left(\tau\right)}{
\left(t-\tau\right)^{\alpha-k+1}}\, d\tau} \\
&=& \frac{1}{\Gamma\left(k-\alpha\right)}
\int_0^t {\left(\sum_{j=1}^{+\infty} {\lambda^j\frac{\alpha j\left(\alpha j
-1\right)\ldots\left(\alpha j-k+1\right)}{\Gamma\left(\alpha j
+1\right)} \tau^{\alpha j-k}}\right)\frac{d\tau}{\left(t-\tau\right)^{\alpha-k+1}}} \\
&=& \sum_{j=1}^{+\infty} {\lambda^j\,\frac{\alpha j\left(\alpha j-1\right)
\ldots\left(\alpha j-k+1\right)}{\Gamma\left(k-\alpha\right)
\Gamma\left(\alpha j+1\right)}}
\int_0^t {\tau^{\alpha j-k}\left(t-\tau\right)^{k-\alpha-1} \,d\tau}.
\end{eqnarray*}
Hence, using the change of variable $\tau=t\sigma$, we obtain that
\begin{equation}\label{H:1}
D_{t,a}^{\alpha} u\left(t\right) =
\sum_{j=1}^{+\infty} {\lambda^j\,
\frac{\alpha j\left(\alpha j-1\right)\ldots\left(\alpha j-k+1\right)}{
\Gamma\left(k-\alpha\right)\Gamma\left(\alpha j+1\right)}}
t^{\alpha j-\alpha}\int_0^1 {\sigma^{\alpha j-k}\left(1-\sigma\right)^{k-\alpha-1}
\, d\tau}.\end{equation}
On the other hand, from the basic properties of the Beta function, it is known
that if~$\Re(z)$, $\Re(w)>0$, then
\begin{equation}\label{H:2} \int_0^1 {\sigma^{z-1}\left(1-\sigma\right)^{w-1}\, dt}
=\frac{\Gamma\left(z\right)\Gamma\left(w\right)}{\Gamma\left(z+w\right)}.
\end{equation}
In particular, taking~$z:=\alpha j-k+1\in(\alpha-k+1,+\infty)\subseteq(0,+\infty)$
and~$w:=k-\alpha\in(0,+\infty)$, and substituting~\eqref{H:2} into~\eqref{H:1},
we conclude that
\begin{equation}\label{H:3}\begin{split}
D_{t,a}^{\alpha} u\left(t\right)=
& \sum_{j=1}^{+\infty} {\lambda^j\frac{\alpha j\left(\alpha j-1\right)\ldots\left(\alpha j-k+1\right)}{\Gamma\left(k-\alpha\right)\Gamma\left(\alpha j+1\right)}}\frac{\Gamma\left(\alpha j-k+1\right)\Gamma\left(k-\alpha\right)}{\Gamma\left(\alpha j-\alpha+1\right)}\,t^{\alpha j-\alpha} \\
=& \sum_{j=1}^{+\infty} {\lambda^j\frac{\alpha j\left(\alpha j-1\right)\ldots\left(\alpha j-k+1\right)}{\Gamma\left(\alpha j+1\right)}}\frac{\Gamma\left(\alpha j-k+1\right)}{\Gamma\left(\alpha j-\alpha+1\right)}\,t^{\alpha j-\alpha}.
\end{split}\end{equation}
Now we use the fact that $z\Gamma\left(z\right)=\Gamma\left(z+1\right)$
for any $z\in\mathbb{C}$ with $\Re\left(z\right)>-1$, so, we have
\begin{equation*}
\alpha j\left(\alpha j-1\right)\ldots\left(\alpha j-k+1\right)\Gamma\left(\alpha j-k+1\right)=\Gamma\left(\alpha j+1\right).
\end{equation*}
Plugging this information into~\eqref{H:3}, we thereby find that
\begin{equation*}
D_{t,a}^{\alpha} u\left(t\right)=
\sum_{j=1}^{+\infty} {\frac{\lambda^j}{\Gamma\left(\alpha j-\alpha+1\right)}t^{\alpha j-\alpha}}=\sum_{j=0}^{+\infty} {\frac{\lambda^{j+1}}{\Gamma\left(\alpha j+1\right)}t^{\alpha j}}=\lambda u(t).
\end{equation*}
This proves~\eqref{CHE:2}.

Then, in view of~\eqref{CHE:1} and~\eqref{CHE:2} we obtain that~$u$
is a solution of~\eqref{CHE:0}.
Hence, to complete the proof of the desired result, we have to show
that such a solution is unique. To this end, supposing that we have two
solutions of~\eqref{CHE:0}, we consider their difference~$w$, and
we observe that~$w$ is a solution of
\begin{equation*}\left\{
\begin{matrix}
D_{t,0}^{\alpha}w(t)=\lambda \,w(t) &
\mbox{ for any }t\in (0,+\infty),\\
\partial^m_t w(0)=0&\mbox{ for any }
m\in\{0,\dots,k-1\}.
\end{matrix}\right.
\end{equation*}
By Theorem~4.1 in~\cite{MR3563609}, it follows that~$w$ vanishes
identically, and this proves the desired uniqueness result.
\end{proof}
\end{lemma}

%%%%			Another useful property we take into account is the scaling property of the Caputo derivative.
%%%%			\begin{proposition}
%%%%			\label{capsca}
%%%%			Let $r>0$ and $u_r(t):=u(rt)$. Then
%%%%			\begin{equation}
%%%%			D_{t,a}^{\alpha}u_r(t)=r^\alpha \,D_{t,ra}^{\alpha}u (rt).
%%%%			\end{equation}
%%%%			\begin{proof}
%%%%			With a simple computation, we have
%%%%			\begin{align*}
%%%%			D_{t,a}^{\alpha}u_r(t)&=\frac{1}{\Gamma(k-\alpha)}\int_a^t \frac{u_r^{(k)}(\tau)}{(t-\tau)^{\alpha-k+1}}\, d\tau=\frac{r^k}{\Gamma(k-\alpha)}\int_a^t \frac{u^{(k)}(r\tau)}{(t-\tau)^{\alpha-k+1}} \,d\tau \\
%%%%			&=\frac{r^{k-1}}{\Gamma(k-\alpha)}\int_{ra}^{rt} \frac{u^{(k)}(w)}{\left(t-\frac{w}{r}\right)^{\alpha-k+1}}\, dw=\frac{r^\alpha}{\Gamma(k-\alpha)}\int_{ra}^{rt} \frac{u^{(k)}(w)}{\left(rt-w\right)^{\alpha-k+1}} \,dw \\
%%%%			&=r^\alpha \,D_{t,ra}^{\alpha}u(rt),
%%%%			\end{align*}
%%%%			where we have also used the change of variable $w:=r\tau$.
%%%%			\end{proof}
%%%%			\end{proposition}

The boundary behaviour of the Mittag-Leffler function for different values
of the fractional parameter~$\alpha$ is depicted in Figure~\ref{FIG1}.
In light of~\eqref{Mittag},
we notice in particular that, near~$z=0$,
$$ E_{\alpha,\beta}\left(z\right)=
\frac{1}{\Gamma\left(\beta\right)}+\frac{z}{\Gamma\left(\alpha +\beta\right)}+O(z^2)$$
and therefore, near~$t=a$,
\begin{equation*}
E_{\alpha,1}\left(\lambda \left(t-a\right)^\alpha\right)
=1+\frac{\lambda \left(t-a\right)^\alpha}{\Gamma\left(\alpha +1\right)}+O\big(\lambda^2\left(t-a\right)^{2\alpha}\big).
\end{equation*}

\begin{figure}[h]
\centering
\includegraphics[width=8 cm]{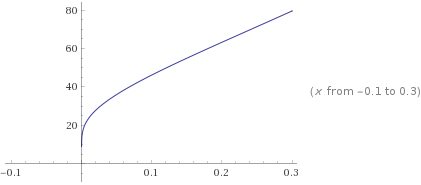}
\includegraphics[width=8 cm]{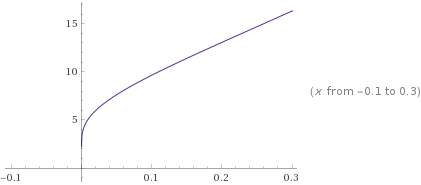}\\
\bigskip
\includegraphics[width=8 cm]{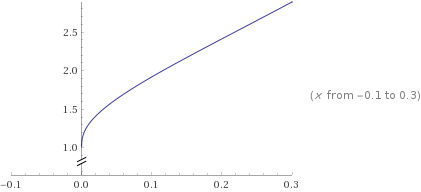}
\includegraphics[width=8 cm]{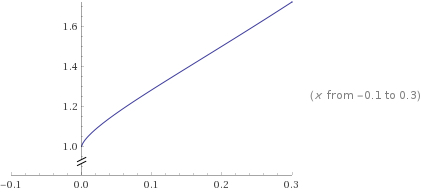}\\
\bigskip
\includegraphics[width=8 cm]{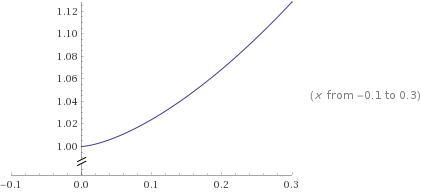}
\includegraphics[width=8 cm]{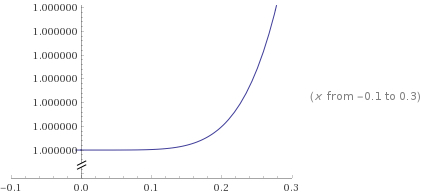}
\caption{\footnotesize\it Behaviour of the
Mittag-Leffler
function $E_{\alpha,1}\left(t^\alpha\right)$ near the origin
for $\alpha=\frac{1}{100}$, $\alpha=\frac{1}{20}$, $\alpha=\frac{1}{3}$, 
$\alpha=\frac{2}{3}$
$\alpha=\frac{3}{2}$ and~$\alpha=\frac{11}{2}$.}
\label{FIG1}
\end{figure}

%\chapter{Scale invariance for fractional Laplacian}\label{s:third}
%
%In this section we show that the operators taken into account in $\mathcal{L}$
%in~\eqref{ILPOAU-2}
%are scale-invariant. Namely, we point out the details
%of the following elementary observation:
%
%\begin{lemma}
%Let~$s_j\in(0,+\infty)$ and~$u\in C^\infty({\mathbb{R}}^{m_j})$.
%Fixed~$r>0$,
%let~$u_r(y_j):=u(ry_j)$. Then
%$$ (-\Delta)^{s_j}_{y_j} u_r(y_j)= r^{2s_j} (-\Delta)^{s_j}_{y_j}u(ry_j).$$
%\end{lemma}
%
%\begin{proof}
%{F}rom~\eqref{nonlocop} and~\eqref{898989ksdc}, we have that
%\begin{eqnarray*}
%(-\Delta)^{s_j}_{y_j} u_r(y_j) &=&
%\int_{\mathbb{R}^{m_j}} 
%\frac{\left(\delta_{h_j} u_r\right)(y_j,Y_j)}{|Y_j|^{m_j+2s_j}}\; \,dY_j
%\\ &=&
%\sum_{k=-h_j}^{h_j}
%(-1)^k \binom{2h_j}{h_j-k}
%\int_{\mathbb{R}^{m_j}} 
%\frac{u_r(y_j+kY_j)}{|Y_j|^{m_j+2s_j}}\, \,dY_j
%\\ &=& \sum_{k=-h_j}^{h_j}
%(-1)^k \binom{2h_j}{h_j-k}
%\int_{\mathbb{R}^{m_j}} 
%\frac{u_r(y_j+kY_j)}{|Y_j|^{m_j+2s_j}}\, \,dY_j.
%\end{eqnarray*}
%Then, using the change of variable~$rY_j=:Z_j$, we obtain that
%\begin{eqnarray*}
%(-\Delta)^{s_j}_{y_j} u_r(y_j) &=&
%r^{m_j+2s_j}\sum_{k=-h_j}^{h_j}
%(-1)^k \binom{2h_j}{h_j-k}
%\int_{\mathbb{R}^{m_j}} 
%\frac{u(ry_j+kZ_j)}{|Z_j|^{m_j+2s_j}}\, \,\frac{dZ_j}{r^{m_j}}
%\\ &=& r^{2s_j}\,
%\sum_{k=-h_j}^{h_j}
%(-1)^k \binom{2h_j}{h_j-k}
%\int_{\mathbb{R}^{m_j}} 
%\;\frac{u(ry_j+kZ_j)}{\left|Z_j\right|^{m_j+2s_j}}
%\,dZ_j\\
%&=& r^{2s_j}\,
%\int_{\mathbb{R}^{m_j}} 
%\frac{\left(\delta_{h_j} u\right)(ry_j,Z_j)}{\left|Z_j\right|^{m_j+2s_j}}\;\,dZ_j\\
%&=& r^{2s_j} (-\Delta)^{s_j}_{y_j} u(ry_j),
%\end{eqnarray*}
%as desired.
%\end{proof}

\section{Sharp boundary behaviour\index{boundary behaviour} for the time-fractional harmonic functions}\label{s:grf0}

In this section, we detect the optimal boundary behaviour of time-fractional
harmonic functions and of their derivatives.
The result that we need for our purposes is the following:

\begin{lemma}\label{LF}
Let~$\alpha\in(0,+\infty)\setminus\mathbb{N}$.
There exists a function~$\psi:\mathbb{R}\to\mathbb{R}$ such
that~$\psi\in C^\infty((1,+\infty))$ and
\begin{eqnarray}
\label{LAp1}&&D^\alpha_0\psi(t)=0 \qquad{\mbox{for all }}t\in(1,+\infty),\\
\label{LAp2}{\mbox{and }}&&\lim_{\epsilon\searrow0}
\epsilon^{\ell-\alpha} \partial^\ell\psi(1+\epsilon t)=\kappa_{\alpha,\ell}\, t^{\alpha-\ell},
\qquad{\mbox{for all }}\ell\in\mathbb{N},
\end{eqnarray}
for some~$\kappa_{\alpha,\ell}\in\mathbb{R}\setminus\{0\}$, where~\eqref{LAp2}
is taken in the sense of distribution for~$t\in(0,+\infty)$.
\end{lemma}

\begin{proof} We use Lemma~2.5
in~\cite{CDV18}, according to which
(see in particular formula~(2.16)
in~\cite{CDV18}) the claim in~\eqref{LAp1}
holds true. Furthermore (see formulas~(2.19)
and~(2.20) in~\cite{CDV18}), we can write that, for all~$t>1$,
\begin{equation}\label{VBVBHJSnb}
\psi(t)=-\frac1{\Gamma(\alpha)\Gamma([\alpha]+1-\alpha)}
\iint_{[1,t]\times[0,3/4]}\partial^{[\alpha]+1}\psi_0(\sigma)\,(\tau-\sigma)^{[\alpha]-\alpha}\,
(t-\tau)^{\alpha-1}\,d\tau\,d\sigma,
\end{equation}
for a suitable~$\psi_0\in C^{[\alpha]+1}([0,1])$.

In addition, by Lemma~2.6 in~\cite{CDV18}, we can write that
\begin{equation}\label{0oLAMJA}
\lim_{\epsilon\searrow0}
\epsilon^{ -\alpha} \psi(1+\epsilon)=\kappa,
\end{equation}
for some~$\kappa\ne0$.
Now we set
$$ (0,+\infty)\ni t\mapsto f_\epsilon(t):=
\epsilon^{\ell-\alpha} \partial^\ell\psi(1+\epsilon t).$$
We observe that, for any~$\varphi\in C^\infty_0((0,+\infty))$,
\begin{equation}\label{UHAikAJ678OKA}
\begin{split}& \int_0^{+\infty} f_\epsilon(t)\,\varphi(t)\,dt=
\epsilon^{\ell-\alpha} \int_0^{+\infty}\partial^\ell\psi(1+\epsilon t)\varphi(t)\,dt\\
&=
\epsilon^{-\alpha} \int_0^{+\infty}\frac{d^\ell}{dt^\ell}\big(\psi(1+\epsilon t)\big)\varphi(t)\,dt
=(-1)^\ell\,
\epsilon^{-\alpha} \int_0^{+\infty}\psi(1+\epsilon t)\,\partial^\ell\varphi(t)\,dt.
\end{split}\end{equation}
Also, in view of~\eqref{VBVBHJSnb},
\begin{eqnarray*}&&
\epsilon^{-\alpha}|\psi(1+\epsilon t)|\\
&=&\left|\frac{\epsilon^{-\alpha}}{\Gamma(\alpha)\Gamma([\alpha]+1-\alpha)}
\iint_{[1,1+\epsilon t]\times[0,3/4]}\partial^{[\alpha]+1}\psi_0(\sigma)\,(\tau-\sigma)^{[\alpha]-\alpha}\,
(1+\epsilon t-\tau)^{\alpha-1}\,d\tau\,d\sigma\right|
\\&\le&C\,\epsilon^{-\alpha}\,
\int_{[1,1+\epsilon t]}(1+\epsilon t-\tau)^{\alpha-1}\,d\tau
\\ &=& Ct^\alpha,
\end{eqnarray*}
which is locally bounded in~$t$, where~$C>0$ here above may vary from line to line.

As a consequence, we can pass to the limit in~\eqref{UHAikAJ678OKA}
and obtain that
$$\lim_{\epsilon\searrow0} \int_0^{+\infty} f_\epsilon(t)\,\varphi(t)\,dt
=
(-1)^\ell\, \int_0^{+\infty} \lim_{\epsilon\searrow0}
\epsilon^{-\alpha}
\psi(1+\epsilon t)\,\partial^\ell\varphi(t)\,dt.
$$
This and~\eqref{0oLAMJA} give that
$$\lim_{\epsilon\searrow0} \int_0^{+\infty} f_\epsilon(t)\,\varphi(t)\,dt
=
(-1)^\ell\, \kappa\,\int_0^{+\infty} t^\alpha\,
\partial^\ell\varphi(t)\,dt=\kappa\,\alpha\dots(\alpha-\ell+1)\int_0^{+\infty} t^{\alpha-\ell}\,\varphi(t)\,dt,$$
which establishes~\eqref{LAp2}.
\end{proof}

\chapter{Boundary behaviour of solutions of space-fractional equations}\label{CH4}

In this chapter, we give precise asymptotics for the boundary behaviour of
solutions of space-fractional equations. The cases of the eigenfunctions
and of the Dirichlet problem with vanishing forcing term
will be studied in detail. To this end, we will also
exploit useful representation formulas of the solutions
in terms of suitable Green functions.

\section{Green representation formulas and solution of $(-\Delta)^s u=f$ in $B_1$ with homogeneous Dirichlet datum}
\label{s:grf}
Our goal is to provide some representation results on the solution
of~$(-\Delta)^s u=f$ in a ball, with~$u=0$ outside this ball, for all~$s>0$.
Our approach is an extension of the Green formula methods introduced
in~\cite{MR3673669} and \cite{AJS3}:
differently from the previous literature, we are not assuming here that~$f$
is regular in the whole of the ball, but merely that it is H\"older continuous
near the boundary and sufficiently integrable inside. Given the type of singularity
of the Green function, these assumptions are sufficient to obtain meaningful
representations, which in turn will be useful to deal with
the eigenfunction problem in the subsequent Chapter~\ref{SEC:eigef}.

\subsection{Solving $(-\Delta)^s u=f$ in $B_1$ for discontinuous~$f$ vanishing near $\partial B_1$}

Now, we want to extend the representation
results of \cite{MR3673669} and \cite{AJS3} to
the case in which the right hand side is not H\"older continuous,
but merely in a Lebesgue space, but it has the additional
property of vanishing near the boundary of the domain.
To this end, fixed~$s>0$,
we consider the polyharmonic Green\index{function!Green}
function in $B_1\subset{\mathbb{R}}^n$, given,
for every~$x\ne y\in\mathbb{R}^n$, by 
\begin{equation}\label{GREEN}
\begin{split}& 
\mathcal{G}_s\left(x,y\right):=\frac{k(n,s)}{\left|x-y\right|^{n-2s}}\,
\int_0^{r_0\left(x,y\right)} 
\frac{\eta^{s-1}}{\left(\eta+1\right)^{\frac{n}{2}}} \,d\eta,
\\ {\mbox{where }}\quad&r_0\left(x,y\right):=
\frac{\left(1-\left|x\right|^2\right)_+\left(1-\left|y\right|^2\right)_+}{\left|x-y\right|^2}, \\
{\mbox{with }}\quad&
k(n,s):=\frac{\Gamma\left(\frac{n}{2}\right)}{
\pi^{\frac{n}{2}}\,4^s\Gamma^2\left(s\right)}. 
\end{split}
\end{equation}
Given~$x\in B_1$, we also set
\begin{equation}\label{GREEN2} d(x):=1-|x|.\end{equation}
In this setting, we have:

\begin{proposition}\label{LONTANO}
Let~$r\in(0,1)$ and~$f\in L^2(B_1)$, with~$f=0$ in~${\mathbb{R}}^n\setminus B_r$. Let
\begin{equation}\label{DEF uG} u(x):=
\begin{cases}
\displaystyle\int_{B_1} \mathcal{G}_s\left(x,y\right)\,f(y)\,dy & {\mbox{ if }}x\in B_1,\\
0&{\mbox{ if }}x\in{\mathbb{R}}^n\setminus B_1.
\end{cases}
\end{equation}
Then:
\begin{equation}
\label{LON1}
u\in L^1(B_1), {\mbox{ and }} \|u\|_{L^1(B_1)}\le C\,\|f\|_{L^1(B_1)},
\end{equation}
\begin{equation}\label{LON2}
{\mbox{for every $R\in(r,1)$, }} \sup_{x\in B_1\setminus B_R}
d^{-s}(x)\,|u(x)|\le C_R\,\|f\|_{L^1(B_1)},\end{equation}
\begin{equation}\label{LON3}
{\mbox{$u$ satisfies }}(-\Delta)^s u=f{\mbox{ in }}B_1 {\mbox{ in the sense of
distributions,}} 
\end{equation}
and
\begin{equation}\label{LON4}
u\in W^{2s,2}_{loc}(B_1).
\end{equation}
Here above, 
$C>0$ is a constant depending on~$n$, $s$ and~$r$,
$C_R>0$ is a constant depending on~$n$, $s$, $r$ and~$R$
and $C_\rho>0$ is a constant depending on~$n$, $s$, $r$ and~$\rho$.
\end{proposition}

When~$f\in C^\alpha(B_1)$ for some~$\alpha\in(0,1)$,
Proposition~\ref{LONTANO} boils down to the main results of \cite{MR3673669}
and~\cite{AJS3}.

\begin{proof}[Proof of Proposition~\ref{LONTANO}] 
We recall the following useful estimate, see Lemma~3.3 in 
\cite{AJS3}:
for any~$\epsilon\in\left(0,\,\min\{n,s\}\right)$, and any~$\bar R$, $\bar r>0$,
$$ \frac1{\bar R^{n-2s}}\,\int_0^{\bar r/\bar R^2}\frac{\eta^{s-1}}{
\left(\eta+1\right)^{\frac{n}{2}}} \,d\eta\le
\frac{2}{s}\;\frac{\bar r^{s-(\epsilon/2)}}{\bar R^{n-\epsilon}},$$
and so, by~\eqref{GREEN} and~\eqref{GREEN2},
for every~$x$, $y\in B_1$,
$$ \mathcal{G}_s\left(x,y\right)\le
\frac{C\,d^{s-(\epsilon/2)}(x)\,d^{s-(\epsilon/2)}(y)}{|x-y|^{n-\epsilon}}
$$
for some~$C>0$. Hence, recalling~\eqref{DEF uG},
\begin{eqnarray*}
\int_{B_1} |u(x)|\,dx&\le& \int_{B_1}
\left(\int_{B_1} \mathcal{G}_s\left(x,y\right)\,|f(y)|\,dy\right)\,dx
\\ &\le& C \int_{B_1}
\left(\int_{B_1} \frac{|f(y)|}{|x-y|^{n-\epsilon}}\,dy\right)\,dx\\&=&
C \int_{B_1}
\left(\int_{B_1} \frac{|f(y)|}{|x-y|^{n-\epsilon}}\,dx\right)\,dy\\&=&
C \int_{B_1}|f(y)|\,dy,
\end{eqnarray*}
up to renaming~$C>0$ line after line, and this proves~\eqref{LON1}.

Now, if~$x\in B_1\setminus B_R$ and~$y\in B_r$, with~$0<r<R<1$, we have that
$$|x-y|\ge |x|-|y|\ge R-r$$
and accordingly
$$ r_0\left(x,y\right)\le
\frac{2d(x)}{(R-r)^2},$$
which in turn implies that
$$ \mathcal{G}_s\left(x,y\right)\le\frac{k(n,s)}{\left|x-y\right|^{n-2s}}\,
\int_0^{{2d(x)}/{(R-r)^2}} 
\frac{\eta^{s-1}}{\left(\eta+1\right)^{\frac{n}{2}}} \,d\eta,
\le {C\,d^s(x)},$$
for some~$C>0$.
As a consequence,
since~$f$ vanishes outside~$B_r$, we see that, for any~$x\in B_1\setminus B_R$,
\begin{eqnarray*}
|u(x)|\le
\int_{B_r} \mathcal{G}_s\left(x,y\right)\,|f(y)|\,dy\le C\,d^s(x)
\int_{B_r} |f(y)|\,dy,
\end{eqnarray*}
which proves~\eqref{LON2}.

Now, we fix~$\hat r\in(r,1)$ and consider a mollification of~$f$,
that we denote by~$f_j\in C^\infty_0(B_{\hat r})$, with~$f_j\to f$
in~$L^2(B_1)$ as~$j\to+\infty$. 
We also write~$\mathcal{G}_s * f$ as a short notation for the right hand side
of~\eqref{DEF uG}. Then, by \cite{MR3673669} and \cite{AJS3},
we know that~$u_j:=\mathcal{G}_s * f_j$ is a (locally smooth, hence distributional) solution of~$(-\Delta)^s u_j=f_j$.
Furthermore, if we set~$\tilde u_j:=u_j-u$ and~$\tilde f_j:=f_j-f$
we have that
$$ \tilde u_j=\mathcal{G}_s * (f_j-f)=\mathcal{G}_s * \tilde f_j,$$
and therefore, by~\eqref{LON1},
$$ \|\tilde u_j\|_{L^1(B_1)}\le C\,\|\tilde f_j\|_{L^1(B_1)},$$
which is infinitesimal as~$j\to+\infty$. This says that~$u_j\to u$
in~$L^1(B_1)$ as~$j\to+\infty$, and consequently, for any~$\varphi\in C^\infty_0(B_1)$,
\begin{eqnarray*}
&&\int_{B_1} u(x)\,(-\Delta)^s\varphi(x)\,dx=\lim_{j\to+\infty}
\int_{B_1} u_j(x)\,(-\Delta)^s\varphi(x)\,dx
\\&&\qquad=\lim_{j\to+\infty}
\int_{B_1} f_j(x)\,\varphi(x)\,dx=\int_{B_1} f(x)\,\varphi(x)\,dx,
\end{eqnarray*}
thus completing the proof of~\eqref{LON3}.

Now, to prove~\eqref{LON4}, we can suppose that~$s\in(0,+\infty)\setminus{\mathbb{N}}$,
since the case of integer~$s$ is classical, see e.g. \cite{MR1814364}.
First of all, we claim that
\begin{equation}\label{08}
{\mbox{\eqref{LON4} holds true for every~$s\in(0,1)$.}}
\end{equation}
For this, we first claim that
if~$g\in C^\infty(B_1)$ and~$v$ is a (locally smooth) solution of~$(-\Delta)^s v=g$
in~$B_1$, with~$v=0$ outside~$B_1$, then~$v\in W^{2s,2}_{loc}(B_1)$,
and, for any~$\rho\in(0,1)$, 
\begin{equation}\label{BIX}
\|v\|_{W^{2s,2}(B_\rho)}\le C_\rho\,\|g\|_{L^2(B_1)}.
\end{equation}
This claim can be seen as a localization of
Lemma~3.1 of~\cite{MR2863859},
or a quantification of the last claim in Theorem~1.3 of~\cite{MR3641649}.
To prove~\eqref{BIX}, we let~$R_-<R_+\in(\rho,1)$,
and consider~$\eta\in C^\infty_0(B_{R_+})$ with~$\eta=1$ in~$B_{R_-}$.
We let~$v^*:=v\eta$, and we recall formulas~(3.2), (3.3)
and~(A.5) in~\cite{MR3641649}, according to which
\begin{equation*}\begin{split}&
(-\Delta)^s v^*-\eta(-\Delta)^s v=g^* \quad{\mbox{ in }}{\mathbb{R}}^n,\\
{\mbox{with }}\quad&\| g^*\|_{L^2({\mathbb{R}}^n)}\le C\,\|v\|_{W^{s,2}({\mathbb{R}}^n)}
,\end{split}\end{equation*}
for some~$C>0$. 

Moreover, 
using a notation taken from \cite{MR3641649}
we denote by~$W^{s,2}_0(\overline{B_1})$
the space of functions in~$W^{s,2}({\mathbb{R}}^n)$ vanishing outside~$B_1$
and we consider the dual space~$
W^{-s,2}_0(\overline{B_1})$. We remark that if~$h\in L^2(B_1)$
we can naturally identify~$h$ as an element of~$
W^{-s,2}_0(\overline{B_1})$ by considering the action of~$h$
on any~$\varphi\in W^{s,2}_0(\overline{B_1})$ as defined by
$$ \int_{B_1} h(x)\,\varphi(x)\,dx.$$
With respect to this, we have that
\begin{equation}\label{DUE} \|h\|_{W^{-s,2}_0(\overline{B_1})}=\sup_{{\varphi\in
W^{s,2}_0(\overline{B_1})}\atop{\|\varphi\|_{W^{s,2}_0(\overline{B_1})}=1}}
\int_{B_1}h(x)\,\varphi(x)\,dx\le\|h\|_{L^2(B_1)}.\end{equation}
We notice also that
$$ \|v\|_{W^{s,2}({\mathbb{R}}^n)}\le C\,\| g\|_{W^{-s,2}(\overline{B_1})},$$
in light of Proposition~2.1 of~\cite{MR3641649}. This and~\eqref{DUE}
give that
$$ \|v\|_{W^{s,2}({\mathbb{R}}^n)}\le C\,\| g\|_{L^2(B_1)}.$$
Then, by Lemma~3.1 of~\cite{MR2863859}
(see in particular formula~(3.2) there, applied here with~$\lambda:=0$),
we obtain that
\begin{equation}\label{BIX0}
\begin{split}
\|v^*\|_{W^{2s,2}({\mathbb{R}}^n)}\,&\le C\,\|
\eta(-\Delta)^s v+
g^*\|_{L^2({\mathbb{R}}^n)}\\
&\le C\,\big( \| (-\Delta)^s v\|_{L^2(B_{R_+})}+\|g^*\|_{L^2({\mathbb{R}}^n)}\big)
\\&=C\,\big( \| g\|_{L^2(B_{R_+})}+\|g^*\|_{L^2({\mathbb{R}}^n)}\big)\\
&\le C\,\big(\| g\|_{L^2(B_{1})}+
\|v\|_{W^{s,2}({\mathbb{R}}^n)}\big)\\
&\le C\, \| g\|_{L^2(B_{1})}
,\end{split}\end{equation}
up to renaming~$C>0$ step by step.
On the other hand, since~$v^*=v$ in~$B_\rho$,
$$ \|v\|_{W^{2s,2}(B_\rho)}=
\|v^*\|_{W^{2s,2}(B_\rho)}\le \|v^*\|_{W^{2s,2}({\mathbb{R}}^n)}.$$
{F}rom this and~\eqref{BIX0}, we obtain \eqref{BIX}, as desired.

Now, we let~$f_j$, $\tilde f_j$, $u_j$ and~$\tilde u_j$
as above and make use
of~\eqref{BIX} to write
\begin{equation}\label{qwe89:BBA}
\begin{split}
&\|u_j\|_{W^{2s,2}(B_\rho)}\le C_\rho\,\|f_j\|_{L^2(B_1)}
\\ {\mbox{and }}\quad&
\|\tilde u_j\|_{W^{2s,2}(B_\rho)}\le C_\rho\,\|\tilde f_j\|_{L^2(B_1)}.
\end{split}
\end{equation}
As a consequence,
taking the limit as~$j\to+\infty$
we obtain that
$$ \|u\|_{W^{2s,2}(B_\rho)}\le C_\rho\,\|f\|_{L^2(B_1)},$$
that is~\eqref{LON4} in this case, namely the claim in~\eqref{08}.

Now, to prove~\eqref{LON4}, we argue by induction on
the integer part of~$s$. When the integer part of~$s$
is zero, the basis of the induction is warranted by~\eqref{08}.
Then, to perform the inductive step, given~$s\in(0,+\infty)\setminus{\mathbb{N}}$,
we suppose that~\eqref{LON4} holds true for~$s-1$,
namely
\begin{equation}\label{0lL:AN1}
\mathcal{G}_{s-1} * f
\in W^{2s-2,2}_{loc}(B_1).
\end{equation}
Then, following~\cite{AJS3},
it is convenient to introduce the notation
$$ [x,y]:=\sqrt{|x|^2|y|^2-2x\cdot y+1}$$
and consider
the auxiliary kernel given, for every~$x\ne y\in B_1$, by
\begin{equation}
\label{aux}
P_{s-1}(x,y):=\frac{(1-|x|^2)^{s-2}_+(1-|y|^2)^{s-1}_+(1-|x|^2|y|^2)}{[x,y]^n}.
\end{equation} 
We point out that if~$x\in B_r$ with~$r\in(0,1)$,
then
\begin{equation}
\label{klop}
[x,y]^2\ge|x|^2|y|^2-2|x|\,| y|+1=(1-|x|\,|y|)^2\ge(1-r)^2>0.
\end{equation}
Consequently, since~$f$ is supported in~$B_r$,
\begin{equation}\label{0lL:AN2}
P_{s-1}*f\in C^\infty({\mathbb{R}}^n).\end{equation}
Then,
we recall that
\begin{equation}\label{0lL:AN3}
-\Delta_x\mathcal{G}_s(x,y)=\mathcal{G}_{s-1}(x,y)-CP_{s-1}(x,y),\end{equation}
for some~$C\in{\mathbb{R}}$, see Lemma~3.1 in \cite{AJS3}.

As a consequence, in view of~\eqref{0lL:AN1}, \eqref{0lL:AN2}, \eqref{0lL:AN3},
we conclude that
$$ -\Delta (\mathcal{G}_s*f)= (-\Delta_x\mathcal{G}_s)*f
\in W^{2s-2,2}_{loc}(B_1).$$
This and the classical elliptic regularity theory (see e.g. \cite{MR1814364})
give that~$\mathcal{G}_s*f\in W^{2s,2}_{loc}(B_1)$, which
completes the inductive proof and establishes~\eqref{LON4}.
\end{proof}

\subsection{Solving $(-\Delta)^s u=f$ in $B_1$ for~$f$ H\"older continuous near $\partial B_1$}

Our goal is now
to extend the representation
results of \cite{MR3673669} and \cite{AJS3} to
the case in which the right hand side is not H\"older continuous
in the whole of the ball,
but merely in a neighborhood of the boundary.
This result is obtained here by superposing the
results in \cite{MR3673669} and \cite{AJS3}
with Proposition~\ref{LONTANO} here, taking advantage of the linear
structure of the problem.

\begin{proposition} \label{LEJOS}
Let~$f\in L^2(B_1)$.
Let~$\alpha$, $r\in(0,1)$ and assume that
\begin{equation}\label{CHlaIA}
f\in C^\alpha(B_1\setminus B_r).\end{equation}
In the notation of~\eqref{GREEN}, let
\begin{equation} \label{0olwsKA}
u(x):=
\begin{cases}
\displaystyle\int_{B_1} \mathcal{G}_s\left(x,y\right)\,f(y)\,dy & {\mbox{ if }}x\in B_1,\\
0&{\mbox{ if }}x\in{\mathbb{R}}^n\setminus B_1.
\end{cases}
\end{equation}
Then, in the notation of~\eqref{GREEN2}, we have that:
\begin{equation}\label{VIC2}
{\mbox{for every $R\in(r,1)$, }} \sup_{x\in B_1\setminus B_R}
d^{-s}(x)\,|u(x)|\le C_R\,\big(\|f\|_{L^1(B_1)}+
\|f\|_{L^\infty(B_1\setminus B_r)}\big)
,\end{equation}
\begin{equation}\label{VIC3}
{\mbox{$u$ satisfies }}(-\Delta)^s u=f{\mbox{ in }}B_1 {\mbox{ in the sense of
distributions,}} \end{equation}
and
\begin{equation}\label{VIC4} 
u\in W^{2s,2}_{loc}(B_1).
\end{equation}
Here above, 
$C>0$ is a constant depending on~$n$, $s$ and~$r$,
$C_R>0$ is a constant depending on~$n$, $s$, $r$ and~$R$
and $C_\rho>0$ is a constant depending on~$n$, $s$, $r$ and~$\rho$.
\end{proposition}

\begin{proof} We take~$r_1\in(r,1)$ and~$\eta\in C^\infty_0(B_{r_1})$
with~$\eta=1$ in~$B_r$.
Let also
$$f_1:=f\eta\qquad{\mbox{and}}\qquad f_2:=f-f_1.$$
We observe that~$f_1\in L^2(B_1)$, and that~$f_1=0$ outside~$B_{r_1}$.
Therefore, we are in the position of applying Proposition~\ref{LONTANO}
and find a function~$u_1$ (obtained by convolving~$\mathcal{G}_s$
against~$f_1$) such that
\begin{eqnarray}
&& \label{XLON2}
{\mbox{for every $R\in(r_1,1)$, }} \sup_{x\in B_1\setminus B_R}
d^{-s}(x)\,|u_1(x)|\le C_R\,\|f_1\|_{L^1(B_1)},\\ 
\label{XLON3}
&& {\mbox{$u_1$ satisfies }}(-\Delta)^s u_1=f_1{\mbox{ in }}B_1 {\mbox{ in the sense of
distributions,}} 
\\
\label{XLON4}{\mbox{and }}
&& u_1\in W^{2s,2}_{loc}(B_1).
\end{eqnarray}
On the other hand, we have that~$f_2=f(1-\eta)$
vanishes outside~$B_1\setminus B_r$
and it is H\"older continuous. Accordingly,
we can apply Theorem~1.1 of~\cite{AJS3}
and find a function~$u_2$ (obtained by convolving~$\mathcal{G}_s$
against~$f_2$) such that
\begin{eqnarray}
&& \label{YLON2}
{\mbox{for every $R\in(r_1,1)$, }} \sup_{x\in B_1\setminus B_R}
d^{-s}(x)\,|u_2(x)|\le C_R\,\|f_2\|_{L^\infty(B_1)},\\ 
\label{YLON3}
&& {\mbox{$u_2$ satisfies }}(-\Delta)^s u_2=f_2{\mbox{ in }}B_1 {\mbox{ in the sense of
distributions,}} 
\\
\label{YLON4}
{\mbox{and }}&& u_2\in C^{2s+\alpha}_{loc}(B_1).
\end{eqnarray}
Then, $f=f_1+f_2$, and thus,
in view of~\eqref{0olwsKA}, we have that~$
u=u_1+u_2$. Also, $u$ satisfies~\eqref{VIC2},
thanks to~\eqref{XLON2} and~\eqref{YLON2},
\eqref{VIC3},
thanks to~\eqref{XLON3} and~\eqref{YLON3}, and~\eqref{VIC4},
thanks to~\eqref{XLON4} and~\eqref{YLON4}.
\end{proof}

\section{Existence and regularity for the first eigenfunction of the higher order fractional Laplacian}\label{SEC:eigef}

The goal of these pages is
to study the boundary behaviour of the first Dirichlet
eigenfunction of higher order fractional equations.

For this, writing~$s=m+\sigma$, with~$m\in{\mathbb{N}}$ and~$\sigma\in(0,1)$,
we define the energy space 
\begin{equation}
\label{energy}
H_0^s\left(B_1\right):=\left\{u\in H^s\left(\mathbb{R}^n\right);\; u=0 \;\text{in}\; \mathbb{R}^n\setminus B_1\right\},
\end{equation}
endowed with the Hilbert norm
\begin{equation}
\label{energynorm} \left\|u\right\|_{H_0^s\left(B_1\right)}:=
\left(\sum_{\left|\alpha\right|\leq m} {\left\|\partial^\alpha u
\right\|^2_{L^2\left(B_1\right)}}+\mathcal{E}_s
\left(u,u\right)\right)^{\frac{1}{2}},\end{equation}
where
%%		$$\mathcal{E}_s\left(u,v\right):=\begin{cases}
%%		\mathcal{E}_{\sigma}\left(\Delta^{\frac{m}{2}}u,\Delta^{\frac{m}{2}}v\right),\hspace{0,4 cm}\text{if}\hspace{0.1 cm}m\hspace{0.1 cm}\text{is even} \\
%%		\sum_{k=1}^N {\mathcal{E}_{\sigma}\left(\partial_k\Delta^{\frac{m-1}{2}}u,\partial_k\Delta^{\frac{m-1}{2}}v\right)},\hspace{0,4 cm}\text{if}\hspace{0.1 cm}m\hspace{0.1 cm}\text{is odd}.
%%		\end{cases}$$
%%		In both cases, it reduces again to
\begin{equation}\label{ENstut}
\mathcal{E}_s\left(u,v\right)=\int_{\mathbb{R}^n} {
\left|\xi\right|^{2s}\mathcal{F}u\left(\xi\right)\overline{\mathcal{F}v\left(\xi\right)} \,
d\xi},\end{equation}
being~$\mathcal{F}$ the Fourier transform and using the notation~$\overline{z}$ to denote the complex conjugated
of a complex number~$z$.

In this setting, we consider~$u\in H^s_0(B_1)$ to be 
such that
\begin{equation}\label{dirfun}\begin{cases}
\left(-\Delta\right)^s u=\lambda_1 u &\quad\text{ in }B_1, \\
u=0 &\quad\text{ in } \mathbb{R}^n\setminus\overline{B_1},
\end{cases}\end{equation}
for every~$s>0$, with~$\lambda_1$ as small as possible.

The existence of solutions of \eqref{dirfun} is ensured
via variational techniques, as stated in the following result:

\begin{lemma}\label{VARIA}
The functional~$\mathcal{E}_s\left(u,u\right)$
attains its minimum~$\lambda_1$ on the functions in~$H^s_0(B_1)$
with unit norm in~$L^2(B_1)$.

The minimizer satisfies~\eqref{dirfun}.

In addition, $\lambda_1>0$.

\begin{proof} The proof is based on the direct method
in the calculus of variations. We provide some details for completeness.
Let~$s=m+\sigma$,
with~$m\in\mathbb{N}$ and~$\sigma\in(0,1)$.
Let us consider a minimizing sequence~$u_j\in H^s_0(B_1)\subseteq H^m({\mathbb{R}}^n)$
such that~$\|u_j\|_{L^2(B_1)}=1$ and
$$ \lim_{j\to+\infty}\mathcal{E}_s\left(u_j,u_j\right)=\inf_{{u\in H^s_0(B_1)}\atop{
\|u\|_{L^2(B_1)}=1}}\mathcal{E}_s\left(u,u\right).$$
In particular, we have that~$u_j$ is bounded in~$H^s_0(B_1)$ uniformly in~$j$,
so, up to a subsequence, it converges to some~$u_\star$
weakly in~$H^s_0(B_1)$ and strongly in~$L^2(B_1)$
as~$j\to+\infty$.

The weak lower semicontinuity of the seminorm $\mathcal{E}_s\left(\cdot,\cdot\right)$
then implies that~$u_\star$ is the desired minimizer.

Then, given~$\phi\in C^\infty_0(B_1)$, we have that
$$ \mathcal{E}_s\left(u_\star+\epsilon\phi,u_\star+\epsilon\phi\right)
\ge\mathcal{E}_s\left(u_\star,u_\star\right),$$
for every~$\epsilon\in{\mathbb{R}}$, and this gives that~\eqref{dirfun} is
satisfied in the sense of distributions,
and also in the classical sense by the elliptic regularity theory.

Finally, we have that~$\mathcal{E}_s\left(u_\star,u_\star\right)>0$,
since~$u_\star$ (and thus~${\mathcal{F}}u_\star$) does not vanish identically.
Consequently,
$$ \lambda_1=\frac{ \mathcal{E}_s
\left(u_\star,u_\star\right)}{\|u_\star\|_{L^2(B_1)}^2}=
\mathcal{E}_s\left(u_\star,u_\star\right)>0,$$
as desired.
\end{proof}
\end{lemma}

Our goal is now to apply Proposition~\ref{LEJOS}
to solutions of~\eqref{dirfun}, taking~$f:=\lambda u$. To this end,
we have to check that condition~\eqref{CHlaIA}
is satisfied, namely that solutions of~\eqref{dirfun} are
H\"older continuous in $B_1\setminus B_r$, for any $0<r<1$.

To this aim, we prove that polyharmonic operators of any order~$s>0$ 
always admit
a first eigenfunction in the ball which does not change sign
and which is radially symmetric. For this, we start discussing
the sign property:

\begin{lemma}\label{ikAHHPKAK}
There exists a nontrivial solution of~\eqref{dirfun} that does not change sign.

\begin{proof} We exploit a method explained in detail in
Section~3.1 of~\cite{MR2667016}. As a matter of fact,
when~$s\in{\mathbb{N}}$, the desired result is exactly Theorem~3.7
in~\cite{MR2667016}.

Let~$u$ be as in
Lemma~\ref{VARIA}.
If either~$u\ge0$ or~$u\le0$, then the desired result is proved.
Hence, we argue by contradiction,
assuming that~$u$ attains strictly positive and strictly negative values.
We define
$$ {\mathcal{K}}:=\{ w:{\mathbb{R}}^n\to{\mathbb{R}} {\mbox{ s.t. 
$\mathcal{E}_s\left(w,w\right)<+\infty$, and
$w\ge0$ in $B_1$}} \}.$$
Also, we set
$$ {\mathcal{K}}^\star
:=\{ w\in H^s_0(B_1) {\mbox{ s.t. }}
\mathcal{E}_s\left(w,v\right)\le0
{\mbox{ for all }}v\in {\mathcal{K}}\}.$$
We claim that
\begin{equation}\label{kstar po}
{\mbox{if~$w\in{\mathcal{K}}^\star$, then $w\le0$}}.
\end{equation}
To prove this, we recall the notation in~\eqref{GREEN},
take~$\phi\in C^\infty_0(B_1)\cap{\mathcal{K}}$,
and let
$$ v_\phi(x):=
\begin{cases}
\displaystyle\int_{B_1} \mathcal{G}_s\left(x,y\right)\,\phi(y)\,dy & {\mbox{ if }}x\in B_1,\\
0&{\mbox{ if }}x\in{\mathbb{R}}^n\setminus B_1.
\end{cases}
$$
Then~$v_\phi\in{\mathcal{K}}$ and it satisfies~$
\left(-\Delta\right)^s v_\phi=\phi$ in~$B_1$, thanks
to~\cite{MR3673669} or~\cite{AJS3}.

Consequently, we can write, for every~$x\in B_1$,
$$ \phi(x)={\mathcal{F}}^{-1}(|\xi|^{2s}{\mathcal{F}}v_\phi)(x).$$
Hence, for every~$w\in{\mathcal{K}}^\star$,
\begin{eqnarray*}
0&\ge&\mathcal{E}_s\left(w,v_\phi\right)\\
&=&
\int_{\mathbb{R}^n} 
\left|\xi\right|^{2s}\mathcal{F}v_\phi\left(\xi\right)
\overline{\mathcal{F}w\left(\xi\right)} \,d\xi
\\&=&
\int_{\mathbb{R}^n} {\mathcal{F}}^{-1}(
\left|\xi\right|^{2s}\mathcal{F}v_\phi)(x)\,
w\left(x\right) \,dx\\&=&
\int_{B_1} {\mathcal{F}}^{-1}(
\left|\xi\right|^{2s}\mathcal{F}v_\phi)(x)\,
w\left(x\right) \,dx\\&=&
\int_{B_1}\phi(x)\,w\left(x\right) \,dx.
\end{eqnarray*}
Since~$\phi$ is arbitrary and nonnegative, this gives that~$w\le0$,
and this establishes~\eqref{kstar po}.

Furthermore, by Theorem~3.4 in~\cite{MR2667016}, we can write
$$ u=u_1+u_2,$$
with~$u_1\in {\mathcal{K}}\setminus\{0\}$,
$u_2\in{\mathcal{K}}^\star\setminus\{0\}$, and~$\mathcal{E}_s\left(u_1,u_2\right)=0$.

We observe that
$$ \mathcal{E}_s\left(u_1-u_2,u_1-u_2\right)=
\mathcal{E}_s\left(u_1,u_1\right)+\mathcal{E}_s\left(u_2,u_2\right)
+2\mathcal{E}_s\left(u_1,u_2\right)=\mathcal{E}_s\left(u_1,u_1\right)+\mathcal{E}_s\left(u_2,u_2\right).$$
In the same way,
$$ \mathcal{E}_s\left(u,u\right)=
\mathcal{E}_s\left(u_1+u_2,u_1+u_2\right)=
\mathcal{E}_s\left(u_1,u_1\right)+\mathcal{E}_s\left(u_2,u_2\right),$$
and therefore
\begin{equation}\label{7yhbAxcvTFV}
\mathcal{E}_s\left(u_1-u_2,u_1-u_2\right)=\mathcal{E}_s\left(u,u\right).
\end{equation}
On the other hand,
\begin{eqnarray*}
\| u_1-u_2\|_{L^2(B_1)}^2-\| u\|_{L^2(B_1)}^2&=&\| u_1-u_2\|_{L^2(B_1)}^2-\| u_1+u_2\|_{L^2(B_1)}^2\\
&=& -4\int_{B_1} u_1(x)\,u_2(x)\,dx.
\end{eqnarray*}
As a consequence, since~$u_2\le0$ in view of~\eqref{kstar po},
we conclude that
$$ \| u_1-u_2\|_{L^2(B_1)}^2-\| u\|_{L^2(B_1)}^2\ge0.$$
This and~\eqref{7yhbAxcvTFV} say that the function~$u_1-u_2$
is also a minimizer for the variational problem in Lemma~\ref{VARIA}.
Since now~$u_1-u_2\ge0$, the desired result follows.
\end{proof}
\end{lemma}

Now, we define the spherical mean\index{spherical mean} of a function~$v$ by
$$ v_\sharp(x):=
\frac{1}{\left|\mathbb{S}^{n-1}\right|}
\int_{\mathbb{S}^{n-1}} v({\mathcal{R}}_\omega\,x)
\,d{\mathcal{H}}^{n-1}(\omega)
$$
where~${\mathcal{R}}_\omega$ is the rotation corresponding to the solid angle~$\omega
\in{\mathbb{S}^{n-1}}$, ${\mathcal{H}}^{n-1}$ is the standard
Hausdorff measure, and~$\left|\mathbb{S}^{n-1}\right|=
{\mathcal{H}}^{n-1}(\mathbb{S}^{n-1})$.
Notice that~$v_\sharp(x)=v_\sharp({\mathcal{R}}_\varpi x)$
for any~$\varpi \in\mathbb{S}^{n-1}$, that is~$v_\sharp$
is rotationally invariant.

Then, we have:

\begin{lemma}
\label{lapsfercom}
Any positive power of the Laplacian commutes
with the spherical mean, that is
$$ \big( (-\Delta)^s v\big)_\sharp(x)=(-\Delta)^s v_\sharp(x).$$
\begin{proof} By density,
we prove the claim for a function~$v$ in the
Schwartz space of smooth and rapidly decreasing functions.
In this setting, writing~${\mathcal{R}}_\omega^T$
to denote the transpose of the rotation~${\mathcal{R}}_\omega$,
and changing variable~$\eta:={\mathcal{R}}_\omega^T\,\xi$,
we have that
\begin{equation}\label{RFA2}
\begin{split} (-\Delta)^s v({\mathcal{R}}_\omega\,x)\,&=
\int_{{\mathbb{R}}^n} |\xi|^{2s} {\mathcal{F}}v(\xi)\,
e^{2\pi i {\mathcal{R}}_\omega\,x\cdot\xi}\,d\xi\\ &=
\int_{{\mathbb{R}}^n} |\xi|^{2s} {\mathcal{F}}v(\xi)\,
e^{2\pi i x\cdot{\mathcal{R}}_\omega^T\,\xi}\,d\xi\\
&=
\int_{{\mathbb{R}}^n} |\eta|^{2s}
{\mathcal{F}}v({\mathcal{R}}_\omega\,\eta)\,
e^{2\pi i x\cdot\eta}\,d\eta.
\end{split}\end{equation}
On the other hand, using the substitution~$y:={\mathcal{R}}_\omega^T\,x$,
\begin{eqnarray*}
{\mathcal{F}}v({\mathcal{R}}_\omega\,\eta)
&=&\int_{{\mathbb{R}}^n} v(x)\,
e^{-2\pi i x\cdot{\mathcal{R}}_\omega\,\eta}\,dx\\
&=& \int_{{\mathbb{R}}^n} v(x)\,
e^{-2\pi i {\mathcal{R}}_\omega^T\,x\cdot\eta}\,dx\\
&=& \int_{{\mathbb{R}}^n} v({\mathcal{R}}_\omega\,y)\,
e^{-2\pi i y\cdot\eta}\,dy,
\end{eqnarray*}
and therefore, recalling~\eqref{RFA2},
$$ (-\Delta)^s v({\mathcal{R}}_\omega\,x)=
\iint_{{\mathbb{R}}^n\times{\mathbb{R}}^n} |\eta|^{2s}
v({\mathcal{R}}_\omega\,y)\,
e^{2\pi i (x-y)\cdot\eta}\,dy\,d\eta.$$
As a consequence,
\begin{eqnarray*}
\big( (-\Delta)^s v\big)_\sharp(x)&=&
\frac{1}{\left|\mathbb{S}^{n-1}\right|}
\int_{\mathbb{S}^{n-1}} (-\Delta)^s v({\mathcal{R}}_\omega\,x)
\,d{\mathcal{H}}^{n-1}(\omega)
\\ &=&
\frac{1}{\left|\mathbb{S}^{n-1}\right|}
\iiint_{\mathbb{S}^{n-1}\times{\mathbb{R}}^n\times{\mathbb{R}}^n} 
|\eta|^{2s}
v({\mathcal{R}}_\omega\,y)\,
e^{2\pi i (x-y)\cdot\eta}
\,d{\mathcal{H}}^{n-1}(\omega)\,dy\,d\eta\\
&=& \iint_{{\mathbb{R}}^n\times{\mathbb{R}}^n} 
|\eta|^{2s}
v_\sharp(y)\,
e^{2\pi i (x-y)\cdot\eta}\,dy\,d\eta\\&=&
\int_{{\mathbb{R}}^n} 
|\eta|^{2s} {\mathcal{F}} (v_\sharp)(\eta)\,
e^{2\pi i x\cdot\eta}\,d\eta\\&=&(-\Delta)^s v_\sharp(x),
\end{eqnarray*}
as desired.
\end{proof}
\end{lemma}

It is also useful to observe that the spherical mean is compatible with the energy bounds.
In particular we have the following observation:

\begin{lemma}
We have that
\begin{equation}\label{ENblaoe1}
\mathcal{E}_s\left(v_\sharp,v_\sharp\right)\le
\mathcal{E}_s\left(v,v\right).\end{equation}
Moreover,
\begin{equation}\label{ENblaoe2}
{\mbox{if~$v\in H^s_0(B_1)$, then so does~$v_\sharp$.}}\end{equation}

\begin{proof} We see that
\begin{eqnarray*}
{\mathcal{F}}(v_\sharp)(\xi)&=&
\int_{\mathbb{R}^n} v_\sharp(x)\,e^{-2\pi ix\cdot\xi}\,dx\\
&=&
\frac{1}{\left|\mathbb{S}^{n-1}\right|}
\iint_{\mathbb{S}^{n-1}\times{\mathbb{R}^n}} 
v({\mathcal{R}}_\omega\,x)\,e^{-2\pi ix\cdot\xi}
\,d{\mathcal{H}}^{n-1}(\omega)\,dx
\end{eqnarray*}
and therefore, taking the complex conjugated,
$$ \overline{ {\mathcal{F}}(v_\sharp)(\xi) }=
\frac{1}{\left|\mathbb{S}^{n-1}\right|}
\iint_{\mathbb{S}^{n-1}\times{\mathbb{R}^n}} 
v({\mathcal{R}}_\omega\,x)\,e^{2\pi ix\cdot\xi}
\,d{\mathcal{H}}^{n-1}(\omega)\,dx.$$
Hence, by~\eqref{ENstut}, and exploiting the changes of variables~$y:=
{\mathcal{R}}_{\omega}\,x$ and~$\tilde y:=
{\mathcal{R}}_{\tilde\omega}\,\tilde x$,
\begin{eqnarray*}&&
\mathcal{E}_s\left(v_\sharp,v_\sharp\right)\\&=&\int_{\mathbb{R}^n} 
\left|\xi\right|^{2s}\mathcal{F}(v_\sharp)\left(\xi\right)
\overline{\mathcal{F}(v_\sharp)\left(\xi\right)} \,d\xi\\
&=&\frac{1}{\left|\mathbb{S}^{n-1}\right|^2}
\iiiint\!\!\!\int_{\mathbb{S}^{n-1}\times\mathbb{S}^{n-1}\times{\mathbb{R}^n}\times{\mathbb{R}^n}\times{\mathbb{R}^n}} 
|\xi|^{2s}v({\mathcal{R}}_\omega\,x)\,v({\mathcal{R}}_{\tilde\omega}\,\tilde x)\,e^{2\pi i (\tilde x-x)\cdot\xi}
\,d{\mathcal{H}}^{n-1}(\omega)
\,d{\mathcal{H}}^{n-1}(\tilde\omega)\,dx\,d\tilde x\,d\xi\\
&=& \frac{1}{\left|\mathbb{S}^{n-1}\right|^2}
\iiiint\!\!\!\int_{\mathbb{S}^{n-1}\times\mathbb{S}^{n-1}\times{\mathbb{R}^n}\times{\mathbb{R}^n}\times{\mathbb{R}^n}} 
|\xi|^{2s}v(y)\,v(\tilde y)\,e^{2\pi i \tilde y\cdot{\mathcal{R}}_{\tilde\omega}\,\xi}
e^{-2\pi i y\cdot{\mathcal{R}}_\omega\,\xi}
\,d{\mathcal{H}}^{n-1}(\omega)
\,d{\mathcal{H}}^{n-1}(\tilde\omega)\,dy\,d\tilde y\,d\xi\\
&=& \frac{1}{\left|\mathbb{S}^{n-1}\right|^2}
\iiint_{\mathbb{S}^{n-1}\times\mathbb{S}^{n-1}\times{\mathbb{R}^n}} 
|\xi|^{2s}\,{\mathcal{F}}v({\mathcal{R}}_\omega\,\xi)\,\overline{
{\mathcal{F}}v({\mathcal{R}}_{\tilde\omega}\,\xi)}\,
\,d{\mathcal{H}}^{n-1}(\omega)
\,d{\mathcal{H}}^{n-1}(\tilde\omega)\,d\xi.
\end{eqnarray*}
Consequently, using the Cauchy-Schwarz Inequality,
and the substitutions~$\eta:=
{\mathcal{R}}_{\omega}\,\xi$ and~$\tilde\eta:=
{\mathcal{R}}_{\tilde\omega}\,\xi$,
\begin{eqnarray*}
\mathcal{E}_s\left(v_\sharp,v_\sharp\right)&\le&
\frac{1}{\left|\mathbb{S}^{n-1}\right|^2}
\iiint_{\mathbb{S}^{n-1}\times\mathbb{S}^{n-1}\times{\mathbb{R}^n}} 
|\xi|^{2s}\,\big|{\mathcal{F}}v({\mathcal{R}}_\omega\,\xi)\big|
\,\big|
{\mathcal{F}}v({\mathcal{R}}_{\tilde\omega}\,\xi)\big|\,
\,d{\mathcal{H}}^{n-1}(\omega)
\,d{\mathcal{H}}^{n-1}(\tilde\omega)\,d\xi
\\ &\le&
\frac{1}{\left|\mathbb{S}^{n-1}\right|^2}
\left(
\iiint_{\mathbb{S}^{n-1}\times\mathbb{S}^{n-1}\times{\mathbb{R}^n}} 
|\xi|^{2s}\,\big|{\mathcal{F}}v({\mathcal{R}}_\omega\,\xi)\big|^2
\,d{\mathcal{H}}^{n-1}(\omega)
\,d{\mathcal{H}}^{n-1}(\tilde\omega)\,d\xi\right)^{\frac12}\\&&\qquad\cdot
\left(
\iiint_{\mathbb{S}^{n-1}\times\mathbb{S}^{n-1}\times{\mathbb{R}^n}} 
|\xi|^{2s}\,
\big|{\mathcal{F}}v({\mathcal{R}}_{\tilde\omega}\,\xi)\big|^2
\,d{\mathcal{H}}^{n-1}(\omega)
\,d{\mathcal{H}}^{n-1}(\tilde\omega)\,d\xi\right)^{\frac12}\\
\\ &=&
\frac{1}{\left|\mathbb{S}^{n-1}\right|^2}
\left(
\iiint_{\mathbb{S}^{n-1}\times\mathbb{S}^{n-1}\times{\mathbb{R}^n}} 
|\eta|^{2s}\,\big|{\mathcal{F}}v(\eta)\big|^2
\,d{\mathcal{H}}^{n-1}(\omega)
\,d{\mathcal{H}}^{n-1}(\tilde\omega)\,d\eta\right)^{\frac12}\\&&\qquad\cdot
\left(
\iiint_{\mathbb{S}^{n-1}\times\mathbb{S}^{n-1}\times{\mathbb{R}^n}} 
|\tilde\eta|^{2s}\,\big|
{\mathcal{F}}v(\tilde\eta)\big|^2
\,d{\mathcal{H}}^{n-1}(\omega)
\,d{\mathcal{H}}^{n-1}(\tilde\omega)\,d\tilde\eta\right)^{\frac12}
\\ &=&
\left(
\int_{{\mathbb{R}^n}} 
|\eta|^{2s}\,\big|{\mathcal{F}}v(\eta)\big|^2
\,d\eta\right)^{\frac12}
\left(
\int_{\mathbb{R}^{n}} 
|\tilde\eta|^{2s}\,\big|
{\mathcal{F}}v(\tilde\eta)\big|^2
\,d\tilde\eta\right)^{\frac12}
\\&=&
\mathcal{E}_s\left(v,v\right).
\end{eqnarray*}
This proves~\eqref{ENblaoe1}.

Now, we prove~\eqref{ENblaoe2}. For this, we observe that
$$ \frac{\partial^\ell v_\sharp}{\partial x_{j_1}\dots\partial x_{j_\ell}}(x)=
\frac{1}{\left|\mathbb{S}^{n-1}\right|}\sum_{k_1,\dots,k_\ell=1}^n
\int_{\mathbb{S}^{n-1}} 
\frac{\partial^\ell v}{\partial x_{k_1}\dots\partial x_{k_\ell}}
({\mathcal{R}}_\omega\,x)\;{\mathcal{R}}_\omega^{k_1j_1}\dots{\mathcal{R}}_\omega^{k_\ell j_\ell}
\;d{\mathcal{H}}^{n-1}(\omega),$$
for every $\ell\in{\mathbb{N}}$ and~$j_1,\dots,j_\ell\in\{1,\dots,n\}$,
where~${\mathcal{R}}_\omega^{jk}$ denotes the $(j,k)$ component of the matrix~${\mathcal{R}}_\omega$.
In particular,
$$ \left|\frac{\partial^\ell v_\sharp}{\partial x_{j_1}\dots\partial x_{j_\ell}}(x)\right|\le
C\,\sum_{k_1,\dots,k_\ell=1}^n
\int_{\mathbb{S}^{n-1}} 
\left|\frac{\partial^\ell v}{\partial x_{k_1}\dots\partial x_{k_\ell}}
({\mathcal{R}}_\omega\,x)\right|
\;d{\mathcal{H}}^{n-1}(\omega),$$
for some~$C>0$ only depending on~$n$ and~$\ell$, and hence
\begin{eqnarray*} \left\|\frac{\partial^\ell v_\sharp}{\partial x_{j_1}\dots
\partial x_{j_\ell}}(x)\right\|_{L^2(B_1)}^2&\le&
C\,\sum_{k_1,\dots,k_\ell=1}^n
\iint_{\mathbb{S}^{n-1}\times B_1} 
\left|\frac{\partial^\ell v}{\partial x_{k_1}\dots\partial x_{k_\ell}}
({\mathcal{R}}_\omega\,x)\right|^2
\;d{\mathcal{H}}^{n-1}(\omega)\,dx\\&=&
C\,\sum_{k_1,\dots,k_\ell=1}^n
\iint_{\mathbb{S}^{n-1}\times B_1} 
\left|\frac{\partial^\ell v}{\partial x_{k_1}\dots\partial x_{k_\ell}}
(y)\right|^2
\;d{\mathcal{H}}^{n-1}(\omega)\,dy\\&=&
C\,\sum_{k_1,\dots,k_\ell=1}^n
\left\|\frac{\partial^\ell v}{\partial x_{k_1}\dots\partial x_{k_\ell}}
\right\|_{L^2(B_1)}^2,
\end{eqnarray*}
up to renaming~$C$.

This, together with~\eqref{energynorm} and~\eqref{ENblaoe1},
gives~\eqref{ENblaoe2}, as desired.
\end{proof}
\end{lemma}

With this preliminary work, we can now find a nontrivial,
nonnegative
and radial solution of~\eqref{dirfun}.

\begin{proposition}\label{2.52.5}
There exists a solution of~\eqref{dirfun} in~$H^s_0(B_1)$
which is radial, nonnegative
and with unit norm in~$L^2(B_1)$.

\begin{proof} Let~$u\ge0$ be a nontrivial solution of~\eqref{dirfun},
whose existence is warranted by Lemma~\ref{ikAHHPKAK}.

Then, we have that~$u_\sharp\ge0$.
Moreover,
\begin{eqnarray*}&& \int_{B_1}u_\sharp(x)\,dx=
\frac{1}{\left|\mathbb{S}^{n-1}\right|}
\iint_{\mathbb{S}^{n-1}\times B_1} u({\mathcal{R}}_\omega\,x)
\,d{\mathcal{H}}^{n-1}(\omega)\,dx\\&&\qquad=
\frac{1}{\left|\mathbb{S}^{n-1}\right|}
\iint_{\mathbb{S}^{n-1}\times B_1} u(y)
\,d{\mathcal{H}}^{n-1}(\omega)\,dy=
\int_{ B_1} u(y)
\,dy>0,
\end{eqnarray*}
and therefore~$u_\sharp$ does not vanish identically.

As a consequence, we can define
$$u_\star:=\frac{u_\sharp}{\|u_\sharp\|_{L^2(B_1)}}.$$
We know that~$u_\star\in H^s_0(B_1)$, due to~\eqref{ENblaoe2}.
Moreover, in view of Lemma~\ref{lapsfercom},
$$ (-\Delta)^s u_\star=
\frac{(-\Delta)^s u_\sharp}{\|u_\sharp\|_{L^2(B_1)}}=
\frac{\big((-\Delta)^s u\big)_\sharp}{\|u_\sharp\|_{L^2(B_1)}}
=\frac{\lambda_1\,u_\sharp}{\|u_\sharp\|_{L^2(B_1)}}
=\lambda_1\,u_\star,$$
which gives the desired result.
\end{proof}
\end{proposition}

Now, we are in the position of proving the following result.

\begin{lemma}
\label{onesob}
Let $s\ge1$ and~$r\in(0,1)$. If $u\in H_0^s\left(B_1\right)$ and $u$ is radial, then $u\in C^\alpha\left({\mathbb{R}}^n\setminus B_r\right)$ for any $\alpha\in\left[0,\frac{1}{2}\right]$.
\begin{proof}
We write
\begin{equation}
u\left(x\right)=u_0\left(\left|x\right|\right),\qquad\mbox{for some }
\;u_0:[0,+\infty)\rightarrow\mathbb{R}
\end{equation}
and we observe that~$u\in H_0^s\left(B_1\right)\subset H^1\left({\mathbb{R}}^n\right) $.

Accordingly, for any $0<r<1$, we have
\begin{equation}
\label{funz}
\infty>\int_{{\mathbb{R}}^n\setminus B_r} |u(x)|^2 \,dx
=\int_r^{+\infty} |u_0(\rho)|^2\rho^{n-1} \,d\rho\geq 
r^{n-1}\int_r^{+\infty}|u_0(\rho)|^2 \,d\rho
\end{equation}
and
\begin{equation}
\label{grad}
\infty>\int_{{\mathbb{R}}^n\setminus B_r} {|\nabla u(x)|^2 \,dx}=\int_r^{+\infty}
{\left|u_0' (\rho)\right|^2\rho^{n-1} \,d\rho}
\geq r^{n-1}\int_r^{+\infty}{\left|u_0' (\rho)\right|^2 \,d\rho}.
\end{equation}
Thanks to \eqref{funz} and \eqref{grad} we have that $u_0\in H^1\left(\left(r,+\infty\right)\right)$,
with~$u_0=0$ in~$[1,+\infty)$.

Then, from 
the Morrey Embedding Theorem, it follows that
$u_0\in C^\alpha\left(\left(r,+\infty\right)\right)$ for any
$\alpha\in\left[0,\frac{1}{2}\right]$, which leads to the desired result.
\end{proof}
\end{lemma}

\begin{corollary}\label{QUESTCP} Let~$s\in(0,+\infty)$.
There exists a radial, nonnegative
and nontrivial solution of \eqref{dirfun}
which belongs to~$H^s_0(B_1)\cap C^{\alpha}({\mathbb{R}}^n\setminus B_{1/2})$,
for some~$\alpha\in(0,1)$.

\begin{proof} If~$s\in(0,1)$, the desired claim follows from 
Corollary~8 in~\cite{DSV1}.

If instead~$s\ge1$, we obtain the desired result as a consequence of
Proposition \ref{2.52.5} and Lemma \ref{onesob}.
\end{proof}
\end{corollary}

\section{Boundary asymptotics\index{boundary behaviour} of the first eigenfunctions of~$(-\Delta)^s$}
\label{sec5}

In Lemma 4 of~\cite{DSV1}, some precise asymptotics
at the boundary for the first Dirichlet eigenfunction
of~$(-\Delta)^s$ have been established in the range~$s\in(0,1)$.

Here, we obtain a related expansion in the range~$s>0$
for the eigenfunction provided in Corollary~\ref{QUESTCP}.
The result that we obtain is the following:

\begin{proposition}
\label{sharbou}
There exists a nontrivial solution $\phi_*$ of \eqref{dirfun}
which belongs to~$H^s_0(B_1)\cap C^{\alpha}
({\mathbb{R}}^n\setminus B_{1/2})$,
for some~$\alpha\in(0,1)$, and such that, for every~$e\in\partial B_1$
and~$\beta=(\beta_1,\dots,\beta_n)\in{\mathbb{N}}^n$,
\[\lim_{\epsilon\searrow 0}\epsilon^{\left|\beta\right|-s}\partial^\beta\phi_*\left(e+\epsilon X\right)=\left(-1\right)^{\left|\beta\right|}k_*\, s\left(s-1\right)\ldots\left(s-\left|\beta\right|+1\right)e_1^{\beta_1}\ldots e_n^{\beta_n}\left(-e\cdot X\right)_+^{s-\left|\beta\right|},\]
in the sense of distribution, with~$|\beta|:=\beta_1+\dots+\beta_n$
and~$k_*>0$.\end{proposition}

The proof of Proposition~\ref{sharbou} relies on Proposition~\ref{LEJOS}
and some auxiliary computations on the Green function in~\eqref{GREEN}.
We start with the following result:

\begin{lemma}
\label{lklkl}
Let $0<r<1$, $e\in\partial B_1$, $s>0$,
$f\in C^\alpha(\mathbb{R}^n\setminus B_r)\cap L^2(\mathbb{R}^n)$ 
for some $\alpha\in(0,1)$, and $f=0$ outside $B_1$. Then the integral
\begin{equation}
\label{I1I2}
\int_{B_1} f(z)\frac{(1-|z|^2)^s}{s|z-e|^n} \,dz
\end{equation}
is finite.
\begin{proof}
We denote by~$I$ the integral in~\eqref{I1I2}. We let
$$ I_1:=
\int_{B_1\setminus B_r} f(z)\frac{(1-|z|^2)^s}{s|z-e|^n} \,dz
\qquad{\mbox{and}}\qquad I_2:=
\int_{B_r} f(z)\frac{(1-|z|^2)^s}{s|z-e|^n} \,dz.$$
Then, we have that
\begin{equation}\label{SsalKAM:1} I=I_1+I_2.\end{equation}
Now, if $z\in B_1\setminus B_r$, we have that
\begin{equation*}
f(z)\leq|f(z)-f(e)|\leq C|z-e|^\alpha,
\end{equation*}
therefore
\begin{equation}\label{SsalKAM:2}
I_1\leq\int_{B_1\setminus B_r}\frac{(1-|z|^2)^s}{s|z-e|^{n-\alpha}} \,dz<\infty.
\end{equation}
If instead~$z\in B_r$,
\begin{equation*}
|z-e|\geq 1-r>0,
\end{equation*}
and consequently
\begin{equation}\label{SsalKAM:3}
I_2\leq \frac{1}{s\,(1-r)^n}\int_{B_r}f(z)\, dz<\infty.
\end{equation}
The desired result follows from~\eqref{SsalKAM:1},
\eqref{SsalKAM:2} and~\eqref{SsalKAM:3}.
\end{proof}
\end{lemma}

Next result gives
a precise boundary behaviour of the Green function
for any $s>0$ (the case in which~$s\in(0,1)$ and~$f\in
C^\alpha(\mathbb{R}^n)$ was considered in Lemma~6 of~\cite{DSV1},
and in fact the proof presented here also simplifies the one in
Lemma~6 of~\cite{DSV1} for the setting considered there).

\begin{lemma}
\label{lemsix}
Let $e$, $\omega\in\partial B_1$, $\epsilon_0>0$ and~$r\in(0,1)$. 
Assume that
\begin{equation}\label{CHIAmahdfn}
e+\epsilon\omega\in B_1,\end{equation}
for any $\epsilon\in(0,\epsilon_0]$. Let $f\in C^\alpha(\mathbb{R}^n\setminus B_r)\cap L^2(\mathbb{R}^n)$ for some $\alpha\in(0,1)$, with $f=0$ outside $B_1$. \\
Then
\begin{equation}
\lim_{\epsilon\searrow 0}\epsilon^{-s}\int_{B_1} f(z)
\mathcal{G}_s(e+\epsilon\omega,z) \,dz=k(n,s)\,
\int_{B_1} f(z)\frac{(-2e\cdot\omega)^s(1-|z|^2)^s}{s|z-e|^n}\, dz,
\end{equation}
for a suitable normalizing constant~$k(n,s)>0$.

\begin{proof} In light of~\eqref{CHIAmahdfn}, we have that
\begin{equation*}1>
|e+\epsilon\omega|^2=1+\epsilon^2+2\epsilon e\cdot\omega,
\end{equation*}
and therefore
\begin{equation}\label{RGHHCicj}
-e\cdot\omega>\frac{\epsilon}{2}>0.
\end{equation}
Moreover, if $r_0$ is as given in \eqref{GREEN}, we have that, for all~$z\in B_1$,
\begin{equation}
\label{kfg}
r_0(e+\epsilon\omega,z)=\frac{\epsilon(-\epsilon-2e\cdot\omega)(1-|z|^2)}{|z-e-\epsilon\omega|^2}\leq\frac{3\epsilon}{|z-e-\epsilon\omega|^2}.
\end{equation}
Also, a Taylor series representation allows us to write, for any~$t\in(-1,1)$,
\begin{equation}\label{7uJAJMMA.aA}
\frac{t^{s-1}}{(t+1)^{\frac{n}{2}}}=\sum_{k=0}^\infty \binom{-n/2}{k}t^{k+s-1}.
\end{equation}
We also notice that
\begin{equation}
\label{bound}\begin{split}&
\left|\binom{-n/2}{k}\right|=
\left|
\frac{-\frac{n}2
\left(-\frac{n}2-1\right)\,...\,\left(-\frac{n}2-k+1\right)
}{k!}
\right|
=
\frac{\frac{n}2
\left(\frac{n}2+1\right)\,...\,\left(\frac{n}2+(k-1)\right)
}{k!}\\
&\qquad\le\frac{n
\left(n+1\right)\,...\,\left(n+(k-1)\right)
}{k!}
\le\frac{\left(n+(k-1)\right)!
}{k!}= (k+1)\,...\,\left(n+(k-1)\right)\\&\qquad
\le (n+k+1)^{n+1}.
\end{split}\end{equation}
This and the Root Test
give that the series in~\eqref{7uJAJMMA.aA}
is uniformly convergent on compact sets in $(-1,1)$.

As a consequence, if we set
\begin{equation}
\label{min}
r_1(x,z):=\min\left\{\frac{1}{2},r_0(x,z)\right\},
\end{equation}
we can switch integration and summation signs and obtain that
\begin{equation}
\label{mmno}
\int_0^{r_1(x,z)} \frac{t^{s-1}}{(t+1)^{\frac{n}{2}}}\, dt=\sum_{k=0}^\infty c_k(r_1(x,z))^{k+s},
\end{equation}
where
\begin{equation*}
c_k:=\frac{1}{k+s}\binom{-n/2}{k}.
\end{equation*}
Once again, the bound in~\eqref{bound}, together with~\eqref{min},
give that the series in~\eqref{mmno} is convergent.

Now, we omit for simplicity the normalizing constant~$k(n,s)$
in the definition of the Green function in~\eqref{GREEN},
and
we define
\begin{equation}\label{GNAKDDEF}
\mathcal{G}(x,z):=|z-x|^{2s-n}\sum_{k=0}^\infty c_k(r_1(x,z))^{k+s}
\end{equation}
and
\begin{equation*}
g(x,z):=|z-x|^{2s-n}\int_{r_1(x,z)}^{r_0(x,z)} \frac{t^{s-1}}{(t+1)^{\frac{n}{2}}} \,dt.
\end{equation*}
Using~\eqref{GREEN}
and \eqref{mmno}, and dropping dimensional constants for the sake of shortness,
we can write
\begin{equation}
\label{splitgreen}
\mathcal{G}_s(x,z)=\mathcal{G}(x,z)+g(x,z).
\end{equation}
Now, we show that
\begin{equation}\label{VKLACL}
g(x,z)\leq
\begin{cases} C\chi(r,z)\,|z-x|^{2s-n}&\quad\text{if}\quad n>2s, \\ 
C\chi(r,z)\,\log r_0(x,z)&\quad\text{if}\quad n=2s, \\
C\chi(r,z)\,|z-x|^{2s-n}(r_0(x,z))^{s-\frac{n}{2}}&\quad\text{if}\quad n<2s, \end{cases}
\end{equation}
where~$\chi(r,z)=1$ if~$r_0(x,z)> \frac{1}{2}$
and~$\chi(r,z)=0$ if~$r_0(x,z)\leq \frac{1}{2}$.
To check this,
we notice that if~$r_0(x,z)\leq \frac{1}{2}$ we have that~$r_1(x,z)=r_0(x,z)$,
due to~\eqref{min}, and therefore~$g(x,z)=0$. 

On the other hand, if $r_0(x,z)>\frac{1}{2}$, 
we deduce from~\eqref{min} that~$r_1(x,z)=\frac12$, and consequently
\begin{equation*}
g(x,z)\leq |z-x|^{2s-n}\int_{1/2}^{r_0(x,z)} t^{s-\frac{n}{2}-1} dt\leq
\begin{cases} C|z-x|^{2s-n}&\quad\text{if}\quad n>2s, \\ 
C\log r_0(x,z)&\quad\text{if}\quad n=2s, \\
C|z-x|^{2s-n}(r_0(x,z))^{s-\frac{n}{2}}&\quad\text{if}\quad n<2s, \end{cases}
\end{equation*}
for some constant $C>0$. This completes the proof of~\eqref{VKLACL}.

Now, we exploit the bound in~\eqref{VKLACL} when $x=e+\epsilon\omega$. 
For this, we notice that 
if~$r_0(e+\epsilon\omega,z)>\frac12$,
recalling \eqref{kfg}, we find that
\begin{equation}\label{75g8676769}
|z-e-\epsilon\omega|^2\leq 6\epsilon<9\epsilon,
\end{equation}
and therefore $z\in B_{3\sqrt{\epsilon}}(e+\epsilon\omega)$.

Hence, using~\eqref{VKLACL},
\begin{equation}
\label{estimates1}
\begin{split}
&\left|\int_{B_1} f(z)g(e+\epsilon\omega,z) dz\right|\leq 
\int_{B_{3\sqrt{\epsilon}}(e+\epsilon\omega)}
|f(z)||g(e+\epsilon\omega,z)| dz \\ &
\leq\begin{cases} C\displaystyle\int_{B_{3\sqrt{\epsilon}}(e+\epsilon\omega)}
|f(z)||z-e-\epsilon\omega|^{2s-n} dz&\quad\text{if}\quad n>2s,\\
C\displaystyle\int_{B_{3\sqrt{\epsilon}}(e+\epsilon\omega)}|f(z)|\log r_0(
e+\epsilon\omega,z) dz&\quad\text{if}\quad n=2s, \\
C\displaystyle\int_{B_{3\sqrt{\epsilon}}(e+\epsilon\omega)}
|f(z)||z-e-\epsilon\omega|^{2s-n}(r_0(e+\epsilon\omega,z))^{
s-\frac{n}{2}} dz &\quad\text{if}\quad n<2s .\end{cases}
\end{split}
\end{equation}
Now, if $z\in B_{3\sqrt{\epsilon}}(e+\epsilon\omega)$,
then \begin{equation}\label{z0okdscxi}
|z-e|\leq |z-e-\epsilon\omega|+|\epsilon\omega|
\leq 3\sqrt{\epsilon}+\epsilon<4\sqrt{\epsilon}.\end{equation}
Furthermore,
for a given~$r\in(0,1)$, we have that~$
B_{3\sqrt{\epsilon}}(e+\epsilon\omega)\subseteq{\mathbb{R}}^n\setminus
B_r$, provided that~$\epsilon$ is sufficiently small.

Hence, if~$z\in B_{3\sqrt{\epsilon}}(e+\epsilon\omega)$,
we can exploit the regularity of~$f$ and deduce that
$$ |f(z)|=|f(z)-f(e)|\leq C|z-e|^\alpha.$$
This and~\eqref{z0okdscxi} lead to
\begin{equation}
\label{su}
|f(z)|\leq C\epsilon^{\frac{\alpha}{2}},
\end{equation}
for every~$z\in B_{3\sqrt{\epsilon}}(e+\epsilon\omega)$.

Thanks to \eqref{kfg}, \eqref{estimates1} and \eqref{su}, we have that
\begin{equation*}
\begin{split}
\left|\int_{B_1} f(z)g(e+\epsilon\omega,z) dz\right|&\leq\begin{cases} C\epsilon^{\frac{\alpha}{2}}
\displaystyle\int_{B_{3\sqrt{\epsilon}}(e+\epsilon\omega)}|z-e-\epsilon\omega|^{2s-n} dz&\quad\text{if}\quad n>2s,\\ C\epsilon^{\frac{\alpha}{2}}
\displaystyle\int_{B_{3\sqrt{\epsilon}}(e+\epsilon\omega)}\log \frac{3\epsilon}{|z-e-\epsilon\omega|^2} dz&\quad\text{if}\quad n=2s ,\\ C\epsilon^{\frac{\alpha}{2}+s-\frac{n}{2}}\displaystyle\int_{B_{3\sqrt{\epsilon}}(e+\epsilon\omega)} dz &\quad\text{if}\quad n<2s \end{cases} \\
&\leq C\epsilon^{\frac{\alpha}{2}+s},
\end{split}
\end{equation*}
up to renaming~$C$.

This and \eqref{splitgreen} give that
\begin{equation}
\label{alalal}
\int_{B_1} f(z)\mathcal{G}_s(e+\epsilon\omega,z) dz=\int_{B_1} f(z)\mathcal{G}(e+\epsilon\omega,z) dz+o(\epsilon^s).
\end{equation}
Now, we consider the series in~\eqref{GNAKDDEF},
and we split the contribution coming from the index $k=0$ from the ones coming from the indices $k>0$, namely
we write
\begin{equation}\label{IN16p}
\begin{split}
&\mathcal{G}(x,z)=\mathcal{G}_0(x,z)+\mathcal{G}_1(x,z), \\
\quad\text{with}\quad &\mathcal{G}_0(x,z):=\frac{|z-x|^{2s-n}}{s}(r_1(x,z))^s \\
\quad\text{and}\quad &\mathcal{G}_1(x,z):=|z-x|^{2s-n}\sum_{k=1}^{+\infty} c_k\,(r_1(x,z))^{k+s}.
\end{split}
\end{equation}
Firstly, we consider the contribution given by the term $\mathcal{G}_1$. Thanks to \eqref{min} and \eqref{su}, we have that
\begin{equation}
\label{estimates5}
\begin{split}
&\left|\int_{B_1\cap B_{3\sqrt{\epsilon}}(e+\epsilon\omega)}f(z)\mathcal{G}_1(e+\epsilon\omega,z) dz\right|\leq \int_{B_{3\sqrt{\epsilon}}(e+\epsilon\omega)} |f(z)|\mathcal{G}_1(e+\epsilon\omega,z) dz \\
&\quad\quad\leq C\epsilon^{\frac{\alpha}{2}}\int_{B_{3\sqrt{\epsilon}}(e+\epsilon\omega)}|z-e-\epsilon\omega|^{2s-n}\sum_{k=1}^{+\infty} |c_k|\,(r_1(e+\epsilon\omega,z))^{k+s} dz \\
&\quad\quad\leq C\epsilon^{\frac{\alpha}{2}}\int_{B_{3\sqrt{\epsilon}}(e+\epsilon\omega)}|z-e-\epsilon\omega|^{2s-n}\sum_{k=1}^{+\infty} |c_k|\,\left(\frac{1}{2}\right)^{k+s} dz \\
&\quad\quad\leq C\epsilon^{\frac{\alpha}{2}}\int_{B_{3\sqrt{\epsilon}}(e+\epsilon\omega)}|z-e-\epsilon\omega|^{2s-n} dz \\
&\quad\quad\leq C\epsilon^{\frac{\alpha}{2}+s},
\end{split}
\end{equation}
up to renaming the constant $C$ step by step. 

On the other hand, for every~$z\in{\mathbb{R}}^n$,
\[|z|=|e+\epsilon\omega+z-e-\epsilon\omega|
\geq|e+\epsilon\omega|-|z-e-\epsilon\omega|
\geq 1-\epsilon-|z-e-\epsilon\omega|.\] 
Therefore,
for every~$z\in B_1\setminus\left(B_r\cup B_{3\sqrt{\epsilon}}
(e+\epsilon\omega)\right)$, we can take~$e_*:=\frac{z}{|z|}$ and obtain that
\begin{equation}
\label{estimate2}
\begin{split}&
|f(z)|=|f(z)-f(e_*)|\leq
C|z-e_*|^\alpha= C(1-|z|)^\alpha
\\&\qquad\leq C(\epsilon+|z-e-\epsilon\omega|)^\alpha
\leq C|z-e-\epsilon\omega|^\alpha,
\end{split}\end{equation}
up to renaming~$C>0$.

Also, using \eqref{kfg}, we see that, for any $k>0$,
\begin{equation}
\label{estimates3}
\begin{split}
%% &(r_1(e+\epsilon\omega,z))^{k+s}=(r_1(e+\epsilon\omega,z))^{s+\frac{\alpha}{4}}(r_1(e+\epsilon\omega,z))^{k-\frac{\alpha}{4}} \\
%% {\mbox{and }}\quad\quad
&(r_0(e+\epsilon\omega,z))^{s+\frac{\alpha}{4}}\left(\frac{1}{2}\right)^{k-\frac{\alpha}{4}}\leq\frac{C\epsilon^{s+\frac{\alpha}{4}}}{2^k|z-e-\epsilon\omega|^{2s+\frac{\alpha}{2}}}.
\end{split}
\end{equation}
This, \eqref{min} and \eqref{estimate2} give that if $z\in B_1\setminus\left(B_r\cup B_{3\sqrt{\epsilon}}(e+\epsilon\omega)\right)$, then
\begin{equation*}
\begin{split}
|f(z)\mathcal{G}_1(e+\epsilon\omega,z)| \,&
\leq C|z-e-\epsilon\omega|^{\alpha+2s-n}
\sum_{k=1}^{+\infty}|c_k|\,(r_1(e+\epsilon\omega,z))^{k+s} \\
&=C|z-e-\epsilon\omega|^{\alpha+2s-n}
\sum_{k=1}^{+\infty}|c_k|\,(r_1(e+\epsilon\omega,z))^{s+\frac\alpha4}
(r_1(e+\epsilon\omega,z))^{k-\frac\alpha4}\\
&\le C|z-e-\epsilon\omega|^{\alpha+2s-n}
\sum_{k=1}^{+\infty}|c_k|\,(r_0(e+\epsilon\omega,z))^{s+\frac\alpha4}
\left(\frac12\right)^{k-\frac\alpha4}\\
&\leq C\epsilon^{s+\frac{\alpha}{4}}|z-e-\epsilon\omega|^{\frac{\alpha}{2}-n}\sum_{k=1}^{+\infty}\frac{|c_k|}{2^k},
\end{split}
\end{equation*}
where the latter series is absolutely convergent thanks to \eqref{bound}. 

This implies that, if we
set $E:=B_1\setminus\left(B_r\cup B_{3\sqrt{\epsilon}}
(e+\epsilon\omega)\right)$, it holds that
\begin{equation}
\label{estimates4}
\begin{split}
&\left|\int_E f(z)\mathcal{G}_1(e+\epsilon\omega,z) dz\right|\leq C\epsilon^{s+\frac{\alpha}{4}}\int_E |z-e-\epsilon\omega|^{\frac{\alpha}{2}-n} dz \\
\quad\quad &\qquad\qquad\leq C\epsilon^{s+\frac{\alpha}{4}}\int_{B_1} |z-e-\epsilon\omega|^{\frac{\alpha}{2}-n} dz\leq C\epsilon^{s+\frac{\alpha}{4}}\int_{B_3} |z|^{\frac{\alpha}{2}-n} dz \leq C\epsilon^{s+\frac{\alpha}{4}}.
\end{split}
\end{equation}
Moreover, if $z\in B_r$, we have that
\begin{equation*}
|e+\epsilon\omega-z|\geq 1-\epsilon-r,
\end{equation*}
and therefore, recalling~\eqref{estimates3},
\begin{equation*}
\begin{split}
\sup_{z\in B_r} |\mathcal{G}_1(e+\epsilon\omega,z)|\,&\leq |z-e-\epsilon\omega|^{2s-n}\sum_{k=1}^{+\infty}
|c_k|\,\big(r_1(e+\epsilon\omega,z)\big)^{s+\frac\alpha4}
\big(r_1(e+\epsilon\omega,z)\big)^{k-\frac\alpha4}\\&\le
|z-e-\epsilon\omega|^{2s-n}\sum_{k=1}^{+\infty}
|c_k|\,\big(r_0(e+\epsilon\omega,z)\big)^{s+\frac\alpha4}
\left(\frac12\right)^{k-\frac\alpha4}
\\ &\le C\,
|z-e-\epsilon\omega|^{-n-\frac\alpha2}\sum_{k=1}^{+\infty}
\frac{|c_k|}{2^k}
\\ &\le
C(1-\epsilon-r)^{-n-\frac{\alpha}{2}}\,
\epsilon^{s+\frac{\alpha}{4}},
\end{split}\end{equation*}
up to renaming~$C$.

As a consequence, we find that
\begin{equation}
\label{estimates6}
\begin{split}
\left|\int_{B_r} f(z)\mathcal{G}_1(e+\epsilon\omega,z) dz\right|
&\leq 
\sup_{z\in B_r} |\mathcal{G}_1
(e+\epsilon\omega,z)|\,
\left\|f\right\|_{L^1(B_r)}
\\ &\leq \left\|f\right\|_{L^1(B_r)}(1-\epsilon-r)^{-n-\frac{\alpha}{2}}\epsilon^{s+\frac{\alpha}{4}}
\\
&\leq \left\|f\right\|_{L^1(B_r)}2^{n+\frac{\alpha}{2}}(1-r)^{-n-\frac{\alpha}{2}}\epsilon^{s+\frac{\alpha}{4}}
\\
&=C\epsilon^{s+\frac{\alpha}{4}},
\end{split}
\end{equation}
as long as $\epsilon$ is suitably
small with respect to $r$, and $C$ is a positive constant
which depends on $\|f\|_{L^1(B_r)}$, $r$, $n$ and~$\alpha$.

Then, by \eqref{estimates5}, \eqref{estimates4} and \eqref{estimates6} we conclude that
\begin{equation}
\int_{B_1} f(z)\mathcal{G}_1(e+\epsilon\omega,z) dz=o(\epsilon^s).
\end{equation}
Inserting this information into \eqref{alalal}, 
and recalling~\eqref{IN16p},
we obtain
\begin{equation}
\label{ololol}
\int_{B_1} f(z)\mathcal{G}_s(e+\epsilon\omega,z) dz=\int_{B_1} f(z)\mathcal{G}_0(e+\epsilon\omega,z) dz+o(\epsilon^s).
\end{equation}
Now, we define \[\mathcal{D}_1:=\left\{z\in B_1\quad\text{s.t.}\quad r_0(e+\epsilon\omega,z)>1/2\right\}\] and \[\mathcal{D}_2:=\left\{z\in B_1\quad\text{s.t.}\quad r_0(e+\epsilon\omega,z)\leq 1/2\right\}.\]
If $z\in\mathcal{D}_1$, then $z\in B_1\setminus B_r$, thanks to~\eqref{75g8676769},
and hence we can use \eqref{estimates1} and \eqref{su} and write 
\[|f(z)\mathcal{G}_0(e+\epsilon\omega,z)|\leq 
C\epsilon^{\frac{\alpha}{2}}|z-e-\epsilon\omega|^{2s-n}.\] 
Then, recalling again \eqref{estimates1},
\begin{equation}
\label{D1}
\left|\int_{\mathcal{D}_1} f(z)\mathcal{G}_1(e+\epsilon\omega,z) dz\right|\leq C\epsilon^{\frac{\alpha}{2}}\int_{B_{3\sqrt{\epsilon}}(e+\epsilon\omega)} |z-e-\epsilon\omega|^{2s-n} dz=C\epsilon^{\frac{\alpha}{2}+s},
\end{equation}
up to renaming the constant $C>0$. This information and \eqref{ololol} give that
\begin{equation*}
\int_{B_1} f(z)\mathcal{G}_s(e+\epsilon\omega,z) dz=
\int_{\mathcal{D}_2} f(z)\mathcal{G}_0(e+\epsilon\omega,z) dz+o(\epsilon^s).
\end{equation*}
Now, by \eqref{kfg} and \eqref{min}, if $z\in\mathcal{D}_2$,
\begin{equation*}
\mathcal{G}_0(e+\epsilon\omega,z)=\frac{|z-e-\epsilon\omega|^{2s-n}}{s}(r_0(e+\epsilon\omega))^s=\frac{\epsilon^s(-\epsilon-2e\cdot\omega)^s(1-|z|^2)^s}{s|z-e-\epsilon\omega|^n}.
\end{equation*}
Hence, we have
\begin{equation}
\label{sesa}
\begin{split}
&\lim_{\epsilon\searrow 0}\epsilon^{-s}\int_{B_1} f(z)\mathcal{G}_s(e+\epsilon\omega,z) dz \\
=\;&\lim_{\epsilon\searrow 0}\epsilon^{-s}\int_{\mathcal{D}_2} f(z)\mathcal{G}_0(e+\epsilon\omega,z) dz
\\
=\;&\lim_{\epsilon\searrow 0} \int_{\left\{2\epsilon(-\epsilon-2e\cdot\omega)(1-|z|^2)\leq|z-e-\epsilon\omega|^2\right\}} f(z)\frac{(-\epsilon-2e\cdot\omega)^s(1-|z|^2)^s}{s|z-e-\epsilon\omega|^n} dz.
\end{split}
\end{equation}
Now we set
\begin{equation}\label{H7uJA78JsadA}
F_\epsilon(z):=\begin{cases}f(z)
\displaystyle\frac{(-\epsilon-2e\cdot\omega)^s(1-|z|^2)^s}{
s|z-e-\epsilon\omega|^n}&\quad\text{if}\quad 2\epsilon
(-\epsilon-2e\cdot\omega)(1-|z|^2)\leq|z-e-\epsilon\omega|^2, \\
0&\quad\text{otherwise}, \end{cases}
\end{equation}
and we prove that for any $\eta>0$ there exists $\delta>0$ independent
of~$\epsilon$ such that, for any $E\subset\mathbb{R}^n$ with $|E|\leq\delta$, we have
\begin{equation}
\label{eta}
\int_{B_1\cap E} |F_\epsilon(z)| dz\leq\eta.
\end{equation}
To this aim, given~$\eta$ and~$E$ as above,
we define
\begin{equation} \label{RGANrho}
\rho:= \min\left\{
\epsilon (-\epsilon-2e\cdot\omega),\,
\sqrt{{ 2\epsilon (-\epsilon-2e\cdot\omega)}(1-r)},\,
\left(
\frac{
2^{s+\alpha} s^2\,\epsilon^{s+\alpha}\,(-\epsilon-2e\cdot\omega)^\alpha
\eta}{3^{2s}\,\|f\|_{C^\alpha(B_1\setminus B_r)}\,|\partial B_1|}
\right)^{\frac1{2\alpha}}
\right\}.\end{equation}
We stress that the above definition is well-posed, thanks to~\eqref{RGHHCicj}.
In addition, using the integrability of~$f$, we take~$\delta>0$
such that if~$A\subseteq B_1$ and~$|A|\le\delta$ then
\begin{equation}\label{PPKAjnaOP} \int_{A} |f(x)|\,dx\le \frac{s\rho^n\eta}{2\cdot 3^s}.
\end{equation}
We set 
\begin{equation}\label{9ikjendE} E_1:=E\cap B_{\rho}(e+\epsilon\omega)\qquad{\mbox{and}}\qquad
E_2:=E\setminus B_{\rho}(e+\epsilon\omega).\end{equation}
{F}rom~\eqref{H7uJA78JsadA}, we see that
$$ |F_\epsilon(z)|\le 
\frac{|f(z)|\,\chi_\star(z)}{2^s s\,\epsilon^s
|z-e-\epsilon\omega|^{n-2s}},$$
where
$$ \chi_\star(z):=\begin{cases}
1 & \quad\text{if}\quad 2\epsilon
(-\epsilon-2e\cdot\omega)(1-|z|^2)\leq|z-e-\epsilon\omega|^2, \\
0&\quad\text{otherwise},
\end{cases}$$
and therefore
\begin{equation}\label{THAnaoa}
\int_{B_1\cap E_1} |F_\epsilon(z)|\,dz
\le \int_{B_1\cap E_1}\frac{|f(z)|\,\chi_\star(z)}{2^s s\,\epsilon^s
|z-e-\epsilon\omega|^{n-2s}}\,dz.
\end{equation}
Now, for every~$z\in B_1\cap E_1\subseteq B_{\rho}(e+\epsilon\omega)$ for which~$\chi_\star(z)\ne0$,
we have that
$$ 2\epsilon
(-\epsilon-2e\cdot\omega)(1-|z|^2)\le|z-e-\epsilon\omega|^2\le\rho^2,$$
and hence
$$ |z|\ge \sqrt{1-\frac{\rho^2}{ 2\epsilon (-\epsilon-2e\cdot\omega)}}
\ge 1-\frac{\rho^2}{ 2\epsilon (-\epsilon-2e\cdot\omega)},$$
which in turn gives that~$|z|\ge r$, recall~\eqref{RGANrho}.

{F}rom this and~\eqref{THAnaoa} we deduce that
\begin{equation}\label{9ikjendE2}
\begin{split}&
\int_{B_1\cap E_1} |F_\epsilon(z)|\,dz
\le \int_{1-\frac{\rho^2}{ 2\epsilon (-\epsilon-2e\cdot\omega)}\le|z|<1
}\frac{\|f\|_{C^\alpha(B_1\setminus B_r)}\,(1-|z|)^\alpha}{2^s s\,\epsilon^s
|z-e-\epsilon\omega|^{n-2s}}\,dz\\&\qquad
\le \frac{\|f\|_{C^\alpha(B_1\setminus B_r)}}{2^s s\,\epsilon^s}\,
\left(\frac{\rho^2}{ 2\epsilon (-\epsilon-2e\cdot\omega)}\right)^\alpha
\int_{1-\frac{\rho^2}{ 2\epsilon (-\epsilon-2e\cdot\omega)}\le|z|<1
}\frac{dz}{|z-e-\epsilon\omega|^{n-2s}}\\&\qquad
\le \frac{\|f\|_{C^\alpha(B_1\setminus B_r)}}{2^s s\,\epsilon^s}\,
\left(\frac{\rho^2}{ 2\epsilon (-\epsilon-2e\cdot\omega)}\right)^\alpha
\int_{B_3}\frac{dx}{|x|^{n-2s}}
\\&\qquad=\frac{3^{2s}\,\|f\|_{C^\alpha(B_1\setminus B_r)}\,
|\partial B_1|}{
2^{s+\alpha+1} s^2\,\epsilon^{s+\alpha}\,(-\epsilon-2e\cdot\omega)^\alpha }\,
\;\rho^{2\alpha}\\&\qquad\le
\frac{\eta}{2},
\end{split}\end{equation}
where~\eqref{RGANrho} has been exploited in the last inequality.

We also point out that, by~\eqref{H7uJA78JsadA},
\eqref{PPKAjnaOP} and~\eqref{9ikjendE},
\begin{eqnarray*}
\int_{B_1\cap E_2}|F_\epsilon(z)|\,dz
&\le&\int_{(B_1\setminus
B_{\rho}(e+\epsilon\omega))\cap E}
|f(z)|\,\frac{(-\epsilon-2e\cdot\omega)^s(1-|z|^2)^s}{
s|z-e-\epsilon\omega|^n}\,dz\\
&\le&\frac{3^s}{s\rho^n}\int_{B_1\cap E}|f(z)|\,dz\\&\le&\frac{\eta}{2}.
\end{eqnarray*}
This, \eqref{9ikjendE} and~\eqref{9ikjendE2} give~\eqref{eta},
as desired.

Notice also that the sequence $F_\epsilon(z)$ converges pointwise
to the function \[F(z):=f(z)\frac{(-2e\cdot\omega)^s(1-|z|^2)^s}{s|z-e|^n}.\]
Hence \eqref{sesa}, \eqref{eta} and the Vitali Convergence Theorem allow us to conclude that
\begin{equation}\label{lkl5678kl}
\begin{split}
\lim_{\epsilon\searrow 0}\int_{B_1} f(z)\mathcal{G}_s(e+\epsilon\omega,z) dz&=\lim_{\epsilon\searrow 0}\int_{B_1} F_\epsilon(z) dz \\
&=\int_{B_1} f(z)\frac{(-2e\cdot\omega)^s(1-|z|^2)^s}{s|z-e|^n} dz,
\end{split}
\end{equation}
which establishes the claim of Lemma \ref{lemsix}
(notice that the finiteness of the latter quantity in~\eqref{lkl5678kl}
follows from~\eqref{lklkl}).
\end{proof}
\end{lemma}

With this preliminary work,
we can now establish the boundary behaviour of solutions which is needed
in our setting. As a matter of fact, from Lemma~\ref{lemsix} we immediately deduce that:

\begin{corollary}
\label{propsev}
Let $e$, $\omega\in\partial B_1$, $\epsilon_0>0$
and~$r\in(0,1)$. 

Assume that $e+\epsilon\omega\in B_1$, for any $\epsilon\in(0,\epsilon_0]$. Let $f\in C^\alpha(\mathbb{R}^n\setminus B_r)\cap L^2(\mathbb{R}^n)$ for some $\alpha\in(0,1)$, with $f=0$ outside $B_1$. 

Let $u$ be as in~\eqref{0olwsKA}.
Then,
\begin{equation*}
\lim_{\epsilon\searrow 0}\epsilon^{-s}u(e+\epsilon\omega)=k(n,s)(-2e\cdot\omega)^s\int_{B_1} f(z)\frac{(1-|z|^2)^s}{s|z-e|^n} dz,
\end{equation*}
where $k(n,s)$ denotes a positive normalizing constant.
\end{corollary}

Now we apply the previous results to detect the
boundary growth of a suitable
first eigenfunction. For our purposes, the statement that we need
is the following:

\begin{corollary}\label{8iJJAUMPAAAxc}
There exists a nontrivial solution $\phi_*$ of \eqref{dirfun}
which belongs to~$H^s_0(B_1)\cap C^{\alpha}
({\mathbb{R}}^n\setminus B_{1/2})$,
for some~$\alpha\in(0,1)$, and such that, for every~$e\in\partial B_1$,
\begin{equation}\label{CHaert}
\lim_{\epsilon\searrow 0}\epsilon^{-s}\phi_*(e+\epsilon\omega)=k_*\,
(-e\cdot\omega)^s_+,
\end{equation}
for a suitable constant~$k_*>0$.

Furthermore,
for every $R\in(r,1)$, there exists~$C_R>0$ such that
\begin{equation}\label{COSI}
\sup_{x\in B_1\setminus B_R}
d^{-s}(x)\,|\phi_*(x)|\le C_R.\end{equation}

\begin{proof}
Let~$\alpha\in(0,1)$ and~$\phi\in H^s_0(B_1)\cap C^{\alpha}({\mathbb{R}}^n\setminus B_{1/2})$
be the nonnegative
and nontrivial solution of \eqref{dirfun}, as given in Corollary~\ref{QUESTCP}.

In the spirit of~\eqref{0olwsKA},
we define
$$
\phi_*(x):=
\begin{cases}
\displaystyle\lambda_1\int_{B_1} \mathcal{G}_s\left(x,y\right)\,\phi(y)\,dy & {\mbox{ if }}x\in B_1,\\
0&{\mbox{ if }}x\in{\mathbb{R}}^n\setminus B_1.
\end{cases}$$
We stress that we can use Proposition~\ref{LEJOS}
in this context, with~$f:=\lambda_1\phi$,
since condition~\eqref{CHlaIA} is satisfied in this case.

Then, from~\eqref{VIC2} and~\eqref{VIC4}, we know that~$\phi_*\in H^s_0(B_1)$
and, from~\eqref{VIC3},
$$ (-\Delta)^s \phi_*=\lambda_1\,\phi{\mbox{ in }}B_1.$$
In particular, we have that~$(-\Delta)^s (\phi-\phi_*)=0$ in~$B_1$,
and~$\phi-\phi_*\in H^s_0(B_1)$, which give that~$\phi-\phi_*$ vanishes identically.
Hence, we can write that~$\phi=\phi_*$, and thus~$\phi_*$
is a solution of~\eqref{dirfun}.

Now, we check~\eqref{CHaert}.
For this, we distinguish two cases.
If~$e\cdot\omega\ge 0$, we have that
$$ |e+\epsilon\omega|^2 =1+2\epsilon e\cdot\omega+\epsilon^2>1,$$
for all~$\epsilon>0$. Then, in this case~$e+\epsilon\omega\in
{\mathbb{R}}^n\setminus B_1$, and therefore~$\phi_*(e+\epsilon\omega)=0$.
This gives that, in this case,
\begin{equation}\label{GIANDIACNKS}
\lim_{\epsilon\searrow 0}\epsilon^{-s}\phi_*(e+\epsilon\omega)=0.\end{equation}
If instead~$e\cdot\omega<0$, we see that
$$ |e+\epsilon\omega|^2 =1+2\epsilon e\cdot\omega+\epsilon^2<1,$$
for all~$\epsilon>0$ sufficiently small. Hence, we can exploit
Corollary~\ref{propsev} and find that
\begin{equation}\label{GIANDIACNKS2}
\lim_{\epsilon\searrow 0}\epsilon^{-s}\phi_*(e+\epsilon\omega)=\lambda_1\,
k(n,s)(-2e\cdot\omega)^s\int_{B_1} \phi(z)\frac{(1-|z|^2)^s}{s|z-e|^n} \,dz,\end{equation}
with~$k(n,s)>0$. Then, we define
$$ k_*:=2^s\,k(n,s)\int_{B_1} \phi(z)\frac{(1-|z|^2)^s}{s|z-e|^n} \,dz.$$
We observe that~$k_*$ is positive by construction,
with~$k(n,s)>0$. Also, in light of
Lemma~\ref{lklkl}, we know that $k_*$ is finite. 
Hence, from~\eqref{GIANDIACNKS} and~\eqref{GIANDIACNKS2}
we obtain~\eqref{CHaert}, as desired.

It only remains to check~\eqref{COSI}.
For this, we use~\eqref{VIC3}, and we see that,
for every~$R\in(r,1)$, 
$$ \sup_{x\in B_1\setminus B_R}
d^{-s}(x)\,|\phi_*(x)|\le C_R\,\lambda_1\,\big(\|\phi\|_{L^1(B_1)}+
\|\phi\|_{L^\infty(B_1\setminus B_r)}\big),
$$
and this gives~\eqref{COSI} up to renaming~$C_R$.
\end{proof}
\end{corollary}

Now, we can complete the proof of Proposition~\ref{sharbou}, by arguing as follows.

\begin{proof}[Proof of Proposition~\ref{sharbou}]
Let $\psi$ be a test function in $C^\infty_0(\mathbb{R}^n)$.
Let also~$R:=\frac{r+1}{2}\in(r,1)$ and
$$ g_\epsilon(X):=
\epsilon^{-s}\phi_*(e+\epsilon X)\partial^{\beta}\psi(X).$$
We claim that
\begin{equation}\label{7UHSNs9oKN}
\sup_{{X\in{\mathbb{R}}^n}}|g_\epsilon(X)|\le C,\end{equation}
for some~$C>0$ independent of~$\epsilon$.
To prove this, we distinguish three cases.
If~$e+\epsilon X\in{\mathbb{R}}^n\setminus B_1$,
we have that~$\phi_*(e+\epsilon X)=0$ and thus~$g_\epsilon(X)=0$.
If instead~$e+\epsilon X\in B_R$,
we observe that
$$ R>|e+\epsilon X|\ge 1-\epsilon|X|,$$
and therefore~$|X|\ge \frac{1-R}{\epsilon}$. In particular,
in this case~$X$ falls outside the support of~$\psi$, as long as~$\epsilon>0$
is sufficiently small, and consequently~$\partial^{\beta}\psi(X)=0$
and~$g_\epsilon(X)=0$.

Hence, to complete the proof of~\eqref{7UHSNs9oKN},
we are only left with the case in which~$
e+\epsilon X\in B_1\setminus B_R$. In this situation,
we make use of~\eqref{COSI} and we find that
\begin{eqnarray*}
&& |\phi_*(e+\epsilon X)|\le C\,d^{s}(e+\epsilon X)=
C\,(1-|e+\epsilon X|)^s\\&&\qquad
\le C\,(1-|e+\epsilon X|)^s(1+|e+\epsilon X|)^s=
C\,(1-|e+\epsilon X|^2)^s\\&&\qquad=
C\,\epsilon^s(-2e\cdot X-\epsilon|X|^2)^s\le C\epsilon^s,
\end{eqnarray*}
for some~$C>0$ possibly varying from line to line,
and this completes the proof of~\eqref{7UHSNs9oKN}.

Now, from~\eqref{7UHSNs9oKN} and the
Dominated Convergence Theorem, we obtain that
\begin{equation}\label{eq567a8s81n} \lim_{\epsilon\searrow0}\int_{\mathbb{R}^n}
\epsilon^{-s}\phi_*(e+\epsilon X)\partial^{\beta}\psi(X) dX
=\int_{\mathbb{R}^n} \lim_{\epsilon\searrow0}
\epsilon^{-s}\phi_*(e+\epsilon X)\partial^{\beta}\psi(X) dX.\end{equation}
On the other hand, by Corollary~\ref{8iJJAUMPAAAxc},
used here with~$\omega:=\frac{X}{|X|}$, we know that
\begin{eqnarray*}&& \lim_{\epsilon\searrow0}
\epsilon^{-s}\phi_*(e+\epsilon X)
=\lim_{\epsilon\searrow0}
\epsilon^{-s}\phi_*(e+\epsilon |X|\omega)=|X|^s
\lim_{\epsilon\searrow 0}\epsilon^{-s}\phi_*(e+\epsilon\omega)\\&&\qquad=k_*\,|X|^s\,
(-e\cdot\omega)^s_+=k_*\,(-e\cdot X)^s_+.
\end{eqnarray*}
Substituting this into~\eqref{eq567a8s81n}, we thus find that
$$ \lim_{\epsilon\searrow0}\int_{\mathbb{R}^n}
\epsilon^{-s}\phi_*(e+\epsilon X)\partial^{\beta}\psi(X) dX
=k_*\,\int_{\mathbb{R}^n} (-e\cdot X)^s_+\partial^{\beta}\psi(X) dX.$$
As a consequence, integrating by parts twice,
\begin{equation*}
\begin{split}
&\lim_{\epsilon\searrow 0}\epsilon^{|\beta|-s}\int_{\mathbb{R}^n}
\partial^\beta\phi_*(e+\epsilon X)\psi(X) dX=
\lim_{\epsilon\searrow 0}\int_{\mathbb{R}^n}\partial^\beta
\Big(\epsilon^{-s}\phi_*(e+\epsilon X)\Big)\psi(X) dX \\
&\qquad=(-1)^{|\beta|}\lim_{\epsilon\searrow 0}\int_{\mathbb{R}^n}
\epsilon^{-s}\phi_*(e+\epsilon X)\partial^{\beta}\psi(X) dX \\
&\qquad=(-1)^{|\beta|}\,k_*\,\int_{\mathbb{R}^n} (-e\cdot X)^s_+\partial^{\beta}\psi(X) dX\\
&\qquad=k_*\,\int_{\mathbb{R}^n} \partial^{\beta}(-e\cdot X)^s_+\psi(X) dX
\\
&\qquad=(-1)^{|\beta|}\,k_*\, s(s-1)\ldots(s-|\beta|+1)e_1^{\beta_1}\ldots e_n^{\beta_n}\int_{\mathbb{R}^n}(-e\cdot X)^{s-|\beta|}_+\psi(X) dX.
\end{split}
\end{equation*}
Since the test function $\psi$ is arbitrary, the claim in
Proposition~\ref{sharbou} is proved.
\end{proof}

\section{Boundary behaviour\index{boundary behaviour} of~$s$-harmonic functions\index{function!$s$-harmonic}}
\label{s:hwb}

In this section we analyze the asymptotic behaviour of $s$-harmonic
functions, with a ``spherical bump function'' as exterior Dirichlet datum.

The result needed for our purpose is the following:

\begin{lemma}
\label{hbump} 
Let $s>0$. Let~$m\in\mathbb{N}_0$
and~$\sigma\in(0,1)$ such that~$s=m+\sigma$.

Then, there exists
\begin{equation}\label{0oHKNSSH013oe2urjhfe}
{\mbox{$\psi\in H^s(\mathbb{R}^n)\cap C^s_0(\mathbb{R}^n)$
such that $
(-\Delta)^s \psi=0$ in~$B_1$,}}\end{equation} and, for
every $x\in\partial B_{1-\epsilon}$,
\begin{equation}\label{0oHKNSSH013oe2urjhfe:2}
\psi(x)=k\,\epsilon^s+o(\epsilon^s),\end{equation}
as $\epsilon\searrow 0$, for some $k>0$.
\begin{proof}

Let~$\overline{\psi}\in C^\infty(\mathbb{R},\,[0,1])$ such that $\overline{\psi}=0$ in $\mathbb{R}\setminus(2,3)$ and $\overline{\psi}>0$ in $(2,3)$. Let $\psi_0(x):=(-1)^m\overline{\psi}(|x|)$.
We recall the Poisson kernel\index{kernel!Poisson}
$$ 
\Gamma_s(x,y):=(-1)^m\frac{\gamma_{n,\sigma}}{
|x-y|^n}\frac{(1-|x|^2)^s_+}{(|y|^2-1)^s},$$
for $x\in\mathbb{R}^n$, $y\in\mathbb{R}^n\setminus\overline{B_1}$, and
a suitable normalization constant~$\gamma_{n,\sigma}>0$ (see formulas~(1.10) and~(1.30)
in~\cite{ABX}).
We define
$$ \psi(x):=
\displaystyle\int_{{\mathbb{R}}^n\setminus B_1} \Gamma_s(x,y)\,\psi_0(y)\,dy+\psi_0(x).$$
Notice that~$\psi_0=0$ in~$B_{3/2}$ and therefore we can
exploit Theorem
in~\cite{ABX} and obtain that~\eqref{0oHKNSSH013oe2urjhfe}
is satisfied
(notice also that~$\psi=\psi_0$ outside~$B_1$, hence~$\psi$ is compactly supported).

Furthermore, to prove~\eqref{0oHKNSSH013oe2urjhfe:2}
we borrow some ideas from Lemma 2.2 in~\cite{MR3626547}
and we see that, for any $x\in \partial B_{1-\epsilon}$,
\begin{equation*}
\begin{split}
\psi(x)
&=c(-1)^m\int_{\mathbb{R}^n\setminus B_1} \frac{\psi_0(y)(1-|x|^2)^s}{(|y|^2-1)^s|x-y|^n} dy+\psi_0(x) \\
&=c(-1)^m\int_{\mathbb{R}^n\setminus B_1} \frac{\psi_0(y)(1-|x|^2)^s}{(|y|^2-1)^s|x-y|^n} dy \\
&=c\,(1-|x|^2)^s\int_2^3\left[\int_{\mathbb{S}^{n-1}} \frac{\rho^{n-1}\overline{\psi}(\rho)}{(\rho^2-1)^s|x-\rho\omega|^n} d\omega\right] d\rho
\\&=c\,(2\epsilon-\epsilon^2)^s\int_2^3\left[\int_{\mathbb{S}^{n-1}} \frac{\rho^{n-1}\overline{\psi}(\rho)}{(\rho^2-1)^s|(1-\epsilon)e_1-\rho\omega|^n} d\omega\right] d\rho \\
&=2^sc\,\epsilon^s\int_2^3\left[\int_{\mathbb{S}^{n-1}} \frac{\rho^{n-1}\overline{\psi}(\rho)}{(\rho^2-1)^s|e_1-\rho\omega|^n} d\omega\right] d\rho+o(\epsilon^s) \\
&=c\epsilon^s+o(\epsilon^s),
\end{split}
\end{equation*}
where~$c>0$ is a constant possibly varying from line to line, and this establishes~\eqref{0oHKNSSH013oe2urjhfe:2}.
\end{proof}
\end{lemma}

\begin{remark}\label{RUCAPSJD} {\rm
As in Proposition~\ref{sharbou}, one can extend~\eqref{0oHKNSSH013oe2urjhfe:2}
to higher derivatives (in the distributional sense), obtaining, for any~$e\in\partial B_1$
and~$\beta\in\mathbb{N}^n$
$$ \lim_{\epsilon\searrow0} \epsilon^{|\beta|-s}\partial^\beta\psi(e+\epsilon X)=k_\beta\,
e_1^{\beta_1}\dots e_n^{\beta_n}(-e\cdot X)_+^{s-|\beta|}
,$$
for some~$\kappa_\beta\ne0$.}\end{remark}

Using Lemma \ref{hbump}, in the spirit of \cite{MR3626547}, we
can construct a sequence of $s$-harmonic functions
approaching~$(x\cdot e)^s_+$ for a fixed unit vector $e$,
by using a blow-up argument. Namely, we prove the following:

\begin{corollary}
\label{lapiog}
Let $e\in\partial B_1$. There exists a sequence $v_{e,j}\in H^s(\mathbb{R}^n)\cap C^s(\mathbb{R}^n)$ such that $(-\Delta)^s v_{e,j}=0$ in $B_1(e)$, $v_{e,j}=0$ in $\mathbb{R}^n\setminus B_{4j}(e)$, and \[v_{e,j}\to\kappa(x\cdot e)^s_+\quad\mbox{in}\quad L^1(B_1(e)),\] as $j\to+\infty$, for some $\kappa>0$.
\begin{proof}
Let $\psi$ be as in Lemma \ref{hbump} and
define \[v_{e,j}(x):=j^s\psi\left(\frac{x}{j}-e\right).\]
The $s$-harmonicity and the property of being compactly supported follow
by the ones of $\psi$. We now prove the convergence.
To this aim, given $x\in B_1(e)$, we write $p_j:=\frac{x}{j}-e$ and $\epsilon_j:=1-|p_j|$. Recall that since $x\in B_1(e)$, then $|x-e|^2<1$, which implies that $|x|^2<2x\cdot e$ and $x\cdot e>0$ for any $x\in B_1(e)$. \\
As a consequence \[|p_j|^2=\left|\frac{x}{j}-e\right|^2=
\frac{|x|^2}{j^2}+1-2\frac{x}{j}\cdot e=1-\frac{2}{j}(x\cdot e)_+
+o\left(\frac{1}{j}\right)(x\cdot e)^2_+,\]
and so \[\epsilon_j=\frac{(1+o(1))}{j}(x\cdot e)_+.\]
Therefore, using \eqref{0oHKNSSH013oe2urjhfe:2},
\begin{equation*}
\begin{split}
v_{e,j}(x)&=j^s\psi(p_j) \\ &=j^s\kappa(\epsilon_j^s+o(\epsilon^s_j)) \\
&=j^s\left(\frac{\kappa}{j^s}(x\cdot e)_+^s+o\left(\frac{1}{j^s}\right)\right) \\
&=\kappa(x\cdot e)^s_+ +o(1).
\end{split}
\end{equation*}
Integrating over $B_1(e)$, we obtain the desired $L^1$-convergence.
\end{proof}
\end{corollary}

Now, we show that, as in the case $s\in (0,1)$ proved in
Theorem~3.1 of \cite{MR3626547}, we can find an $s$-harmonic function
with an arbitrarily large number of derivatives prescribed at some point.

\begin{proposition}
\label{maxhlapspan}
For any $\beta\in\mathbb{N}^n$, there exist~$p\in\mathbb{R}^n$, $R>r>0$, and~$v\in H^s(\mathbb{R}^n)\cap C^s(\mathbb{R}^n)$ such that
\begin{equation}
\label{csi}
\begin{cases}
(-\Delta)^s v=0&\quad\text{in}\quad B_r(p), \\
v=0&\quad\text{in}\quad\mathbb{R}^n\setminus B_R(p),
\end{cases}
\end{equation}
\begin{equation*}
D^\alpha v(p)=0\quad{\mbox{ for any }}\; \alpha\in\mathbb{N}^n\quad{\mbox{ with }}\;|\alpha|\leq|\beta|-1,\end{equation*}
\begin{equation*}
D^\alpha v(p)=0\quad{\mbox{ for any }}\; \alpha\in\mathbb{N}^n\quad{\mbox{ with }}\;|\alpha|=|\beta|\quad{\mbox{ and }}\;\alpha\neq\beta\end{equation*}  
and 
\begin{equation*}
D^\beta v(p)=1.\end{equation*}

\begin{proof}
Let $\mathcal{Z}$ be the set of all pairs~$(v,x)\in\left(H^s(\mathbb{R}^n)\cap C^s(\mathbb{R}^n)\right)\times B_r(p)$ that satisfy \eqref{csi} for some $R>r>0$ and $p\in\mathbb{R}^n$. 

To each pair $(v,x)\in\mathcal{Z}$ we associate the vector
$\left(D^\alpha v(x)\right)_{|\alpha|\leq|\beta|}\in\mathbb{R}^{K'}$, for some $K'=K'_{|\beta|}$ and 
consider~$\mathcal{V}$ to be the vector space spanned by this construction, namely
we set
$$\mathcal{V}:=\Big\{\left(D^\alpha v(x)\right)_{|\alpha|\leq|\beta|},
\quad{\mbox{ with }}\; (v,x)\in\mathcal{Z}
\Big\}.$$
We claim that
\begin{equation}\label{CLSPAZ}
\mathcal{V}=\mathbb{R}^{K'}.\end{equation}
To check this, we suppose by contradiction that~$\mathcal{V}$ lies in a 
proper subspace of~$\mathbb{R}^{K'}$. Then, $\mathcal{V}$ must lie in a
hyperplane, hence there exists 
\begin{equation}
\label{cnonull0}
c=(c_\alpha)_{|\alpha|\leq|\beta|}\in\mathbb{R}^{K'}\setminus\left\{0\right\} 
\end{equation}
which is orthogonal to any vector $\left(D^\alpha v(x)\right)_{|\alpha|\leq|\beta|}$
with~$(v,x)\in\mathcal{Z}$, that is
\begin{equation}
\label{prp18}
\sum_{|\alpha|\leq|\beta|} c_\alpha D^\alpha v(x)=0.
\end{equation}
We notice that the pair $(v_{e,j},x)$, with $v_j$ as in Corollary \ref{lapiog},
$e\in\partial B_1$
and $x\in B_1(e)$, belongs to~$\mathcal{Z}$. Consequently,
fixed~$\xi\in\mathbb{R}^n\setminus B_{1/2}$
and set~$e:=\frac{\xi}{|\xi|}$, we have that~\eqref{prp18} holds true when~$v:=v_{e,j}$ and $x\in B_1(e)$, namely
$$
\sum_{|\alpha|\leq|\beta|} c_\alpha D^\alpha v(x)=0.
$$
Let now~$\varphi\in C_0^\infty(B_1(e))$. 
Integrating by parts,
by Corollary \ref{lapiog} and the Dominated Convergence Theorem, 
we have that
\begin{eqnarray*}
&&0=\lim_{j\to+\infty}\int_{\mathbb{R}^n}\sum_{|\alpha|\leq|\beta|}
c_\alpha D^\alpha v_{e,j}(x)\varphi(x)\,dx
=\lim_{j\to+\infty}\int_{\mathbb{R}^n}\sum_{|\alpha|\leq|\beta|}(-1)^{|\alpha|}
c_\alpha v_{e,j}(x)D^\alpha\varphi(x)\,dx\\
&&\qquad=\kappa\int_{\mathbb{R}^n}\sum_{|\alpha|\leq|\beta|}(-1)^{|\alpha|}
c_\alpha(x\cdot e)^s_+D^\alpha\varphi(x)\,dx
=\kappa\int_{\mathbb{R}^n}\sum_{|\alpha|\leq|\beta|}c_\alpha
D^\alpha(x\cdot e)^s_+\varphi(x)\,dx.\end{eqnarray*}
This gives that, for every $x\in B_1(e)$,
\[\sum_{|\alpha|\leq|\beta|}c_\alpha D^\alpha(x\cdot e)^s_+=0.\] 
Moreover, for every $x\in B_1(e)$,
\[D^\alpha(x\cdot e)^s_+=s(s-1)\ldots(s-|\alpha|+1)
(x\cdot e)^{s-|\alpha|}_+e_1^{\alpha_1}\ldots e_n^{\alpha_n}.\]
In particular, for $x=\frac{e}{|\xi|}\in B_1(e)$,
\[D^\alpha(x\cdot e)^s_+\big|_{|_{x=e/{|\xi|}}}=s(s-1)
\ldots(s-|\alpha|+1)|\xi|^{-s}\xi_1^{\alpha_1}\ldots \xi_n^{\alpha_n}.\] And, using the usual multi-index notation, we write
\begin{equation}
\label{sampspal}
\sum_{|\alpha|\leq|\beta|}c_\alpha s(s-1)\ldots(s-|\alpha|+1)
\xi^\alpha=0,
\end{equation}
for any $\xi\in\mathbb{R}^n\setminus B_{1/2}$.
The identity \eqref{sampspal} describes
a polynomial in $\xi$ which vanishes for any~$\xi$ in an open subset of $\mathbb{R}^n$. As a result, the Identity Principle
for polynomials leads to
$$ c_\alpha s(s-1)\ldots(s-|\alpha|+1)=0,$$
for all~$|\alpha|\leq|\beta|$.

Consequently, since $s\in\mathbb{R}\setminus\mathbb{N}$, 
the product $s(s-1)\ldots(s-|\alpha|+1)$ never vanishes, 
and so the coefficients $c_\alpha$ are forced to be null for any $|\alpha|\leq|\beta|$. 
This is in contradiction with~\eqref{cnonull0}, and therefore the proof
of~\eqref{CLSPAZ} is complete.

{F}rom this, the desired claim in
Proposition~\ref{maxhlapspan} plainly follows.
\end{proof}
\end{proposition}

\chapter{Proof of the main result}\label{CH5}

This chapter is devoted to the proof of the main result in Theorem~\ref{theone}.
This will be accomplished by an auxiliary result of purely nonlocal type
which will allow us to prescribe an arbitrarily large number of derivatives
at a point for the solution of a fractional equation.

\section{A result which implies Theorem \ref{theone}}\label{s:fourthE}

We will use the notation
\begin{equation}\label{NEOAAJKin1a}
\Lambda_{-\infty}:=\Lambda_{(-\infty,\dots,-\infty)},\end{equation}
that is we exploit~\eqref{1.6BIS} with~$a_1:=\dots:=a_l:=-\infty$.
This section presents the following statement:

\begin{theorem}\label{theone2}
Suppose that 
\begin{equation*}\begin{split}&
{\mbox{either there exists~$i\in\{1,\dots,M\}$ such that~$\XB_i\ne0$
and~$s_i\not\in{\mathbb{N}}$,}}\\
&{\mbox{or there exists~$i\in\{1,\dots,l\}$ such that~$\XC_i\ne0$ and $\alpha_i\not\in{\mathbb{N}}$.}}\end{split}
\end{equation*}
Let $\ell\in\mathbb{N}$, $f:\mathbb{R}^N\rightarrow\mathbb{R}$,
with $f\in C^{\ell}\big(\overline{B_1^N}\big)$. Fixed $\epsilon>0$,
there exist
\begin{equation*}\begin{split}&
u=u_\epsilon\in C^\infty\left(B_1^N\right)\cap C\left(\mathbb{R}^N\right),\\
&a=(a_1,\dots,a_l)=(a_{1,\epsilon},\dots,a_{l,\epsilon})
\in(-\infty,0)^l,\\ {\mbox{and }}\quad&
R=R_\epsilon>1\end{split}\end{equation*} such that:
\begin{itemize}
\item for every~$h\in\{1,\dots,l\}$ and~$(x,y,t_1,\dots,t_{h-1},t_{h+1},\dots,t_l)$
\begin{equation}\label{SPAZIO}
{\mbox{the map ${\mathbb{R}}\ni t_h\mapsto u(x,y,t)$
belongs to~$C^{k_h,\alpha_h}_{-\infty}$,}}
\end{equation}
in the notation of formula~(1.4) of~\cite{CDV18},
\item it holds that
\begin{equation}\label{MAIN EQ:2}\left\{\begin{matrix}
\Lambda_{-\infty} u=0 &\mbox{ in }\;B_1^{N-l}\times(-1,+\infty)^l, \\
u(x,y,t)=0&\mbox{ if }\;|(x,y)|\ge R,
\end{matrix}\right.\end{equation}
\begin{equation}\label{ESTENSIONE}
\partial^{k_h}_{t_h} u(x,y,t)=0\qquad{\mbox{if }}t_h\in(-\infty,a_h),\qquad{\mbox{for all }}h\in\{1,\dots,l\},
\end{equation}
and
\begin{equation}\label{IAzofm:2}
\left\|u-f\right\|_{C^{\ell}(B_1^N)}<\epsilon.
\end{equation}\end{itemize}
\end{theorem}

The proof of Theorem~\ref{theone2} will basically occupy the
rest of this work, and this will lead us to the completion of the
proof of Theorem~\ref{theone}. Indeed, we have that:

\begin{lemma}\label{GRAT}
If the statement of Theorem~\ref{theone2} holds true,
then the statement in Theorem~\ref{theone} holds true.
\end{lemma}

\begin{proof} Assume that the claims
in Theorem~\ref{theone2} are satisfied. Then, by~\eqref{SPAZIO} and~\eqref{ESTENSIONE},
we are in the position of exploting Lemma~A.1 in~\cite{CDV18}
and conclude that, in~$B_1^{N-l}\times(-1,+\infty)^l$,
$$ D^{\alpha_h}_{t_h ,a_h} u=D^{\alpha_h}_{t_h,-\infty} u,$$
for every~$h\in\{1,\dots,l\}$. This and~\eqref{MAIN EQ:2}
give that
\begin{equation} \label{33ujNAKS}
\Lambda_{a}u=\Lambda_{-\infty} u=0 \qquad\mbox{ in }\;B_1^{N-l}
\times(-1,+\infty)^l.\end{equation}
We also define
$$\underline{a}:=\min_{h\in\{1,\dots,l\}} a_h$$
and take~$\tau\in C^\infty_0 ([-\underline{a}-2,3])$ with~$\tau=1$
in~$[-\underline{a}-1,1]$. Let
\begin{equation}\label{UJNsdA} U(x,y,t):=u(x,y,t)\,\tau(t_1)\dots\tau(t_l).\end{equation}
Our goal is to prove that~$U$ satisfies the theses of Theorem~\ref{theone}.
To this end, we observe that~$u=U$ in~$B^N_1$, therefore~\eqref{IAzofm}
for~$U$
plainly follows from~\eqref{IAzofm:2}.

In addition, from~\eqref{defcap}, we see that~$
D^{\alpha_h}_{t_h,a_h}$ at a point~$t_h\in(-1,1)$
only depends on the values of the function
between~$a_h$ and~$1$. Since the cutoffs in~\eqref{UJNsdA} do not
alter these values, we see that~$D^{\alpha_h}_{t_h,a_h}U=D^{\alpha_h}_{t_h,a_h}u$
in~$B_1^N$, and accordingly~$\Lambda_a U=\Lambda_a u$ in~$B_1^N$.
This and~\eqref{33ujNAKS} say that
\begin{equation}\label{9OAJA}
\Lambda_a U=0\qquad{\mbox{in }}B_1^N.\end{equation}
Also, since~$u$ in Theorem~\ref{theone2} is compactly supported
in the variable~$(x,y)$, we see from~\eqref{UJNsdA} that~$U$
is compactly supported in the variables~$(x,y,t)$.
This and~\eqref{9OAJA} give that~\eqref{MAIN EQ} is satisfied by~$U$
(up to renaming~$R$).
\end{proof} 

\section{A pivotal span result towards the proof of Theorem \ref{theone2}}\label{s:fourth0}

In what follows, we let~$\Lambda_{-\infty}$
be as in~\eqref{NEOAAJKin1a}, we recall the setting in~\eqref{1.0},
and we
use the following multi-indices notations:
\begin{equation}\label{mulPM}
\begin{split}
& \iota=\left(i,I,\mathfrak{I}\right)=\left(i_1,\ldots,i_n,I_1,\ldots,I_M,\mathfrak{I}_1,
\ldots,\mathfrak{I}_l\right)\in\mathbb{N}^N\\
{\mbox{and }} &
\partial^\iota w=\partial^{i_1}_{x_1}\ldots\partial^{i_n}_{x_n}
\partial^{I_1}_{y_1}\ldots\partial^{I_M}_{y_M}\partial^{\mathfrak{I}_1}_{t_1}
\ldots\partial^{\mathfrak{I}_l}_{t_l}w.
\end{split}\end{equation}
Inspired by Lemma 5 of~\cite{DSV1},
we consider the span of the derivatives of functions in~$
\ker\Lambda_{-\infty}$, with derivatives up to a fixed order $K\in{\mathbb{N}}$.
We want to prove that the derivatives of such functions span
a maximal vectorial space. 

For this, we denote by $\partial^K w(0)$
the vector with entries given,
in some prescribed order,
by~$
\partial^\iota w(0)$ with $\left|\iota\right|\leq K$.

We notice that
\begin{equation}\label{8iokjKK}
{\mbox{$\partial^K w(0)\in\mathbb{R}^{K'}$ for some $K'\in{\mathbb{N}}$,}}
\end{equation}
with~$K'$ depending on~$K$.

Now, we adopt the notation in formula~(1.4) of~\cite{CDV18},
and
we denote by~$
\mathcal{A}$ \label{CALSASS}
the set of all functions~$w=w(x,y,t)$
such that for all~$h\in\{1,\ldots, l\}$ and all~$
(x,y,t_1,\ldots,t_{h-1},t_{h+1},\ldots,t_l)\in\mathbb{R}^{N-1}$,
the map~${\mathbb{R}}\ni t_h\mapsto w(x,y,t)$ belongs to~$
C^{\infty}((a_h,+\infty))\cap C^{k_h,\alpha_h}_{-\infty}$,
and~\eqref{ESTENSIONE} holds true for some~$a_h\in (-2,0)$.

We also set
\begin{equation*}
\begin{split}
\mathcal{H}:=\Big\{w\in C(\mathbb{R}^N)
\cap C_0(\mathbb{R}^{N-l})\cap C^\infty(\mathcal{N})\cap\mathcal{A},
\text{ for some neighborhood 
$\mathcal{N} $
of the origin, } \\
 \text{ such that }
\Lambda_{-\infty} w=0 \text{ in }\mathcal{N}\Big\}
\end{split}
\end{equation*}
and, for any $w\in\mathcal{H}$, let $\mathcal{V}_K$ be the vector space spanned by the vector $\partial^K w(0)$. 

By \eqref{8iokjKK}, we know that~$\mathcal{V}_K\subseteq\mathbb{R}^{K'}$.
In fact, we show that equality holds in this inclusion, as
stated in the following\footnote{Notice that results
analogous to Lemma~\ref{lemcin}
cannot hold for solutions of local operators: for instance,
pure second derivatives of harmonic functions have to satisfy
a linear equation, so they are forced to lie in a proper subspace.
In this sense, results such as Lemma~\ref{lemcin} here reveal a truly nonlocal
phenomenon.}
result:
\begin{lemma}
\label{lemcin}
It holds that $\mathcal{V}_K=\mathbb{R}^{K'}$.
\end{lemma}

The proof of Lemma~\ref{lemcin} is 
by contradiction. Namely, if $\mathcal{V}_K$ does not exhaust the whole of $\mathbb{R}^{K'}$ there exists 
\begin{equation}
\label{tetaort}
\theta\in\partial B_1^{K'}
\end{equation}
such that
\begin{equation}
\label{aza}
\mathcal{V}_K\subseteq\left\{\zeta\in\mathbb{R}^{K'}
\text{ s.t. } \theta\cdot\zeta=0\right\}.
\end{equation}
In coordinates, recalling~\eqref{mulPM},
we write~$\theta$
as~$\theta_\iota=\theta_{i,I,\mathfrak{I}}$,
with~$i\in\mathbb{N}^{p_1+\dots+p_n}$,
$I\in\mathbb{N}^{m_1+\dots+m_M}$
and~$\mathfrak{I}\in\mathbb{N}^l$.
We consider
\begin{equation}\label{IBARRA}
\begin{split}&
{\mbox{a multi-index $\overline{I}\in\mathbb{N}^{m_1+\dots+m_M}$ 
such that it maximizes~$|I|$}}\\
&{\mbox{among all the multi-indexes~$(i,I,\mathfrak{I})$
for which~$\left|i\right|+\left|I\right|+|\mathfrak{I}|\leq K$}}\\
&{\mbox{and~$\theta_{i,I,\mathfrak{I}}\ne0$
for some~$(i,\mathfrak{I})$.}}
\end{split}\end{equation}
Some comments on the setting\label{FFOAK}
in~\eqref{IBARRA}. We stress that, by~\eqref{tetaort},
the set~$\mathcal{S}$
of indexes~$I$ for which there exist indexes~$
(i,\mathfrak{I})$ such that~$|i|+|I|+|\mathfrak{I}|\le K$
and~$\theta_{i,I,\mathfrak{I}}\ne0$ is not empty.
Therefore, since~${\mathcal{S}}$ is a finite set,
we can take
$$ S:=\sup_{I\in {\mathcal{S}}} |I|=\max_{I\in {\mathcal{S}}}
|I|\in{\mathbb{N}}\cap [0,K].$$
Hence, we consider a multi-index $\overline{I}$ for
which~$|\overline I|=S$ to obtain the
setting in~\eqref{IBARRA}. By construction, we have that
\begin{itemize}
\item $|i|+|\overline I|+|\mathfrak{I}|\le K$,
\item  if~$|I|>|\overline I|$, then $\theta_{i,I,\mathfrak{I}}=0$,
\item
and there exist multi-indexes~$i$ and~$\mathfrak{I}$
such that~$\theta_{i,\overline I,\mathfrak{I}}\ne0$.\end{itemize}

As a variation of the setting in~\eqref{IBARRA},
we can also consider
\begin{equation}\label{IBARRA2}
\begin{split}&
{\mbox{a multi-index $\overline{\mathfrak{I}}\in\mathbb{N}^{l}$ 
such that it maximizes~$|\mathfrak{I}|$}}\\
&{\mbox{among all the multi-indexes~$(i,I,\mathfrak{I})$
for which~$\left|i\right|+\left|I\right|+|\mathfrak{I}|\leq K$}}\\
&{\mbox{and~$\theta_{i,I,\mathfrak{I}}\ne0$
for some~$(i,I)$.}}
\end{split}\end{equation}
In the setting of~\eqref{IBARRA} and~\eqref{IBARRA2},
we claim that there exists an open set
of~$\mathbb{R}^{p_1+\ldots+p_n}\times\mathbb{R}^{m_1+\ldots+m_M}\times\mathbb{R}^{l}$
such that for every~$(\XX,\XY,\XT)$ in such open set we have that
\begin{equation}
\label{ipop}\begin{split}
{\mbox{either }}\qquad&
0=\sum_{{|i|+|I|+|\mathfrak{I}|\le K}\atop{|I| = |\overline{I}|}}
C_{i,I,\mathfrak{I}}\;\theta_{i,I,\mathfrak{I}}\;
\XX^i \XY^{{I}}\XT^{\mathfrak{I}},\qquad{\mbox{ with }}\qquad
C_{i,I,\mathfrak{I}}\ne0,\\
{\mbox{or }}\qquad&
0=\sum_{{|i|+|I|+|\mathfrak{I}|\le K}\atop{|\mathfrak{I}| = |\overline{\mathfrak{I}}|}}
C_{i,I,\mathfrak{I}}\;\theta_{i,I,\mathfrak{I}}\;
\XX^i \XY^{{I}}\XT^{\mathfrak{I}},\qquad{\mbox{ with }}\qquad
C_{i,I,\mathfrak{I}}\ne0.\end{split}
\end{equation}
In our framework, the claim in~\eqref{ipop} will be pivotal
towards the completion of the proof of Lemma~\ref{lemcin}.
Indeed, let us suppose for the moment that~\eqref{ipop}
is established and let us complete the proof of Lemma~\ref{lemcin}
by arguing as follows.

Formula \eqref{ipop} says that $\theta\cdot\partial^K w(0)$
is a polynomial which vanishes for any triple $(\XX,\XY,\XT)$
in an open subset of $\mathbb{R}^{p_1+\ldots+p_n}\times\mathbb{R}^{m_1+\ldots+m_M}\times\mathbb{R}^{l}$.
Hence, using the identity principle of polynomials, we have that each $C_{i,I,\mathfrak{I}}\;\theta_{i,I,\mathfrak{I}}$ is equal to zero
whenever~$|i|+|I|+|\mathfrak{I}|\le K$
and either~$|I|=|\overline I|$ (if the first identity in~\eqref{ipop}
holds true) or~$|\mathfrak{I}|=|\overline{\mathfrak{I}}|$
(if the second identity in~\eqref{ipop}
holds true). Then, since~$C_{i,I,\mathfrak{I}}\neq 0$,
we conclude that each $\theta_{i,I,\mathfrak{I}}$ is zero
as long as either~$|I|=|\overline{I}|$ (in the first case)
or~$|\mathfrak{I}|=|\overline{\mathfrak{I}}|$
(in the second case), but this contradicts either
the definition of $\overline{I}$
in~\eqref{IBARRA} (in the first case)
or the definition of~$\overline{\mathfrak{I}}$
in~\eqref{IBARRA2} (in the second case). This would therefore complete the proof
of Lemma~\ref{lemcin}.
\medskip

In view of the discussion above,
it remains to prove~\eqref{ipop}.
To this end, we distinguish the following four
cases:
\begin{enumerate}

\item\label{itm:case1} there exist $i\in\{1,\dots,n\}$ and
$j\in\{1,\dots,M\}$ such that~$\XA_i\ne0$ and~$\XB_j\ne0$,
\item\label{itm:case2} there exist $i\in\{1,\dots,n\}$ and
$h\in\{1,\dots,l\}$ such that~$\XA_i\ne0$ and~$\XC_h\ne0$,
\item\label{itm:case3} we have that~$\XA_1=\dots=\XA_n=0$,
and there exists~$j\in\{1,\dots,M\}$ such that~$\XB_j\ne0$,
\item\label{itm:case4} we have that~$\XA_1=\dots=\XA_n=0$,
and there exists~$h\in\{1,\dots,l\}$ such that~$\XC_h\ne0$.

\end{enumerate}

Notice that cases~\ref{itm:case1} and~\ref{itm:case3}
deal with the case in which space-fractional diffusion is present
(and in case~\ref{itm:case1} one also has classical
derivatives, while in case~\ref{itm:case3}
the classical derivatives are absent).

Similarly, cases~\ref{itm:case2} and~\ref{itm:case4}
deal with the case in which time-fractional diffusion is present
(and in case~\ref{itm:case2} one also has classical
derivatives, while in case~\ref{itm:case4}
the classical derivatives are absent).

Of course, the case in which both space- and time-fractional diffusion occur is already comprised by the
previous cases (namely, it is comprised in
both cases~\ref{itm:case1} and~\ref{itm:case2}
if classical derivatives are also present,
and in both cases~\ref{itm:case3}
and~\ref{itm:case4} if classical derivatives are absent).

\begin{proof}[Proof of \eqref{ipop}, case \ref{itm:case1}]
For any $j\in\left\{1,\ldots,M\right\}$ we denote by $\tilde{\phi}_{\star,j}$
the first eigenfunction for $(-\Delta)^{s_j}_{y_j}$
vanishing outside $B_1^{m_j}$ given in Corollary \ref{QUESTCP}. 
We normalize it such that $ \|\tilde{\phi}_{\star,j}\|_{L^2(\mathbb{R}^{m_j})}=1$,
and we write $\lambda_{\star,j}\in(0,+\infty)$ to indicate the corresponding first eigenvalue
(which now depends on~$s_j$), namely we write
\begin{equation}
\label{lambdastarj}
\begin{cases}
(-\Delta)^{s_j}_{y_j}\tilde{\phi}_{\star,j}=\lambda_{\star,j}\tilde{\phi}_{\star,j}&\quad\text{in}\,B_1^{m_j} ,\\
\tilde{\phi}_{\star,j}=0&\quad\text{in}\,\mathbb{R}^{m_j}\setminus\overline{B_1^{m_j}}.
\end{cases}
\end{equation}
Up to reordering the variables and/or
taking the operators to the other side of the equation, 
given the assumptions of case~\ref{itm:case1},
we can suppose that 
\begin{equation}\label{AGZ}
{\mbox{$\XA_1\ne0$}}\end{equation}
and
\begin{equation}\label{MAGGZ}
{\mbox{$\XB_M>0$}}.\end{equation} 
In view of~\eqref{AGZ}, we can define
\begin{equation}\label{MAfghjkGGZ} R:=\left( 
\frac{ 1 }{|\XA_1|}\displaystyle\left(\sum_{j=1}^{M-1}{|\XB_j|\lambda_{\star,j}}
+\sum_{h=1}^l|\XC_h|\right)\right)^{1/|r_{1}|}.\end{equation}
Now, we fix two sets of free parameters 
\begin{equation}\label{FREExi}
\XX_1\in(R+1,R+2)^{p_1},\ldots,\XX_n\in(R+1,R+2)^{p_n},\end{equation}
and
\begin{equation}\label{FREEmustar}
\XT_{\star,1}\in\left(\frac12,1\right),\dots,\XT_{\star,l}\in\left(\frac12,1\right).\end{equation}
We also set 
\begin{equation}\label{1.6md}
{\mbox{$\lambda_j:=\lambda_{\star,j}$ for $j\in\left\{1,\ldots,M-1\right\}$, }}\end{equation}
where $\lambda_{\star,j}$ is defined as in \eqref{lambdastarj}, and
\begin{equation}
\label{alp}
\lambda_M\,:=\,\frac{1}{\XB_M}\left(
\sum_{j=1}^n {\left|\XA_j\right|\XX_j^{r_j}}-\sum_{j=1}^{M-1}
{\XB_j\lambda_j}-\sum_{h=1}^l\XC_h\XT_{\star,h}\right).\end{equation}
Notice that this definition is well-posed, thanks to~\eqref{MAGGZ}.
In addition, from~\eqref{FREExi}, we can write~$\XX_{j}=(
\XX_{j1},\dots,\XX_{jp_j})$, and
we know that~$\XX_{j\ell}>R+1$
for any~$j\in\{1,\dots,n\}$ and any~$\ell\in\{1,\dots,p_j\}$.
Therefore,
\begin{equation}\label{1.15bis} \XX_j^{r_j}= \XX_{j1}^{r_{j1}}\dots\XX_{jp_j}^{r_{jp_j}}\ge0.\end{equation}
{F}rom this, \eqref{MAfghjkGGZ} and~\eqref{FREEmustar}, we deduce that
\begin{eqnarray*}&& \sum_{j=1}^n {\left|\XA_j\right|\XX_j^{r_j}}
\ge \left|\XA_1\right|\XX_1^{r_1}\ge
\left|\XA_1\right| (R+1)^{|r_1|}>
\left|\XA_1\right| R^{|r_1|}\\&&\qquad
=\sum_{j=1}^{M-1}
{|\XB_j|\lambda_j}+\sum_{h=1}^l |\XC_h|\geq\sum_{j=1}^{M-1}
{\XB_j\lambda_j}+\sum_{h=1}^l \XC_h\XT_{\star,h},\end{eqnarray*}
and consequently, by~\eqref{alp},
\begin{equation}
\label{alp-0}
\lambda_M>0.\end{equation}
We also set
\begin{equation}\label{OMEj}
\omega_j:=\begin{cases}1&\quad\text{if }\,j=1,\dots,M-1 ,\\
\displaystyle\frac{\lambda_{\star,M}^{1/2s_M}}{
\lambda_M^{1/2s_M}}&\quad\text{if }\,j=M.\end{cases}
\end{equation}
Notice that this definition is well-posed, thanks to~\eqref{alp-0}.
In addition, 
by~\eqref{lambdastarj}, we have that,
for any $j\in\{1,\dots,M\}$, the functions
\begin{equation}
\label{autofun1}
\phi_j\left(y_j\right):=\tilde{\phi}_{\star,j}\left(\frac{y_j}{\omega_j}\right)
%\label{autofun2}
%{\mbox{and }}\quad \psi_h\left(t_h\right)&:=&\tilde{\psi}_{\star,h}\left(\frac{t_h}{\varrho_h}\right)
\end{equation}
are eigenfunctions of $(-\Delta)^{s_j}_{y_j}$ in $B_{\omega_j}^{m_j}$ 
with external homogenous Dirichlet boundary condition\index{external!boundary condition},
and eigenvalues\index{Dirichlet!eigenvalues} $\lambda_j$:
namely, we can rewrite~\eqref{lambdastarj} as
\begin{equation}\label{REGSWYS-A}
\begin{cases}
(-\Delta)^{s_j}_{y_j} {\phi}_{j}=\lambda_{j} {\phi}_{j}&\quad\text{in}\,B_{\omega_j}^{m_j} ,\\
{\phi}_{j}=0&\quad\text{in}\,\mathbb{R}^{m_j}\setminus\overline{B_{\omega_j}^{m_j}}.
\end{cases}
\end{equation}
Now, we define
\begin{equation}
\label{chosofpsistar}
\psi_{\star,h}(t_h):=E_{\alpha_h,1}(t_h^{\alpha_h}),
\end{equation}
where~$E_{\alpha_h,1}$ denotes the Mittag-Leffler function
with parameters $\alpha:=\alpha_h$ and $\beta:=1$ as defined 
in \eqref{Mittag}.

Moreover, 
we consider~$a_h\in(-2,0)$, for every~$h=1,\dots,l$,
to be chosen appropriately in what follows
(the precise choice will be performed in~\eqref{pata7UJ:AKK}),
and, recalling~\eqref{FREEmustar},
we let 
\begin{equation}\label{TGAdef}
\XT_h:=\XT_{\star,h}^{1/{\alpha_h}},\end{equation} and we define
\begin{equation}
\label{autofun2}
\psi_h(t_h):=\psi_{\star,h}\big(\XT_h (t_h-a_h)\big)=
E_{\alpha_h,1}\big(\XT_{\star,h} (t_h-a_h)^{\alpha_h}\big).
\end{equation}
We point out that, thanks to Lemma~\ref{MittagLEMMA}, the function in~\eqref{autofun2}, solves
\begin{equation}
\label{jhjadwlgh}
\begin{cases}
D^{\alpha_h}_{t_h,a_h}\psi_h(t_h)=\XT_{\star,h}\psi_h(t_h)&\quad\text{in }\,(a_h,+\infty), \\
\psi_h(a_h)=1, \\
\partial^m_{t_h}\psi_h(a_h)=0&\quad\text{for every }\,m\in\{1,\dots,[\alpha_h] \}.
\end{cases}
\end{equation}
Moreover, for any $h\in\{1,\ldots, l\}$, we define
\begin{equation}
\label{starest}
\psi^{\star}_h(t_h):=\begin{cases}
\psi_h(t_h)\qquad\text{ if }\,t_h\in[a_h,+\infty) \\
1\qquad\qquad\text{ if }\,t_h\in(-\infty,a_h).\end{cases}
\end{equation}
Thanks to \eqref{jhjadwlgh} and Lemma A.3 in \cite{CDV18} applied here with $b:=a_h$, $a:=-\infty$, $u:=\psi_h$, $u_\star:=\psi^{\star}_h$, we have that $\psi^{\star}_h\in C^{k_h,\alpha_h}_{-\infty}$, and
\begin{equation}\label{DOBACHA}
D^{\alpha_h}_{t_h,-\infty}\psi_h^\star(t_h)=D^{\alpha_h}_{t_h,a_h}\psi_h(t_h)
=\XT_{\star,h}\psi_h(t_h)
=\XT_{\star,h}\psi_h^\star(t_h)\,\text{ in every interval }\,I\Subset(a_h,+\infty).
\end{equation}
We observe that the setting in \eqref{starest} is compatible with the ones in \eqref{SPAZIO} and \eqref{ESTENSIONE} .

{F}rom~\eqref{Mittag} and~\eqref{autofun2}, we see that
\begin{equation*}
\psi_h(t_h)=
\sum_{j=0}^{+\infty} {\frac{
\XT_{\star,h}^j\, (t_h-a_h)^{\alpha_h j}}{\Gamma\left(\alpha_h j+1\right)}}.
\end{equation*}
Consequently, for every~${\mathfrak{I}_h}\in\mathbb{N}$,
we have that
\begin{equation}\label{STAvca}
\partial^{\mathfrak{I}_h}_{t_h}\psi_h(t_h)=
\sum_{j=0}^{+\infty} {\frac{
\XT_{\star,h}^j\, \alpha_h j(\alpha_h j-1)\dots(\alpha_h j-
\mathfrak{I}_h+1)
(t_h-a_h)^{\alpha_h j-\mathfrak{I}_h}}{\Gamma\left(\alpha_h j+1\right)}}
.\end{equation}
Now, we define, for any $i\in\left\{1,\ldots,n\right\}$,
\begin{equation*}
\overline{\XA}_i:=
\begin{cases}
\displaystyle\frac{\XA_i}{\left|\XA_i\right|}\quad \text{ if }\XA_i\neq 0 ,\\
1 \quad \text{ if }\XA_i=0.
\end{cases}
\end{equation*}
We notice that
\begin{equation}\label{NOZABA}
{\mbox{$\overline{\XA}_i\neq 0$ for all~$i\in\left\{1,\ldots,n\right\}$,}}\end{equation}
and
\begin{equation}\label{NOZABAnd}
{\XA}_i\overline{\XA}_i=|{\XA}_i|.\end{equation}
Now, for each~$i\in\{1,\dots,n\}$, we consider
the multi-index~$r_i=(r_{i1},\dots,r_{i p_i})\in\mathbb{N}^{p_i}$.
This multi-index acts on~$\mathbb{R}^{p_i}$,
whose variables are denoted by~$x_i=(x_{i1},\dots,x_{ip_i})\in\mathbb{R}^{p_i}$.
We let~$\overline{v}_{i1}$ be the solution of
the Cauchy problem
\begin{equation}\label{CAH1}
\begin{cases}
\partial^{r_{i1}}_{x_{i1}}\overline{v}_{i1}=-\overline{\XA}_i\overline{v}_{i1} \\
\partial^{\beta_1}_{x_{i1}}\overline{v}_{i1}\left(0\right)=1\quad
\text{ for every } \beta_1\leq r_{i1}-1.
\end{cases}
\end{equation}
We notice that
the solution of the Cauchy problem in~\eqref{CAH1}
exists at least in a neighborhood of the origin of the form
$[-\rho_{i1},\rho_{i1}]$ for a suitable $\rho_{i1}>0$.

Moreover, if~$p_i\ge2$,
for any $\ell\in\{2,\dots, p_i\}$, we consider
the solution of the following Cauchy problem:
\begin{equation}\label{CAH2}
\begin{cases}
\partial^{r_{i\ell}}_{x_{i\ell}}\overline{v}_{i\ell}=\overline{v}_{i\ell} \\
\partial^{\beta_\ell}_{x_{i\ell}}\overline{v}_{i\ell}\left(0\right)=1\quad
\text{ for every } \beta_\ell\leq r_{i\ell}-1.
\end{cases}
\end{equation}
As above, these solutions
are well-defined at least in a neighborhood of the origin of the form $[-\rho_{i\ell},\rho_{i\ell}]$,
for a suitable $\rho_{i\ell}>0$.

Then, we define
\[\overline{\rho}_i:=\min\{ \rho_{i1},\dots,\rho_{i p_i}\}=\min_{\ell\in\{1,\dots,p_i\}}\rho_{i\ell}.\] 
In this way, for every~$x_i=(x_{i1},\dots,x_{ip_i})\in B^{p_i}_{\overline{\rho}_i}$,
we set
\begin{equation}\label{vba}
\overline{v}_i(x_i):=\overline{v}_{i1}(x_{i1})\ldots
\overline{v}_{i{p_i}}(x_{ip_i}).\end{equation}
By~\eqref{CAH1} and~\eqref{CAH2}, we have that
\begin{equation}
\label{cau}
\begin{cases}
\partial^{r_i}_{x_i}\overline{v}_i=-\overline{\XA}_i\overline{v}_i \\
\\
\partial^{\beta}_{x_i}\overline{v}_i\left(0\right)=1\quad
\begin{matrix}
&\text{ for every $\beta=(\beta_1,\dots\beta_{p_i})\in{\mathbb{N}}^{p_i}$}\\
&\text{ such that~$\beta_{\ell}\leq r_{i\ell}-1$ for each~$\ell\in\{1,\dots,p_i\}$.}\end{matrix}
\end{cases}
\end{equation}
Now, we define \[\rho:=\min\{ \overline{\rho}_1,\dots\overline{\rho}_n\}
=\min_{i\in\{1,\dots,n\}}\overline{\rho}_i.\]
We take 
\begin{equation*}
\overline{\tau}\in C_0^\infty\left(B^{p_1+\ldots+p_n}_{\rho/(R+2)}\right),
\end{equation*}
with $\overline{\tau}=1$ in $B^{p_1+\ldots+p_n}_{\rho/(2(R+2))}$, and,
for every~$x=(x_1,\dots,x_n)\in{\mathbb{R}}^{p_1}\times\dots\times{\mathbb{R}}^{p_n}$, we set 
\begin{equation}
\label{otau1}\tau_1\left(x_1,\ldots,x_n\right):=\overline{\tau}\left(\XX_1 \otimes x_1,\ldots,\XX_n \otimes x_n\right).\end{equation}
We recall that the free parameters~$\XX_1,\dots,\XX_n$
have been introduced in~\eqref{FREExi}, and we have used here the notation
$$ \XX_i\otimes x_i = (\XX_{i1},\dots,\XX_{ip_i})\otimes(x_{i1},\dots,x_{ip_i}):= (\XX_{i1}x_{i1},\dots,\XX_{ip_i}x_{ip_i})\in{\mathbb{R}}^{p_i},$$
for every~$i\in\{1,\dots,n\}$.

We also set, for any~$i\in\{1,\dots,n\}$,
\begin{equation}
\label{vuggei}
v_i\left(x_i\right):=\overline{v}_i\left(\XX_i\otimes x_i\right).
\end{equation}
We point out that if~$x_i\in B^{p_i}_{\overline\rho_i/(R+2)}$ we have that
$$ |\XX_i\otimes x_i|^2=\sum_{\ell=1}^{p_i}(\XX_{i\ell}x_{i\ell})^2\le
(R+2)^2\sum_{\ell=1}^{p_i}x_{i\ell}^2<\overline\rho_i^2,$$
thanks to~\eqref{FREExi}, and therefore the setting in~\eqref{vuggei}
is well-defined for every~$x_i\in B^{p_i}_{\overline\rho_i/(R+2)}$.

Recalling~\eqref{cau} and~\eqref{vuggei}, we see that, for any~$i\in\{1,\dots,n\}$,
\begin{equation}\label{QYHA0ow1dk} \partial_{x_i}^{r_i}v_i(x_i)=
\XX_i^{r_i}\partial_{x_i}^{r_i}\overline{v}_i\left(\XX_i\otimes x_i\right)
=
-\overline\XA_i \XX_i^{r_i}\overline{v}_i\left(\XX_i\otimes x_i\right)
=
-\overline\XA_i \XX_i^{r_i} {v}_i( x_i).\end{equation}
%Let also~\begin{equation}\label{otau2}\tau_2\in C^\infty_0([-30,30]^{l})\end{equation}
%with~$\tau_2=1$ in~$[-20,20]^{l}$.
We take $e_1,\ldots,e_M$, with
\begin{equation}
\label{econ}
e_j\in\partial B_{\omega_j}^{m_j},
\end{equation}
and we introduce an additional set of free parameters
$Y_1,\ldots,Y_M$ with 
\begin{equation}
\label{eq:FREEy}
Y_j\in\mathbb{R}^{m_j}\qquad{\mbox{
and }}\qquad e_j\cdot Y_j<0. \end{equation}
We let~$\epsilon>0$, to be taken small
possibly depending on the  free parameters $e_j$, $Y_j$ and $\XT_h$, 
and we define
\begin{equation}
\label{svs}
\begin{split}
w\left(x,y,t\right): 
=&\tau_1\left(x\right)v_1\left(x_1\right)\cdot
\ldots \cdot
v_n\left(x_n\right)\phi_1\left(y_1+e_1+\epsilon Y_1\right)\cdot
\ldots\cdot
\phi_M\left(y_M+e_M+\epsilon Y_M\right) \\
&\times\psi^\star_1(t_1)\cdot\ldots\cdot\psi^\star_l(t_l),
\end{split}
\end{equation}
where the setting in~\eqref{autofun1}, %% \eqref{autofun2},
\eqref{starest}, \eqref{otau1} and~\eqref{vuggei} has been exploited.

%We remark that, for every~$h\in\{1,\dots,l\}$,
%we have that~$\tau_2 (t_1,\dots,t_{h-1},\eta,t_{h+1},\dots,t_l)=1$
%when~$\eta\in(a_h,t_h)$ and~$|t_1|,\dots,|t_l|\le18$, and consequently
%\begin{equation}\label{TH92KA}
%\begin{split}&
%\Gamma([\alpha_h]+1-\alpha_h)\,D^{\alpha_h}_{t_h,a_h}
%\Big(\psi_1(t_1)\dots\psi_l(t_l)\tau_2(t)\Big)\\
%=\,&
%\int_{a_h}^{t_h} \frac{\partial^{[\alpha_h]+1}_{\eta}
%\Big(
%\psi_1(t_1)\dots\psi_{h-1}(t_{h-1})
%\psi_h(\eta)\psi_{h+1}(t_{h+1})\dots\psi_l(t_l)\;\tau_2 (t_1,\dots,t_{h-1},\eta,t_{h+1},\dots,t_l)
%\Big)}{(t_h-\eta)^{\alpha_h-[\alpha_h]}}\,d\eta\\
%=\,&
%\int_{a_h}^{t_h} \frac{\partial^{[\alpha_h]+1}_{\eta}
%\big(\psi_1(t_1)\dots\psi_{h-1}(t_{h-1})
%\psi_h(\eta)\psi_{h+1}(t_{h+1})\dots\psi_l(t_l)
%\big)}{(t_h-\eta)^{\alpha_h-[\alpha_h]}}\,d\eta\\
%=\,&\Gamma([\alpha_h]+1-\alpha_h)\,D^{\alpha_h}_{t_h,a_h}
%\Big(\psi_1(t_1)\dots\psi_l(t_l)\Big)\\
%=\,&\Gamma([\alpha_h]+1-\alpha_h)\,
%\psi_1(t_1)\dots\psi_{h-1}(t_{h-1})\,
%D^{\alpha_h}_{t_h,a_h}\psi_h(t_h)\,\psi_{h+1}(t_{h+1})\dots
%\psi_l(t_l),
%\end{split}\end{equation}
%as long as~$|t|\le 18$.

We also notice that $w\in C\left(\mathbb{R}^N\right)\cap C_0\left(\mathbb{R}^{N-l}\right)\cap\mathcal{A}$. Moreover,
if
\begin{equation}\label{pata7UJ:AKK}
a=(a_1,\dots,a_l):=\left(-\frac{\epsilon}{\XT_1},\dots,-\frac{\epsilon}{\XT_l}
\right)\in(-\infty,0)^l\end{equation}
and~$(x,y)$ is sufficiently close to the origin
and~$t\in(a_1,+\infty)\times\dots\times(a_l,+\infty)$, we have that
\begin{eqnarray*}&&
\Lambda_{-\infty} w\left(x,y,t\right)\\&=&
\left( \sum_{i=1}^n \XA_i \partial^{r_i}_{x_i}
+\sum_{j=1}^{M} \XB_j (-\Delta)^{s_j}_{y_j}+
\sum_{h=1}^{l} \XC_h D^{\alpha_h}_{t_h,-\infty}\right) w\left(x,y,t\right)\\
&=&
\sum_{i=1}^n \XA_i 
v_1\left(x_1\right)
\ldots  
v_{i-1}\left(x_{i-1}\right)
\partial^{r_i}_{x_i}v_{i}\left(x_{i}\right)v_{i+1}\left(x_{i+1}\right)
\ldots  v_n\left(x_n\right)\\
&&\qquad\times\phi_1\left(y_1+e_1+\epsilon Y_1\right)
\ldots\phi_M\left(y_M+e_M+\epsilon Y_M\right)\psi^\star_1\left(t_1\right)
\ldots\psi^\star_l\left(t_l\right)\\
&&+\sum_{j=1}^{M} \XB_j 
v_1\left(x_1\right)
\ldots v_n\left(x_n\right)\phi_1\left(y_1+e_1+\epsilon Y_1\right)
\ldots\phi_{j-1}\left(y_{j-1}+e_{j-1}+\epsilon Y_{j-1}\right)\\&&\qquad\times
(-\Delta)^{s_j}_{y_j}\phi_j\left(y_j+e_j+\epsilon Y_j\right)
\phi_{j+1}\left(y_{j+1}+e_{j+1}+\epsilon Y_{j+1}\right)
\ldots\phi_M\left(y_M+e_M+\epsilon Y_M\right)\\
&&\qquad\times\psi^\star_1\left(t_1\right)
\ldots\psi^\star_l\left(t_l\right)\\
&&+\sum_{h=1}^{l} \XC_h 
v_1\left(x_1\right)
\ldots v_n\left(x_n\right)\phi_1\left(y_1+e_1+\epsilon Y_1\right)\ldots\phi_M\left(y_M+e_M+\epsilon Y_M\right)\psi^\star_1\left(t_1\right)
\ldots\psi^\star_{h-1}\left(t_{h-1}\right)\\
&&\qquad\times D^{\alpha_h}_{t_h,-\infty}\psi^\star_h\left(t_h\right)
\psi^\star_{h+1}(t_{h+1})\ldots\psi^\star_l\left(t_l\right)\\
&=&
-\sum_{i=1}^n \XA_i \overline\XA_i\XX_i^{r_i}
v_1\left(x_1\right)
\ldots v_n\left(x_n\right)\phi_1\left(y_1+e_1+\epsilon Y_1\right)
\ldots\phi_M\left(y_M+e_M+\epsilon Y_M\right)\psi^\star_1(t_1)\ldots\psi^\star_l(t_l)\\
&&+\sum_{j=1}^{M} \XB_j \lambda_j
v_1\left(x_1\right)
\ldots v_n\left(x_n\right)\phi_1\left(y_1+e_1+\epsilon Y_1\right)
\ldots\phi_M\left(y_M+e_M+\epsilon Y_M\right)\psi^\star_1(t_1)\ldots\psi^\star_l(t_l)
\\ 
&&+\sum_{h=1}^{l} \XC_h \XT_{\star,h}
v_1\left(x_1\right)
\ldots v_n\left(x_n\right)\phi_1\left(y_1+e_1+\epsilon Y_1\right)
\ldots\phi_M\left(y_M+e_M+\epsilon Y_M\right)\psi^\star_1(t_1)\ldots\psi^\star_l(t_l)
\\
&=&\left( -\sum_{i=1}^n \XA_i \overline\XA_i\XX_i^{r_i}
+\sum_{j=1}^{M} \XB_j \lambda_j+\sum_{h=1}^l \XC_h\XT_{\star,h}\right) w(x,y,t),
\end{eqnarray*}
thanks to \eqref{REGSWYS-A}, \eqref{DOBACHA} and~\eqref{QYHA0ow1dk}
.

Consequently, making use of~\eqref{1.6md}, \eqref{alp} and~\eqref{NOZABAnd},
if~$(x,y)$ lies near the origin and~$t\in(a_1,+\infty)\times\dots\times(a_l,+\infty)$,
we have that
\begin{eqnarray*}&&
\Lambda_{-\infty} w\left(x,y,t\right)=
\left( -\sum_{i=1}^n |\XA_i |\XX_i^{r_i}
+\sum_{j=1}^{M-1} \XB_j \lambda_{j}+\XB_M\lambda_M+\sum_{h=1}^l \XC_h\XT_{\star,h}\right) w(x,y,t)
\\&&\qquad=
\left( -\sum_{i=1}^n |\XA_i |\XX_i^{r_i}
+\sum_{j=1}^{M-1} \XB_j \lambda_{\star,j}+\XB_M\lambda_M+\sum_{h=1}^l \XC_h\XT_{\star,h}\right) w(x,y,t)=0.
\end{eqnarray*}
This says that~$w\in\mathcal{H}$. Thus, in light of~\eqref{aza} we have that
\begin{equation}
\label{eq:ort}
0=\theta\cdot\partial^K w\left(0\right)=
\sum_{\left|\iota\right|\leq K}
{\theta_{\iota}\partial^\iota w\left(0\right)}=\sum_{|i|+|I|+|\mathfrak{I}|\le K}
\theta_{i,I,\mathfrak{I}}\,\partial_{x}^{i}\partial_{y}^{I}\partial_{t}^{\mathfrak{I}}w\left(0\right) 
.\end{equation}
%Since~$w$ is constant in~$t$ near the origin,
%recalling the notation in~\eqref{mulPM}
%we have that~$\partial^\iota w(0)=
%\partial^{(i,I,\mathfrak{I})} w(0)=0$ if~$|\mathfrak{I}|>0$.
Now, we recall~\eqref{vba} and we claim that,
for any $j\in\{1,\dots,n\}$,
any~$\ell\in\{1,\dots, p_j\}$ and any~$
i_{j\ell}\in\mathbb{N}$, we have that
\begin{equation}
\label{notnull}
\partial^{i_{j\ell}}_{x_{j\ell}}\overline{v}_{j\ell}(0)\neq 0.
\end{equation}
We prove it by induction over~$i_{j\ell}$. Indeed, if 
$i_{j\ell}\in\left\{0,\ldots,r_{j\ell}-1\right\}$, then
the initial condition in \eqref{CAH1} (if~$\ell=1$)
or~\eqref{CAH2} (if~$\ell\ge2$) gives that~$
\partial^{i_{j\ell}}_{x_{i\ell}}\overline{v}_{i\ell}\left(0\right)=1$, and
so~\eqref{notnull}
is true in this case.

To perform the inductive step,
let us now
suppose that the claim in~\eqref{notnull}
still holds for all $i_{j\ell}\in\left\{0,\ldots,i_0\right\}$
for some~$i_0$ such that $i_0\geq r_{j\ell}-1$.
Then, using the equation in~\eqref{CAH1} (if~$\ell=1$)
or in~\eqref{CAH2} (if~$\ell\ge2$), we have that
\begin{equation}\label{8JAMANaoaksd}
\partial^{i_0+1}_{x_{j\ell}}\overline{v}_j=
\partial^{i_0+1-r_{j\ell}}_{x_{j\ell}}\partial^{r_{j\ell}}_{x_{j\ell}}
\overline{v}_j=-\tilde{a}_j\partial^{i_0+1-r_{j\ell}}_{x_{j\ell}}\overline{v}_j,
\end{equation}
with
$$\tilde{a}_j:=\begin{cases}
\overline{a}_j & {\mbox{ if }}\ell=1,\\
-1 & {\mbox{ if }}\ell\ge2.
\end{cases}$$
Notice that~$\tilde{a}_j\ne0$, in view of~\eqref{NOZABA},
and~$\partial^{i_0+1-r_{j\ell}}_{x_{j \ell}}\overline{v}_j\left(0\right)\neq 0$,
by the inductive assumption. These considerations and~\eqref{8JAMANaoaksd}
give that~$\partial^{i_0+1}_{x_{j\ell}}\overline{v}_j\left(0\right)\neq 0$,
and this proves~\eqref{notnull}. 

Now, using \eqref{vba} and~\eqref{notnull} we have that,
for any $j\in\{1,\dots,n\}$ and any~$
i_{j}\in\mathbb{N}^{p_j}$,
\begin{equation*}
\partial^{i_j}_{x_{j}}\overline{v}_{j}(0)\neq 0.
\end{equation*}
This, \eqref{FREExi} and the computation in~\eqref{QYHA0ow1dk} give that,
for any $j\in\{1,\dots,n\}$ and any~$
i_{j}\in\mathbb{N}^{p_j}$,
\begin{equation}
\label{eq:adoajfiap}
\partial^{i_j}_{x_{j}} {v}_{j}(0)=\XX_j^{i_j}\partial^{i_j}_{x_{j}}\overline{v}_{j}(0)\neq 0.
\end{equation}
We also notice that, in light of~\eqref{starest}, \eqref{svs} and~\eqref{eq:ort},
\begin{equation}\label{7UJHASndn2weirit}\begin{split}&
0=\sum_{|i|+|I|+|\mathfrak{I}|\le K}
\theta_{i,I,\mathfrak{I}}\,
\partial^{i_1}_{x_{1}} {v}_{1}(0)\ldots
\partial^{i_n}_{x_{n}} {v}_{n}(0)
\,\partial_{y_1}^{I_1}\phi_1\left(e_1+\epsilon Y_1\right)
\ldots\partial_{y_M}^{I_M}\phi_M\left(e_M+\epsilon Y_M\right)\\
&\qquad\qquad\qquad\times\partial^{\mathfrak{I}_1}_{t_1}
\psi_1(0)\ldots\partial^{\mathfrak{I}_l}\psi_l(0).\end{split}\end{equation}
Now, by~\eqref{autofun1}
and Proposition \ref{sharbou}
(applied to $s:=s_j$, $\beta:=I_j$, $e:=\frac{e_j}{\omega_j}\in\partial B_1^{m_j}$, due to~\eqref{econ}, and $X:=\frac{Y_j}{\omega_j}$),
we see that, for any $j=1,\ldots, M$,
\begin{equation}
\label{alppoo}
\begin{split}\omega_j^{|I_j|}
\lim_{\epsilon\searrow 0}\epsilon^{|I_j|-s_j}\partial_{y_j}^{I_j}
\phi_j\left( e_j+\epsilon Y_j \right)
\; =\;&
\lim_{\epsilon\searrow 0}\epsilon^{|I_j|-s_j}\partial_{y_j}^{I_j}
\tilde\phi_{\star,j}\left(\frac{e_j+\epsilon Y_j}{\omega_j}\right)
\\=\;& \kappa_j \frac{e_j^{I_j}}{\omega_j^{|I_j|}}\left(-\frac{e_j}{\omega_j}\cdot\frac{Y_j}{\omega_j}\right)_+^{s_j-|I_j|}
,\end{split}\end{equation}
with~$\kappa_j\ne0$, in the sense of distributions (in the coordinates~$Y_j$).

Moreover, using~\eqref{STAvca} and \eqref{pata7UJ:AKK},
it follows that
\begin{eqnarray*}
\partial^{\mathfrak{I}_h}_{t_h}\psi_h(0)&=&
\sum_{j=0}^{+\infty} {\frac{
\XT_{\star,h}^j\, \alpha_h j(\alpha_h j-1)\dots(\alpha_h j-
\mathfrak{I}_h+1)
(0-a_h)^{\alpha_h j-\mathfrak{I}_h}}{\Gamma\left(\alpha_h j+1\right)}}\\
&=& \sum_{j=0}^{+\infty} {\frac{
\XT_{\star,h}^j\, \alpha_h j(\alpha_h j-1)\dots(\alpha_h j-
\mathfrak{I}_h+1)
\,{\epsilon}^{\alpha_h j-\mathfrak{I}_h}}{\Gamma\left(\alpha_h j+1\right)\;
{\XT_h}^{\alpha_h j-\mathfrak{I}_h}}}\\
&=& \sum_{j=1}^{+\infty} {\frac{
\XT_{\star,h}^j\, \alpha_h j(\alpha_h j-1)\dots(\alpha_h j-
\mathfrak{I}_h+1)
\,{\epsilon}^{\alpha_h j-\mathfrak{I}_h}}{\Gamma\left(\alpha_h j+1\right)\;
{\XT_h}^{\alpha_h j-\mathfrak{I}_h}}}.
\end{eqnarray*}
Accordingly, recalling~\eqref{TGAdef}, we find that
\begin{equation}
\label{limmittl}\begin{split}&
\lim_{\epsilon\searrow 0}\epsilon^{\mathfrak{I}_h-\alpha_h}
\partial^{\mathfrak{I}_h}_{t_h}\psi_h(0)
=\lim_{\epsilon\searrow 0}
\sum_{j=1}^{+\infty} {\frac{
\XT_{\star,h}^j\, \alpha_h j(\alpha_h j-1)\dots(\alpha_h j-
\mathfrak{I}_h+1)
\,{\epsilon}^{\alpha_h (j-1)}}{\Gamma\left(\alpha_h j+1\right)\;
{\XT_h}^{\alpha_h j-\mathfrak{I}_h}}}
\\&\qquad=
{\frac{ \XT_{\star,h}\,\alpha_h (\alpha_h -1)\dots(\alpha_h -\mathfrak{I}_h+1)
}{\Gamma\left(\alpha_h +1\right)\;
\XT_h^{\alpha_h -\mathfrak{I}_h}}} =
{\frac{ \XT_h^{\mathfrak{I}_h}\,\alpha_h (\alpha_h -1)\dots(\alpha_h -\mathfrak{I}_h+1)
}{\Gamma\left(\alpha_h +1\right)}}.
\end{split}\end{equation}
Also, recalling~\eqref{IBARRA}, we can write~\eqref{7UJHASndn2weirit} as
\begin{equation}
\label{eq:ort:X}\begin{split}&
0=\sum_{{|i|+|I|+|\mathfrak{I}|\le K}\atop{|I|\le|\overline{I}|}}
\theta_{i,I,\mathfrak{I}}\,
\partial^{i_1}_{x_{1}} {v}_{1}(0)\ldots
\partial^{i_n}_{x_{n}} {v}_{n}(0)
\,\partial_{y_1}^{I_1}\phi_1\left(e_1+\epsilon Y_1\right)
\ldots\partial_{y_M}^{I_M}\phi_M\left(e_M+\epsilon Y_M\right)\\&\qquad\qquad\times
\partial^{\mathfrak{I}_1}_{t_1}\psi_{1}(0)\ldots
\partial^{\mathfrak{I}_l}_{t_l}\psi_{l}(0).\end{split}
\end{equation}
Moreover, we define
\begin{equation*}
\Xi:=\left|\overline{I}\right|-\sum_{j=1}^M {s_j}+|\mathfrak{I}|-\sum_{h=1}^l {\alpha_h}.
\end{equation*}
Then, we
multiply~\eqref{eq:ort:X} by $\epsilon^{\Xi}\in(0,+\infty)$, and we
send~$\epsilon$ to zero. In this way, we obtain from~\eqref{alppoo},
\eqref{limmittl}
and~\eqref{eq:ort:X} that
\begin{eqnarray*}
0&=&\lim_{\epsilon\searrow0}
\epsilon^{\Xi}
\sum_{{|i|+|I|+|\mathfrak{I}|\le K}\atop{|I|\le|\overline{I}|}}
\theta_{i,I,\mathfrak{I}}\,
\partial^{i_1}_{x_{1}} {v}_{1}(0)\ldots
\partial^{i_n}_{x_{n}} {v}_{n}(0)
\,\partial_{y_1}^{I_1}\phi_1\left(e_1+\epsilon Y_1\right)
\ldots\partial_{y_M}^{I_M}\phi_M\left(e_M+\epsilon Y_M\right)
\\ &&\qquad\times\partial^{\mathfrak{I}_1}_{t_1}\psi_1(0)\ldots\partial^{\mathfrak{I}_l}_{t_l}\psi_l(0)
\\ &=&\lim_{\epsilon\searrow0}
\sum_{{|i|+|I|+|\mathfrak{I}|\le K}\atop{|I|\le|\overline{I}|}}
\epsilon^{|\overline{I}|-|I|}
\theta_{i,I,\mathfrak{I}}\,
\partial^{i_1}_{x_{1}} {v}_{1}(0)\ldots
\partial^{i_n}_{x_{n}} {v}_{n}(0)\\
&&\qquad\times\epsilon^{|I_1|-s_1}\partial_{y_1}^{I_1}\phi_1\left(e_1+\epsilon Y_1\right)
\ldots\epsilon^{|I_M|-s_M}\partial_{y_M}^{I_M}\phi_M\left(e_M+\epsilon Y_M\right)
\\ &&\qquad\times\epsilon^{\mathfrak{I}_1-\alpha_1}\partial^{\mathfrak{I}_1}_{t_1}\psi_1(0)\ldots\epsilon^{\mathfrak{I}_l-\alpha_l}\partial^{\mathfrak{I}_l}_{t_l}\psi_l(0)
\\&=&
\sum_{{|i|+|I|+|\mathfrak{I}|\le K}\atop{|I| = |\overline{I}|}}
\tilde C_{i,I,\mathfrak{I}}\,\theta_{i,I,\mathfrak{I}}\,
\partial^{i_1}_{x_{1}} {v}_{1}(0)\ldots
\partial^{i_n}_{x_{n}} {v}_{n}(0)\\\
&&\qquad\times e_1^{I_1}\ldots e_M^{I_M}\,
\left(- e_1 \cdot Y_1 \right)_+^{s_1-|I_1|}
\ldots
\left(- e_M \cdot Y_M\right)_+^{s_M-|I_M|}\XT_1^{\mathfrak{I}_1}\ldots\XT_l^{\mathfrak{I}_l}
,\end{eqnarray*}
for a suitable~$\tilde C_{i,I,\mathfrak{I}}\ne0$
(strictly speaking, the above identity holds
in the sense of distribution with respect to the coordinates~$Y$
and~$\XT$, but since the left hand side vanishes,
we can consider it also a pointwise identity).

Hence, recalling~\eqref{eq:adoajfiap},
\begin{equation}\label{GANAfai}\begin{split}
0&\,=\,\sum_{{|i|+|I|+|\mathfrak{I}|\le K}\atop{|I| = |\overline{I}|}}
C_{i,I,\mathfrak{I}}\,\theta_{i_1,\dots,i_n,I_1,\dots,I_M,\mathfrak{I}_1,\dots,\mathfrak{I}_l}\;
\XX_1^{i_1}\ldots\XX_n^{i_n}\\&\qquad\qquad\times
e_1^{I_1}\ldots e_M^{I_M}\,
\left(- e_1 \cdot Y_1 \right)_+^{s_1-|I_1|}
\ldots
\left(- e_M \cdot Y_M\right)_+^{s_M-|I_M|}
\XT_1^{\mathfrak{I}_1}\ldots\XT_l^{\mathfrak{I}_l}\\
&\,=\,
\left(- e_1 \cdot Y_1 \right)_+^{s_1}
\ldots
\left(- e_M \cdot Y_M\right)_+^{s_M}\\&\qquad\qquad\times
\sum_{{|i|+|I|+|\mathfrak{I}|\le K}\atop{|I| = |\overline{I}|}}
C_{i,I,\mathfrak{I}}\,\theta_{i,I,\mathfrak{I}}\;
\XX^i\,e^{I}\,
\left(- e_1 \cdot Y_1 \right)_+^{-|I_1|}
\ldots
\left(- e_M \cdot Y_M\right)_+^{-|I_M|}
\XT^{\mathfrak{I}}
,\end{split}\end{equation}
for a suitable~$C_{i,I,\mathfrak{I}}\ne0$.

We observe that the equality in~\eqref{GANAfai}
is valid for any choice of the free parameters~$(\XX,Y,\XT)$
in an open subset of~$\mathbb{R}^{p_1+\dots+p_n}\times
\mathbb{R}^{m_1+\dots+m_M}\times\mathbb{R}^l$,
as prescribed in~\eqref{FREExi},~\eqref{FREEmustar}
and~\eqref{eq:FREEy}. 

Now, we take new free parameters, $\XY_1,\ldots,\XY_M$ with $\XY_j\in\mathbb{R}^{m_j}\setminus\{0\}$, and we
define
\begin{equation}\label{COMPATI}
e_j:=\frac{\omega_j\XY_j}{|\XY_j|}\quad {\mbox{ and }}\quad 
Y_j:=-\frac{\XY_j}{|\XY_j|^2}.\end{equation}
We stress that the setting in~\eqref{COMPATI} is compatible with that in~\eqref{eq:FREEy}, since
$$ e_j\cdot Y_j=-\frac{\omega_j\XY_j}{|\XY_j|}\cdot
\frac{\XY_j}{|\XY_j|^2}=-\frac{\omega_j}{|\XY_j|}<0,$$
thanks to~\eqref{OMEj}. We also notice that, for all~$j\in\{1,\dots,M\}$,
$$ e_j^{I_j}\left(- e_j \cdot Y_j \right)_+^{-|I_j|}=
\frac{\omega_j^{|I_j|}\XY_j^{I_j}}{|\XY_j|^{|I_j|}}\,
\frac{|\XY_j|^{|I_j|}}{\omega_j^{|I_j|}}=\XY_j^{I_j},
$$
and hence
$$ e^{I}\,
\left(- e_1 \cdot Y_1 \right)_+^{-|I_1|}
\ldots
\left(- e_M \cdot Y_M\right)_+^{-|I_M|}=
\XY^I.$$
Plugging this into formula \eqref{GANAfai},
we obtain the first identity in~\eqref{ipop},
as desired.
Hence, 
the proof of~\eqref{ipop} in case~\ref{itm:case1}
is complete.
\end{proof}

\begin{proof}[Proof of \eqref{ipop}, case \ref{itm:case2}]
Thanks to the assumptions given in case~\ref{itm:case2}, we can suppose
that formula~\eqref{AGZ} still holds, and also that
\begin{equation}\label{MAGGZC}
{\mbox{$\XC_l>0$}}.\end{equation} In addition,
for any $j\in\{1,\ldots,M\}$, we consider $\lambda_j$
and~$\phi_j$ as in \eqref{REGSWYS-A}.

Then, we define
\begin{equation}\label{MAfghjkGGZC} R:=\left( 
\frac{ 1 }{|\XA_1|}\displaystyle\left(\sum_{h=1}^{l-1}|\XC_h|+\sum_{j=1}^M|\XB_j|\lambda_j\right)\right)^{1/|r_{1}|}.\end{equation}
We notice that, in light of \eqref{AGZ}, the setting in~\eqref{MAfghjkGGZC} is well-defined. 

Now, we fix two sets of free parameters $\XX_1,\ldots,\XX_n$
as in~\eqref{FREExi}
and~$\,\XT_{\star,1},\dots,\XT_{\star,l}$ as in~\eqref{FREEmustar}, here taken with~$R$
as in~\eqref{MAfghjkGGZC}.
Moreover, we define
\begin{equation}
\label{alpc}
\lambda\,:=\,\frac{1}{\XC_l\,\XT_{\star,l}}\left(
\sum_{j=1}^n {\left|\XA_j\right|\XX_j^{r_j}}-\sum_{j=1}^M\XB_j\lambda_j-\sum_{h=1}^{l-1}\XC_h\XT_{\star,h}\right).\end{equation}
We notice that \eqref{alpc} is well-defined,
thanks to~\eqref{FREEmustar}
and~\eqref{MAGGZC}.
Furthermore, recalling~\eqref{FREExi}, \eqref{1.15bis} and~\eqref{MAfghjkGGZC},
we find that
\begin{equation*}
\begin{split}
\sum_{i=1}^n&|\XA_i|\XX_i^{r_i}\ge|\XA_1|\XX_1^{r_1}>
|\XA_1|(R+1)^{|r_1|}>|\XA_1|R^{|r_1|} \\
&=\sum_{h=1}^{l-1}|\XC_h|+\sum_{j=1}^M|\XB_j|\lambda_j
\ge\sum_{h=1}^{l-1}\XC_h\XT_{\star,h}+\sum_{j=1}^M\XB_j\lambda_j.
\end{split}
\end{equation*}
Consequently, by~\eqref{alpc},
\begin{equation}
\label{alp-0c}
\lambda>0.\end{equation}
Hence, we can define
\begin{equation}\label{OVlam}
\overline{\lambda}:=\lambda^{1/{\alpha_l}}.\end{equation}
Moreover, 
we consider~$a_h\in(-2,0)$, for every~$h\in\{ 1,\dots,l\}$,
to be chosen appropriately in what follows
(the exact choice will be performed in~\eqref{pata7UJ:AKKcc}),
and, using the notation in~\eqref{chosofpsistar}
and~\eqref{TGAdef}, we define
\begin{equation}
\label{autofun2c}
\psi_h(t_h):=\psi_{\star,h}\big(\XT_h (t_h-a_h)\big)=
E_{\alpha_h,1}\big(\XT_{\star,h} (t_h-a_h)^{\alpha_h}\big)\quad\text{if }\,h\in\{1,\dots,l-1\}
\end{equation}
and
\begin{equation}
\label{psielle}
\psi_{l}(t_l):=\psi_{\star,l}\big(\overline{\lambda}\,\XT_l (t_l-a_l)\big)=
E_{\alpha_l,1}\big(\lambda\,\XT_{\star,l} (t_l-a_l)^{\alpha_l}\big).
\end{equation}
We recall that, thanks to Lemma~\ref{MittagLEMMA}, the function in~\eqref{autofun2c}
solves~\eqref{jhjadwlgh} and satisfies~\eqref{STAvca} for any $h\in\{1,\dots,l-1\}$, while the function
in~\eqref{psielle} solves
\begin{equation}
\label{jhjadwlghplop}
\begin{cases}
D^{\alpha_l}_{t_l,a_l}\psi_l(t_l)=\lambda\,\XT_{\star,l}\psi_l(t_l)&\quad\text{in }\,(a_l,+\infty), \\
\psi_l(a_l)=1, \\
\partial^m_{t_l}\psi_l(a_l)=0&\quad\text{for every }\,m\in\{1,\dots,[\alpha_l] \}.
\end{cases}
\end{equation}
As in~\eqref{starest}, we extend the functions~$\psi_{h}$
constantly in~$(-\infty,a_h)$, calling~$\psi^\star_h$ this
extended function. In this way,
Lemma~A.3 in~\cite{CDV18} translates~\eqref{jhjadwlghplop} into
\begin{equation}
\label{jhjadwlghplop2}
D^{\alpha_h}_{t_h,-\infty}\psi_h^\star(t_h)=\XT_{\star,h}\psi_h(t_h)
=\XT_{\star,h}\psi_h^\star(t_h)\,\text{ in every interval }\,I\Subset(a_h,+\infty).
\end{equation}
Now, we let~$\epsilon>0$, to be taken small
possibly depending on the  free parameters,
and we exploit the functions defined in \eqref{otau1} and \eqref{vuggei},
provided that
one replaces the positive constant $R$ defined in \eqref{MAfghjkGGZ}
with the one in \eqref{MAfghjkGGZC}, when necessary. 

With this idea in mind, for any $j\in\{1,\ldots,M\}$, we let\footnote{Comparing~\eqref{freee2}
with \eqref{econ}, we observe that~\eqref{econ} reduces to~\eqref{freee2} with the choice~$\omega_j:=1$.}
\begin{equation}
\label{freee2}
e_j\in\partial B^{m_j}_1,
\end{equation}
and we define
\begin{equation}
\label{svsc}
\begin{split}
w\left(x,y,t\right): 
=&\tau_1\left(x\right)v_1\left(x_1\right)\cdot
\ldots \cdot
v_n\left(x_n\right)\phi_1\left(y_1+e_1+\epsilon Y_1\right)\cdot
\ldots\cdot
\phi_M\left(y_M+e_M+\epsilon Y_M\right) \\
&\times\psi^\star_1(t_1)\cdot\ldots\cdot\psi^\star_l(t_l),
\end{split}
\end{equation}
where the setting in~\eqref{REGSWYS-A}, %% \eqref{starest}, 
\eqref{otau1},
\eqref{vuggei}, \eqref{eq:FREEy}, \eqref{autofun2c}
and~\eqref{psielle} has been exploited. 

We also notice that $w\in C\left(\mathbb{R}^N\right)\cap C_0(\mathbb{R}^{N-l})\cap\mathcal{A}$. Moreover,
if
\begin{equation}\label{pata7UJ:AKKcc}
a=(a_1,\dots,a_l):=\left(-\frac{\epsilon}{\XT_1},\dots,-\frac{\epsilon}{\,\XT_l}\right)
\in(-\infty,0)^l\end{equation}
and~$(x,y)$ is sufficiently close to the origin
and~$t\in(a_1,+\infty)\times\dots\times(a_l,+\infty)$, we have that
\begin{eqnarray*}&&
\Lambda_{-\infty} w\left(x,y,t\right)\\&=&
\left( \sum_{i=1}^n \XA_i \partial^{r_i}_{x_i}
+\sum_{j=1}^{M} \XB_j (-\Delta)^{s_j}_{y_j}+
\sum_{h=1}^{l} \XC_h D^{\alpha_h}_{t_h,-\infty}\right) w\left(x,y,t\right)\\
&=&
\sum_{i=1}^n \XA_i 
v_1\left(x_1\right)
\ldots  
v_{i-1}\left(x_{i-1}\right)
\partial^{r_i}_{x_i}v_{i}\left(x_{i}\right)v_{i+1}\left(x_{i+1}\right)
\ldots  v_n\left(x_n\right)\\
&&\qquad\times\phi_1\left(y_1+e_1+\epsilon Y_1\right)
\ldots\phi_M\left(y_M+e_M+\epsilon Y_M\right)\psi^\star_1\left(t_1\right)
\ldots\psi^\star_{l-1}\left(t_{l-1}\right)\psi^\star_l\left(t_l\right)\\
&&+\sum_{j=1}^{M} \XB_j 
v_1\left(x_1\right)
\ldots v_n\left(x_n\right)\phi_1\left(y_1+e_1+\epsilon Y_1\right)
\ldots\phi_{j-1}\left(y_{j-1}+e_{j-1}+\epsilon Y_{j-1}\right)\\&&\qquad\times
(-\Delta)^{s_j}_{y_j}\phi_j\left(y_j+e_j+\epsilon Y_j\right)
\phi_{j+1}\left(y_{j+1}+e_{j+1}+\epsilon Y_{j+1}\right)
\ldots\phi_M\left(y_M+e_M+\epsilon Y_M\right)\\
&&\qquad\times\psi^\star_1\left(t_1\right)
\ldots\psi^\star_{l-1}\left(t_{l-1}\right)\psi^\star_l\left(t_l\right)\\
&&+\sum_{h=1}^{l} \XC_h 
v_1\left(x_1\right)
\ldots v_n\left(x_n\right)\phi_1\left(y_1+e_1+\epsilon Y_1\right)\ldots\phi_M\left(y_M+e_M+\epsilon Y_M\right)\psi^\star_1\left(t_1\right)
\ldots\psi^\star_{h-1}\left(t_{h-1}\right)\\
&&\qquad\times D^{\alpha_h}_{t_h,-\infty}\psi^\star_h\left(t_h\right)
\psi^\star_{h+1}(t_{h+1})\ldots\psi^\star_{l-1}\left(t_{l-1}\right)\psi^\star_l\left(t_l\right)\\
&=&
-\sum_{i=1}^n \XA_i \overline\XA_i\XX_i^{r_i}
v_1\left(x_1\right)
\ldots v_n\left(x_n\right)\phi_1\left(y_1+e_1+\epsilon Y_1\right)
\ldots\phi_M\left(y_M+e_M+\epsilon Y_M\right)\\
&&\qquad\times\psi^\star_1(t_1)\ldots\psi^\star_{l-1}(t_{l-1})\psi^\star_l(t_l)\\
&&+\sum_{j=1}^{M} \XB_j \lambda_j
v_1\left(x_1\right)
\ldots v_n\left(x_n\right)\phi_1\left(y_1+e_1+\epsilon Y_1\right)
\ldots\phi_M\left(y_M+e_M+\epsilon Y_M\right)\\&&\qquad\times\psi^\star_1(t_1)\ldots\psi^\star_{l-1}\left(t_{l-1}\right)\psi^\star_l(t_l)
\\ 
&&+\sum_{h=1}^{l-1} \XC_h \XT_{\star,h}
v_1\left(x_1\right)
\ldots v_n\left(x_n\right)\phi_1\left(y_1+e_1+\epsilon Y_1\right)
\ldots\phi_M\left(y_M+e_M+\epsilon Y_M\right)\\&&\qquad\times\psi^\star_1(t_1)\ldots\psi^\star_{l-1}(t_{l-1})\psi^\star_l(t_l)
\\
&&+\XC_l\lambda\XT_{\star,l}v_1\left(x_1\right)
\ldots v_n\left(x_n\right)\phi_1\left(y_1+e_1+\epsilon Y_1\right)
\ldots\phi_M\left(y_M+e_M+\epsilon Y_M\right)\\
&&\qquad\times\psi^\star_1(t_1)\ldots\psi^\star_{l-1}(t_{l-1})\psi^\star_l(t_l) \\
&=&\left( -\sum_{i=1}^n \XA_i \overline\XA_i\XX_i^{r_i}
+\sum_{j=1}^M\XB_j\lambda_j+\sum_{h=1}^{l-1} \XC_h\XT_{\star,h}+\XC_l\lambda\XT_{\star,l}\right) w(x,y,t),
\end{eqnarray*}
thanks to \eqref{REGSWYS-A}, \eqref{jhjadwlgh}, \eqref{QYHA0ow1dk}
and~\eqref{jhjadwlghplop2}.

Consequently, making use of \eqref{NOZABAnd} and~\eqref{alpc}, when~$(x,y)$ is
near the origin and~$t\in(a_1,+\infty)\times\dots\times(a_l,+\infty)$,
we have that
$$
\Lambda_{-\infty} w\left(x,y,t\right)=
\left( -\sum_{i=1}^n |\XA_i |\XX_i^{r_i}
+\sum_{j=1}^M\XB_j\lambda_j+\sum_{h=1}^{l-1} \XC_h\XT_{\star,h}+\lambda\XC_l\XT_{\star,l}\right) w(x,y,t)=0
.$$
This says that~$w\in\mathcal{H}$. Thus, in light of~\eqref{aza} we have that
\begin{equation*}
0=\theta\cdot\partial^K w\left(0\right)=
\sum_{\left|\iota\right|\leq K}
{\theta_{\iota}\partial^\iota w\left(0\right)}=\sum_{|i|+|I|+|\mathfrak{I}|\le K}
\theta_{i,I,\mathfrak{I}}\,\partial_{x}^{i}\partial_{y}^I\partial_{t}^{\mathfrak{I}}w\left(0\right) 
.\end{equation*}
Hence, in view of~\eqref{eq:adoajfiap} and~\eqref{svsc},
\begin{equation}\label{1.64bis}
\begin{split}
0\,&=\,\sum_{|i|+|I|+|\mathfrak{I}|\le K}
\theta_{i,I,\mathfrak{I}}\,
\partial^{i_1}_{x_{1}} {v}_{1}(0)\ldots
\partial^{i_n}_{x_{n}} {v}_{n}(0)\\&\qquad\times
\partial^{I_1}_{y_{1}} {\phi}_{1}(e_1+\epsilon Y_1)\ldots
\partial^{I_M}_{y_{M}} {\phi}_{M}(e_M+\epsilon Y_M)
\,\partial_{t_1}^{\mathfrak{I}_1}\psi_1(0)
\ldots\partial_{t_l}^{\mathfrak{I}_l}\psi_l(0)
\\&=
\,\sum_{|i|+|I|+|\mathfrak{I}|\le K}
\theta_{i,I,\mathfrak{I}}\,\XX_1^{r_1}\ldots\XX_n^{r_n}\,
\partial^{i_1}_{x_{1}} \overline{v}_{1}(0)\ldots
\partial^{i_n}_{x_{n}} \overline{v}_{n}(0)\\&\qquad\times
\partial^{I_1}_{y_{1}} {\phi}_{1}(e_1+\epsilon Y_1)\ldots
\partial^{I_M}_{y_{M}} {\phi}_{M}(e_M+\epsilon Y_M)
\,\partial_{t_1}^{\mathfrak{I}_1}\psi_1(0)
\ldots\partial_{t_l}^{\mathfrak{I}_l}\psi_l(0)
.\end{split}\end{equation}
Moreover, using~\eqref{Mittag}, \eqref{psielle}
and \eqref{pata7UJ:AKKcc},
it follows that
\begin{eqnarray*}
\partial^{\mathfrak{I}_l}_{t_l}\psi_l(0)&=&
\sum_{j=0}^{+\infty} {\frac{
\lambda^j\,\XT_{\star,l}^j\, \alpha_l j(\alpha_l j-1)\dots(\alpha_l j-
\mathfrak{I}_l+1)
(0-a_l)^{\alpha_l j-\mathfrak{I}_l}}{\Gamma\left(\alpha_l j+1\right)}}\\
&=& \sum_{j=0}^{+\infty} {\frac{
\lambda^j\,\XT_{\star,l}^j\, \alpha_l j(\alpha_l j-1)\dots(\alpha_l j-
\mathfrak{I}_l+1)
\,{\epsilon}^{\alpha_l j-\mathfrak{I}_l}}{\Gamma\left(\alpha_l j+1\right)\;
{\XT_l}^{\alpha_l j-\mathfrak{I}_l}}}\\
&=& \sum_{j=1}^{+\infty} {\frac{
\lambda^j\,\XT_{\star,l}^j\, \alpha_l j(\alpha_l j-1)\dots(\alpha_l j-
\mathfrak{I}_l+1)
\,{\epsilon}^{\alpha_l j-\mathfrak{I}_l}}{\Gamma\left(\alpha_l j+1\right)\;
{\XT_l}^{\alpha_l j-\mathfrak{I}_l}}}.
\end{eqnarray*}
Accordingly, by~\eqref{TGAdef}, we find that
\begin{equation}
\label{limmittlc}\begin{split}&
\lim_{\epsilon\searrow 0}\epsilon^{\mathfrak{I}_l-\alpha_l}
\partial^{\mathfrak{I}_l}_{t_l}\psi_l(0)
=\lim_{\epsilon\searrow 0}
\sum_{j=1}^{+\infty} {\frac{
\lambda^j\,\XT_{\star,l}^j\, \alpha_l j(\alpha_l j-1)\dots(\alpha_l j-
\mathfrak{I}_l+1)
\,{\epsilon}^{\alpha_l (j-1)}}{\Gamma\left(\alpha_l j+1\right)\;
{\XT_l}^{\alpha_l j-\mathfrak{I}_l}}}
\\&\qquad=
{\frac{ \lambda\,\XT_{\star,l}\,\alpha_l (\alpha_l -1)\dots(\alpha_l -\mathfrak{I}_l+1)
}{\Gamma\left(\alpha_l +1\right)\;
\XT_l^{\alpha_l -\mathfrak{I}_l}}} =
{\frac{ \lambda\,\XT_l^{\mathfrak{I}_l}\,\alpha_l (\alpha_l -1)\dots(\alpha_l -\mathfrak{I}_l+1)
}{\Gamma\left(\alpha_l +1\right)}}.
\end{split}\end{equation}
Hence, recalling~\eqref{IBARRA2}, we can write~\eqref{1.64bis} as
\begin{equation}
\label{eq:ort:Xc}\begin{split}&
0=\sum_{{|i|+|I|+|\mathfrak{I}|\le K}\atop{|\mathfrak{I}|\le|\overline{\mathfrak{I}}|}}
\theta_{i,I,\mathfrak{I}}\,\XX_1^{r_1}\ldots\XX_n^{r_n}\,
\partial^{i_1}_{x_{1}} \overline{v}_{1}(0)\ldots
\partial^{i_n}_{x_{n}} \overline{v}_{n}(0)\\
&\qquad\qquad\times
\partial^{I_1}_{y_{1}} {\phi}_{1}(e_1+\epsilon Y_1)\ldots
\partial^{I_M}_{y_{M}} {\phi}_{M}(e_M+\epsilon Y_M) \partial^{\mathfrak{I}_1}_{t_1}\psi_{1}(0)\ldots
\partial^{\mathfrak{I}_l}_{t_l}\psi_{l}(0).\end{split}
\end{equation}
Moreover, we define
\begin{equation*}
\Xi:=|\overline{\mathfrak{I}}|-\sum_{h=1}^l {\alpha_h}+\left|I\right|-
\sum_{j=1}^M {s_j}.
\end{equation*}
Then, we
multiply~\eqref{eq:ort:Xc} by $\epsilon^{\Xi}\in(0,+\infty)$, and we
send~$\epsilon$ to zero. In this way, we obtain from~\eqref{limmittl}, used here for~$h\in\{1,\ldots, l-1\}$,
\eqref{limmittlc} and~\eqref{eq:ort:Xc} that
\begin{eqnarray*}
0&=&\lim_{\epsilon\searrow0}
\epsilon^{\Xi}
\sum_{{|i|+|I|+|\mathfrak{I}|\le K}\atop{|\mathfrak{I}|\le|\overline{\mathfrak{I}}|}}
\theta_{i,I,\mathfrak{I}}\,
\,\XX_1^{r_1}\ldots\XX_n^{r_n}\,
\partial^{i_1}_{x_{1}} \overline{v}_{1}(0)\ldots
\partial^{i_n}_{x_{n}} \overline{v}_{n}(0)\\
&&\qquad\times\partial_{y_1}^{I_1}\phi_1\left(e_1+\epsilon Y_1\right)
\ldots\partial_{y_M}^{I_M}\phi_M\left(e_M+\epsilon Y_M\right)
\\ &&\qquad\times\partial^{\mathfrak{I}_1}_{t_1}\psi_1(0)\ldots\partial^{\mathfrak{I}_l}_{t_l}\psi_l(0)
\\ &=&\lim_{\epsilon\searrow0}
\sum_{{|i|+|I|+|\mathfrak{I}|\le K}\atop{|\mathfrak{I}|\le|\overline{\mathfrak{I}}|}}
\epsilon^{|\overline{\mathfrak{I}}|-|\mathfrak{I}|}
\theta_{i,I,\mathfrak{I}}
\,\XX_1^{r_1}\ldots\XX_n^{r_n}\,
\partial^{i_1}_{x_{1}} \overline{v}_{1}(0)\ldots
\partial^{i_n}_{x_{n}} \overline{v}_{n}(0)\\
&&\qquad\times\epsilon^{|I_1|-s_1}\partial_{y_1}^{I_1}\phi_1\left(e_1+\epsilon Y_1\right)
\ldots\epsilon^{|I_M|-s_M}\partial_{y_M}^{I_M}\phi_M\left(e_M+\epsilon Y_M\right)
\\ &&\qquad\times\epsilon^{\mathfrak{I}_1-\alpha_1}\partial^{\mathfrak{I}_1}_{t_1}\psi_1(0)\ldots\epsilon^{\mathfrak{I}_l-\alpha_l}\partial^{\mathfrak{I}_l}_{t_l}\psi_l(0)
\\&=&
\sum_{{|i|+|I|+|\mathfrak{I}|\le K}\atop{|\mathfrak{I}| = |\overline{\mathfrak{I}}|}}
\lambda\,\tilde C_{i,I,\mathfrak{I}}\,\theta_{i,I,\mathfrak{I}}\,\XX_1^{r_1}\ldots\XX_n^{r_n}\,
\partial^{i_1}_{x_{1}} \overline{v}_{1}(0)\ldots
\partial^{i_n}_{x_{n}} \overline{v}_{n}(0)\\\
&&\qquad\times e_1^{I_1}\ldots e_M^{I_M}\,
\left(- e_1 \cdot Y_1 \right)_+^{s_1-|I_1|}
\ldots
\left(- e_M \cdot Y_M\right)_+^{s_M-|I_M|}\XT_1^{\mathfrak{I}_1}\ldots\XT_l^{\mathfrak{I}_l}
,\end{eqnarray*}
for a suitable~$\tilde C_{i,I,\mathfrak{I}}$. We stress that~$\tilde C_{i,I,\mathfrak{I}}\ne0$,
thanks also to~\eqref{alppoo},
applied here with~$\omega_j:=1$, $\tilde{\phi}_{\star,j}:=\phi_j$
and $e_j$ as in \eqref{freee2} for any $j\in\{1,\ldots,M\}$. 

Hence, recalling~\eqref{alp-0c},
\begin{equation}\label{GANAfaic}
\begin{split}
0&\,=\,\sum_{{|i|+|I|+|\mathfrak{I}|\le K}\atop{|\mathfrak{I}| = |\overline{\mathfrak{I}}|}}
C_{i,I,\mathfrak{I}}\,\theta_{i_1,\dots,i_n,I_1,\dots,I_M,\mathfrak{I}_1,\dots,\mathfrak{I}_l}\;
\XX_1^{i_1}\ldots\XX_n^{i_n}\\&\qquad\qquad\times
e_1^{I_1}\ldots e_M^{I_M}\,
\left(- e_1 \cdot Y_1 \right)_+^{s_1-|I_1|}
\ldots
\left(- e_M \cdot Y_M\right)_+^{s_M-|I_M|}
\XT_1^{\mathfrak{I}_1}\ldots\XT_l^{\mathfrak{I}_l}\\
&\,=\,
\left(- e_1 \cdot Y_1 \right)_+^{s_1}
\ldots
\left(- e_M \cdot Y_M\right)_+^{s_M}\\&\qquad\qquad\times
\sum_{{|i|+|I|+|\mathfrak{I}|\le K}\atop{|\mathfrak{I}| = |\overline{\mathfrak{I}}|}}
C_{i,I,\mathfrak{I}}\,\theta_{i,I,\mathfrak{I}}\;
\XX^i\,e^{I}\,
\left(- e_1 \cdot Y_1 \right)_+^{-|I_1|}
\ldots
\left(- e_M \cdot Y_M\right)_+^{-|I_M|}
\XT^{\mathfrak{I}}
,\end{split}\end{equation}
for a suitable~$C_{i,I,\mathfrak{I}}\ne0$.

We observe that the equality in~\eqref{GANAfaic}
is valid for any choice of the free parameters~$(\XX,Y,\XT)$
in an open subset of~$\mathbb{R}^{p_1+\dots+p_n}\times
\mathbb{R}^{m_1+\ldots+m_M}\times\mathbb{R}^l$,
as prescribed in~\eqref{FREExi}, \eqref{FREEmustar}
and~\eqref{eq:FREEy}.

Now, we take new free parameters $\XY_j$ with $\XY_j\in\mathbb{R}^{m_j}\setminus\{0\}$ for any $j=1,\ldots,M$, and perform in \eqref{GANAfaic} the same change of variables done in \eqref{COMPATI}, obtaining that
$$
0=\sum_{{|i|+|I|+|\mathfrak{I}|\le K}\atop{|\mathfrak{I}| = |\overline{\mathfrak{I}}|}}
C_{i,I,\mathfrak{I}}\,\theta_{i,I,\mathfrak{I}}\;
\XX^i\XY^I\XT^{\mathfrak{I}},
$$
for some $C_{i,I,\mathfrak{I}}\ne 0$. 

Hence, 
the second identity in~\eqref{ipop} is obtained as desired, and the proof of Lemma \ref{lemcin} in case~\ref{itm:case2}
is completed.
\end{proof}

\begin{proof}[Proof of \eqref{ipop}, case \ref{itm:case3}]
We divide the proof of case \ref{itm:case3} into two subcases, namely
either 
\begin{equation}\label{SC-va1}
{\mbox{there exists $h\in\{1,\ldots,l\}$ such that $\XC_h\ne 0$,}}\end{equation}
or
\begin{equation}\label{SC-va2}
{\mbox{$\XC_h=0\,$ for every $h\in\{1,\ldots,l\}$.}}\end{equation}
We start by dealing with the case in~\eqref{SC-va1}.
Up to relabeling and reordering the coefficients $\XC_h$, we can assume that
\begin{equation}\label{CNONU}
{\mbox{$\XC_1\ne 0$}}.\end{equation}
Also, thanks to the assumptions given in case~\ref{itm:case3}, we can suppose that
\begin{equation}\label{NEGB}
{\mbox{$\XB_M<0$}},\end{equation}
and, for any $j\in\{1,\ldots,M\}$, we consider $\lambda_{\star,j}$ and $\tilde{\phi}_{\star,j}$ as in \eqref{lambdastarj}. 
Then, we take~$\omega_j:=1$ and~$\phi_j$ as in~\eqref{autofun1},
so that~\eqref{REGSWYS-A} is satisfied.
In particular, here we have that
\begin{equation}\label{7yHSSIKnNSJS}
\lambda_{j}=
\lambda_{\star,j}\qquad{\mbox{ and }}\qquad\phi_j=\tilde\phi_{\star,j}.\end{equation}
We define
\begin{equation}\label{MAfghjkGGZ3} R:= 
\frac{ 1 }{|\XC_1|}\displaystyle\sum_{j=1}^{M-1}|\XB_j|\lambda_{\star,j}.\end{equation}
We notice that, in light of \eqref{CNONU},
the setting in~\eqref{MAfghjkGGZ3} is well-defined. 

Now, we fix a set of free parameters 
\begin{equation}\label{FREEmustar3}
\XT_{\star,1}\in(R+1,R+2),\dots\,\XT_{\star,l}\in(R+1,R+2).\end{equation}
Moreover, we define
\begin{equation}
\label{alp3}
\lambda_M\,:=\,\frac{1}{\XB_M}\left(
-\sum_{j=1}^{M-1}\XB_j\lambda_{\star,j}-\sum_{h=1}^{l}|\XC_h|\XT_{\star,h}\right).\end{equation}
We notice that \eqref{alp3} is well-defined thanks to \eqref{NEGB}.
{F}rom~\eqref{MAfghjkGGZ3} we deduce that
\begin{equation*}
\begin{split}
\sum_{h=1}^l&|\XC_h|\XT_{\star,h}+\sum_{j=1}^{M-1}\XB_j\lambda_{\star,j}\ge
|\XC_1|\XT_{\star,1}-\sum_{j=1}^{M-1}|\XB_j|\lambda_{\star,j} \\
&>|\XC_1|R-\sum_{j=1}^{M-1}|\XB_j|\lambda_{\star,j}=0.
\end{split}
\end{equation*}
Consequently, by~\eqref{NEGB} and~\eqref{alp3},
\begin{equation}
\label{alp-03}
\lambda_M>0.\end{equation}
Now, we define, for any $h\in\left\{1,\ldots,l\right\}$,
\begin{equation*}
\overline{\XC}_h:=
\begin{cases}
\displaystyle\frac{\XC_h}{\left|\XC_h\right|}\quad& \text{ if }\XC_h\neq 0 ,\\
1 \quad &\text{ if }\XC_h=0.
\end{cases}
\end{equation*}
We notice that
\begin{equation}\label{ahflaanhuf}
{\mbox{$\overline{\XC}_h\neq 0$ for all~$h\in\left\{1,\ldots,l\right\}$,}}\end{equation}
and
\begin{equation}\label{ahflaanhuf3}
{\XC}_h\overline{\XC}_h=|{\XC}_h|.\end{equation}
Moreover, 
we consider~$a_h\in(-2,0)$, for every~$h=1,\dots,l$,
to be chosen appropriately in what follows (see~\eqref{pata7UJ:AKK3}
for a precise choice).

Now, for every $h\in\{1,\ldots,l\}$, we define
\begin{equation}
\label{autofun3}
\psi_h(t_h):=E_{\alpha_h,1}(\overline{\XC}_h\XT_{\star,h}(t_h-a_h)^{\alpha_h}),
\end{equation}
where~$E_{\alpha_h,1}$ denotes the Mittag-Leffler function
with parameters $\alpha:=\alpha_h$ and $\beta:=1$ as defined 
in \eqref{Mittag}. By Lemma~\ref{MittagLEMMA},
we know that
\begin{equation}
\label{HFloj}
\begin{cases}
D^{\alpha_h}_{t_h,a_h}\psi_h(t_h)=\overline{\XC}_h\XT_{\star,h}\psi_h(t_h)\quad\text{in }\,(a_h,+\infty) ,\\
\psi_h(a_h)=1 ,\\
\partial^m_{t_h}\psi_h(a_h)=0\quad\text{for any}\,m=1,\dots,[\alpha_h],
\end{cases}
\end{equation}
and we consider again the extension $\psi^\star_h$ given in \eqref{starest}. 
By Lemma~A.3 in~\cite{CDV18}, we know that~\eqref{HFloj}
translates into
\begin{equation}\label{HFloj2}
D^{\alpha_h}_{t_h,-\infty}\psi_h^\star(t_h)
=\overline{\XC}_h\XT_{\star,h}\psi_h^\star(t_h)\,\text{ in every interval }\,I\Subset(a_h,+\infty).
\end{equation}
Now, we 
consider auxiliary parameters~$\XT_h$, $e_j$ and $Y_j$
as in \eqref{TGAdef}, \eqref{econ} and~\eqref{eq:FREEy}. 
Moreover, we introduce an additional set of free parameters 
\begin{equation}
\label{FREEXXXXX}
\XX=(\XX_1,\ldots,\XX_n)\in
\mathbb{R}^{p_1}\times\ldots\times\mathbb{R}^{p_n}.
\end{equation}
We
let~$\epsilon>0$, to be taken small
possibly depending on the 
free parameters.
We take~$\tau\in C^\infty(\mathbb{R}^{p_1+\ldots+p_n},[0+\infty))$ such that
\begin{equation}
\label{otau3}
\tau(x):=\begin{cases}
\exp\left({\,\XX\cdot x}\right)&\quad\text{if }\,x\in B_1^{p_1+\ldots+p_n} ,\\
0&\quad\text{if }\,x\in \mathbb{R}^{p_1+\ldots+p_n}\setminus B_2^{p_1+\ldots+p_n},
\end{cases}
\end{equation}
where 
$$ \XX\cdot x:=\sum_{j=1}^{n} \XX_i\cdot x_i$$
denotes the standard scalar product.

We notice that, for any $i\in\mathbb{N}^{p_1}\times\ldots\times\mathbb{N}^{p_n}$,
\begin{equation}
\label{resczzz}
\partial^{i}_{x}\tau(0)=\partial^{i_1}_{x_1}\ldots\partial^{i_n}_{x_n}\tau(0)=
\XX^{i_{11}}_{11}\ldots\XX^{i_{1p_1}}_{1p_1}\ldots\XX^{i_{n1}}_{n1}\ldots\XX^{i_{np_n}}_{np_n}
=\XX^i.
\end{equation}
We define
\begin{equation}
\label{svs3}
w\left(x,y,t\right):=\tau(x)\phi_1\left(y_1+e_1+\epsilon Y_1\right)\cdot
\ldots\cdot
\phi_M\left(y_M+e_M+\epsilon Y_M\right) 
\psi^\star_1(t_1)\cdot\ldots\cdot\psi^\star_l(t_l),
\end{equation}
where the setting in~\eqref{REGSWYS-A}
has also been exploited. 

We also notice that $w\in C\left(\mathbb{R}^N\right)\cap C_0\left(\mathbb{R}^{N-l}\right)\cap\mathcal{A}$. Moreover,
if
\begin{equation}\label{pata7UJ:AKK3}
a=(a_1,\dots,a_l):=\left(-\frac{\epsilon}{\XT_1},\dots,
-\frac{\epsilon}{\,\XT_l}\right)\in(-\infty,0)^l\end{equation}
and~$\left(x,y\right)$ is sufficiently close to the origin
and~$t\in(a_1,+\infty)\times\dots\times(a_l,+\infty)$, we have that
\begin{eqnarray*}&&
\Lambda_{-\infty} w\left(x,y,t\right)\\&=&
\left( \sum_{j=1}^{M} \XB_j (-\Delta)^{s_j}_{y_j}+
\sum_{h=1}^{l} \XC_h D^{\alpha_h}_{t_h,-\infty}\right) w\left(x,y,t\right)\\
&=&\sum_{j=1}^{M} \XB_j 
\tau(x)\phi_1\left(y_1+e_1+\epsilon Y_1\right)
\ldots\phi_{j-1}\left(y_{j-1}+e_{j-1}+\epsilon Y_{j-1}\right)(-\Delta)^{s_j}_{y_j}\phi_j\left(y_j+e_j+\epsilon Y_j\right)\\&&\qquad\times
\phi_{j+1}\left(y_{j+1}+e_{j+1}+\epsilon Y_{j+1}\right)
\ldots\phi_M\left(y_M+e_M+\epsilon Y_M\right)\psi^\star_1\left(t_1\right)
\ldots\psi^\star_l\left(t_l\right)\\
&&+\sum_{h=1}^{l} \XC_h\tau(x)\phi_1\left(y_1+e_1+\epsilon Y_1\right)\ldots\phi_M\left(y_M+e_M+\epsilon Y_M\right)\psi^\star_1\left(t_1\right)
\ldots\psi^\star_{h-1}\left(t_{h-1}\right)\\
&&\qquad\times D^{\alpha_h}_{t_h,-\infty}\psi^\star_h\left(t_h\right)
\psi^\star_{h+1}(t_{h+1})\ldots\psi^\star_l\left(t_l\right)\\
&=&\sum_{j=1}^{M} \XB_j \lambda_j
\tau(x)\phi_1\left(y_1+e_1+\epsilon Y_1\right)
\ldots\phi_M\left(y_M+e_M+\epsilon Y_M\right)\psi^\star_1(t_1)\ldots\psi^\star_l(t_l)
\\ 
&&+\sum_{h=1}^{l} \XC_h\overline{\XC}_h\XT_{\star,h}
\tau(x)\phi_1\left(y_1+e_1+\epsilon Y_1\right)
\ldots\phi_M\left(y_M+e_M+\epsilon Y_M\right)\psi^\star_1(t_1)\ldots\psi^\star_l(t_l) \\
&=&\left( \sum_{j=1}^M\XB_j\lambda_j+\sum_{h=1}^{l} \XC_h\overline{\XC}_h\XT_{\star,h}\right) w(x,y,t),
\end{eqnarray*}
thanks to \eqref{REGSWYS-A} and~\eqref{HFloj2}.

Consequently, making use of \eqref{7yHSSIKnNSJS}, \eqref{alp3} and~\eqref{ahflaanhuf3},
if~$(x,y)$ is near the origin and~$t\in(a_1,+\infty)\times\dots\times(a_l,+\infty)$,
we have that
$$
\Lambda_{-\infty} w\left(x,y,t\right)=
\left( \sum_{j=1}^M\XB_j\lambda_{\star,j}+\XB_M\lambda_M+\sum_{h=1}^{l} |\XC_h|\XT_{\star,h}\right) w(x,y,t)=0
.$$
This says that~$w\in\mathcal{H}$. Thus, in light of~\eqref{aza} we have that
\begin{equation*}
0=\theta\cdot\partial^K w\left(0\right)=
\sum_{\left|\iota\right|\leq K}
{\theta_{\iota}\partial^\iota w\left(0\right)}=\sum_{|i|+|I|+|\mathfrak{I}|\le K}
\theta_{i,I,\mathfrak{I}}\,\partial_{x}^i\partial_{y}^I\partial_{t}^{\mathfrak{I}}w\left(0\right) 
.\end{equation*}
{F}rom this and~\eqref{svs3}, we obtain that
\begin{equation}\label{eq:orthsi3}
0=\sum_{|i|+|I|+|\mathfrak{I}|\le K}
\theta_{i,I,\mathfrak{I}}\,
\partial^{i}_x\tau(0)
\partial^{I_1}_{y_{1}} {\phi}_{1}(e_1+\epsilon Y_1)\ldots
\partial^{I_M}_{y_{M}} {\phi}_{M}(e_M+\epsilon Y_M)
\,\partial_{t_1}^{\mathfrak{I}_1}\psi_1(0)\ldots\partial_{t_l}^{\mathfrak{I}_l}\psi_l(0).\end{equation}
Moreover, using~\eqref{autofun3} and \eqref{pata7UJ:AKK3},
it follows that, for every $\mathfrak{I}_h\in\mathbb{N}$
\begin{eqnarray*}
\partial^{\mathfrak{I}_h}_{t_h}\psi_h(0)&=&
\sum_{j=0}^{+\infty} {\frac{
\overline{\XC}_h^j\,\XT_{\star,h}^j\, \alpha_h j(\alpha_h j-1)\dots(\alpha_h j-
\mathfrak{I}_l+1)
(0-a_h)^{\alpha_h j-\mathfrak{I}_h}}{\Gamma\left(\alpha_h j+1\right)}}\\
&=& \sum_{j=0}^{+\infty} {\frac{
\overline{\XC}_h^j\,\XT_{\star,h}^j\, \alpha_h j(\alpha_h j-1)\dots(\alpha_h j-
\mathfrak{I}_h+1)
\,{\epsilon}^{\alpha_h j-\mathfrak{I}_h}}{\Gamma\left(\alpha_h j+1\right)\;
{\XT_h}^{\alpha_h j-\mathfrak{I}_h}}}\\
&=& \sum_{j=1}^{+\infty} {\frac{
\overline{\XC}_h^j\,\XT_{\star,h}^j\, \alpha_h j(\alpha_h j-1)\dots(\alpha_h j-
\mathfrak{I}_h+1)
\,{\epsilon}^{\alpha_h j-\mathfrak{I}_h}}{\Gamma\left(\alpha_h j+1\right)\;
{\XT_h}^{\alpha_h j-\mathfrak{I}_h}}}.
\end{eqnarray*}
Accordingly, recalling~\eqref{TGAdef}, we find that
\begin{equation}
\label{limmittl3}\begin{split}&
\lim_{\epsilon\searrow 0}\epsilon^{\mathfrak{I}_h-\alpha_h}
\partial^{\mathfrak{I}_h}_{t_h}\psi_h(0)
=\lim_{\epsilon\searrow 0}
\sum_{j=1}^{+\infty} {\frac{
\overline{\XC}_h^j\,\XT_{\star,h}^j\, \alpha_h j(\alpha_h j-1)\dots(\alpha_h j-
\mathfrak{I}_h+1)
\,{\epsilon}^{\alpha_h (j-1)}}{\Gamma\left(\alpha_h j+1\right)\;
{\XT_h}^{\alpha_h j-\mathfrak{I}_h}}}
\\&\qquad=
{\frac{ \overline{\XC}_h\,\XT_{\star,h}\,\alpha_h (\alpha_h -1)\dots(\alpha_h -\mathfrak{I}_h+1)
}{\Gamma\left(\alpha_h +1\right)\;
\XT_h^{\alpha_h -\mathfrak{I}_h}}} =
{\frac{ \overline{\XC}_h\,\XT_h^{\mathfrak{I}_h}\,\alpha_h (\alpha_h -1)\dots(\alpha_h -\mathfrak{I}_h+1)
}{\Gamma\left(\alpha_h +1\right)}}.
\end{split}\end{equation}
Also, recalling~\eqref{IBARRA}, we can write~\eqref{eq:orthsi3} as
\begin{equation}
\label{eq:ort:X3}
0=\sum_{{|i|+|I|+|\mathfrak{I}|\le K}\atop{|I|\le|\overline{I}|}}
\theta_{i,I,\mathfrak{I}}\,\partial^{i}_x\tau(0)
\partial^{I_1}_{y_{1}} {\phi}_{1}(e_1+\epsilon Y_1)\ldots
\partial^{I_M}_{y_{M}} {\phi}_{M}(e_M+\epsilon Y_M) 
\partial^{\mathfrak{I}_1}_{t_1}\psi_{1}(0)\ldots
\partial^{\mathfrak{I}_l}_{t_l}\psi_{l}(0).
\end{equation}
Moreover, we define
\begin{equation*}
\Xi:=\left|\overline{I}\right|-\sum_{j=1}^M {s_j}+|\mathfrak{I}|-\sum_{h=1}^l {\alpha_h}
.\end{equation*}
Then, we
multiply~\eqref{eq:ort:X3} by $\epsilon^{\Xi}\in(0,+\infty)$, and we
send~$\epsilon$ to zero. In this way, we obtain from~\eqref{alppoo},
\eqref{resczzz}, \eqref{limmittl3} and~\eqref{eq:ort:X3} that
\begin{eqnarray*}
0&=&\lim_{\epsilon\searrow0}
\epsilon^{\Xi}
\sum_{{|i|+|I|+|\mathfrak{I}|\le K}\atop{|I|\le|\overline{I}|}}
\theta_{i,I,\mathfrak{I}}\,\partial_{x}^i\tau(0)
\partial_{y_1}^{I_1}\phi_1\left(e_1+\epsilon Y_1\right)
\ldots\partial_{y_M}^{I_M}\phi_M\left(e_M+\epsilon Y_M\right)
\partial^{\mathfrak{I}_1}_{t_1}\psi_1(0)\ldots\partial^{\mathfrak{I}_l}_{t_l}\psi_l(0)
\\ &=&\lim_{\epsilon\searrow0}
\sum_{{|i|+|I|+|\mathfrak{I}|\le K}\atop{|I|\le|\overline{I}|}}
\epsilon^{|\overline{I}|-|I|}
\theta_{i,I,\mathfrak{I}}\,\partial_x^i\tau(0)
\epsilon^{|I_1|-s_1}\partial_{y_1}^{I_1}\phi_1\left(e_1+\epsilon Y_1\right)
\ldots\epsilon^{|I_M|-s_M}\partial_{y_M}^{I_M}\phi_M\left(e_M+\epsilon Y_M\right)
\\ &&\qquad\times\epsilon^{\mathfrak{I}_1-\alpha_1}\partial^{\mathfrak{I}_1}_{t_1}\psi_1(0)\ldots\epsilon^{\mathfrak{I}_l-\alpha_l}\partial^{\mathfrak{I}_l}_{t_l}\psi_l(0)
\\&=&
\sum_{{|i|+|I|+|\mathfrak{I}|\le K}\atop{|I| = |\overline{I}|}}
C_{i,I,\mathfrak{I}}\,\theta_{i,I,\mathfrak{I}}\,
\XX_1^{i_1}\ldots\XX_n^{i_n}\,
e_1^{I_1}\ldots e_M^{I_M}\,
\left(- e_1 \cdot Y_1 \right)_+^{s_1-|I_1|}
\ldots
\left(- e_M \cdot Y_M\right)_+^{s_M-|I_M|}\XT_1^{\mathfrak{I}_1}\ldots\XT_l^{\mathfrak{I}_l}
\\&=&
\left(- e_1 \cdot Y_1 \right)_+^{s_1}
\ldots
\left(- e_M \cdot Y_M\right)_+^{s_M}\\&&\qquad\qquad\times
\sum_{{|i|+|I|+|\mathfrak{I}|\le K}\atop{|I| = |\overline{I}|}}
C_{i,I,\mathfrak{I}}\,\theta_{i,I,\mathfrak{I}}\;
\XX^i\,e^{I}\,
\left(- e_1 \cdot Y_1 \right)_+^{-|I_1|}
\ldots
\left(- e_M \cdot Y_M\right)_+^{-|I_M|}
\XT^{\mathfrak{I}}
,\end{eqnarray*}
for a suitable~$C_{i,I,\mathfrak{I}}\ne0$.

We observe that the latter equality
is valid for any choice of the free parameters~$(\XX,Y,\XT)$
in an open subset of~$\mathbb{R}^{p_1+\ldots+p_n}\times\mathbb{R}^{m_1+\ldots+m_M}\times\mathbb{R}^l$,
as prescribed in~\eqref{eq:FREEy},~\eqref{FREEmustar3} and~\eqref{FREEXXXXX}. 

Now, we take new free parameters $\XY_j$ with $\XY_j\in\mathbb{R}^{m_j}\setminus\{0\}$ for any $j=1,\ldots,M$, and perform in the latter identity the same change of variables done in \eqref{COMPATI}, obtaining that
$$
0=\sum_{{|i|+|I|+|\mathfrak{I}|\le K}\atop{|I| = |\overline{I}|}}
C_{i,I,\mathfrak{I}}\,\theta_{i,I,\mathfrak{I}}\;
\XX^i\XY^I\XT^{\mathfrak{I}},
$$
for some $C_{i,I,\mathfrak{I}}\ne 0$.
This completes the proof of~\eqref{ipop} in case~\eqref{SC-va1} is satisfied.\medskip

Hence, we now focus on the case in which~\eqref{SC-va2} holds true.
For any $j\in\{1,\ldots,M\}$,
we consider the function~$\psi\in H^{s_j}(\mathbb{R}^{m_j})\cap C^{s_j}_0(\mathbb{R}^{m_j})$
constructed in Lemma \ref{hbump} and we call such function~$\phi_j$,
to make it explicit its dependence on~$j$ in this case.
We recall that
\begin{equation}
\label{harmony}
(-\Delta)^{s_j}_{y_j}\phi_j(y_j)=0\quad\text{ in }\,B_1^{m_j}.
\end{equation}
Also, for every $j\in\{1,\ldots,M\}$, we let $e_j$
and~$Y_j$ be as in \eqref{econ} and \eqref{eq:FREEy}.
Thanks to Lemma~\ref{hbump} and Remark~\ref{RUCAPSJD},
for any $I_j\in\mathbb{N}^{m_j}$, we know that
\begin{equation}
\label{psveind}
\lim_{\epsilon\searrow 0}\epsilon^{|I_j|-s_j}\partial_{y_j}^{I_j}\phi_j(e_j+\epsilon Y_j)=\kappa_{s_j}e_j^{I_j}(-e_j\cdot Y_j)_+^{s_j-|I_j|},
\end{equation}
for some $\kappa_{s_j}\ne 0$.

Moreover, for any $h=1,\ldots, l$, we define $\overline{\tau}_h(t_h)$ as
\begin{equation}
\label{tita}
\overline{\tau}_h(t_h):=\begin{cases}
e^{\XT_h t_h}\quad&\text{if}\quad t_h\in[-1,+\infty), \\
\displaystyle e^{-\XT_h}\sum_{i=0}^{k_h-1}\frac{\XT_h^i}{i!}(t_h+1)^i\quad&\text{if}\quad t_h\in(-\infty,-1),\end{cases}
\end{equation}
where
$\XT=(\XT_1,\ldots,\XT_l)\in(1,2)^l$ are free parameters.

We notice that, for any $h\in\{1,\ldots, l\}$ and $\mathfrak{I}_h\in\mathbb{N}$,
\begin{equation}
\label{rescttt}
\partial^{\mathfrak{I}_h}_{t_h}\overline{\tau}_h(0)=\XT_h^{\mathfrak{I}_h}.
\end{equation}
%%	
%%	Moreover, 
%%	% setting $f_h(t_h):=e^{\XT_h t_h}$, 
%%	%	a simple computation shows that
%%	%	%\begin{equation}
%%	%	%\label{tito}
%%	%	%D^{\alpha_h}_{t_h,-1}f_h(t_h)\leq \frac{2^{k_h}}{\Gamma(k_h-\alpha_h+1)}\left(t_h^{k_h-\alpha_h}(e^{2t_h}-1)+(t_h+1)^{k_h-\alpha_h}\right),
%%	%	%\end{equation}
%%	%	%$|D^{\alpha_h}_{t_h,-1}f_h(t_h)|$ is finite for every $t_h>-1$.
%%	%	%Hence, from \eqref{tita}, it follows that 
%%	%	$f_h\in C^{k_h,\alpha_h}_{-1}$,
%%	using Lemma A.3 in \cite{CDV18} with $u(t_h):=e^{\XT_h t_h}$, $a:=-\infty$, $b:=-1$ and $u_\star:=\overline{\tau}_h$, we have that 
%%	\begin{equation}
%%	\label{railey}
%%	\overline{\tau}_h\in C^{k_h,\alpha_h}_{-\infty}\cap C^{\infty}((-1,+\infty)),\,\text{ and}\quad\partial^{k_h}_{t_h}\overline{\tau}_h(t_h)=0\quad\text{in}\quad(-\infty,-1).
%%	\end{equation}
%%	
Now, we
define
\begin{equation}
\label{quiquoqua}
w(x,y,t):=\tau(x)\phi_1(y_1+e_1+\epsilon Y_1)\ldots\phi_M(y_M+e_M+\epsilon Y_M)\overline{\tau}_1(t_1)\ldots\overline{\tau}_l(t_l),
\end{equation}
where the setting of~\eqref{autofun1}, \eqref{otau3} and \eqref{tita} has been exploited.
We have that~$w\in\mathcal{A}$. Moreover, we point out that, since $\tau$, $\phi_1,\ldots,\phi_M$ are 
compactly supported, we have that~$w\in C(\mathbb{R}^N)\cap C_0(\mathbb{R}^{N-l})$, and, using Proposition \ref{maxhlapspan}, for any $j\in\{1,\ldots,M\}$, it holds that~$\phi_j\in C^{\infty}(\mathcal{N}_j)$ for some neighborhood~$\mathcal{N}_j$
of the origin in $\mathbb{R}^{m_j}$.
Hence $w\in C^{\infty}(\mathcal{N})$. 

Furthermore,
using \eqref{harmony}, when~$y$
is in a neighborhood of the origin we have that
\begin{equation*}
\begin{split}
\Lambda_{-\infty} w(x,y,t)&=\tau(x)\left(\XB_1(-\Delta)^{s_1}_{y_1}\phi_1(y_1+e_1+\epsilon Y_1)\right)\ldots\phi_M(y_M+e_M+\epsilon Y_M)\overline{\tau}_1(t_1)\ldots\overline{\tau}_l(t_l) \\
&+\ldots+\tau(x)\phi_1(y_1)\ldots\left(\XB_M(-\Delta)^{s_M}_{Y_M}\phi_M(y_M+
e_M+\epsilon Y_M)\right)\overline{\tau}_1(t_1)\ldots\overline{\tau}_l(t_l)=0,
\end{split}
\end{equation*}
which gives that~$w\in\mathcal{H}$.

In addition, using~\eqref{IBARRA}, \eqref{resczzz} and \eqref{rescttt}, we have that
\begin{eqnarray*}&&
0=\theta\cdot\partial^K w(0)=\sum_{|\iota|\leq K}\theta_{i,I,\mathfrak{I}}
\partial_x^i \partial_y^I \partial_t^{\mathfrak{I}} w(0)
=\sum_{{|\iota|\leq K}\atop{|I| \le |\overline{I}|}}\theta_{i,I,\mathfrak{I}}
\partial_x^i \partial_y^I \partial_t^{\mathfrak{I}} w(0)\\
&&\qquad\qquad=\sum_{{|\iota|\leq K}\atop{|I| \le |\overline{I}|}}\theta_{i,I,\mathfrak{I}}
\,\XX^i\partial_{y_1}^{I_1} \phi_1(e_1+\epsilon Y_1)\ldots
\partial_{y_M}^{I_M}\phi_M(e_M+\epsilon Y_M)\,\XT^\mathfrak{I}.
\end{eqnarray*}
Hence, we set
$$ \Xi:=|\overline{I}|-\sum_{j=1}^M s_j,$$
we multiply the latter identity by~$\epsilon^\Xi$
and we exploit~\eqref{psveind}. In this way, we find that
\begin{eqnarray*}
0 &=& \lim_{\epsilon\searrow0}
\sum_{{|\iota|\leq K}\atop{|I| \le |\overline{I}|}}\epsilon^{|\overline I|-|I|}\theta_{i,I,\mathfrak{I}}
\,\XX^i\,\epsilon^{|I_1|-s_1}\partial_{y_1}^{I_1} \phi_1(e_1+\epsilon Y_1)\ldots\epsilon^{|I_M|-s_M}
\partial_{y_M}^{I_M}\phi_M(e_M+\epsilon Y_M)\,\XT^\mathfrak{I}\\&=&
\sum_{{|\iota|\leq K}\atop{|I| = |\overline{I}|}}\theta_{i,I,\mathfrak{I}}\,\kappa_{s_j}\,
\,\XX^i\,e^{I}\,(-e_1\cdot Y_1)_+^{s_1-|I_1|}\ldots(-e_M\cdot Y_M)_+^{s_M-|I_M|}
\,\XT^\mathfrak{I}\\
&=&
(-e_1\cdot Y_1)_+^{s_1}\ldots(-e_M\cdot Y_M)_+^{s_M}
\sum_{{|\iota|\leq K}\atop{|I| = |\overline{I}|}}\theta_{i,I,\mathfrak{I}}\,\kappa_{s_j}\,
\,\XX^i\,e^{I}\,(-e_1\cdot Y_1)_+^{-|I_1|}\ldots(-e_M\cdot Y_M)_+^{-|I_M|}
\,\XT^\mathfrak{I},
\end{eqnarray*}
and consequently
\begin{equation}\label{7UJHAnAXbansdo}
0=\sum_{{|\iota|\leq K}\atop{|I| = |\overline{I}|}}\theta_{i,I,\mathfrak{I}}\,\kappa_{s_j}\,
\,\XX^i\,e^{I}\,(-e_1\cdot Y_1)_+^{-|I_1|}\ldots(-e_M\cdot Y_M)_+^{-|I_M|}
\,\XT^\mathfrak{I}.
\end{equation}
Now we take
free parameters $\XY\in\mathbb{R}^{m_1+\ldots+m_M}\setminus\{0\}$
and we perform the same change of variables in \eqref{COMPATI}.
In this way, we deduce from~\eqref{7UJHAnAXbansdo} that
\begin{eqnarray*}
0=\sum_{{|i|+|I|+|\mathfrak{I}|\le K}\atop{|I| = |\overline{I}|}}C_{i,I,\mathfrak{I}}\theta_{i,I,\mathfrak{I}}\XX^i\XY^I\XT^\mathfrak{I},
\end{eqnarray*}
for some $C_{i,I,\mathfrak{I}}\ne 0$, and the first claim in \eqref{ipop} is proved
in this case as well.
\end{proof}

\begin{proof}[Proof of \eqref{ipop}, case \ref{itm:case4}]
Notice that if there exists $j\in\{1,\ldots,M\}$ such that $\XB_j\ne 0$, we are in the setting of case \ref{itm:case3}.
Therefore, we assume that $\XB_j=0$ for every $j\in\{1,\ldots,M\}$.

We let $\psi$  
be the function constructed in Lemma~\ref{LF}.
For each $h\in\{1,\dots,l\}$,
we let~$\overline\psi_h(t_h):=\psi(t_h)$,
to make the dependence on $h$ clear and explicit.
Then, by formulas~\eqref{LAp1} and~\eqref{LAp2},
we know that
\begin{equation}
\label{cidivu}
D^{\alpha_h}_{t_h,0}\overline{\psi}_h(t_h)=0\quad\text{ in }\,(1,+\infty)
\end{equation} 
and, for every~$\ell\in\mathbb{N}$,
\begin{equation}\label{FOR2.50}
\lim_{\epsilon\searrow0} \epsilon^{\ell-\alpha_h}\partial^\ell_{t_h}
\overline{\psi}_h(1+\epsilon t_h)=\kappa_{h,\ell}\; t_h^{\alpha_h-\ell}
,\end{equation}
in the sense of distribution,
for some~$\kappa_{h,\ell}\ne0$.

Now,
we introduce a set of auxiliary parameters $\XT=(\XT_1,\ldots,\XT_l)
\in(1,2)^l$, and fix $\epsilon$ sufficiently small
possibly depending on the parameters. Then, we define
\begin{equation}
\label{keanzzz}
a=(a_1,\dots,a_l):=\left(-\frac{\epsilon}{\XT_1}-1,\ldots,-\frac{\epsilon}{\XT_l}-1\right)
\in(-2,0)^l,\end{equation}
and
\begin{equation}
\label{translation}
\psi_h(t_h):=\overline{\psi}_h(t_h-a_h).
\end{equation}
With a simple computation we have that the function
in~\eqref{translation} satisfies
\begin{equation}\label{cgraz}
D^{\alpha_h}_{t_h,a_h}\psi_h(t_h)=D^{\alpha_h}_{t_h,0}
\overline{\psi}_h(t_h-a_h)=0 \quad \text{in }\,(1+a_h,+\infty)=\left(-\frac{\epsilon}{\XT_h},+\infty\right),
\end{equation}
thanks to~\eqref{cidivu}.
In addition, for every~$\ell\in\mathbb{N}$, we have that~$\partial^\ell_{t_h}\psi_h(t_h)=
\partial^\ell_{t_h}\overline{\psi}_h(t_h-a_h)$, and therefore,
in light of~\eqref{FOR2.50} and~\eqref{keanzzz},
\begin{equation}\label{FOR2.5}
\epsilon^{\ell-\alpha_h}\partial^\ell_{t_h}\psi_h(0)=\epsilon^{\ell-\alpha_h}
\partial^\ell_{t_h}\overline{\psi}_h(-a_h)=
\epsilon^{\ell-\alpha_h}\partial^\ell_{t_h}\overline{\psi}_h\left(
1+\frac{\epsilon}{\XT_h}
\right)\to\kappa_{h,\ell}\;\XT_h^{\ell-\alpha_h},
\end{equation}
in the sense of distributions, as~$\epsilon\searrow0$.

Moreover, since for any $h=1,\ldots, l$, $\psi_h\in C^{k_h,\alpha_h}_{a_h}$, we can consider the extension
\begin{equation}\label{estar}
\psi^\star_h(t_h):=\begin{cases}
\psi_h(t_h)&\quad\text{ if }\,t_h\in[a_h,+\infty), \\
\displaystyle\sum_{i=0}^{k_h-1}\frac{\psi_h^{(i)}(a_h)}{i!}(t_h-a_h)^i&\quad\text{ if }\,t_h\in(-\infty,a_h),
\end{cases}
\end{equation}
and, using Lemma A.3 in \cite{CDV18} with $u:=\psi_h$, $a:=-\infty$, $b:=a_h$ and $u_\star:=\psi^\star_h$, we have that 
\begin{equation}
\label{tuck}
\psi^\star_h\in C^{k_h,\alpha_h}_{-\infty}\quad\text{and}\quad D^{\alpha_h}_{t_h,-\infty}\psi_h^\star=D^{\alpha_h}_{t_h,a_h}\psi_h=0\quad\text{in every interval }\,I\Subset\left(-\frac{\epsilon}{\XT_h},+\infty\right).
\end{equation}
Now, we fix a set of free parameters $\XY=\left(\XY_1,\ldots,\XY_M\right)\in\mathbb{R}^{m_1+\ldots+m_M}$, and consider~$\overline\tau\in C^\infty
(\mathbb{R}^{m_1+\ldots+m_M})$, such that
\begin{equation}
\label{otau5}
\overline{\tau}(y):=\begin{cases}
\exp\left({\XY\cdot y}\right)\quad&\text{if }\quad y\in B_1^{m_1+\ldots+m_M} ,\\
0\quad&\text{if }\quad y\in \mathbb{R}^{m_1+\ldots+m_M}\setminus B_2^{m_1+\ldots+m_M},\end{cases}
\end{equation}
where
$$
\XY\cdot y=\sum_{j=1}^M\XY_j\cdot y_j,
$$
denotes the standard scalar product.

We notice that, for any multi-index $I\in\mathbb{N}^{m_1+\ldots m_M}$,
\begin{equation}
\label{rescyyy}
\partial^{I}_{y}\overline{\tau}(0)=\XY^I,
\end{equation}
where the multi-index notation has been used.

Now, we define
\begin{equation}
\label{sinaloa}
w(x,y,t):=\tau(x)\overline{\tau}(y)\psi^\star_1(t_1)\ldots\psi^\star_l(t_l),
\end{equation}
where the setting in \eqref{otau3},
\eqref{estar} and \eqref{otau5} has been exploited.

Using~\eqref{tuck}, we have that, for any $(x,y)$
in a neighborhood of the origin and~$t\in\left(-\frac{\epsilon}{2},+\infty\right)^l$,
\begin{equation*}
\begin{split}
\Lambda_{-\infty} w(x,y,t)&=\tau(x)\overline{\tau}(y)\left(\XC_1 D^{\alpha_1}_{t_1,-\infty}\psi^\star_1(t_1)\right)
\ldots\psi^\star_l(t_l) \\
&+\ldots+\tau(x)\overline{\tau}(y)\psi^\star_1(t_1)\ldots\left(\XC_l D^{\alpha_l}_{t_l,-\infty}\psi^\star_l(t_l)\right)=0.
\end{split}
\end{equation*}
We have that~$w\in\mathcal{A}$, and, since $\tau$ and $\overline{\tau}$ are compactly supported, we
also have that $w\in C(\mathbb{R}^N)\cap C_0(\mathbb{R}^{N-l})$.
Also, from Lemma~\ref{LF}, for any $h\in\{1,\ldots,l\}$, we know that~$\overline{\psi}_h\in C^{\infty}((1,+\infty))$, hence $\psi_h \in C^{\infty}\left(\left(
-\frac{\epsilon}{\XT_h},+\infty\right)\right)$.
Thus, $w\in C^{\infty}(\mathcal{N})$, and consequently~$w\in\mathcal{H}$.

Recalling~\eqref{IBARRA2}, \eqref{resczzz}, and \eqref{rescyyy}, we have that
\begin{equation}\label{LATTER}\begin{split}&
0=\theta\cdot\partial^K w(0)=\sum_{|\iota|\leq K}\theta_{i,I,\mathfrak{I}}
\partial_x^i \partial_y^I \partial_t^{\mathfrak{I}} w(0)
=\sum_{{|\iota|\leq K}\atop{|\mathfrak{I}| \le |\overline{\mathfrak{I}}|}}\theta_{i,I,\mathfrak{I}}
\partial_x^i \partial_y^I \partial_t^{\mathfrak{I}} w(0)\\
&\qquad\qquad=\sum_{{|\iota|\leq K}\atop{|\mathfrak{I}| \le |\overline{\mathfrak{I}}|}}\theta_{i,I,\mathfrak{I}}
\,\XX^i\XY^I\partial_{t_1}^{\mathfrak{I}_1}\psi_1(0)\ldots\partial_{t_l}^{\mathfrak{I}_l}\psi_l(0).
\end{split}\end{equation}
Hence, we set
$$ \Xi:=|\overline{\mathfrak{I}}|-\sum_{h=1}^l \alpha_h
%%%%=\left(|\overline{\mathfrak{I}}|-|\mathfrak{I}|\right)+\left(|\mathfrak{I}|-\sum_{h=1}^l \alpha_h\right)
,$$
we multiply the identity in~\eqref{LATTER} by~$\epsilon^\Xi$
and we exploit~\eqref{FOR2.5}. In this way, we find that
\begin{eqnarray*}
0 &=& \lim_{\epsilon\searrow0}
\sum_{{|\iota|\leq K}\atop{|\mathfrak{I}| \le |\overline{\mathfrak{I}}|}}\epsilon^{|\overline{\mathfrak{I}}|-|\mathfrak{I}|}\theta_{i,I,\mathfrak{I}}
\,\XX^i\,\XY^I\,\epsilon^{\mathfrak{I}_1-\alpha_1}\partial_{t_1}^{\mathfrak{I}_1} {\psi}_1(0)\ldots\epsilon^{\mathfrak{I}_l-\alpha_l}\partial_{t_l}^{\mathfrak{I}_l} {\psi}_l(0)\\&=&
\sum_{{|\iota|\leq K}\atop{|\mathfrak{I}| = |\overline{\mathfrak{I}}|}}\theta_{i,I,\mathfrak{I}}\;
\kappa_{1,\mathfrak{I}_1}\dots
\kappa_{l,\mathfrak{I}_l}\;
\XX^i\,\XY^I\,\XT_1^{\mathfrak{I}_1-\alpha_1}\ldots\XT_l^{\mathfrak{I}_l-\alpha_l}\\
&=&
\XT_1^{-\alpha_1}\ldots\XT_l^{-\alpha_l}
\sum_{{|\iota|\leq K}\atop{|\mathfrak{I}| = |\overline{\mathfrak{I}}|}}\theta_{i,I,\mathfrak{I}}
\;\kappa_{1,\mathfrak{I}_1}\dots
\kappa_{l,\mathfrak{I}_l}\;
\XX^i\,\XY^I\,\XT_1^{\mathfrak{I}_1}\ldots\XT_l^{\mathfrak{I}_l},
\end{eqnarray*}
and consequently
\begin{equation*}
0=\sum_{{|\iota|\leq K}\atop{|\mathfrak{I}| = |\overline{\mathfrak{I}}|}}
\theta_{i,I,\mathfrak{I}}\;
\kappa_{1,\mathfrak{I}_1}\dots
\kappa_{l,\mathfrak{I}_l}\;
\,\XX^i\,\XY^I
\,\XT^\mathfrak{I},
\end{equation*}
and the second claim in \eqref{ipop} is proved
in this case as well.
\end{proof}

\section{Every function is locally $\Lambda_{-\infty}$-harmonic up to a small error,
and completion of the proof of Theorem \ref{theone2}}
\label{s:fourth}

In this section we complete the proof of Theorem \ref{theone2}
(which in turn implies Theorem \ref{theone} via Lemma~\ref{GRAT}).
By standard approximation arguments we can reduce to
the case in which $f$ is a polynomial, and hence, by the linearity
of the operator~$\Lambda_{-\infty}$, to the case in which is a monomial.
The details of the proof are therefore the following:

%	we recall a version of the Stone-Weierstrass Theorem in the $C^\ell$ setting
%	(in our context, this will be useful to reduce approximation
%	problems to the case of polynomials, and hence, by linearity, of monomials).
%	For a short proof of this result, see \cite{MR3626547}.
%	
%	\begin{lemma}
%	\label{stowei}
%	Fix $\ell\in\mathbb{N}$. For any $f\in C^\ell(\overline{B_1})$
%	and any $\epsilon>0$ there exists a polynomial $P$
%	such that $\|f-P\|_{C^\ell(\overline{B_1})}\leq\epsilon$.
%	\end{lemma}

\subsection{Proof of Theorem \ref{theone2} when $f$ is a monomial}\label{7UHASGBSBSBBSB}

We prove Theorem~\ref{theone2} under the initial assumption
that $f$ is a monomial, that is
\begin{equation}\label{iorade}
f\left(x,y,t\right)=\frac{x_1^{i_1}\ldots x_n^{i_n}y_1^{I_1}\ldots y_M^{I_M}t_1^{\mathfrak{I}_1}\ldots t_l^{\mathfrak{I}_l}}{\iota!}=\frac{x^iy^It^{\mathfrak{I}}}{\iota!}=
\frac{(x ,y,t)^{\iota}}{\iota!},
\end{equation}
where $\iota!:=i_1!\ldots i_n!I_1!\ldots I_M!\mathfrak{I}_1!\ldots\mathfrak{I}_l!$ and $I_\beta!:=I_{\beta,1}!\ldots I_{\beta,m_\beta}!$, $i_\chi!:=i_{\chi,1}!\ldots i_{\chi,p_\chi}!$ for all $\beta=1,\ldots M$. and $\chi=1,\ldots,n$. 
To this end, we argue as follows. 
We 
consider $\eta\in\left(0,1\right)$, to be taken
sufficiently small with respect to the
parameter~$\epsilon>0$ which has been fixed
in the statement of Theorem~\ref{theone2}, and we define
%\begin{equation}\label{le-a}
%a_1:=-2-\frac3{\eta^{\frac1{\alpha_1}}},\dots,
%a_l:=-2-\frac3{\eta^{\frac1{\alpha_l}}}\qquad a:=(a_1,\dots,a_l)
%\end{equation}
%and
$$ {\mathcal{T}}_\eta(x,y,t):=\left(
\eta^{\frac{1}{r_1}}x_1,\ldots,\eta^{\frac{1}{r_n}}x_n,\eta^{\frac{1}{2s_1}}y_1,\ldots,\eta^{\frac{1}{2s_M}}y_M,\eta^{\frac{1}{\alpha_1}}t_1,\ldots,\eta^{\frac{1}{\alpha_l}}t_l\right).$$
We also define
\begin{equation}\label{iorade2}
\gamma:=\sum_{j=1}^n {\frac{|i_j|}{r_j}}+\sum_{j=1}^M {\frac{\left|I_j\right|}{2s_j}}+\sum_{j=1}^l {\frac{\mathfrak{I}_j}{\alpha_j}}
,\end{equation}
and
\begin{equation}
\label{am}
\delta:=\min\left\{\frac{1}{r_1},\ldots,\frac{1}{r_n},\frac{1}{2s_1},\ldots,\frac{1}{2s_M},\frac{1}{\alpha_1},\ldots,\frac{1}{\alpha_l}\right\}.
\end{equation}
We also take $K_0\in\mathbb{N}$ such that
\begin{equation}
\label{blim}
K_0\geq\frac{\gamma+1}{\delta}
\end{equation}
and we let
\begin{equation}
\label{blo}
K:=K_0+\left|i\right|+\left|I\right|+\left|\mathfrak{I}\right|+\ell=
K_0+\left|\iota\right|+\ell,
\end{equation}
where $\ell$ is the fixed integer given in the statement of Theorem~\ref{theone2}. 

By Lemma \ref{lemcin}, there exist a neighborhood~$\mathcal{N}$
of the origin and a function~$w\in C\left(\mathbb{R}^N\right)\cap C_0\left(\mathbb{R}^{N-l}\right)\cap C^\infty\left(\mathcal{N}\right)\cap\mathcal{A}$ such that
\begin{equation}\label{7UJHAanna}
{\mbox{$\Lambda_{-\infty} w=0$ in $\mathcal{N}$, }}\end{equation}
and such that 
\begin{equation}\label{9IKHAHSBBNSBAA}
\begin{split}&
{\mbox{all the derivatives of $w$ in 0 up to order $K$ vanish,}}\\
&{\mbox{with the exception of $\partial^\iota w \left(0\right)$ which equals~$1$,}}\end{split}\end{equation}
being~$\iota$ as in~\eqref{iorade}. 
Recalling the definition of~$\mathcal{A}$ on page~\pageref{CALSASS},
we also know that
\begin{equation}\label{SPE12129dd}
{\mbox{$\partial^{k_h}_{t_h}w=0$
in~$(-\infty,a_h)$, }}\end{equation}
for suitable~$a_h\in(-2,0)$,
for all~$h\in\{1,\dots,l\}$.

In this way, setting
\begin{equation}
\label{quaqui}
g:=w-f,
\end{equation}
we deduce from~\eqref{9IKHAHSBBNSBAA} that
$$
\partial^\sigma g\left(0\right)=0\quad \text{ for any }
\sigma\in\mathbb{N}^N \text{ with }\left|\sigma\right|\leq K.
$$
Accordingly, in $\mathcal{N}$ we can write
\begin{equation}
\label{quaqua}
g\left(x,y,t\right)=\sum_{\left|\tau\right|\geq K+1} {x^{\tau_1}y^{\tau_2}t^{\tau_3}h_\tau\left(x,y,t\right)},
\end{equation}
for some $h_\tau$ smooth in $\mathcal{N}$, where the multi-index
notation $\tau=(\tau_1,\tau_2,\tau_3)$ has been used. 

Now, we define
\begin{equation}\label{JAncasxciNasd}
u\left(x,y,t\right):=\frac{1}{\eta^\gamma}
w\left({\mathcal{T}}_\eta(x,y,t)\right).
\end{equation}
In light of~\eqref{SPE12129dd}, we notice that
$\partial^{k_h}_{t_h}u=0$
in~$(-\infty,a_h/\eta^{\frac1{\alpha_h}})$, for all~$h\in\{1,\dots,l\}$,
and therefore~$u\in C\left(\mathbb{R}^N\right)\cap C_0(\mathbb{R}^{N-l})\cap C^\infty
\left({\mathcal{T}}_\eta(\mathcal{N})\right)\cap\mathcal{A}$. We also claim that
\begin{equation}\label{DEN}
{\mathcal{T}}_\eta([-1,1]^{N-l}\times(a_1,+\infty)\times\ldots\times(a_l,+\infty))\subseteq {\mathcal{N}}.
\end{equation}
To check this, let~$(x,y,t)\in[-1,1]^{N-l}\times(a_1+\infty)\times\ldots\times(a_l,+\infty)$
and~$(X,Y,T):={\mathcal{T}}_\eta(x,y,t)$.
Then, we have that~$|X_1|=\eta^{\frac{1}{r_1}}|x_1|\le \eta^{\frac{1}{r_1}}$,
~$|Y_1|=\eta^{\frac{1}{2s_1}}|y_1|\le \eta^{\frac{1}{2s_1}}$,
~$T_1=\eta^{\frac{1}{\alpha_1}}t_1> a_1\eta^{\frac{1}{\alpha_1}}>-1$,
provided~$\eta$ is small enough.
Repeating this argument, we obtain that, for small~$\eta$,
\begin{equation}\label{CLOS}
{\mbox{$(X,Y,T)$ is as
close to the origin as we wish.}}\end{equation} 
%On the other hand, recalling~\eqref{le-a},
%we see that
%$$ t_1-a_1=t_1+2+\frac3{\eta^{\frac1{\alpha_1}}}\in\left[ 1+\frac3{\eta^{\frac1{\alpha_1}}},3+\frac3{\eta^{\frac1{\alpha_1}}}\right].$$
%As a consequence,
%$$ \eta^{\frac{1}{\alpha_1}}(t_1-a_1)\in\left[{\eta^{\frac1{\alpha_1}}}+3,\,
%3{\eta^{\frac1{\alpha_1}}}+3
%\right].
%$$
%This and~\eqref{9iwkfhvksc283AN} give that
%$$ T_1=\eta^{\frac{1}{\alpha_1}}(t_1-a_1)+a^\star_1\in
%\left[{\eta^{\frac1{\alpha_1}}}+1,\,
%3{\eta^{\frac1{\alpha_1}}}+3
%\right]\subset[0,10],
%$$
%provided that~$\eta$ is small enough. Repeating this argument
%for~$T_2,\dots,T_l$, we find that
%$$ T\in [0,10]^l.$$
%Accordingly, recalling the notation in~\eqref{CALAK},
%we have that~$(O,T)\in{\mathcal{K}}\subseteq{\mathcal{N}}^\star$.
{F}rom \eqref{CLOS} and the fact that~${\mathcal{N}}$
is an open set, we infer that~$(X,Y,T)\in{\mathcal{N}}$,
and this proves~\eqref{DEN}.

Thanks to~\eqref{7UJHAanna} and~\eqref{DEN}, we have that, in~$B_1^{N-l}\times(-1,+\infty)^l$,
\begin{align*}
&\eta^{\gamma-1}\,\Lambda_{-\infty} u\left(x,y,t\right) \\
=\;&\sum_{j=1}^n {\XA_j\partial_{x_j}^{r_j}w
\left({\mathcal{T}}_\eta(x,y,t)\right)}
+\sum_{j=1}^M {\XB_j(-\Delta)^{s_j}_{y_j} w
\left({\mathcal{T}}_\eta(x,y,t)\right)}
+\sum_{j=1}^l {\XC_jD_{t_h,-\infty}^{\alpha_h}w
\left({\mathcal{T}}_\eta(x,y,t)\right) }\\=\;&\Lambda_{-\infty}w
\left({\mathcal{T}}_\eta(x,y,t)\right)\\
=\;&0.
\end{align*}
These observations establish that $u$ solves the equation in $B_1^{N-l}\times(-1+\infty)^l$ and $u$ vanishes when $|(x,y)|\ge R$,
for some~$R>1$, and thus the claims in~\eqref{MAIN EQ:2}
and~\eqref{ESTENSIONE}
are proved.

Now we prove that $u$ approximates $f$, as claimed in~\eqref{IAzofm:2}. For this, using the monomial structure of $f$ in~\eqref{iorade}
and the definition of $\gamma$ in~\eqref{iorade2}, we have,
in a multi-index notation,
\begin{equation}
\label{monsca}
\begin{split}
&\frac{1}{\eta^\gamma}f\left({\mathcal{T}}_\eta(x,y,t)\right)=\frac{1}{\eta^\gamma\,\iota!} \;
(\eta^{\frac{1}{r}}x)^i (\eta^{\frac{1}{2s}}y)^I \big(\eta^{\frac{1}{\alpha}}t\big)^\mathfrak{I}
=\frac{1}{\iota!} x^i y^I t^\mathfrak{I}=f(x,y,t).
\end{split}
\end{equation}
%$$
%\frac{1}{\eta^\gamma}f\left(\eta^{\frac{1}{r_1}}x_1,\ldots,\eta^{\frac{1}{r_n}}x_n,\eta^{\frac{1}{2s_1}}y_1,\ldots,\eta^{\frac{1}{2s_M}}y_M,\eta^{\frac{1}{\alpha_1}}t_1,\ldots,\eta^{\frac{1}{\alpha_l}}t_l\right)=f\left(x,y,t\right).
%$$
Consequently, by~\eqref{quaqui}, \eqref{quaqua}, \eqref{JAncasxciNasd} and~\eqref{monsca},
\begin{align*}
u\left(x,y,t\right)-f\left(x,y,t\right) 
&=\frac{1}{\eta^\gamma}g\left(\eta^{\frac{1}{r_1}}x_1,\ldots,\eta^{\frac{1}{r_n}}x_n,\eta^{\frac{1}{2s_1}}y_1,\ldots,\eta^{\frac{1}{2s_M}}y_M,\eta^{\frac{1}{\alpha_1}}t_1,\ldots,\eta^{\frac{1}{\alpha_l}}t_l\right) \\
&=\sum_{\left|\tau\right|\geq K+1} {\eta^{\left|\frac{\tau_1}{r}\right|+\left|\frac{\tau_2}{2s}\right|+\left|\frac{\tau_3}{\alpha}\right|-\gamma}x^{\tau_1}y^{\tau_2}t^{\tau_3}h_\tau\left(\eta^{\frac{1}{r}}x,\eta^{\frac{1}{2s}}y,\eta^{\frac{1}{\alpha}}t\right)}
,\end{align*}
where a multi-index notation has been used, e.g. we have written
$$ \frac{\tau_1}{r}:=\left( \frac{\tau_{1,1}}{r_1},\dots,\frac{\tau_{1,n}}{r_n}\right)
\in\mathbb{R}^n.$$
Therefore, for any multi-index $\beta=\left(\beta_1,\beta_2,\beta_3\right)$ with $\left|\beta\right|\leq \ell$,
\begin{equation}
\label{eq:quaquo}\begin{split}
&\partial^\beta\left(u\left(x,y,t\right)-f\left(x,y,t\right)\right)
\\=\,&\partial^{\beta_1}_{x}\partial^{\beta_2}_{y}\partial^{\beta_3}_{t}\left(u\left(x,y,t\right)-f\left(x,y,t\right)\right)
\\=\,&\sum_{\substack{\left|\beta'_1\right|+\left|\beta''_1\right|=\left|\beta_1\right| \\ \left|\beta'_2\right|+\left|\beta''_2\right|=\left|\beta_2\right| \\ \left|\beta'_3\right|+\left|\beta''_3\right|=\left|\beta_3\right| \\ \left|\tau\right|\geq K+1}} {c_{\tau,\beta}\;\eta^{\kappa_{\tau,\beta}}\; x^{\tau_1-\beta'_1}y^{\tau_2-\beta'_2}t^{\tau_3-\beta'_3}\partial_{x}^{\beta''_1}\partial_{y}^{\beta''_2}\partial_{t}^{\beta''_3}h_\tau\left(\eta^{\frac{1}{r}}x,\eta^{\frac{1}{2s}}y,\eta^{\frac{1}{\alpha}}t\right)},
\end{split}\end{equation}
where
$$
\kappa_{\tau,\beta}:=\left|\frac{\tau_1}{r}\right|+\left|\frac{\tau_2}{2s}\right|+\left|\frac{\tau_3}{\alpha}\right|-\gamma+\left|\frac{\beta''_1}{r}\right|+\left|\frac{\beta''_2}{2s}\right|+\left|\frac{\beta''_3}{\alpha}\right|,
$$
for suitable coefficients $c_{\tau,\beta}$. Thus, to complete the proof of~\eqref{IAzofm:2}, we need to show that this quantity is small if so is $\eta$.
To this aim, we use~\eqref{am}, \eqref{blim}
and~\eqref{blo} to see that
\begin{eqnarray*}\kappa_{\tau,\beta}
%	&=&
%	\left|\frac{\tau_1}{r}\right|
%	+\left|\frac{\tau_2}{2s}\right|
%	+\left|\frac{\tau_3}{\alpha}\right|
%	-\gamma+\left|\frac{\beta''_1}{r}\right|
%	+\left|\frac{\beta''_2}{2s}
%	\right|+\left|\frac{\beta''_3}{\alpha}\right|
&\geq&\left|\frac{\tau_1}{r}\right|+\left|\frac{\tau_2}{2s}\right|
+\left|\frac{\tau_3}{\alpha}\right|-\gamma \\
& \geq&\delta\left(\left|\tau_1\right|+\left|\tau_2\right|
+\left|\tau_3\right|\right)-\gamma\\&\geq& K\delta
-\gamma\\&\geq& K_0\delta-\gamma\\&\geq& 1.
\end{eqnarray*}
Consequently, we deduce from \eqref{eq:quaquo} that $\left\|u-f\right\|_{C^\ell\left(B_1^N\right)}\leq C\eta$ for some $C>0$. By choosing $\eta$ sufficiently small with respect to $\epsilon$, this implies the claim in~\eqref{IAzofm:2}. This completes the proof of
Theorem \ref{theone2} when $f$ is a monomial.

\subsection{Proof of Theorem \ref{theone2} when $f$ is a polynomial}\label{AUJNSLsxcrsd}

Now, we consider the case in which~$f$ is a polynomial. In this case, we can write $f$ as
$$
f\left(x,y,t\right)=\sum_{j=1}^J c_jf_j\left(x,y,t\right),
$$
where each $f_j$ is a monomial, $J\in\mathbb{N}$ and $c_j\in\mathbb{R}$ for all $j=1,\ldots, J$.

Let $$c:=\max_{j\in\{1,\dots,J\}} c_j.$$ Then, 
by the work done in Subsection~\ref{7UHASGBSBSBBSB},
we know that the claim 
in Theorem \ref{theone2} holds true for each $f_j$, and so we can find $a_j\in(-\infty,0)^l$, $u_j\in C^\infty\left(B_1^N\right)\cap C\left(\mathbb{R}^N\right)
\cap\mathcal{A}$
and $R_j>1$ such that $\Lambda_{-\infty} u_j=0$ in $B_1^{N-l}\times(-1,+\infty)^l$, $\left\|u_j-f_j\right\|_{C^\ell\left(B_1^N\right)}\leq\epsilon$ and $u_j=0$ if $|(x,y)|\ge R_j$. 

Hence, we set
$$
u\left(x,y,t\right):=\sum_{j=1}^J c_ju_j\left(x,y,t\right),
$$
and we see that
\begin{equation}
\label{pazxc}
\left\|u-f\right\|_{C^\ell\left(B_1^N\right)}\leq\sum_{j=1}^J {\left|c_j\right|\left\|u_j-f_j\right\|_{C^\ell\left(B_1^N\right)}}\leq cJ\epsilon.
\end{equation}
Also, $\Lambda_{-\infty} u=0$ thanks to the linearity of $\Lambda_{-\infty}$ in $B_1^{N-l}\times(-1,+\infty)^l$. Finally, $u$ is supported in $B_R^{N-l}$ in the variables $(x,y)$, being $$R:=\max_{j\in\{1,\dots,J\}} R_j.$$ This proves 
Theorem \ref{theone2} when $f$ is a polynomial (up to replacing $\epsilon$ with $cJ\epsilon$).

\subsection{Proof of Theorem \ref{theone2} for a general $f$}

Now we deal with the case of a general~$f$. To this end, we exploit
Lemma~2 in~\cite{MR3626547} and we see that
there exists a polynomial $\tilde{f}$ such that
\begin{equation}\label{6ungfbnreog}
\|f-\tilde{f}\|_{C^\ell(B_1^N)}\leq\epsilon.\end{equation}
Then, applying the result already proven in Subsection~\ref{AUJNSLsxcrsd}
to the polynomial $\tilde{f}$,
we can find $a\in(-\infty,0)^l$, $u\in C^\infty\left(B_1^N\right)
\cap C\left(\mathbb{R}^N\right)\cap\mathcal{A}$ and $R>1$ such that 
\begin{eqnarray*}
&& \Lambda_{-\infty} u=0 \quad{\mbox{ in }}B_1^{N-l}\times(-1,+\infty)^l, \\&&
u=0 \qquad\quad{\mbox{ if }}|(x,y)|\ge R,\\&&
\partial^{k_h}_{t_h} u=0 \quad\quad{\mbox{if }}t_h\in(-\infty,a_h),\qquad{\mbox{ for all }}h\in\{1,\dots,l\},
\\ {\mbox{and }}&&\|u-\tilde{f}\|_{C^\ell(B_1^N)}
\leq\epsilon.\end{eqnarray*}
Then, recalling~\eqref{6ungfbnreog}, we see that
$$ \|u-f\|_{C^\ell(B_1^N)}\leq\|u-\tilde{f}
\|_{C^\ell(B_1^N)}+\|f-\tilde{f}
\|_{C^\ell(B_1^N)}\leq2\epsilon.$$ Hence, the proof
of Theorem~\ref{theone2} is complete.
\qed
\begin{appendix}
\appendixpage
\noappendicestocpagenum
\addappheadtotoc

\chapter{Some applications}\label{APPEA}
In this appendix we give
some applications of the approximation results obtained and discussed
in this book. These examples exploit particular cases of the
operator $\Lambda_a$, namely, when $s\in(0,1)$ and $\Lambda_a$ is the fractional Laplacian $(-\Delta)^s$, or the fractional heat operator $\partial_t+(-\Delta)^s$.
Similar applications have been pointed
out in~\cite{MR3579567, AV, 2017arXiv170806300R}.
\begin{example} [The classical Harnack inequality\index{inequality!Harnack} fails for $s$-harmonic functions]{\rm
Harnack inequality, in its classical formulation, says that if $u$ is a nontrivial and nonnegative harmonic function in $B_1$ then, for any $0<r<1$, there exists $0<c=c(n,r)$ such that
\begin{equation}
\label{appendeq}
\sup_{B_r}u\leq c\inf_{B_r}u.
\end{equation}
The same result is not true for $s$-harmonic functions. To construct a counterexample, consider the smooth function $f(x)=|x|^2$, and, for a small $\epsilon>0$, let $v=v_\epsilon$ be the function
provided by Theorem \ref{theone}, where we choose $\ell=0$.
Notice that, if $x\in B_1\setminus B_{r/2}$,
$$
v(x)\geq f(x)-\left\|v-f\right\|_{L^{\infty}(B_1)}\geq \frac{r^2}{4}-\epsilon>\frac{r^2}{8},
$$
provided $\epsilon$ is small enough, while
$$
v(0)\leq f(0)+\left\|v-f\right\|_{L^{\infty}(B_1)}\leq\epsilon<\frac{r^2}{8}.
$$
Hence, we have that $v(0)<v(x)$ for any $x\in B_1\setminus B_{r/2}$, and therefore the minimum of $v$ in $B_1$ is attained in some point $\overline{x}\in \overline{B_{r/2}}$. Then, we define
$$
u(x):=v(x)-v(\overline{x}).
$$
Notice that $u$ is $s$-harmonic in $B_1$ since so does~$v$.
Also, $u\geq 0$ in $B_1$ by construction, and~$u>0$ in $B_1\setminus B_{r/2}$. On the other hand, since $\overline{x}\in B_r$
$$
\inf_{B_r} u=u(\overline{x})=0,
$$
which implies that $u$ cannot satisfies an inequality such as \eqref{appendeq}.

As a matter of fact, in the fractional case, the analogue of the Harnack inequality requires~$u$
to be nonnegative in the whole of~${\mathbb{R}}^n$, hence a ``global'' condition
is needed to obtain a ``local'' oscillation bound. See e.g.~\cite{MR2817382}
and the references therein for a complete discussion
of nonlocal Harnack inequalities.
}\end{example}
\begin{example}[A logistic equation with nonlocal interactions]{\rm
We consider the logistic equation taken into account in \cite{MR3579567}
\begin{equation}
\label{eq:logistic}
-(-\Delta)^s u+(\sigma-\mu u)u+\tau(J*u)=0,
\end{equation}
where $s\in(0,1]$, $\tau\in[0,+\infty)$ and $\sigma,\mu, J$ are nonnegative functions. The symbol $*$ denotes as usual the convolution product between $J$ and $u$. Moreover, the convolution kernel $J$ is assumed to be of unit mass and even, namely
$$
\int_{\mathbb{R}^n} J(x)dx=1
$$
and
$$
J(-x)=J(x)\quad\text{for any}\,x\in\mathbb{R}^n.
$$
In this framework, the solution $u$ denotes the density of a population living in some environment $\Omega\subseteq\mathbb{R}^n$, while the functions $\sigma$ and $\mu$ model respectively the growing and dying effects on the population.
The equation is driven by the fractional Laplacian that models a nonlocal dispersal strategy which has been observed experimentally in nature, and may be related to optimal hunting strategies and adaptation to the environment stimulated by natural selection.

We state here a result which translates the fact that a
population with a nonlocal strategy can plan the distribution of resources in a strategic region better than populations with a local one.

Namely, fixed $\Omega=B_1$, one can find a solution of a slightly perturbed version
of~\eqref{eq:logistic} in~$B_1$, compactly supported in a larger ball~$B_{R_\epsilon}$,
where~$\epsilon\in(0,1)$ denotes the perturbation.

The strategic plan consists in properly adjusting
the resources in $B_{R_{\epsilon}}\setminus B_1$ (that is,
a bounded region in which the equation is not satisfied) in order to consume almost
all the given resources in $B_1$.

The detailed statement goes as follows:

\begin{theorem}\label{VECH}
Let $s\in(0,1)$ and $\ell\in\mathbb{N}$, $\ell\geq 2$. Assume that
$\sigma,\mu\in C^{\ell}(\overline{B_2})$, with
$$
\inf_{\overline{B_2}}\mu>0,\qquad\inf_{\overline{B_2}}\sigma>0.
$$
Fixed $\epsilon\in(0,1)$, there exist
a nonnegative function $u_\epsilon$,
$R_\epsilon>2$ and $\sigma_\epsilon\in C^{\ell}(\overline{B_1})$ such that
\begin{equation*}
%% \label{eq:primeq}
(-\Delta)^s u_\epsilon=(\sigma_\epsilon-\mu u_\epsilon)u_\epsilon+
\tau(J*u_\epsilon)\quad\mbox{ in } B_1,
\end{equation*}
\begin{equation*}
%% \label{eq:secondeq}
u_\epsilon=0\quad\text{in } \mathbb{R}^n\setminus B_{R_{\epsilon}},
\end{equation*}
\begin{equation*}
%%\label{eq:terzeq}
\left\|\sigma_\epsilon-\sigma\right\|_{C^{\ell}(\overline{B_1})}\leq\epsilon,
\end{equation*}
\begin{equation*}
%%\label{eq:quarteq}
u_\epsilon\geq\mu^{-1}\sigma_\epsilon\quad\text{in } B_1.
\end{equation*}
\end{theorem}
}
\end{example}

\begin{example}{\rm Higher order nonlocal equations also appear naturally
in several contexts, see e.g.~\cite{MR3051400} for a nonlocal version
of the Cahn-Hilliard phase coexistence model.
Higher orders operators have also appeared in connection with logistic equations,
see e.g.~\cite{MR3578282}. In this spirit, we point out a version
of Theorem~\ref{VECH} which is new and relies
on Theorem~\ref{theone}. Its content is that nonlocal logistic equations
(of any order and with nonlocality given in either time or space, or both)
admits solutions which can arbitrarily well adapt to any given resource.
The precise statement is the following:

\begin{theorem}
Let $s\in(0,+\infty)$, $\alpha\in(0,+\infty)$ and $\ell\in\mathbb{N}$, $\ell\geq 2$.
Assume that 
\begin{equation}\label{EFVA}
{\mbox{either~$s\not\in{\mathbb{N}}$ or~$\alpha\not\in{\mathbb{N}}$.}}\end{equation}
Let~$\sigma,\mu\in C^{\ell}(\overline{B_2})$, with
\begin{equation}\label{INAsdH}
\inf_{\overline{B_1}}\mu>0.
\end{equation}
Fixed $\epsilon\in(0,1)$, there exist
a nonnegative function $u_\epsilon$,
$R_\epsilon>2$, $a_\epsilon<0$,
and $\sigma_\epsilon\in C^{\ell}(\overline{B_1})$ such that
\begin{equation}
\label{eq:primeq77}\begin{split}&
D^\alpha_{t,a_\epsilon} u_\epsilon(x,t)+
(-\Delta)^s u_\epsilon(x,t)=\Big(\sigma_\epsilon(x,t)-\mu (x,t)u_\epsilon(x,t)\Big)\,u_\epsilon(x,t)\\
&\qquad\mbox{ for all $(x,t)\in{\mathbb{R}}^p\times{\mathbb{R}}$ with $|(x,t)|<1$},\end{split}
\end{equation}
\begin{equation}
\label{eq:secondeq77}
u_\epsilon(x,t)=0\quad\text{if } |(x,t)|\ge R_{\epsilon},
\end{equation}
\begin{equation}
\label{eq:terzeq77}
\left\|\sigma_\epsilon-\sigma\right\|_{C^{\ell}(\overline{B_1})}\leq\epsilon,
\end{equation}
\begin{equation}
\label{eq:quarteq77}
u_\epsilon=\mu^{-1}\sigma_\epsilon\geq\mu^{-1}\sigma-\epsilon\quad\text{in } B_1.
\end{equation}
\end{theorem}

\begin{proof}
We use Theorem \ref{theone} in the case in which $\Lambda_a:=
D^\alpha_{t,a}+(-\Delta)^s$. Let~$f:=\sigma/\mu$. Then,
by Theorem~\ref{theone}, which can be exploited here in view of~\eqref{EFVA},
we obtain the existence of suitable~$u_\epsilon$,
$R_\epsilon>2$ and $a_\epsilon<0$ satisfying~\eqref{eq:secondeq77},
\begin{equation}
\label{eq:primeq7781}\begin{split}&
D^\alpha_{t,a_\epsilon} u_\epsilon(x,t)+
(-\Delta)^s u_\epsilon(x,t)=0\\
&\qquad\mbox{ for all $(x,t)\in{\mathbb{R}}^p\times{\mathbb{R}}$ with $|(x,t)|<1$},\end{split}
\end{equation}
and
\begin{equation}
\label{eq:primeq7789}
\left\|u_\epsilon-f\right\|_{C^{\ell}(\overline{B_1})}\leq\epsilon.\end{equation}
Then, we set~$\sigma_\epsilon:=\mu u_\epsilon$, and then, by~\eqref{eq:primeq7789},
\begin{equation}\label{7hS823} \begin{split}
\left\|\sigma_\epsilon-\sigma\right\|_{C^{\ell}(\overline{B_1})}\,&\leq C\,
\|\mu\|_{C^{\ell}(\overline{B_1})}\,
\left\|\frac{\sigma_\epsilon}{\mu}-\frac{\sigma}{\mu}\right\|_{C^{\ell}(\overline{B_1})}\\
&= C\,
\|\mu\|_{C^{\ell}(\overline{B_1})}\,
\left\|u_\epsilon-f\right\|_{C^{\ell}(\overline{B_1})}\\&\le C\,
\|\mu\|_{C^{\ell}(\overline{B_1})}\,\epsilon,
\end{split}\end{equation}
which gives~\eqref{eq:terzeq77}, up to renaming~$\epsilon$.

Moreover, if~$|(x,t)|<1$,
$$ (\sigma_\epsilon-\mu u_\epsilon)u_\epsilon=0=
D^\alpha_{t,a_\epsilon} u_\epsilon+
(-\Delta)^s u_\epsilon,$$
thanks to~\eqref{eq:primeq7781}, and this proves~\eqref{eq:primeq77}.

In addition, recalling~\eqref{7hS823} and~\eqref{INAsdH},
$$ u_\epsilon=\mu^{-1}\sigma_\epsilon\ge \mu^{-1}\sigma-
\frac{1}{\inf_{\overline{B_1}}\mu}\|\sigma-\sigma_\epsilon\|_{L^\infty(B_1)}
\ge\mu^{-1}\sigma-
\frac{\|\mu\|_{C^{\ell}(\overline{B_1})}\,\epsilon}{\inf_{\overline{B_1}}\mu},$$
in~$B_1$, which proves~\eqref{eq:quarteq77}, up to renaming~$\epsilon$.
\end{proof}
}\end{example}

\end{appendix}

\addcontentsline{toc}{chapter}{Index}
\printindex
\end{document}